\documentclass[12pt,reqno]{amsart}

\usepackage{etex}
\usepackage[utf8]{inputenc}
\usepackage{amsfonts}
\usepackage{amsmath}
\usepackage{amssymb}
\usepackage{amsthm}
\usepackage{mathrsfs}
\usepackage{stmaryrd}
\usepackage{color}
\usepackage[english]{babel}
\usepackage{fontenc}
\usepackage{url}
\usepackage{graphicx}
\usepackage{hyperref}
\usepackage{caption}
\usepackage{epstopdf}
\usepackage{hyphenat}
\usepackage{float}
\usepackage{indentfirst}
\usepackage[export]{adjustbox}
\usepackage{tikz}
\usepackage{color}
\usepackage[all]{xy}
\usetikzlibrary{matrix,arrows,decorations.pathmorphing}
\usepackage{bookmark}
\usepackage{array}
\newcolumntype{P}[1]{>{\centering\arraybackslash}p{#1}}
\usepackage{enumitem}
\usepackage{float}
\floatstyle{plaintop}
\restylefloat{table}
\usepackage{booktabs}

\allowdisplaybreaks

\newtheorem{theorem}{Theorem}[section]
\newtheorem{lemma}[theorem]{Lemma}
\newtheorem{proposition}[theorem]{Proposition}
\newtheorem{corollary}[theorem]{Corollary}
\newtheorem{lemdef}[theorem]{Lemma--Definition}

\theoremstyle{definition}
\newtheorem{remark}[theorem]{Remark}
\newtheorem{example}[theorem]{Example}
\newtheorem{definition}[theorem]{Definition}

\tikzset{
    partial ellipse/.style args={#1:#2:#3}{
        insert path={+ (#1:#3) arc (#1:#2:#3)}
    }
}

\DeclareMathOperator{\Sym}{Sym}
\DeclareMathOperator{\Bl}{Bl}
\DeclareMathOperator{\Spec}{Spec}
\DeclareMathOperator{\GL}{GL}
\DeclareMathOperator{\PGL}{PGL}
\DeclareMathOperator{\Aut}{Aut}
\DeclareMathOperator{\Laf}{Laf}
\DeclareMathOperator{\GP}{GP}
\DeclareMathOperator{\diag}{diag}
\DeclareMathOperator{\Proj}{Proj}
\DeclareMathOperator{\ConvHull}{ConvHull}
\DeclareMathOperator{\Supp}{Supp}
\DeclareMathOperator{\Hilb}{Hilb}
\DeclareMathOperator{\id}{id}
\DeclareMathOperator{\val}{val}

\def\bP{\mathbb{P}}
\def\XGP{\overline{\mathbf{X}}_{\mathrm{GP}}}

\makeatletter
\def\l@subsection{\@tocline{2}{0pt}{2.5pc}{5pc}{}}
\makeatother

\author{Luca Schaffler}
\address{Department of Mathematics, KTH Royal Institute of Technology, SE-100 44 Stockholm, Sweden}
\email{lucsch@math.kth.se}

\author{Jenia Tevelev}
\address{Department of Mathematics and Statistics, University of Massachusetts, Amherst, MA 01003. 
Laboratory of Algebraic Geometry and its Applications, HSE, Moscow, Russia.}
\email{tevelev@math.umass.edu}

\title[Compactifications of moduli of points and lines in the projective plane]{Compactifications of moduli of points and lines\\in the projective plane}

\begin{document}

\begin{abstract}
Projective duality identifies the moduli spaces $\mathbf{B}_n$ and $\mathbf{X}(3,n)$
parametrizing linearly general configurations of $n$  points in $\mathbb{P}^2$
and $n$ lines  in the dual $\mathbb{P}^2$, respectively.
The~space $\mathbf{X}(3,n)$ admits Kapranov's Chow quotient compactification $\overline{\mathbf{X}}(3,n)$, 
studied also by Lafforgue, Hacking, Keel, Tevelev, and Alexeev, which gives an
example of a KSBA moduli space of stable surfaces: it carries a family of 
certain reducible degenerations of~$\mathbb{P}^2$ with $n$ ``broken lines''. 
Gerritzen and Piwek proposed a dual perspective, a compact moduli space
parametrizing certain reducible degenerations of $\mathbb{P}^2$ with $n$ smooth points. We investigate the relation between these approaches,
answering a question of Kapranov from 2003.
\end{abstract}

\maketitle


\section{Introduction}

Projective duality associates to a configuration of $n$ points in $\mathbb{P}^2$  a configuration of $n$ lines in the dual $\mathbb{P}^2$, and conversely. Our motivating question is: how does projective duality between points and lines in the plane behave in families and under degenerations? More precisely, denote by $\mathbf{B}_n$ and $\mathbf{X}(3,n)$ the moduli spaces of $n$ general points in $\mathbb{P}^2$ and $n$ general lines in the dual $(\mathbb{P}^2)^\vee$ respectively. Projective duality gives an identification 
$$\mathbf{B}_n=\mathbf{X}(3,n),$$
but these moduli spaces carry distinct $\mathbb{P}^2$-bundles, one with $n$ points (sections) and another with $n$ lines (divisors). The problem is to find their \emph{modular compactifications}, i.e.~to extend the $\mathbb{P}^2$-bundles to universal families over the compactifications. For $\mathbf{X}(3,n)$, a modular compactification is given by Kapranov's Chow quotient  $\overline{\mathbf{X}}(3,n)$ \cite{Kap93a}, which carries a family of reducible degenerations of $\mathbb{P}^2$ with $n$ ``broken lines''. This is an example of the moduli space of stable pairs (varieties with divisors) of Koll\'ar, Shepherd-Barron, and Alexeev \cite{HKT06,Ale15}. There are no comparable moduli spaces for varieties with marked points.

We start by studying the Gerritzen--Piwek compactification $\overline{\mathbf{B}}_n$.
Although defined differently in \cite{GP91}, $\overline{\mathbf{B}}_n$ 
is isomorphic to the Zariski closure of the image of the map 
$$\mathbf{B}_n\hookrightarrow \prod_{\mathcal{Q}_n}\mathbb{P}^1,$$
where $\mathcal{Q}_n$ is the set of ordered quintuples of distinct labels
$v_1,\ldots,v_5\in \{1,\ldots,n\}$. Every map $\mathbf{B}_n\to\mathbb{P}^1$ sends $(p_1,\ldots,p_n)$ to the cross-ratio of four points obtained by projecting $p_{v_1},\ldots,p_{v_4}$ from $p_{v_5}\in\mathbb{P}^2$.
The analogous construction for points in $\mathbb{P}^1$ instead of $\mathbb{P}^2$ yields $\overline{\mathrm{M}}_{0,n}\cong\overline{\mathbf{X}}(2,n)$ \cite[Theorem~9.18]{HKT09}. 
It follows from Luxton's \cite[Proposition~3.2.8]{Lux08} (see the proof of Theorem~\ref{mainresultforapplications})
that a similar result holds for lines in $\mathbb{P}^2$ up to normalization:

\begin{theorem}\label{isomorphismbetweenkapranovcompandgpcomp}
There exists a finite birational morphism $\overline{\mathbf{X}}(3,n)\rightarrow\overline{\mathbf{B}}_n$. In particular, 
we have
$\overline{\mathbf{X}}(3,n)^\nu\cong\overline{\mathbf{B}}_n^\nu$.
 (For a reduced scheme $X$, we denote by $X^\nu$ its normalization.)
\end{theorem}

For $n=6$, Luxton proves that in fact $\overline{\mathbf{X}}(3,6)\cong\overline{\mathbf{B}}_6$, \cite[Theorem~4.1.15]{Lux08}.
We prove a very general result in combinatorial matroid theory (see \S\ref{aresultinmatroidtheory}) and combine it with the theory of Lafforgue's 
 varieties~$\overline{\mathbf{\Omega}}^Q$ (\cite{Laf03}, see~\S\ref{kapranovandlafforgue}), which generalize Kapranov's construction of 
$\overline{\mathbf{X}}(3,n)$ to arbitrary matroids, to prove perhaps the most general result in this direction.

\begin{theorem}
\label{mainresultforapplications}
Let $\overline{\mathbf{\Omega}}^Q$ be the Lafforgue variety corresponding to a matroid polytope $Q$ of rank $r$ and dimension $n-1$ such that $r\geq3$, or $r=2$ and $n\geq5$. The product of face maps
\begin{equation}
\label{sgsGSRH}
\overline{\mathbf{\Omega}}^Q\rightarrow\prod_{F\in\mathcal{F}}\overline{\mathbf{\Omega}}^F=\prod_{F\in\mathcal{F}}\mathbb{P}^1
\end{equation}
is a finite morphism and its restriction to the main stratum of $\overline{\mathbf{\Omega}}^Q$ is birational onto its image if the matroid is framed. Here $\mathcal{F}$ is the collection of faces of $Q$ in the boundary of the hypersimplex $\Delta(r,n)$ that are equivalent to $\Delta(2,4)$.
\end{theorem}

Next we investigate the claim from \cite{GP91} that $\overline{\mathbf{B}}_n$ is a modular compactification which parametrizes certain degenerations of $\mathbb{P}^2$ with configurations of $n$ points. As common in moduli theory, to construct the family over a compactification, the first step is to identify potential central fibers in one-parameter degenerations. Formally, if $\Bbbk$ is an algebraically closed field, $R=\Bbbk[[t]]$, and $K=\Bbbk((t))$ is the field of fractions of $R$, then given $\Spec(K)\rightarrow\mathbf{B}_n$, we need to construct a family over $\Spec(R)$. We recall the definition of \emph{Mustafin join} \cite{Mus78}.

\begin{definition}
For~a free $R$-submodule $L\subseteq K^3$ of rank $3$, let $\mathbb{P}(L)=\mathrm{Proj}(\mathrm{Sym}(L^\vee))$.
The~submodule $L$~is called a~\emph{lattice}.
Given a finite set of lattices 
$\Sigma=\{L_1,\ldots,L_m\}$, the~corresponding 
Mustafin join $\mathbb{P}(\Sigma)$
is the Zariski closure of the diagonal embedding $$\mathbb{P}_K^2\hookrightarrow \mathbb{P}(L_1)\times_R\ldots\times_R\mathbb{P}(L_m).$$
We denote by $\mathbb{P}(\Sigma)_\Bbbk$ the central fiber of $\mathbb{P}(\Sigma)$,
a degeneration of $\mathbb{P}_K^2$. The \emph{Bruhat--Tits building} $\mathfrak{B}_3^0$ is the set of equivalence classes of lattices modulo rescaling by 
$\lambda\in K^*$. The~Mustafin join $\mathbb{P}(\Sigma)$ only depends on the equivalence classes of lattices $[L_1],\ldots,[L_m]\in \mathfrak{B}_3^0$.
\end{definition}

Mustafin joins were used to determine special fibers of various families of compactified moduli spaces. One of the first examples is \cite{GP91}, where Gerritzen and Piwek  constructed a Mustafin join associated to a one-parameter degeneration of $n$ points in $\mathbb{P}^2$ (see below). Kapranov showed in \cite{Kap93b} that an analogous construction for $n$ points in $\mathbb{P}^1$ gives all fibers of the universal family over~$\overline{\mathrm{M}}_{0,n}$. The interpretation of the family of visible contours over $\overline{\mathbf{X}}(3,n)$ in terms of Mustafin joins was obtained in \cite{KT06}. This work was based on \cite{Fal01,Laf03}, where Mustafin joins were also studied. In \cite{CS10,CHSW11}, it was shown that the central fibers of arbitrary Mustafin joins are reduced and Cohen--Macaulay.

\begin{definition}[\cite{GP91}]
Let $a_1,\ldots,a_n\in\mathbb{P}^2(K)$
be points in general linear position.
We~call $\mathbf{a}=(a_1,\ldots,a_n)$ an~{\em arc}.
Given a lattice $L$, 
the valuative criterion of properness gives unique extensions $\overline{a}^L_1,\ldots,\overline{a}^L_n\in\mathbb{P}(L)(R)$,
which are sections of the $\bP^2$-bundle  $\mathbb{P}(L)\rightarrow\Spec(R)$.
We~say that $L$ is a {\em stable lattice with respect to an arc $\mathbf{a}$} if at 
least four of the limits $\overline{a}^L_1(0),\ldots,\overline{a}^L_n(0)$ in the central fiber $\mathbb{P}(L)_\Bbbk\subseteq\mathbb{P}(L)$ are in general linear position. 
Let 
$$\Sigma_\mathbf{a}\subseteq\mathfrak{B}_3^0$$ 
be the set of stable lattice classes with respect to $\mathbf{a}$. This is a finite set \cite[Lemma~5.19]{KT06}.
We call $\mathbb{P}(\Sigma_\mathbf{a})$ {\em the Mustafin join for an arc} to distinguish it from arbitrary Mustafin joins. 
The valuative criterion of properness gives $n$ sections 
$\overline{a}_1,\ldots,\overline{a}_n$ 
of $\mathbb{P}(\Sigma_\mathbf{a})\rightarrow\Spec(R)$.
\end{definition}

By Theorem~\ref{sectionssmoothcentralfiber},
the Mustafin join $\mathbb{P}(\Sigma_\mathbf{a})\rightarrow\Spec(R)$ for an arc
is smooth
along $n$ disjoint sections $\overline{a}_1,\ldots,\overline{a}_n$.
The central fiber $\mathbb{P}(\Sigma_\mathbf{a})_\Bbbk$ is reduced and Cohen--Macaulay by \cite{CS10}.

\begin{remark}
In \cite{GP91}, the authors claimed 
that $\overline{\mathbf{B}}_n$
is the moduli space for Mustafin joins for arcs, i.e.~there 
exists a family $\overline{\mathbf{F}}_n\rightarrow\overline{\mathbf{B}}_n$ 
with the following universal property. 
Let~$\mathbf{a}\colon\mathrm{Spec}(K)\rightarrow\mathbf{B}_n$ be an arc 
and let 
$\overline{\mathbf{a}}\colon\Spec(R)\to\overline{\mathbf{B}}_n$ 
be its unique extension.
Then~the pullback of $\overline{\mathbf{F}}_n$ via $\overline{\mathbf{a}}$ 
is isomorphic to $\mathbb{P}(\Sigma_\mathbf{a})$. However, we found a mistake in the argument and the statement is also wrong, as there exist arcs $\mathbf{a},\mathbf{b}\colon\mathrm{Spec}(K)\rightarrow\mathbf{B}_n$ with $\overline{\mathbf{a}}(0)=\overline{\mathbf{b}}(0)\in\overline{\mathbf{B}}_n$ such that the central fibers of $\mathbb{P}(\Sigma_\mathbf{a})$ and $\mathbb{P}(\Sigma_\mathbf{b})$ are not isomorphic (see Example~\ref{familyoverBn}). 
\end{remark}

In \S\ref{UniversalMustafin}, 
we construct the correct modular compactification of $\mathbf{B}_n$
and the universal Mustafin join of point configurations over it 
using multigraded Hilbert schemes \cite{HS04}.

\begin{theorem}
There exists a compactification $\mathbf{B}_n\subseteq\overline{\mathbf{X}}_{\mathrm{GP}}(3,n)$ with a proper flat family $\overline{\mathcal{M}}\to\overline{\mathbf{X}}_{\mathrm{GP}}(3,n)$
smooth along $n$ disjoint sections 
satisfying the following universal property.
If $\mathbf{a}\colon\mathrm{Spec}(K)\rightarrow\mathbf{B}_n$ and $\overline{\mathbf{a}}\colon\Spec(R)\to\overline{\mathbf{X}}_{\mathrm{GP}}(3,n)$ is the unique extension, then $\overline{\mathbf{a}}^*\overline{\mathcal{M}}$ is isomorphic to the Mustafin join $\mathbb{P}(\Sigma_\mathbf{a})$ with $n$ sections.
\end{theorem}

In Proposition~\ref{sRGsrgwrH}, we construct forgetful morphisms
$\overline{\mathbf{X}}_{\GP}(3,n)\rightarrow\overline{\mathbf{X}}_{\GP}(3,n-1)$, 
the analogues of forgetful morphisms $\overline{\mathrm{M}}_{0,n}\to\overline{\mathrm{M}}_{0,n-1}$.
We further show that there exists a birational morphism $\overline{\mathbf{X}}_{\mathrm{GP}}(3,n)\to\overline{\mathbf{B}}_n$. In Theorem~\ref{sgsrgRWGwrhr}, we show that 
$$\overline{\mathbf{X}}_{\mathrm{GP}}(3,5)\cong\overline{\mathbf{X}}(3,5)\cong\overline{\mathbf{B}}_5\cong\overline{\mathrm{M}}_{0,5}.$$

However,  $\overline{\mathbf{X}}_{\GP}(3,6)$ is different from Kapranov's Chow quotient compactification $\overline{\mathbf{X}}(3,6)$. In \cite{Lux08}, Luxton showed that $\overline{\mathbf{X}}(3,6)$ is a tropical compactification of $\mathbf{X}(3,6)$ inside the toric variety whose fan $\Sigma(3,6)$ is the tropical Grassmannian  of Speyer and Sturmfels \cite{SS04}. The $15$ singular points of $\overline{\mathbf{X}}(3,6)$ correspond to non simplicial \emph{bipyramid cones} in $\Sigma(3,6)$. 

\begin{theorem}
\label{zh}
$\overline{\mathbf{X}}_{\GP}(3,6)^\nu$ is the blow up of $\overline{\mathbf{X}}(3,6)$ given by the minimal refinement of 
$\Sigma(3,6)$ induced by splitting each bipyramid cone into $12$ subcones as illustrated in Table~\ref{subdivisionofthebipyramid}. The $6$-pointed degenerations of $\mathbb{P}^2$ parametrized by $\overline{\mathbf{X}}_{\GP}(3,6)^\nu$ are listed in Tables~\ref{tbl:nonplanarGPdegenerations} -- \ref{tbl:furtherGPdegenerations}.
\end{theorem}

The main difficulty in dealing with $\overline{\mathbf{X}}(3,n)$
is that its singularities are hard to control.
In~Theorem~\ref{reductionmorphismisoonU123}, we show that
$\overline{\mathbf{X}}(3,n)^\nu$ has a natural open locus 
$\mathbf{U}(3,n)$ with toroidal singularities.
Moreover, its preimage $\mathbf{U}_{\GP}(3,n)$
in $\overline{\mathbf{X}}_{\GP}(3,n)^\nu$ is also toroidal.
We call $\mathbf{U}(3,n)$ and $\mathbf{U}_{\GP}(3,n)$
{\em planar loci} because they correspond to limits of arcs
with all stable lattices in one apartment of the Bruhat--Tits building. Planar loci are covered by charts
which are isomorphic to (non-toric) open subsets in toric varieties
with fans that we denote by ${\mathcal{Q}}_m$ and 
$\widetilde{\mathcal{Q}}_m$ (here $n=m+3$).
While ${\mathcal{Q}}_m$ is nothing but the 
\emph{quotient fan} of Kapranov--Sturmfels--Zelevinsky 
\cite{KSZ91}, see \S\ref{sGsgsGsehsRH},
the fan $\widetilde{\mathcal{Q}}_m$ is new, see
\S\ref{asrgasrhadrhd}.

For $n=6$, $\mathbf{U}(3,6)$ nearly describes all of $\overline{\mathbf{X}}(3,6)$ (including its singular locus),
which allows us to immediately describe $\mathbf{U}_{\GP}(3,6)$ by toric geometry.
This leaves a few closed subsets,
which we deal with an ad hoc analysis.


\subsection*{Acknowledgements}
We would like to thank Valery Alexeev, Dustin Cartwright,  
Patricio Gallardo, and Sean Keel for helpful conversations related to this work. In particular, we would like to thank Valery Alexeev for allowing us to reproduce in the current paper his pictures \cite[Figures~5.11, 5.12, and 5.13]{Ale15}. Finally, we thank the anonymous referee for the valuable comments and feedback. At KTH, the first author was supported by the KTH grant Verg Foundation. The second author was supported by the NSF grant DMS-1701704, Simons Fellowship, and by the HSE University Basic Research Program and Russian Academic Excellence Project 5-100.

\tableofcontents


\section{A result in matroid theory}
\label{aresultinmatroidtheory}

This section is elementary and we only
use standard terminology and facts in matroid theory that can be found in \cite{Oxl11}. 
We use the notation $[n]=\{1,\ldots,n\}$ throughout the paper.
Recall that an example of a matroid on $[n]$ is a configuration of $n$ hyperplanes in $\bP^{r-1}$.
Matroids of this form are called \emph{realizable}.

\begin{definition}
Let $M$ be a matroid on $[n]$ and let $i_0\in[n]$.
\begin{enumerate}
\item The \emph{contraction} $M/i_0$ is the matroid on $[n]\setminus\{i_0\}$ such that $I\subseteq[n]\setminus\{i_0\}$ is independent provided $I\cup\{i_0\}$ is independent for $M$;
\item The \emph{deletion} $M\setminus i_0$ is the matroid on $[n]\setminus\{i_0\}$ such that $I\subseteq[n]\setminus\{i_0\}$ is independent provided $I$ is independent for $M$.
\item $M$ is called \emph{disconnected} if there exists a partition $[n]=E_1\amalg E_2$ and matroids $M_1, M_2$ on $E_1,E_2$ respectively such that $I$ is an independent set for $M$ if and only if $I=I_1\amalg I_2$, where $I_i$ is an independent set for $M_i$, $i=1,2$. A matroid is \emph{connected} provided it is not disconnected.
\end{enumerate}
For realizable matroids, deletion corresponds to removing a hyperplane from the arrangement and contraction corresponds to intersecting the arrangement with one of its hyperplanes.
\end{definition}

\begin{definition}[\cite{Tev07}]
Let $M$ and $M'$ be two matroids on $[n]$. Then we say that $M$ is \emph{more constrained than} $M'$ provided every independent set for $M$ is independent for $M'$, but there exists an independent set for $M'$ that is dependent for $M$.
\end{definition}

 The goal of this section is to prove the following:

\begin{theorem}
\label{mainresultinmatroidtheory}
Let $M$ and $M'$ be two matroids on $[n]$ of the same rank $r$.
Suppose that $r\geq3$ or $r=2$ and $n\geq5$. Then the following statement (CD) holds: \begin{itemize}
    \item[(CD)] If $M$ is connected and more constrained than $M'$ then there exists $i_0\in[n]$ such that either the contraction $M/i_0$ is connected and more constrained than $M'/i_0$, or the deletion $M\setminus i_0$ is connected and more constrained than $M'\setminus i_0$.
\end{itemize}
\end{theorem}

Before we plunge into the proof,
we recast the theorem in 
the language of matroid polytopes.
Let $\{e_1,\ldots,e_n\}$ be the basis of~$\mathbb{R}^n$. The \emph{hypersimplex} $\Delta(r,n)$ is the lattice polytope
\begin{equation*}
\Delta(r,n)=\ConvHull\left(\sum_{i\in I}e_i\mid I\subseteq[n]~\textrm{and}~|I|=r\right)\subseteq\mathbb{R}^n.
\end{equation*}

A \emph{matroid polytope} $P$ in $\mathbb{R}^n$ is a lattice polytope whose vertices are $\sum_{i\in B}e_i$, where $B$ varies among the bases of a matroid on $[n]$. 
If the matroid has rank $r$, then the vertices of $P$ are also vertices of the hypersimplex $\Delta(r,n)$. The hypersimplex is the matroid polytope of the uniform matroid. Recall that $M$ is connected if and only if $P$ has  dimension $n-1$, see \cite[Theorem 1.11]{Laf03} or \cite[Proposition 2.4]{FS05}.
If a matroid $M$ is more constrained than a matroid $M'$, then
we have $P\subsetneq P'$ for the corresponding matroid polytopes.
The converse holds if we assume $M$ and $M'$ have the same rank.
So we have

\begin{corollary}
\label{applicationtheoremmatroidstomatroidpolytopes}
Let $P,P'\subseteq\Delta(r,n)$ be matroid polytopes of rank $r$ and dimension $n-1$ such that $r\geq3$, or $r=2$ and $n\geq5$. Assume $P\subsetneq P'$. Then there exists a facet $F\subseteq\Delta(r,n)$ such that
$P|_F$ and $P'|_F$ are facets of $P$ and $P'$ respectively, and
$P|_F\subsetneq P'|_F$.
\end{corollary}

\begin{proof}
Let $M,M'$ be the matroids corresponding to $P,P'$. By Theorem~\ref{mainresultinmatroidtheory}, there exists $i_0\in[n]$ such that either $M/i_0$ is connected and more constrained than $M'/i_0$, or $M\setminus i_0$ is connected and more constrained than $M'\setminus i_0$.
In the former case, let $F$ be the facet of $\Delta(r,n)$ given by  the hyperplane $x_{i_0}=1$. Then $P|_F$ and $P'|_F$ are the matroid polytopes corresponding to $M/i_0$ and $M'/i_0$ respectively. Therefore, $P|_F\subsetneq P'|_F$. In the latter case we consider the facet of $\Delta(r,n)$ given by  the hyperplane $x_{i_0}=0$.
\end{proof}

Next we consider matroid polytope subdivisions of a matroid polytope.

\begin{corollary}
\label{extensiontomatroidpolytopesubdivisions}
Let $Q\subseteq\Delta(r,n)$ be a matroid polytope of rank $r$ and dimension $n-1$ such that $r\geq3$, or $r=2$ and $n\geq5$. Let $\mathscr{P},\mathscr{P}'$ be two subdivisions of $Q$ into matroid polytopes such that $\mathscr{P}'$ is coarser than $\mathscr{P}$. Then there exists a facet $F\subseteq\Delta(r,n)$ such that $E=Q|_F$ is a facet of $Q$ and $\mathscr{P}'|_E$ is coarser than $\mathscr{P}|_E$.
\end{corollary}

\begin{proof}
Since $\mathscr{P}'$ is coarser than $\mathscr{P}$, we can find maximal dimensional polytopes $P\in\mathscr{P}$ and $P'\in\mathscr{P}'$ such that $P\subsetneq P'$. By Corollary~\ref{applicationtheoremmatroidstomatroidpolytopes} there exists a facet $F\subseteq\Delta(r,n)$ such that $P|_F$ and $P'|_F$ are facets of $P$ and $P'$ respectively, and $P|_F\subsetneq P'|_F$. If we set $E=Q|_F$, it follows that the subdivision $\mathscr{P}'|_E$ is coarser than $\mathscr{P}|_E$.
\end{proof}

The rest of this section is dedicated to the proof
of Theorem~\ref{mainresultinmatroidtheory}.

\begin{lemma}
\label{moreconstrainednesswithcontractionanddeletion}
Let $M$ and $M'$ be two matroids on $[n]$. Assume that $M$ is more constrained than $M'$ and let $I'$ be an independent set for $M'$ which is dependent for $M$. Then 
\begin{itemize}
\item[(a)] If $i_0\in I'$, then $M/i_0$ is more constrained than $M'/i_0$;
\item[(b)] If $j_0\in[n]\setminus I'$, then $M\setminus j_0$ is more constrained than $M'\setminus j_0$.
\end{itemize}
\end{lemma}

\begin{proof}
If $I$ is independent for $M/i_0$, then $I\cup\{i_0\}$ is independent for $M$ by definition. Hence, $I\cup\{i_0\}$ is independent for $M'$, implying that $I$ is independent for $M'/i_0$. Moreover, $I'\setminus\{i_0\}$ is independent for $M'/i_0$, but dependent for $M/i_0$.
If $I$ is independent for $M\setminus j_0$, then $I$ is independent for $M$.
Hence, $I$ is independent for $M'$, implying that $I$ is independent for $M'\setminus j_0$ because $j_0\notin I$. Moreover, $I'$ is independent for $M'\setminus j_0$, but dependent for $M\setminus j_0$.
\end{proof}

\begin{proposition}
\label{casewithtwodistinctindependentsets}
Let $M$ and $M'$ be two matroids on $[n]$ of the same rank. 
Assume that  at least two distinct independent sets $I_1',I_2'$ for $M'$ are dependent for $M$.
Then (CD) holds.
\end{proposition}

\begin{proof}
Up to switching $I_1'$ with $I_2'$, let $e\in I_1'\setminus I_2'$. If $M/e$ is connected, then we are done by Lemma~\ref{moreconstrainednesswithcontractionanddeletion}~(a). Otherwise, by \cite[Theorem 4.3.1]{Oxl11} we have that $M\setminus e$ is connected, and by Lemma~\ref{moreconstrainednesswithcontractionanddeletion}~(b) we have that $M\setminus e$ is more constrained than $M'\setminus e$.
\end{proof}

\begin{proposition}
\label{casewithnosizetwococircuit}
Let $M$ and $M'$ be two matroids on $[n]$ of the same rank. 
Assume that $M$ has no size two cocircuit. 
Then (CD) holds. 
\end{proposition}

\begin{proof}
By Proposition~\ref{casewithtwodistinctindependentsets} we can assume there is a unique independent set $I'$ for $M'$ that is dependent for $M$. Observe that $I'$ must be a basis of $M'$, hence $|I'|=r$, where $r$ is the common rank of $M$ and $M'$. From the connectedness of $M$ we have that $M$ has no coloops: $M$ connected implies that the dual matroid $M^*$ is also connected by \cite[Corollary 4.2.5]{Oxl11}, and a connected matroid has no loops, so $M^*$ has no loops, implying that $M$ has no coloops. Moreover, by hypothesis $M$ has no size two cocircuits, hence $M$ is a cosimple matroid (by definition, a matroid is cosimple provided it has no coloops and no size two cocircuits). This implies that the dual matroid $M^*$ is simple and connected. Hence there are at least $r+1$ elements $e$ such that $M^*/e$ is connected by \cite[\S4.3, Exercise 10 (e)]{Oxl11}. In particular, we can find one of such $e$ in the complement of $I'$. Since $M^*/e$ is connected, also $(M^*/e)^*=M\setminus e$ is connected (here we used \cite[\S3.1, Exercise 1 (b)]{Oxl11} and $(M^*)^*=M$), and by Lemma~\ref{moreconstrainednesswithcontractionanddeletion}~(b) we have that $M\setminus e$ is more constrained than $M'\setminus e$.
\end{proof}

\begin{lemma}
\label{moreconstrainednessanddualmatroids}
Let $M$ and $M'$ be two matroids on $[n]$ of the same rank. Assume that $M$ is more constrained than $M'$. Then the dual matroid $M^*$ is more constrained than $(M')^*$.
\end{lemma}

\begin{proof}
Let $I$ be an independent set for $M^*$. Then there exists a basis $B$ for $M$ disjoint from~$I$. But $B$ is also a basis for $M'$ because $M$ is more constrained than $M'$ and they have the same rank. So $I$ is also independent for $M'$. Let $I'$ be an independent set for $M'$ that is dependent for $M$. Let $B'$ be a basis for $M'$ containing $I'$. Then $B'$ is also dependent for $M$. It follows that $(B')^c$ is a basis for $(M')^*$ and $(B')^c$ is dependent for $M^*$, proving that $M^*$ is more constrained than $(M')^*$.
\end{proof}

\begin{proof}[Proof of Theorem~\ref{mainresultinmatroidtheory}]
Suppose first that $r\ge 3$.
By Lemma~\ref{moreconstrainednessanddualmatroids}, we have that the dual matroid $M^*$ is connected and more constrained than $(M')^*$. If there exists $i_0\in[n]$ such that $M^*/i_0$ is connected and more constrained than $(M')^*/i_0$, or $M^*\setminus i_0$ is connected and more constrained than $(M')^*\setminus i_0$, then $M\setminus i_0$ is connected and more constrained than $M'\setminus i_0$, or $M/i_0$ is connected and more constrained than $M'/i_0$ respectively. Therefore, by Proposition~\ref{casewithnosizetwococircuit} applied to $M^*$ and $(M')^*$, we can assume that $M^*$ has a size two cocircuit $C$, which is a size two circuit of $M$.

By Proposition~\ref{casewithtwodistinctindependentsets} we can assume that there exists a unique independent set for $M'$ that is dependent for $M$. Recall that $I'$ must be a basis of $M'$, hence $|I'|=r$ and in particular $I'\neq C$ (here is where we used the assumption that $r\geq3$). Therefore, $C$ must be dependent for $M'$, otherwise we would contradict the uniqueness of $I'$. In particular, $C\not\subseteq I'$, hence we can find an element $c\in C\setminus I'$. By Lemma~\ref{moreconstrainednesswithcontractionanddeletion}~(b), we only need to show that $M\setminus c$ is connected. But this follows from Lemma~\ref{deleteelementparallelandconnectedness} below.

Finally, suppose $r=2$ and $n\geq5$.
First of all observe that $M$ (and similarly $M'$) is realized by a point arrangement in $\mathbb{P}^1$ (over the complex numbers). Let $\mathcal{H}=\{p_1,\ldots,p_n\}$, and declare $p_i=p_j$ for $i\neq i$ provided $\{i,j\}$ is a circuit of $M$. Then the matroid associated to $\mathcal{H}$ is isomorphic to $M$. Similarly, define $\mathcal{H}'=\{p_1',\ldots,p_n'\}$.

For point configurations on $\mathbb{P}^1$, saying that $\mathcal{H}$ is more constrained than $\mathcal{H}'$ means that if some points coincide in $\mathcal{H}'$, then the corresponding points in $\mathcal{H}$ coincide as well, but not vice versa. In what follows we analyze an exhaustive list of possibilities where in each one we determine an appropriate index $i_0$.

If there exist $i,j\in[n]$ with $i\neq j$ such that $p_i'=p_j'$, then one can take $i_0=i$ and the claim follows. So let us assume all points in $\mathcal{H}'$ are distinct. Under this assumption, what we need to show is that we can always find $i_0\in[n]$ such that $\mathcal{H}\setminus i_0$ is automorphism-free and it contains points appearing multiple times.

If $p_i\in\mathcal{H}$ appears at least three times, then we can set $i_0=i$. Otherwise, assume each point in $\mathcal{H}$ appears at most two times. If $p_i\in\mathcal{H}$ is the only point appearing twice, then let $i_0\in[n]\setminus\{i\}$ (in this case the hypothesis $n>4$ guarantees that $\mathcal{H}\setminus i_0$ is automorphism-free). If there exist two distinct points $p_i,p_j\in\mathcal{H}$ appearing twice, then we can set $i_0=i$. Since we considered all the possibilities, this concludes the proof.
\end{proof}

\begin{lemma}
\label{deleteelementparallelandconnectedness}
Let $M$ be a connected matroid and let $C$ be a size two circuit of $M$. Then for all $c\in C$ we have that $M\setminus c$ is connected.
\end{lemma}

\begin{proof}
Let $C=\{c_1,c_2\}$. Let us show $M\setminus c_1$ is connected (connectedness of $M\setminus c_2$ is proved analogously). We show that for each $x\neq c_1,c_2$ in the ground set, there exists a circuit of $M\setminus c_1$ containing both $x$ and $c_2$, so that $x$ and $c_2$ belong to the same connected component. Since $M$ is connected, there exists a circuit $D$ of $M$ containing $x$ and $c_2$. Observe that $c_1\notin D$ because $C\not\subseteq D$ by the minimality of $D$. Hence, $D$ is also a circuit of $M\setminus c_1$. 
\end{proof}

\begin{remark}
\label{counterexampleinrank2}
The assumption $n\ge5$ if $r=2$ can not be removed.
Let $M'=U_{2,4}$ be the uniform matroid of rank $2$ with ground set $[4]$. 
Let $M=\widetilde{U}_{2,4}$ be the matroid with the same ground set and with bases
\[
\{1,2\},~\{1,3\},~\{1,4\},~\{2,3\},~\{2,4\}.
\]
It is easy to check that (CD) does not hold.
\end{remark}


\section{Cross-ratios on Lafforgue's varieties}
\label{kapranovandlafforgue}


We recall several definitions and facts from \cite{Kap93a,Laf03,KT06}.  
Let $G(r,n)$ denote the Grassmannian of $r$-dimensional linear subspaces in $\Bbbk^n$
embedded in 
$$\mathbb{P}:=\mathbb{P}(\bigwedge^r\Bbbk^n)$$ 
via the Pl\"ucker embedding.
Let $G^0(r,n)\subseteq G(r,n)$ be the subset of points with nonzero Pl\"ucker coordinates. 
The  torus $H=\mathbb{G}_m^n/\diag(\mathbb{G}_m)$  acts on $G(r,n)$ via the action of $\mathbb{G}_m^n$ on~$\Bbbk^n$ and 
$G^0(r,n)$  is an $H$-invariant open subset. Denote the quotient by $\mathbf{X}(r,n)=G^0(r,n)/H$ ($H$~acts freely).
By the Gelfand--MacPherson correspondence \cite{GM82}, ${\mathbf{X}}(r,n)$ is also the moduli space 
of $n$ hyperplanes in general linear position in $\mathbb{P}^{r-1}$.
Namely, a point of $G(r,n)$ can be represented as the row space of an $r\times n$ matrix, and columns of this matrix give $n$ hyperplanes in $\mathbb{P}^{r-1}$.
Kapranov's compactification $\overline{\mathbf{X}}(r,n)$ of $\mathbf{X}(r,n)$ is the Chow quotient 
\begin{equation*}
\overline{\mathbf{X}}(r,n)=G(r,n)/\!/ H.
\end{equation*}
By 
\cite{Kap93a}, $\overline{\mathbf{X}}(r,n)$ is also
isomorphic to the Chow quotient of $(\mathbb{P}^{r-1})^n$
by $\PGL_r$. For example, $\overline{\mathbf{X}}(2,n)$ is isomorphic to $\overline{\mathrm{M}}_{0,n}$, the moduli space 
of stable genus zero $n$-pointed curves. In particular,
$\overline{\mathbf{X}}(2,4)=\bP^1$
via the cross-ratio of four points.

For $x\in\mathbb{P}$, let $\Supp(x)$ be the convex hull of the vertices of $\Delta(r,n)$ corresponding to nonzero coordinates of $x$. Given a matroid polytope $P\subseteq\Delta(r,n)$, we define the algebraic tori
$$
\mathbb{P}^{P,0}=\{x\in\mathbb{P}\mid\Supp(x)=P\}\quad\hbox{\rm and}\quad 
\mathcal{A}^P_0:=\mathbb{P}^{P,0}/H.$$
For instance, $\mathbb{P}^{\Delta(r,n),0}\subseteq\mathbb{P}$ is the maximal torus. We refer to \cite{GKZ08}
for the theory of \emph{regular subdivisions} and \emph{secondary polytopes}.
We define $\mathcal{F}_{\sec}(P)$ to be the normal fan to the secondary polytope of $P$. 
The union of cones in $\mathcal{F}_{\sec}(P)$ corresponding to regular subdivisions of $P$ by matroid polytopes
is a subfan denoted by $\mathcal{F}_{\Laf}(P)$.
The \emph{Lafforgue toric variety}~$\mathcal{A}^P$ 
is the toric variety of $\mathcal{A}^P_0$ associated to the fan $\mathcal{F}_{\Laf}(P)$.
We define $\mathbf{\Omega}^P$ to be the quotient 
$$\mathbf{\Omega}^P=\left(G(r,n)\cap\mathbb{P}^{P,0}\right)/H\subseteq \mathcal{A}^P_0.$$
The \emph{Lafforgue variety} $\overline{\mathbf{\Omega}}^P$ is a certain projective subscheme
of the toric variety $\mathcal{A}^P$,
the precise definition of which we will not need.
$\overline{\mathbf{\Omega}}^P$ contains $\mathbf{\Omega}^P$ as an open subset, but in general $\overline{\mathbf{\Omega}}^P$ is reducible \cite[Proposition~3.10]{KT06}.
The closure of $\mathbf{\Omega}^P$ in $\overline{\mathbf{\Omega}}^P$ is called the {\em main stratum}.
If $P=\Delta(r,n)$, then $G(r,n)\cap\mathbb{P}^{P,0}=G^0(r,n)$. Therefore, $\mathbf{\Omega}^P=\mathbf{X}(r,n)$.
By~\cite[2.9]{KT06}, $\overline{\mathbf{X}}(r,n)$ and the main stratum in 
$\overline{\mathbf{\Omega}}^P$ have the same normalization $\overline{\mathbf{X}}(r,n)^\nu$. 

\begin{remark}
\label{lafforgue'svarietyemptyinnonrealizablecase}
We note that $P\subseteq\Delta(r,n)$ is the matroid polytope of a non-realizable matroid
if and only if $\mathbf{\Omega}^P=\emptyset$  (see \cite[2.6]{KT06}).
The main stratum is then of course also empty but 
we imagine that $\overline{\mathbf{\Omega}}^P$ can still be non-empty, although we do not know an example.
\end{remark}

\begin{definition}
Given a matroid polytope $P$ and a face $F\subseteq P$, we have an induced morphism of Lafforgue's toric varieties $\mathcal{A}^P\rightarrow\mathcal{A}^F$ by restricting piece-wise affine functions. It is called the \emph{face map}
and restricts to $\overline{\mathbf{\Omega}}^Q$ giving a morphism 
$$\overline{\mathbf{\Omega}}^Q\rightarrow\overline{\mathbf{\Omega}}^F,$$
which we also call a face map.
Note that $\Delta(r,n)$ has $2n$ facets given by hyperplanes $x_i=0$ and $x_i=1$. 
These facets are equivalent to $\Delta(r,n-1)$ and $\Delta(r-1,n-1)$, respectively.
\end{definition}

\begin{lemma}
\label{lemmaforlafforgue'storicvarieties}
Let $Q\subseteq\Delta(r,n)$ be a matroid polytope of rank $r$ and dimension $n-1$ such that either $r\geq3$ or $r=2$ and $n\geq5$. Let $\mathcal{F}$ be the collection of facets of $Q$ that are contained in the boundary of $\Delta(r,n)$. 
Then the fibers of the product of face maps
\[
f\colon\mathcal{A}^Q\rightarrow\prod_{F\in\mathcal{F}}\mathcal{A}^F
\]
do not contain any complete subvarieties.
\end{lemma}

\begin{proof}
Let $C\subseteq\mathcal{A}^Q$ be a complete curve. Let $S'\subseteq\mathcal{A}^Q$ be the minimal closed toric stratum containing $C$. More explicitly, we can write $S'=\overline{T'}$ where $\overline{T'}\subseteq\mathcal{A}^Q$ is the closure of a torus $T'$ of the appropriate dimension. Since the torus $T'$ is affine, $C$ is projective, and $C\cap T'\neq\emptyset$ by the minimality of $S'$, we cannot have that $C\subseteq T'$, hence $C$ intersects the closure of another torus orbit $\overline{T}\subsetneq\overline{T'}$. Then $S=\overline{T}$ is a toric stratum properly contained in $S'$ which intersects $C$ nontrivially. Let $p\in S\cap C$ and $p'\in(S'\setminus S)\cap C$.

The stratum $S'$ (resp. $S$) corresponds to a matroid polytope subdivision $\mathscr{P}'$ (resp. $\mathscr{P}$) of~$Q$. Since $S\subsetneq S'$, the subdivision $\mathscr{P}'$ is coarser than $\mathscr{P}$. Then we can apply Corollary~\ref{extensiontomatroidpolytopesubdivisions} and find a facet $E\in\mathcal{F}$ such that $\mathscr{P}'|_E$ is coarser than $\mathscr{P}|_E$. Then $\mathscr{P}'|_E$ and $\mathscr{P}|_E$ correspond to two toric strata $\Sigma\subsetneq\Sigma'$ of $\mathcal{A}^E$. Let $f_E$ be the face map $\mathcal{A}^Q\rightarrow\mathcal{A}^E$. Then $f_E(p')\in\Sigma'\setminus\Sigma$ and $f_E(p)\in\Sigma$, implying that $f(p)\neq f(p')$. Therefore $C$ is not contracted by the morphism~$f$. It follows that $f$ 
does not contract any complete curve.
\end{proof}


\begin{definition}
Let $Q\subseteq\Delta(r,n)$ be a matroid polytope of rank $r$. We say that $Q$ is \emph{framed} provided it is realized by a hyperplane arrangement $\mathcal{H}$ in $\mathbb{P}^{r-1}$ containing $r+1$ hyperplanes in general linear position.
Note that such a matroid is automatically connected.
\end{definition}

\begin{proof}[Proof of Theorem~\ref{mainresultforapplications}]
We apply Lemma~\ref{lemmaforlafforgue'storicvarieties} to $\mathcal{A}^Q$ recursively until we map to a product
\[
f'\colon\mathcal{A}^Q\rightarrow\prod_{F\in\mathcal{F}'}\mathcal{A}^F,
\]
where $\mathcal{F}'$ consists of the faces of $Q$ of rank $2$ that are contained in the boundary of $\Delta(r,n)$. We know that the above morphism does not contract any complete curve.

Let $M$ be the matroid associated to $F\in\mathcal{F}'$. There are the following two possibilities:
\begin{enumerate}

\item $M$ is isomorphic to $U_{2,4}$, hence $F$ is equivalent to $\Delta(2,4)$;

\item $M$ is isomorphic to $\widetilde{U}_{2,4}$ (see Remark~\ref{counterexampleinrank2}).

\end{enumerate}
Therefore, $\mathcal{F}$ is the subset of $\mathcal{F}'$ of facets that are not of type (2). Consider the composition
\[
f\colon\mathcal{A}^Q\xrightarrow{f'}\prod_{F\in\mathcal{F}'}\mathcal{A}^F\xrightarrow{\pi}\prod_{F\in\mathcal{F}}\mathcal{A}^F.
\]
Then $f$ does not contract any complete curve either. Assume by contradiction $C\subseteq\mathcal{A}^Q$ is a complete curve contracted by $f$. We know $f'$ does not contract $C$, so $\pi$ contracts $f'(C)$. Since $\pi$ contracts all the factors $\mathcal{A}^F$ with $F$ of type (2), then $C$ has to map onto some $\mathcal{A}^F$ with $F$ of type (2). But $\mathcal{A}^F\cong\mathbb{G}_m$, which is a contradiction because $C$ is complete and $\mathcal{A}^F$ is affine. Note that $\mathcal{A}^F\cong\mathbb{G}_m$ because $F$ does not have any proper subdivision into matroid subpolytopes. Restricting $f$ to the projective variety $\overline{\mathbf{\Omega}}^Q$, we obtain the claimed finite morphism.

Finally, suppose that $Q$ is framed.
We claim that the product of face maps \eqref{sgsGSRH}
is a locally closed embedding and thus its restriction 
to the main stratum of $\overline{\mathbf{\Omega}}^Q$
is a finite birational morphism.
It suffices to  show $\mathbf{\Omega}^Q\rightarrow\prod_{F\in\mathcal{F}}\overline{\mathbf{\Omega}}^F$ is injective on $R$-points, where $(R,\mathfrak{m})$ is any local $\Bbbk$-algebra with $R/\mathfrak{m}\cong\Bbbk$. Let $p\in\mathbf{\Omega}^Q(R)$ and let $(\mathbb{P}_R^{r-1},\sum_{i=1}^nH_i)$ be the hyperplane arrangement parametrized by it
(the projective bundle is trivial because the frame gives a section of the associated $\PGL_r$-torsor).
The images of $p$ under the face maps (in other words, the cross-ratios of the hyperplane arrangement) are described as follows. Let $a_i$ be the vector of coefficients of the hyperplane $H_i$. Fix distinct indices $i_1,\ldots,i_{r-2},i,j,k,\ell\in[n]$ (these determine a face $F\in\mathcal{F}$). Then the corresponding cross-ratio is given by the point
\[
[|a_{i_1},\ldots,a_{i_{r-2}},a_i,a_j||a_{i_1},\ldots,a_{i_{r-2}},a_k,a_\ell|:|a_{i_1},\ldots,a_{i_{r-2}},a_i,a_k||a_{i_1},\ldots,a_{i_{r-2}},a_j,a_\ell|]\in\mathbb{P}^1.
\]
We~show that the hyperplane arrangement $(\mathbb{P}_R^{r-1},\sum_{i=1}^nH_i)$ can be uniquely reconstructed by the data of these cross-ratios.
Since $R$ is local, the determinant of the matrix of coefficients of any $r$ hyperplanes among $H_1,\ldots,H_{r+1}$ is invertible.  In particular, up to $\PGL_r(R)$-action, we can assume $H_1,\ldots,H_{r+1}$ are the standard hyperplanes. In other words,
we can assume $(a_1,\ldots,a_{r+1})=(\mathrm{Id}_r,\mathbf{1})$, where $\mathrm{Id}_r$ is the $r\times r$ identity matrix and $\mathbf{1}=(1,\ldots,1)$. We show that any $b\in\{a_{r+2},\ldots,a_n\}$ is uniquely determined by the cross-ratios. Consider the cross-ratio corresponding to the choice $(i_1,\ldots,i_{r-2})=(1,\ldots,r-2)$, $i=r+1,a_j=b,k=r-1,\ell=r$, which is given by $[b_{r-1}-b_r:b_{r-1}]$. This fixes $\frac{b_r}{b_{r-1}}\in\mathbb{P}^1$. In general, by letting $(i_1,\ldots,i_{r-2})=(1,\ldots,\widehat{s},\ldots,\widehat{t},\ldots,r)$, $i=r+1,a_j=b,k=s,\ell=s+1$, we can fix the ratio $\frac{b_t}{b_s}$, hence also $\frac{b_s}{b_t}$.
Therefore, to determine $b$, let $b_s$ be one of the nonzero coordinates of $b$ (there exists at least one because $b$ represents the coefficients of a hyperplane). Then all the ratios $\frac{b_1}{b_s},\ldots,\frac{b_r}{b_s}$ are fixed, hence the coefficients $b_1,\ldots,b_r$ are uniquely determined up to a common scalar, proving what we wanted.
\end{proof}


\begin{corollary}\label{productofforgetfulmapsisfinite}
Let $r\geq2$ and $n\geq r+2$. 
The product of forgetful morphisms
\[
\overline{\mathbf{X}}(r,n)\rightarrow\prod_{I\in\binom{[n]}{r+2}}
\overline{\mathbf{X}}(r,I)
\]
and the rational map
$
\overline{\mathbf{X}}(r,n)\rightarrow\prod\limits_{v_1,\ldots,v_{r+2}\in\{1,\ldots,n\}}\bP^1
$
are finite birational morphisms onto their image. For $r\geq3$, each map $\overline{\mathbf{X}}(r,n)\rightarrow\bP^1$ 
is the cross-ratio of $4$ points obtained by intersecting hyperplanes  $H_{v_1},H_{v_2},H_{v_3},H_{v_4}$
with the line $H_{v_5}\cap\ldots\cap H_{v_{r+2}}$.
\end{corollary}

\begin{proof}
Consider the product $\prod_{F\in\mathcal{F}}\overline{\mathbf{\Omega}}^F$ in Theorem~\ref{mainresultforapplications} for $Q=\Delta(r,n)$. Each $\overline{\mathbf{\Omega}}^F\cong \overline{\mathrm{M}}_{0,4}$
and the face $F$
can be described as follows. Choose distinct indices $i,j,k,\ell,i_1,\ldots,i_{r-2}\in[n]$. Then $F$ is the convex hull of the set
$
\{e_a+e_b+e_{i_1}+\ldots+e_{i_{r-2}}\mid a,b\in\{i,j,k,\ell\}\}.
$
This gives the following identification:
\[
\prod_{I\in\binom{[n]}{r+2}}\prod_{J\in\binom{I}{r-2}}\overline{\mathrm{M}}_{0,I\setminus J}\cong\prod_{F\in\mathcal{F}}\overline{\mathbf{\Omega}}^F.
\]

For each $I\in\binom{[n]}{r+2}$, we have isomorphisms $\overline{\mathbf{X}}(r,I)\cong\overline{\mathbf{X}}(2,I)\cong\overline{\mathrm{M}}_{0,I}$. We claim that $f\colon\overline{\mathbf{X}}(r,n)\rightarrow\prod_{I\in\binom{[n]}{r+2}}\overline{\mathrm{M}}_{0,I}$ is finite and birational onto its image.
Given $I\in\binom{[n]}{r+2}$, the product of forgetful morphisms $\overline{\mathrm{M}}_{0,I}\rightarrow\prod_{J\in\binom{I}{r-2}}\overline{\mathrm{M}}_{0,I\setminus J}$ is a closed embedding \cite[Theorem 9.18]{HKT09}. 
Now consider the following commutative diagram:
\begin{center}
\begin{tikzpicture}[>=angle 90]
\matrix(a)[matrix of math nodes,
row sep=2em, column sep=2.5em,
text height=2.0ex, text depth=1.50ex]
{\overline{\mathbf{X}}(r,n)^\nu&\overline{\mathbf{X}}(r,n)&\prod_{I\in\binom{[n]}{r+2}}\overline{\mathrm{M}}_{0,I}\\
(\overline{\mathbf{\Omega}}^Q)^\nu&\overline{\mathbf{\Omega}}^Q&\prod_{F\in\mathcal{F}}\overline{\mathbf{\Omega}}^F.\\};
\path[->] (a-1-1) edge node[]{}(a-1-2);
\path[->] (a-1-1) edge node[left]{$\cong$}(a-2-1);
\path[->] (a-2-1) edge node[]{}(a-2-2);
\path[->] (a-1-2) edge node[above]{$f$}(a-1-3);
\path[->] (a-1-3) edge node[]{}(a-2-3);
\path[->] (a-2-2) edge node[above]{$g$}(a-2-3);
\end{tikzpicture}
\end{center}
Since $g$ is finite and birational onto its image by Theorem~\ref{mainresultforapplications}, the morphism $f$ also is.
\end{proof}


\section{Gerritzen and Piwek's cross-ratio variety \texorpdfstring{$\overline{\mathbf{B}}_n$}{Lg}}
\label{explicitclosedembedding}

We start by recalling the definition of the Gerritzen--Piwek's compactification $\overline{\mathbf{B}}_n$ \cite{GP91}. 
Let $\mathbf{U}_n\subseteq(\mathbb{P}^2)^n$ be an open subset 
parametrizing configurations $(x_1,\ldots,x_n)$ of $n$ distinct points
in $\mathbb{P}^2$ in general linear position and let 
$\mathbf{B}_n$ be the moduli space, the quotient of $\mathbf{U}_n$
by the free action of $\PGL_3$. We have the quotient morphism
$\mathbf{U}_n\to\mathbf{B}_n$.
If we normalize the points 
$x_1,x_2,x_3,x_4$
to be
$e_1=[1:0:0]$, $e_2=[0:1:0]$, $e_3=[0:0:1]$, $e_4=[1:1:1]$
applying a unique projective transformation, we obtain a section 
$$s\colon\mathbf{B}_n\to\mathbf{U}_n.$$
To ``symmetrize'' this map,
let $\mathcal{Q}_n$ be the set of ordered $5$-tuples of points and let $v\in\mathcal{Q}_n$. Define the map $q_v\colon\mathbf{B}_n\rightarrow\mathbb{P}^2$ as the 
image of the point $x_{v_5}$ under a unique projective transformation 
$\psi_{x_{v_1},\ldots,x_{v_4}}\in\PGL_3$
that sends $x_{v_1},x_{v_2},x_{v_3},x_{v_4}$ to
$e_1, e_2, e_3, e_4$, respectively.
The product morphism
\begin{equation*}
\prod_{v\in\mathcal{Q}_n}q_v\colon\mathbf{B}_n\rightarrow\prod_{\mathcal{Q}_n}\mathbb{P}^2
\end{equation*}
is an open immersion onto a closed subscheme of $\prod_{\mathcal{Q}_n}\mathbb{P}^2$. The \emph{Gerritzen--Piwek compactification} $\overline{\mathbf{B}}_n$ is the closure of $\mathbf{B}_n$ in $\prod_{\mathcal{Q}_n}\mathbb{P}^2$ under the above immersion. (Note that in \cite{GP91}, the compact moduli space is denoted by $\mathbf{B}_n$ and its interior is denoted by $\mathbf{B}_n^*$.) A slightly more economical embedding can be obtained as follows.

\begin{lemma}
\label{equivalentdefinitionofBnbarinsmallerproductofP2s}
Let $\mathcal{Q}_n'$ denote the set of ordered $5$-tuples in $[n]$ where the first four elements are in increasing order. 
$\overline{\mathbf{B}}_n$ is isomorphic to the closure of the image of ${\mathbf{B}}_n$ in $\prod_{\mathcal{Q}_n'}\mathbb{P}^2$ under the product of forgetful maps. 
\end{lemma}

\begin{proof}
There is a map $\mu\colon\mathcal{Q}_n\rightarrow\mathcal{Q}_n'$ sending $v=(v_1,\ldots,v_5)$ to $(v_{\tau(v)(1)},\ldots,v_{\tau(v)(4)},v_5)$, where $\tau\colon\mathcal{Q}_n\rightarrow S_4$ is the map that associates to $v\in\mathcal{Q}_n$ the unique permutation $\tau(v)\in S_4$ such that $v_{\tau(v)(1)}<\ldots<v_{\tau(v)(4)}$. 
Given $\sigma\in S_4$, let $\varphi_\sigma$ be the unique projective linear transformation of $\mathbb{P}^2$ such that $\varphi_\sigma(e_i)=e_{\sigma(i)}$ for all $i=1,\ldots,4$. 
Let 
$f\colon\prod_{\mathcal{Q}_n'}\mathbb{P}^2\rightarrow\prod_{\mathcal{Q}_n}\mathbb{P}^2$
be the map $(p_v)_{v\in\mathcal{Q}_n'}\mapsto(\varphi_{\tau(v)}(p_{\mu(v)}))_{v\in\mathcal{Q}_n}$.
Let $\pi\colon\prod_{\mathcal{Q}_n}\mathbb{P}^2\rightarrow\prod_{\mathcal{Q}_n'}\mathbb{P}^2$ be the natural product of projection maps. 
Then $f$ is a section of $\pi$ and we have a commutative diagram
\begin{center}
\begin{tikzpicture}[>=angle 90]
\matrix(a)[matrix of math nodes,
row sep=2em, column sep=2.5em,
text height=1.5ex, text depth=0.50ex]
{&\prod_{\mathcal{Q}_n}\mathbb{P}^2\\
\mathbf{B}_n&\prod_{\mathcal{Q}_n'}\mathbb{P}^2.\\};
\path[right hook->] (a-2-1) edge node[above left]{$\iota$}(a-1-2);
\path[->] (a-2-1) edge node[below]{$\pi\circ\iota$}(a-2-2);
\path[->] (a-2-2) edge node[right]{$f$}(a-1-2);
\end{tikzpicture}
\end{center}
The result follows.
\end{proof}


Next we prove Theorem~\ref{isomorphismbetweenkapranovcompandgpcomp}:
there  exists  a  finite  birational  morphism $\overline{\mathbf{X}}(3,n)\rightarrow\overline{\mathbf{B}}_n$.
\begin{proof}[Proof of Theorem~\ref{isomorphismbetweenkapranovcompandgpcomp}]
Consider the following commutative diagram: 
\begin{center}
\begin{tikzpicture}[>=angle 90]
\matrix(a)[matrix of math nodes,
row sep=2em, column sep=2.5em,
text height=2.0ex, text depth=1.50ex]
{\mathbf{X}(3,n)&\overline{\mathbf{X}}(3,n)&\prod_{I\in\binom{[n]}{5}}\overline{\mathbf{X}}(3,I)&\prod_{I\in\binom{[n]}{5}}\overline{\mathrm{M}}_{0,I}\\
\mathbf{B}_n&\overline{\mathbf{B}}_n&&\prod_{\mathcal{Q}_n'}\mathbb{P}^2.\\};
\path[right hook->] (a-1-1) edge node[]{}(a-1-2);
\path[->] (a-1-1) edge node[left]{$\cong$}(a-2-1);
\path[right hook->] (a-2-1) edge node[]{}(a-2-2);
\path[->] (a-1-2) edge node[]{}(a-1-3);
\path[->] (a-1-3) edge node[above]{$\cong$}(a-1-4);
\path[right hook->] (a-1-4) edge node[]{}(a-2-4);
\path[right hook->] (a-2-2) edge node[]{}(a-2-4);
\end{tikzpicture}
\end{center}
Recall that
the morphism $\overline{\mathbf{X}}(3,n)\rightarrow\prod_{I\in\binom{[n]}{5}}\overline{\mathbf{X}}(3,I)$ is finite and birational onto its image by Corollary~\ref{productofforgetfulmapsisfinite}.
The embedding $\prod_{I\in\binom{[n]}{5}}\overline{\mathrm{M}}_{0,I}\hookrightarrow\prod_{\mathcal{Q}_n'}\mathbb{P}^2$ is obtained by applying Lemma~\ref{productofkapranovmapsn=5} to each copy of $\overline{\mathrm{M}}_{0,I}$.
This gives a finite birational morphism $\overline{\mathbf{X}}(3,n)\rightarrow\overline{\mathbf{B}}_n$.
\end{proof}

For the next lemma, recall that the \emph{$i$-th Kapranov's map} \cite{Kap93b} is a birational morphism $\overline{\mathrm{M}}_{0,n}\rightarrow\mathbb{P}^{n-3}$ given by the linear system~$|\psi_i|$.

\begin{lemma}
\label{productofkapranovmapsn=5}
The product of Kapranov's maps $\overline{\mathrm{M}}_{0,5}\rightarrow\prod_{i=1}^5\mathbb{P}^2$ is a closed embedding. \end{lemma}

\begin{proof}
The boundary of $\overline{\mathrm{M}}_{0,5}$ consists of $10$ irreducible curves isomorphic to $\mathbb{P}^1$. These are denoted by $D_I$, where $I\subseteq[5]$ is a subset of size $2$. Moreover, two distinct boundary divisors $D_I,D_J$ intersect if and only if $I\subseteq J^c$. For $i\in[5]$, the Kapranov's map $\sigma_i$ contracts the boundary divisors $D_I$ such that $i\in I$, and away from these divisors $\sigma_i$ is an isomorphism. Given this description, it is clear that for all $x\in\overline{\mathrm{M}}_{0,5}$, there exists an open subset $U\subseteq\overline{\mathrm{M}}_{0,5}$ containing $x$ and $i\in[5]$ such that $\sigma_i|_U$ is an isomorphism onto its image. This implies that the product of Kapranov's maps $\overline{\mathrm{M}}_{0,5}\rightarrow\prod_{i=1}^5\mathbb{P}^2$ is a closed embedding.
\end{proof}

\begin{remark}
The product of Kapranov's maps $\overline{\mathrm{M}}_{0,8}\rightarrow\prod_{i=1}^8\mathbb{P}^5$ is not injective. Indeed, the curve $C=D_{12}\cap D_{34}\cap D_{56}\cap D_{78}\subseteq\overline{\mathrm{M}}_{0,8}$ is contracted by $\sigma_1$ because $\sigma_1(D_{12})$ is a point. Analogously, $\sigma_2,\ldots,\sigma_8$ contract $C$. Hence, the product of Kapranov's maps contracts~$C$.
\end{remark}

\begin{remark}
We have a commutative  diagram of closed embeddings
\begin{center}
\begin{tikzpicture}[>=angle 90]
\matrix(a)[matrix of math nodes,
row sep=2em, column sep=2.5em,
text height=2.0ex, text depth=1.50ex]
{\overline{\mathbf{B}}_n&\prod_{\mathcal{Q}_n'}\mathbb{P}^2&\\
&\prod_{I\in\binom{[n]}{5}}\overline{\mathrm{M}}_{0,I}&\prod_{\mathcal{Q}_n'}\mathbb{P}^1,\\};
\path[right hook->] (a-1-1) edge node[]{}(a-1-2);
\path[right hook->] (a-1-1) edge node[]{}(a-2-2);
\path[right hook->] (a-2-2) edge node[]{}(a-1-2);
\path[right hook->] (a-2-2) edge node[]{}(a-2-3);
\end{tikzpicture}
\end{center}
and we can embed $\prod_{\mathcal{Q}_n'}\mathbb{P}^1\hookrightarrow\prod_{\mathcal{Q}_n}\mathbb{P}^1$ using a strategy analogous to the proof of 
Lemma~\ref{equivalentdefinitionofBnbarinsmallerproductofP2s}. In conclusion, we can view $\overline{\mathbf{B}}_n$ as a closed subvariety of $\prod_{\mathcal{Q}_n}\mathbb{P}^1$ and the embedding 
can be interpreted as follows: given an ordered quintuple $v\in\mathcal{Q}_n$, define $\mathbf{B}_n\rightarrow\mathbb{P}^1$ by sending $(p_1,\ldots,p_n)$ to the cross-ratio of four points in $\mathbb{P}^1$ obtained by projecting $p_{v_1},\ldots,p_{v_4}$ from~$p_{v_5}$.
As a consequence of this discussion,  $\overline{\mathbf{B}}_5
\cong\overline{\mathrm{M}}_{0,5}\cong \overline{\mathbf{X}}(3,5)$.
\end{remark}


\section{Mustafin joins: general theory}

Mustafin joins were defined in the Introduction. Here we collect some basic facts.

\begin{remark}[{\cite[\S2]{CHSW11}}]
\label{equationsmustafinvariety}
Let $\Sigma=\{L_1,\ldots,L_s\}\subseteq \mathfrak{B}_3^0$. 
Fix a $K$-basis $e_1,e_2,e_3$ for $K^3$. 
For each $j=1,\ldots,s$, let $L_j=f_{1j}R+f_{2j}R+f_{3j}R$ and let $g_j\in\GL_3(K)$ be the matrix with columns $f_{1j}, f_{2j}, f_{3j}$.
Denote by $X$ the matrix $(x_{ij})_{1\leq i\leq3,1\leq j\leq s}$, 
where we interpret the $j$-th column as homogeneous coordinates on the $j$-th copy of $\mathbb{P}^2$. Let $g(X)$ be the matrix obtained by applying $g_j$ to the $j$-th column of $X$ for all $j$. Then $\mathbb{P}(\Sigma)$ is isomorphic to the subscheme of $(\mathbb{P}_R^2)^s$ cut out by the multihomogeneous ideal $I_2(g(X))\cap R[X]$, where $I_2$ denotes the ideal generated by the $2\times2$ minors.
For example, 
let $\Sigma=\{[L_1],[L_2]\}$, where
\[
L_1=e_1R+e_2R+e_3R,~L_2=te_1R+e_2R+e_3R.
\]
The ideal of $\mathbb{P}(\Sigma)\subseteq\mathbb{P}(L_1)\times_R\mathbb{P}(L_2)$ is generated by the $2\times2$ minors of the matrix
\begin{displaymath}
\left( \begin{array}{cc}
x_{11}&tx_{12}\\
x_{21}&x_{22}\\
x_{31}&x_{32}
\end{array} \right).
\end{displaymath}
The special fiber $\mathbb{P}(\Sigma)_\Bbbk$ 
is defined 
by the following equations:
\[
x_{11}x_{22}=0,~x_{11}x_{32}=0,~x_{21}x_{32}-x_{31}x_{22}=0.
\]
Its irreducible components  are given by $V(x_{22},x_{32})\cong\mathbb{P}^2$ and $V(x_{11},x_{21}x_{32}-x_{31}x_{22})\cong\mathbb{F}_1$. These are glued along a line in $\mathbb{P}^2$ and the exceptional divisor in $\mathbb{F}_1$ as shown in Figure~\ref{degeneration5points1}.
\end{remark}

\begin{figure}[hbtp]
\begin{tikzpicture}[scale=0.5]
	\draw[line width=1pt] (0,6) -- (6,6);
	\draw[line width=1pt] (6,0) -- (6,6);
	\draw[line width=1pt] (0,6) -- (6,0);
	\draw[line width=1pt] (3,6) -- (6,3);

	\node at (3.7,3.7) {$\mathbb{F}_1$};

	\node at (5.2,5.2) {$\mathbb{P}^2$};
\end{tikzpicture}
\caption{Central fiber of the Mustafin join $\mathbb{P}(\Sigma)$ in Remark~\ref{equationsmustafinvariety}.}
\label{degeneration5points1}
\end{figure}
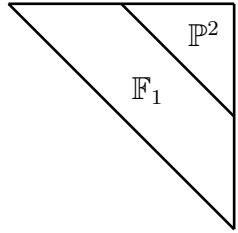

\begin{remark}
\label{mustafindegenerationsofP2andP2dualarenotthesame}
For a lattice $L$ 
we can define $\mathbb{P}^\vee(L)$ as $\Proj(\Sym L)$, and given a finite subset $\Sigma\subseteq\mathfrak{B}_3^0$, we can define $\mathbb{P}^\vee(\Sigma)$ accordingly. In general, the Mustafin joins $\mathbb{P}(\Sigma)$ and $\mathbb{P}^\vee(\Sigma)$ have non-isomorphic central fibers. For instance, consider $\Sigma=\{e_1R+e_2R+e_3R,te_1R+e_2R+t^2e_3R,e_1R+te_2R+t^2e_3R\}$. Then $\mathbb{P}(\Sigma)_\Bbbk$ can be found in \cite[Figure 6 (iii), first picture]{CHSW11}, and $\mathbb{P}^\vee(\Sigma)_\Bbbk$ is in \cite[Figure 6 (i), first picture]{CHSW11}.
\end{remark}

\begin{remark}
\label{degenerationfrompicture}
An \emph{apartment} $A\subseteq\mathfrak{B}_3^0$ corresponding to a basis $e_1,e_2,e_3$ of $K^3$ is the subset consisting of all lattice classes of the form
\[
[t^\alpha e_1R+t^\beta e_2R+t^\gamma e_3R],
\]
for $\alpha,\beta,\gamma\in\mathbb{Z}$. Given a finite subset $\Sigma\subseteq A$, the central fiber of the Mustafin join $\mathbb{P}(\Sigma)$ can be computed as follows \cite[\S4]{CHSW11}. The apartment $A$ is identified with the tropical torus $\mathbb{R}^3/\mathbb{R}\mathbf{1}$, where $\mathbf{1}=(1,1,1)$, under the following bijection:
\[
[t^\alpha e_1R+t^\beta e_2R+t^\gamma e_3R]\mapsto(-\alpha,-\beta,-\gamma).
\]
To each point $p\in\mathbb{R}^3/\mathbb{R}\mathbf{1}$, one can associate a tropical line (a spider with three legs)
\[
\ell_p=\{v\in\mathbb{R}^3/\mathbb{R}\mathbf{1}\mid\max_{i=1,2,3}\{v_i-p_i\}~\textrm{is achieved at least twice}\}.
\]
Under the identification $A=\mathbb{R}^3/\mathbb{R}\mathbf{1}$, a lattice class $[L]\in\Sigma$ determines a tropical line $\ell_{[L]}$. The union of the bounded regions determined by the tropical lines $\ell_{[L]}$, $[L]\in\Sigma$, gives a tropical polytope $P_{\Sigma}$, which is the min-convex hull of $\Sigma$. Consider the set of points $p\in P_{\Sigma}$ which correspond to a lattice class in $\Sigma$, or $p$ is the intersection of at least two tropical lines $\ell_{[L]}$, $[L]\in\Sigma$. Such a point $p$ determines a projective toric variety whose polytope has edges orthogonal to the rays generating from $p$. Gluing all these polytopes we obtain a regular mixed polyhedral subdivision of $m\Delta_2$, where $m=|\Sigma|$ and $\Delta_2$ is the standard $2$-dimensional simplex, which  determines the central fiber of $\mathbb{P}(\Sigma)$. We illustrate this procedure in Figure~\ref{centralfibermustafinjoinconfigurationinoneapartment}.
\end{remark}

\begin{figure}[hbtp]
\begin{tikzpicture}[scale=0.5]
	\draw[->,line width=1pt] (0-4,0+2) -- (1-4,1+2);
	\draw[->,line width=1pt] (0-4,0+2) -- (-1-4,0+2);
	\draw[->,line width=1pt] (0-4,0+2) -- (0-4,-1+2);

	\node at (-5.5-1,2) {{\tiny$+(1,0,0)$}};
	\node at (-3.2-1,.5) {{\tiny$+(0,1,0)$}};
	\node at (-1-1,3.7) {{\tiny$+(0,0,1)$}};

	\draw[line width=1pt] (0,0) -- (9,0);
	\draw[line width=1pt] (0,0) -- (0,3);
	\draw[line width=1pt] (0,0) -- (-3,-3);

	\node at (0,0) {$\bullet$};
	\node at (0.7,-0.5) {{\tiny$(0,0,0)$}};

	\draw[line width=1pt] (6,-3) -- (6+3,-3);
	\draw[line width=1pt] (6,-3) -- (6,-3+6);
	\draw[line width=1pt] (6,-3) -- (6-2,-3-2);

	\node at (6,-3) {$\bullet$};
	\node at (6+1.3,-3-0.5) {{\tiny$(-3,0,-1)$}};

	\draw[dashed,line width=1pt] (-2,1) -- (1,1);
	\draw[dashed,line width=1pt] (1,1) -- (1,-2);
	\draw[dashed,line width=1pt] (-2,1) -- (1,-2);

	\draw[dashed,line width=1pt] (6-2,-3+1) -- (6+1,-3+1);
	\draw[dashed,line width=1pt] (6+1,-3+1) -- (6+1,-3-2);
	\draw[dashed,line width=1pt] (6-2,-3+1) -- (6+1,-3-2);

	\draw[dashed,line width=1pt] (6-1,0+1) -- (6+1,0+1);
	\draw[dashed,line width=1pt] (6-1,0-1) -- (6+1,0-1);
	\draw[dashed,line width=1pt] (6+1,0+1) -- (6+1,0-1);
	\draw[dashed,line width=1pt] (6-1,0+1) -- (6-1,0-1);

	\draw[line width=1pt] (0+11,0+1) -- (6+11,0+1);
	\draw[line width=1pt] (6+11,0+1) -- (6+11,-6+1);
	\draw[line width=1pt] (0+11,0+1) -- (6+11,-6+1);
	\draw[line width=1pt] (3+11,0+1) -- (3+11,-3+1);
	\draw[line width=1pt] (3+11,-3+1) -- (6+11,-3+1);
\end{tikzpicture}
\caption{Illustration of the procedure described in Remark~\ref{degenerationfrompicture} to compute $\mathbb{P}(\Sigma)_\Bbbk$, where $\Sigma=\{[e_1R+e_2R+e_3R],[t^3e_1R+e_2R+te_3R]\}$.}
\label{centralfibermustafinjoinconfigurationinoneapartment}
\end{figure}
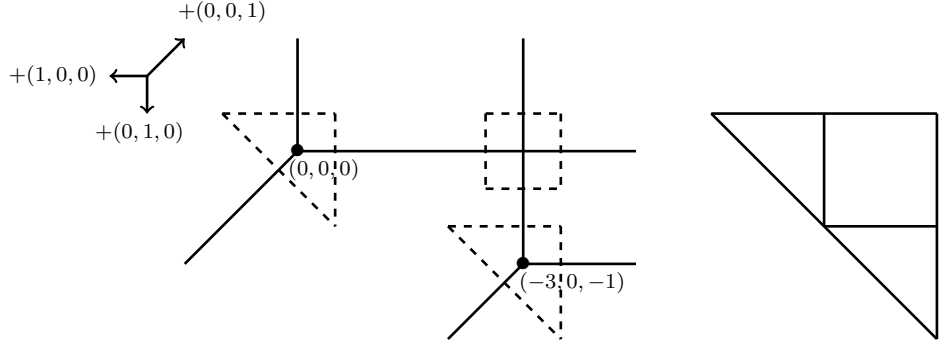

\begin{definition}
Let $[L_0],[L]\in\mathfrak{B}_3^0$. 
Let $z\in\mathbb{Z}$ be a unique integer such that $tL_0\subseteq t^zL$, but $L_0\nsubseteq t^zL$. 
Following \cite[\S5]{CHSW11}, we define a linear subspace
\emph{induced by $[L]$} as
\[
W_{L_0}(L)=\mathbb{P}\left((t^zL\cap L_0)/tL_0\right)\subseteq\mathbb{P}(L_0)_\Bbbk.
\]
\end{definition}

\begin{remark}
\label{allabouttwolattices}
Let $\Sigma=\{[L_1],[L_2]\}$, $L_1\ne L_2$. 
Consider the diagram
\begin{center}
\begin{tikzpicture}[>=angle 90]
\matrix(a)[matrix of math nodes,
row sep=2em, column sep=1.5em,
text height=1.5ex, text depth=0.25ex]
{&\mathbb{P}(\Sigma)_\Bbbk&\\
\mathbb{P}(L_1)_\Bbbk&&\mathbb{P}(L_2)_\Bbbk,\\};
\path[->] (a-1-2) edge node[]{}(a-2-1);
\path[->] (a-1-2) edge node[]{}(a-2-3);
\end{tikzpicture}
\end{center}
where the morphisms from $\mathbb{P}(\Sigma)_\Bbbk\subseteq\mathbb{P}(L_1)_\Bbbk\times\mathbb{P}(L_2)_\Bbbk$ are induced by the usual projections.
There are two options for the central fiber $\mathbb{P}(\Sigma)_\Bbbk$:
\begin{enumerate}

\item $\mathbb{P}(\Sigma)_\Bbbk$ is the gluing of $\mathbb{P}^2$ and $\mathbb{F}^1$ along a line and the exceptional divisor, respectively;

\item $\mathbb{P}(\Sigma)_\Bbbk$ is the union of two copies of $\mathbb{P}^2$ and $\mathbb{F}_0$. Each $\mathbb{P}^2$ is glued along a line to a ruling of $\mathbb{F}_0$, these two rulings intersect.

\end{enumerate}
A comprehensive list of possibilities for the induced linear subspaces $W_{L_1}(L_2)\subseteq\mathbb{P}(L_1)_\Bbbk$ and $W_{L_2}(L_1)\subseteq\mathbb{P}(L_2)_\Bbbk$ is shown in Figure~\ref{diagramsfortwolatticesandinducedlinearspaces}, where $W_{L_1}(L_2)$ and $W_{L_2}(L_1)$ are the images of appropriate irreducible components of $\mathbb{P}(\Sigma)_\Bbbk$
in $\mathbb{P}(L_1)_\Bbbk$ and $\mathbb{P}(L_1)_\Bbbk$, respectively.
To prove these claims, 
we can assume that $[L_1],[L_2]$ lie in the same apartment \cite[Proposition~4.11]{Wer01} 
\[
L_1=e_1R+e_2R+e_3R,~L_2=t^\alpha e_1R+t^\beta e_2R+t^\gamma e_3R,
\]
for some $K$-basis $e_1,e_2,e_3$ of $K^3$. 
If two among $\{\alpha,\beta,\gamma\}$ are equal, then without loss of generality
$0=\alpha=\beta\neq\gamma$. In this case, $\mathbb{P}(\Sigma)_\Bbbk$ is as in (1) above. Moreover, 
\begin{itemize}

\item If $\gamma<0$, then $W_{L_1}(L_2)$ is point, and $W_{L_2}(L_1)$ is a line;

\item If $\gamma>0$, then $W_{L_1}(L_2)$ is a line, and $W_{L_2}(L_1)$ is a point.

\end{itemize}
If $\alpha,\beta,\gamma$ are all distinct, then without loss of generality 
$0=\alpha<\beta<\gamma$. Then $\mathbb{P}(\Sigma)_\Bbbk$ is as in (2) and the induced linear spaces $W_{L_1}(L_2)$ and $W_{L_2}(L_1)$ are both lines.
\end{remark}

\begin{figure}[hbtp]
\begin{tikzpicture}[scale=0.5]
	\draw[line width=1pt] (0,4) -- (4,4);
	\draw[line width=1pt] (4,0) -- (4,4);
	\draw[line width=1pt] (0,4) -- (4,0);
	\node at (4,4) {$\bullet$};
	\node at (2,5) {$\mathbb{P}(L_1)_\Bbbk$};
	\node at (6,4) {{\small$W_{L_1}(L_2)$}};

	\draw[line width=1pt] (0+6,4+3) -- (4+6,4+3);
	\draw[line width=1pt] (4+6,0+3) -- (4+6,4+3);
	\draw[line width=1pt] (0+6,4+3) -- (4+6,0+3);
	\draw[line width=1pt] (8,4+3) -- (10,4+1);
	\fill[gray!50,nearly transparent] (8,4+3) -- (10,4+1) -- (4+6,4+3) -- cycle;
	\node at (8,8) {$\mathbb{P}(\Sigma)_\Bbbk$};

	\draw[line width=1pt] (0+12,4) -- (4+12,4);
	\draw[line width=1pt] (4+12,0) -- (4+12,4);
	\draw[line width=2.5pt] (0+12,4) -- (4+12,0);
	\fill[gray!50,nearly transparent] (0+12,4) -- (4+12,4) -- (4+12,0) -- cycle;
	\node at (14,5) {$\mathbb{P}(L_2)_\Bbbk$};
	\node at (12.5,1.5) {{\small$W_{L_2}(L_1)$}};

	\draw[->,line width=1pt] (5,5.5) -- (4,4.5);
	\draw[->,line width=1pt] (10.5,5.5) -- (12,4.5);




	\draw[line width=1pt] (0,4-9) -- (4,4-9);
	\draw[line width=2.5pt] (4,0-9) -- (4,4-9);
	\draw[line width=1pt] (0,4-9) -- (4,0-9);
	\node at (6,2-9) {{\small$W_{L_1}(L_2)$}};
	\draw[line width=1pt] (0+6,4+3-9) -- (4+6,4+3-9);
	\draw[line width=1pt] (4+6,0+3-9) -- (4+6,4+3-9);
	\draw[line width=1pt] (0+6,4+3-9) -- (4+6,0+3-9);
	\draw[line width=1pt] (8,4+3-9) -- (8,4+3-9-2);
	\draw[line width=1pt] (8,4+1-9) -- (10,4+1-9);
	\fill[gray!50,nearly transparent] (8,4+1-9) -- (10,4+1-9) -- (4+6,0+3-9) -- cycle;
	\draw[line width=2.5pt] (0+12,4-9) -- (4+12,4-9);
	\draw[line width=1pt] (4+12,0-9) -- (4+12,4-9);
	\draw[line width=1pt] (0+12,4-9) -- (4+12,0-9);
	\fill[gray!50,nearly transparent] (0+12,4-9) -- (4+12,4-9) -- (4+12,0-9) -- cycle;
	\node at (14,5-9) {$W_{L_2}(L_1)$};
    \draw[->,line width=1pt] (5,5.5-9) -- (4,4.5-9);
	\draw[->,line width=1pt] (10.5,5.5-9) -- (12,4.5-9);

\end{tikzpicture}
\caption{The linear subspaces $W_{L_1}(L_2)$ and $W_{L_2}(L_1)$ in relation to $\mathbb{P}(\Sigma)_\Bbbk$.}
\label{diagramsfortwolatticesandinducedlinearspaces}
\end{figure}
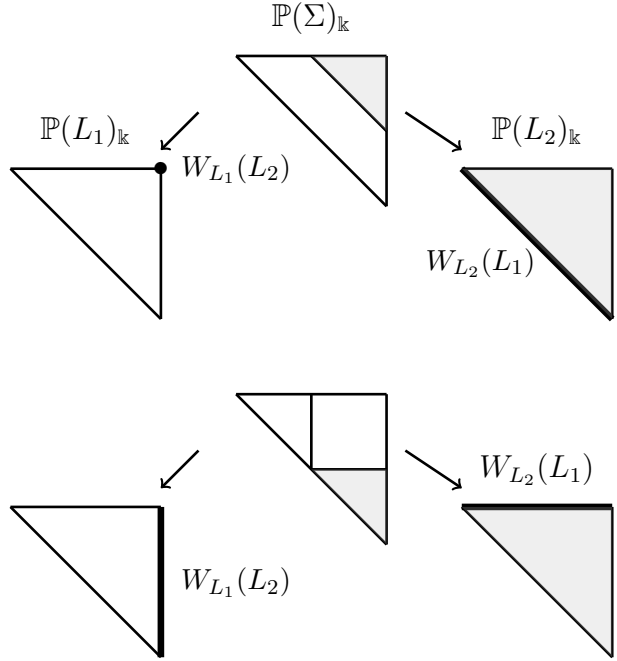

\begin{definition}[\cite{CHSW11}]
An irreducible component of the central fiber $\mathbb{P}(\Sigma)_\Bbbk$ 
is  {\em primary} if it projects birationally onto one of the $\mathbb{P}(L)_\Bbbk$, $L\in\Sigma$.
Other components are  {\em secondary}.
\end{definition}

\begin{lemma}
\label{smoothsectionifminimum}
Let $s=[v]\in\mathbb{P}^2(K)$ and let $\Sigma$ be a finite subset of $\mathfrak{B}_3^0$. 
We can write
$$\Sigma=\{[L_1],\ldots,[L_m]\}$$
for a unique choice of lattices
if we assume that $v\in L_j\setminus tL_j$ for all $j$. 
Suppose that one of the lattices $L_{i_0}\subseteq L_1\cap \ldots\cap L_m$. Then 
$\overline{s}(0)\in\mathbb{P}(\Sigma)_\Bbbk$ is a smooth point contained in the primary component corresponding to~$L_{i_0}$.
In other words, $\mathbb{P}(\Sigma)$ is smooth along the section~$s$.
\end{lemma}

\begin{proof}
For each $j\neq i_0$, fix an integer $z_j$ such that $tL_{i_0}\subseteq t^{z_j}L_j\setminus t^{z_j+1}L_j$. 
We  have $z_j>0$. Otherwise, if $z_j\leq0$, then $L_{i_0}\subseteq L_j\subseteq t^{z_j}L_j$, which gives a contradiction.
By \cite[\S5]{CHSW11}, the primary component corresponding to $L_{i_0}$ is obtained by blowing up $\mathbb{P}(L_{i_0})_\Bbbk$ along the proper linear subspaces $W_{L_{i_0}}(L_j)$ for $j\neq i_0$ (in particular, it is smooth). So, we would be done if we can show that the image of $v$ in $L_{i_0}/tL_{i_0}$ is not contained in $(t^{z_j}L_j\cap L_{i_0})/tL_{i_0}$ for all $j\neq i_0$. But this is clear, because $t^{z_j}L_j\subseteq tL_j$ 
and $v\notin tL_j$ by hypothesis.
\end{proof}

Finally, we study how the Mustafin join changes 
when we add a lattice to the configuration.

\begin{lemma}
\label{criterionforlinescontracted}
Let 
$\Sigma=\{[L_1],\ldots,[L_m]\}\subseteq \mathfrak{B}_3^0$ and $\Sigma'=\Sigma\cup\{[L_0]\}$. Let  $\widetilde{C}\subseteq\mathbb{P}(\Sigma')_\Bbbk$ be the primary component corresponding to $[L_0]$. Let $\ell\subseteq\mathbb{P}(L_0)_\Bbbk$ be a line and denote by $\widetilde{\ell}$ its strict transform in $\widetilde{C}$. Then $\widetilde{\ell}$ is contracted by the projection $\pi\colon\mathbb{P}(\Sigma')\rightarrow\mathbb{P}(\Sigma)$ if and only if $\ell$ intersects every linear subspace $W_{L_0}(L_j)\subseteq\mathbb{P}(L_0)_\Bbbk$. 
Otherwise, $\pi|_{\widetilde{\ell}}$ is injective.
\end{lemma}

\begin{proof}
Suppose $\ell$ intersects every linear subspace $W_{L_0}(L_j)$. For every~$j$, let $\Lambda=\{[L_0],[L_j]\}$, and consider the diagram $\mathbb{P}(L_0)_\Bbbk\leftarrow\mathbb{P}(\Lambda)_\Bbbk\rightarrow\mathbb{P}(L_j)_\Bbbk$. After inspecting all the possibilities for $\mathbb{P}(\Lambda)_\Bbbk$ (see Remark~\ref{allabouttwolattices}) and $W_{L_0}(L_j)$ (which is a point or a line), we can see that the image of $\widetilde{\ell}$ in $\mathbb{P}(L_j)_\Bbbk$ of $\widetilde{\ell}$ is a point. Since 
this is true for every $j$, 
$\pi$ contracts $\widetilde{\ell}$.

Conversely, if $\ell$ does not intersect a linear space $W_{L_0}(L_j)$ for some $j$ then by considering the same diagram (where now we only have to consider the case where $W_{L_0}(L_j)$ is a point), we see that the image of $\widetilde{\ell}$ in $\mathbb{P}(L_j)_\Bbbk$ is not contracted, in fact  $\pi|_{\widetilde{\ell}}$ is injective. 
\end{proof}

\begin{remark}
An almost identical statement appears in \cite[Lemma 5.10, Proposition~5.11]{CHSW11}, but with stronger hypotheses: 
we do not  require that $\pi|_{\widetilde{C}}$ is birational onto its image.
\end{remark}

\section{Stable lattices of arcs and their Mustafin joins}
\label{Mustafinjoinsofstablelattices}

Here we focus on Mustafin joins for arcs considered in \cite{GP91}.
Consider $K$-points 
$$a_1,\ldots,a_n\in\mathbb{P}^2(K)$$ in general linear position, 
i.e.~an arc $\mathbf{a}=(a_1,\ldots,a_n)\colon\Spec(K)\rightarrow\mathbf{B}_n$.  
We denote by $\Sigma_\mathbf{a}$ the subset of $\mathfrak{B}_3^0$ of stable lattice classes with respect to $\mathbf{a}$. 
Recall from the introduction that a lattice $\Lambda$ is stable
if at least four of the limits $\overline{a}^L_1(0),\ldots,\overline{a}^L_n(0)$ in the central fiber $\mathbb{P}(L)_\Bbbk\subseteq\mathbb{P}(L)$ are in general linear position. 
There exists a unique lattice class stabilizing any given quadruple \cite{KT06}, although the same lattice can stabilize several quadruples.
The Mustafin join $\mathbb{P}(\Sigma_\mathbf{a})$ has extra structure,
namely $n$ {\em sections} $\overline{a}_1,\ldots,\overline{a}_n$ which are defined as follows.
Every component $a_i$ can be viewed as a section of $\mathbb{P}(\Sigma_\mathbf{a})_K\rightarrow\Spec(K)$
using an isomorphism $\mathbb{P}(\Sigma_\mathbf{a})_K\cong\mathbb{P}_K^2$. By the valuative criterion of properness, 
$a_i$ admits a unique extension $\overline{a}_i\colon\Spec(R)\rightarrow\mathbb{P}(\Sigma_\mathbf{a})$, which is the claimed section.

\begin{lemma}\label{disjointsectionsmustafinjoin}
The sections $\overline{a}_1,\ldots,\overline{a}_n$ of the Mustafin join $\mathbb{P}(\Sigma_\mathbf{a})$ 
are pairwise disjoint.
\end{lemma}

\begin{proof}
It suffices to check the claim on the central fiber $\mathbb{P}(\Sigma_\mathbf{a})_\Bbbk$.
Let $1\le i<j\le n$. Write $\Sigma_\mathbf{a}=\{L_1,\ldots,L_m\}$. By the definition of stable lattices, 
there exists $h\in\{1,\ldots,m\}$ such that $\overline{a}_i(0),\overline{a}_j(0)\in\mathbb{P}(L_h)_\Bbbk$ are distinct. Then from the commutativity of the following diagram:
\begin{center}
\begin{tikzpicture}[>=angle 90]
\matrix(a)[matrix of math nodes,
row sep=2em, column sep=2em,
text height=1.5ex, text depth=0.25ex]
{\mathbb{P}(\Sigma_\mathbf{a})&\mathbb{P}(L_1)\times\cdots\times\mathbb{P}(L_m)\\
\Spec(R)&\mathbb{P}(L_h),\\};
\path[right hook->] (a-1-1) edge node[above]{}(a-1-2);
\path[->] (a-1-2) edge node[]{}(a-2-2);
\path[->] (a-2-1) edge node[]{}(a-1-1);
\path[->] (a-2-1) edge node[]{}(a-2-2);
\end{tikzpicture}
\end{center}
we can conclude that $\overline{a}_i(0),\overline{a}_j(0)\in\mathbb{P}(\Sigma_\mathbf{a})_k$ are also distinct.
\end{proof}

\begin{example}
\label{seconddegeneration5points}
Let $(a_1,\ldots,a_5)=([1:0:0],[0:1:0],[0:0:1],[1:1:1],[t^2:1:t])$.
Then
\[
\Sigma_\mathbf{a}=\{[e_1R+e_2R+e_3R],\ [t^2e_1R+e_2R+te_3R],\ [te_1R+e_2R+te_3R]\}.
\]
To prove this, 
we list all  possible quadruples ${i_1},\ldots,{i_4}$ and the corresponding stable lattice 
\begin{align*}
&\{1234\},\quad \{1345\} & ~e_1R+e_2R+e_3R,\\
&\{1235\}, \quad \{2345\} & ~t^2e_1R+e_2R+te_3R,\\
&\{1245\}&~te_1R+e_2R+te_3R.
\end{align*}
Using the procedure of Remark~\ref{degenerationfrompicture}, 
the central fiber 
$\mathbb{P}(\Sigma_\mathbf{a})_\Bbbk$ with sections is shown in Figure~\ref{degeneration5points2}.

\end{example}

\begin{figure}[hbtp]
\begin{tikzpicture}[scale=0.6]
	\draw[line width=1pt] (0,0) -- (0,1);
	\draw[line width=1pt] (0,0) -- (1,0);
	\draw[line width=1pt] (0,0) -- (-1,-1);
	\draw[line width=1pt] (1,0) -- (1+2,0);
	\draw[line width=1pt] (1,0) -- (1,3);
	\draw[line width=1pt] (1,0) -- (1-1.5,-1.5);
	\draw[line width=1pt] (0,1) -- (3,1);
	\draw[line width=1pt] (0,1) -- (0,1+2);
	\draw[line width=1pt] (0,1) -- (-1.5,1-1.5);

	\fill (0,0) circle (5pt);
	\fill (1,0) circle (5pt);
	\fill (0,1) circle (5pt);
\end{tikzpicture}
\hspace{1in}
\begin{tikzpicture}[scale=0.5]
	\draw[line width=1pt] (0,6) -- (6,6);
	\draw[line width=1pt] (6,0) -- (6,6);
	\draw[line width=1pt] (0,6) -- (6,0);
	\draw[line width=1pt] (2,4) -- (6,4);
	\draw[line width=1pt] (4,6) -- (4,2);

	\node at (3,5.2) {$\mathbb{F}_1$};

	\node at (5.2,3) {$\mathbb{F}_1$};

	\node at (5,5) {$\mathbb{F}_0$};
	
	\node at (3.47,3.47) {$\mathbb{P}^2$};

	\fill (0,6) circle (5pt);
	\fill (6,0) circle (5pt);
	\fill (6,6) circle (5pt);
	\fill (2,5) circle (5pt);
	\fill (5,2) circle (5pt);

	\node at (-0.5,6) {$1$};
	\node at (6.5,0) {$2$};
	\node at (6.5,6) {$3$};
	\node at (1.5,5.5) {$4$};
	\node at (5.5,1.5) {$5$};
\end{tikzpicture}
\caption{
The lattice classes in $\Sigma_\mathbf{a}$ (in the  apartment corresponding to the basis $e_1,e_2,e_3$) and the central fiber of the Mustafin join $\mathbb{P}(\Sigma_\mathbf{a})$ of Example~\ref{seconddegeneration5points}.} 
\label{degeneration5points2}
\end{figure}
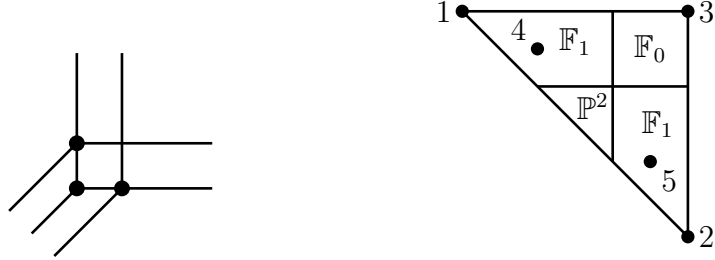


In the remainder of \S\ref{Mustafinjoinsofstablelattices} we will show that the Mustafin join of stable lattices $\mathbb{P}(\Sigma_\mathbf{a})\rightarrow\Spec(R)$
is smooth along 
the sections 
$\overline{a}_1,\ldots,\overline{a}_n$,
i.e.~
$\overline{a}_1(0),\ldots,\overline{a}_n(0)$
are smooth points on the central fiber $\mathbb{P}(\Sigma_\mathbf{a})_\Bbbk$. 
This was claimed in \cite[\S4.2]{GP91}, but no proof was given.
We fix vectors 
$$v_1,\ldots,v_n\in K^3$$ 
such that $a_i=[v_i]\in\bP^2(K)$ are in general linear position. 
After reordering, it suffices to prove this for the first section~$\overline{a}_1$.
We start by introducing 
canonical normalizations of vectors $v_1,\ldots,v_n$ as well as lattices $\Lambda$ for all lattice classes $[\Lambda]$ in the building.



\begin{definition}
We call a lattice $\Lambda$ \emph{normalized with respect to $v_1$} if $v_1\in\Lambda\setminus t\Lambda$.
\end{definition}

\begin{lemdef}\label{xfvxfbfb}
For any quadruple  $I\subseteq\{1,\ldots,n\}$,
write $\sum\limits_{i\in I}k_iv_i=0$, $k_i\in K\setminus\{0\}$.
The lattice $L_I:=\sum\limits_{i\in I}k_iv_iR$
stabilizes the quadruple $\{a_i\mid i\in I\}$. 
After rescaling the coefficients $k_i$, we can assume that all the lattices $L_I$ are normalized with respect to $v_1$.
\end{lemdef}

\begin{proof}
Since the vectors $\{v_i\mid i\in I\}$ are linearly dependent, 
we can write $\sum\limits_{i\in I}k_iv_i=0$ for some $k_i\in K$ not all 
equal to $0$.
In fact $k_i\neq0$ for all $i$ because any three vectors are linearly independent. In the basis $k_1v_1, k_2v_2, k_3v_3$ of $L_I$, 
the limits $a_i(0)$, $i\in I$, are 
$[1:0:0]$, $[0:1:0]$, $[0:0:1]$, $[1:1:1]$, respectively. Thus  $L_I$ stabilizes the quadruple $\{a_i\mid i\in I\}$.
\end{proof}

\begin{definition}
Fix an isomorphism $\varphi\colon\bigwedge^3K^3\xrightarrow{\cong}K$ and
let $\val\colon K\setminus\{0\}\to\mathbb{Z}$ be the usual valuation. 
For three linearly independent vectors $u,v,w$ in $K^3$, consider the log volume
\[
\|u,v,w\|=\val(\varphi(u\wedge v\wedge w)).
\]
\end{definition}

\begin{lemdef}
\label{vectorsnormalizedwrtanotherone}
We can rescale $v_2,\ldots,v_n$ uniquely
by $t^{p_2},\ldots,t^{p_n}$ so that 
\begin{enumerate}

\item For any quadruple $1,i,j,k$, we have $\|v_i,v_j,v_k\|\geq\|v_1,v_j,v_k\|$;

\item For all $i\neq1$, there exists a quadruple $1,i,j,k$ such that $\|v_i,v_j,v_k\|=\|v_1,v_j,v_k\|$.

\end{enumerate}
In this case we say that $v_1,\ldots,v_n$ are \emph{normalized with respect to $v_1$}.
\end{lemdef}

\begin{proof}
Let $i\in\{2,\ldots,n\}$ and define $p_i$ to be the least  integer such that
\[
\|t^{p_i}v_i,v_j,v_k\|=p_i+\|v_i,v_j,v_k\|\geq\|v_1,v_j,v_k\|,
\]
for all distinct $j,k\neq1,i$.
Note that the inequality depends not on $v_j, v_k$ but only on the corresponding points $a_j,a_k\in\bP^2(K)$.
By minimality of $p_i$, equality is achieved for at least one pair $j,k$. 
Therefore, $t^{p_2}v_2,\ldots,t^{p_n}v_n$ satisfy properties (1) and (2).
\end{proof}

\begin{lemma}
\label{latticeusingglobalnormalizationisstable}
Assume $v_1,\ldots,v_n$ are normalized with respect to $v_1$. If $1\in I\in\binom{[n]}{4}$, then 
\[
\Lambda_I=\{u\in K^3\mid\|u,v_j,v_k\|\geq\|v_1,v_j,v_k\|~\mathrm{for~all}~j,k\in I\setminus\{1\}\}
\]
is a lattice normalized with respect to $v_1$ and stabilizing the quadruple $a_i$, $i\in I$.
\end{lemma}

\begin{proof}
Assume $I=\{1,\ldots,4\}$. 
Write $v_1=c_2v_2+c_3v_3+c_4v_4$ for $c_2,c_3,c_4\in K$.
By Lemma--Definition~\ref{xfvxfbfb}, the lattice $L_I=c_2v_2R+c_3v_3R+c_4v_4R$ stabilizes $a_1,\ldots,a_4$.  By Cramer's rule,
\begin{equation}
\label{cramer1}
c_2=\frac{\varphi(v_1\wedge v_3\wedge v_4)}{\varphi(v_2\wedge v_3\wedge v_4)},~c_3=-\frac{\varphi(v_1\wedge v_2\wedge v_4)}{\varphi(v_2\wedge v_3\wedge v_4)},~c_4=\frac{\varphi(v_1\wedge v_2\wedge v_3)}{\varphi(v_2\wedge v_3\wedge v_4)}.
\end{equation}
Let $u\in L_I$. Then $u=r_2c_2v_2+r_3c_3v_3+r_4c_4v_4$ for some $r_2,r_3,r_4\in R$. By Cramer's rule,
\begin{equation}
\label{cramer2}
r_2c_2=\frac{\varphi(u\wedge v_3\wedge v_4)}{\varphi(v_2\wedge v_3\wedge v_4)},~r_3c_3=-\frac{\varphi(u\wedge v_2\wedge v_4)}{\varphi(v_2\wedge v_3\wedge v_4)},~r_4c_4=\frac{\varphi(u\wedge v_2\wedge v_3)}{\varphi(v_2\wedge v_3\wedge v_4)}.
\end{equation}
Since $\val(r_ic_i)\geq\val(c_i)$,  \eqref{cramer1} and \eqref{cramer2} imply that $\|u,v_j,v_k\|\geq\|v_1,v_j,v_k\|$. Thus $L_I\subseteq\Lambda_I$. For the other containment, let $u\in\Lambda_I$
and write $u=d_2v_2+d_3v_3+d_4v_4$ 
for some $d_2,d_3,d_4\in K$. Since $\|u,v_j,v_k\|\geq\|v_1,v_j,v_k\|$ for all $j,k\in I\setminus\{1\}$, we have
$\val(d_i)\geq\val(c_i)$ for all $i=2,3,4$ by Cramer's rule.
Thus $d_i=r_ic_i$ for some $r_i\in R$, implying that $u\in L_I$.
\end{proof}

\begin{corollary}
\label{intersectionallstablelatticeswrt1}
Assume $v_1,\ldots,v_n$ are normalized with respect to $v_1$. 
Write $\Sigma_\mathbf{a}=\{[L_j]\}$, where every $L_j$ is normalized with respect to $v_1$.
Then
\[
\Lambda=\bigcap_{[L_j]\in\Sigma_\mathbf{a}}L_j=\{u\in K^3\mid\|u,v_j,v_k\|\geq\|v_1,v_j,v_k\|~\mathrm{for~all}~j,k\in [n]\setminus\{1\}\}.
\]
\end{corollary}

\begin{proof}
We use the notation of Lemma--Definition~\ref{xfvxfbfb}. By Lemma~\ref{latticeusingglobalnormalizationisstable},
it suffices to prove that $\bigcap\limits_{I\in\binom{[n]}{4}}L_I=\bigcap\limits_{1\in I\in\binom{[n]}{4}}L_I$.

Let $I\in\binom{[n]}{4}$, $1\not\in I$. It suffices to show that there exists $J\in\binom{[n]}{4}$ such that $1\in J$ and $L_J\subseteq L_I$. By assumption, $v_1\in L_I$, hence we can write
\[
v_1=\sum_{i\in I}r_ik_iv_i,
\]
for some $r_i\in R$. Let $j\in I$ be arbitrary
and write $k_jv_j=\sum_{i\in I\setminus\{j\}}(-k_i)v_i$. By substituting into the previous expression for $v_1$ we obtain
\begin{equation}
\label{relationfornewlattice}
v_1=\sum_{i\in I\setminus\{j\}}(r_i-r_j)k_iv_i. 
\end{equation}
Let $\Lambda=\sum_{i\in I\setminus\{j\}}(r_i-r_j)k_iv_iR$. 
Since $\Lambda$ is contained in $L_I$, it suffices to show 
that $L_{(I\setminus\{j\})\cup\{1\}}=\Lambda$. By Lemma--Definition~\ref{xfvxfbfb} and \eqref{relationfornewlattice}, this is true up to rescaling, so we only have to show that $\Lambda$ is normalized with respect to $v_1$. By \eqref{relationfornewlattice} we have that $v_1\in\Lambda$, so we just need to show that $t^{-1}v_1\notin\Lambda$. But this is true because $t^{-1}v_1\notin L_I$.
\end{proof}

\begin{theorem}
\label{sectionssmoothcentralfiber}
Let $\mathbf{a}=(a_1,\ldots,a_n)\in(\mathbb{P}^2)^n(K)$ be in general linear position. Then the Mustafin join $\mathbb{P}(\Sigma_\mathbf{a})\rightarrow\Spec(R)$ is smooth
along the $n$ disjoint sections $\overline{a}_1,\ldots,\overline{a}_n$.
\end{theorem}

\begin{proof}
We need to show that the sections determine 
smooth points of the central fiber.
Up~to permuting the indices, it suffices to show that $\overline{a}_1(0)\in\mathbb{P}(\Sigma_a)_\Bbbk$ is a smooth point. Write
\[
a_1=[v_1],\ldots,a_n=[v_n],
\]
where $v_1,\ldots,v_n$ are normalized with respect to $v_1$ (Lemma--Definition~\ref{vectorsnormalizedwrtanotherone}). Using Lemma~\ref{latticeusingglobalnormalizationisstable}, we can write the set of stable lattices with respect to $\mathbf{a}$ as
\[
\Sigma_\mathbf{a}=\left\{[\Lambda_I]\;\middle|\;I\in\binom{[n]}{4}\right\},
\]
where  $\Lambda_I$  is normalized with respect to $v_1$ and 
stabilizes the quadruple $a_i$, $i\in I$. If the set $\Sigma_\mathbf{a}$ has a minimal element with respect to inclusion of lattices then we are done by Lemma~\ref{smoothsectionifminimum}. 

Suppose there is no minimal
element in $\Sigma_\mathbf{a}$ and consider $\Lambda$ of Corollary~\ref{intersectionallstablelatticeswrt1},  the intersection of all the $\Lambda_I$'s.
Then $[\Lambda]\not\in\Sigma_\mathbf{a}$.
Define $\Sigma_\mathbf{a}'=\Sigma_\mathbf{a}\cup\{[\Lambda]\}$. Let $\widetilde{C}$ be the primary component of $\mathbb{P}(\Sigma_\mathbf{a}')_\Bbbk$ corresponding to~$[\Lambda]$. Recall that $\widetilde{C}$ is the blow up of $\mathbb{P}(\Lambda)_\Bbbk$ at the linear subspaces $W_{\Lambda}(\Lambda_I)$ for $I\in\binom{[n]}{4}$ (see \cite[Proposition 5.6]{CHSW11}). By Lemma~\ref{smoothsectionifminimum}, 
$\overline{a}_i(0)\in\mathbb{P}(\Sigma_\mathbf{a}')_\Bbbk$ is a smooth point, and $\widetilde{C}$ is the primary component containing $\overline{a}_1(0)\in\mathbb{P}(\Sigma_\mathbf{a}')_\Bbbk$. In particular, $\overline{a}_1(0)\in\mathbb{P}(\Lambda)_\Bbbk$ is not contained in any of the linear subspaces $W_{\Lambda}(\Lambda_I)$, $I\in\binom{[n]}{4}$.

We have a morphism $\pi\colon\mathbb{P}(\Sigma_\mathbf{a}')\rightarrow\mathbb{P}(\Sigma_\mathbf{a})$. We will show later on in the proof
that $\pi|_{\widetilde{C}}$ is birational onto its image. For now, let us assume it. If $\ell\subseteq\mathbb{P}(\Lambda)_\Bbbk$ is a line, then we denote by $\widetilde{\ell}\subseteq\widetilde{C}$ its strict transform. 
By \cite[Lemma 5.10]{CHSW11}, $\widetilde{\ell}$ is contracted by $\pi$ if and only if $\ell$ intersects all linear subspaces $W_{\Lambda}(\Lambda_I)$. Moreover, by \cite[Proposition~5.11]{CHSW11}, the union of all contracted $\widetilde{\ell}$'s is the exceptional locus of $\pi|_{\widetilde{C}}$. So it suffices to show that there is no line $\ell\subseteq\mathbb{P}(\Lambda)_\Bbbk$ such that $\overline{a}_1(0)\in\widetilde{\ell}$ and 
$\ell$ intersects all the linear subspaces $W_{\Lambda}(\Lambda_I)$, $I\in\binom{[n]}{4}$. 
Arguing by contradiction, suppose $\ell$ exists. Since $\pi|_{\widetilde{C}}$ is birational, there exist $\Lambda_J,\Lambda_K$ inducing two distinct points
$
W_{\Lambda}(\Lambda_J),W_{\Lambda}(\Lambda_K)\in\mathbb{P}(\Lambda)_\Bbbk.
$
These  points lie on~$\ell$. 
By~Lemma~\ref{latticerestrictionpoint},  $W_{\Lambda}(\Lambda_J)=\overline{a}_j(0)$ and $W_{\Lambda}(\Lambda_K)=\overline{a}_k(0)$ for some sections $\overline{a}_j$ and $\overline{a}_k$. Therefore, $\overline{a}_1(0),\overline{a}_j(0),\overline{a}_k(0)$ are three distinct limit points on $\ell$. By Lemma~\ref{alignedlimitsstablelattice},  $\ell$ is a linear subspace of $\mathbb{P}(\Lambda)_\Bbbk$ induced by a stable lattice. But this is a contradiction, since we know that $\overline{a}_1(0)\in\mathbb{P}(\Lambda)_\Bbbk$ cannot lie on any of the linear spaces $W_{\Lambda}(\Lambda_I)$.

To conclude the proof, it remains to show that $\pi|_{\widetilde{C}}$ is birational onto its image. 
For any  $i\in\{2,\ldots,n\}$, there is a
quadruple $I=\{1,i,j,k\}$ such that 
$\|v_i,v_j,v_k\|=\|v_1,v_j,v_k\|$ (Lemma--Definition~\ref{vectorsnormalizedwrtanotherone}). 
We claim that $\overline{a}_i(0)\in\mathbb{P}(\Lambda)_\Bbbk$ is disjoint from $W_{\Lambda}(\Lambda_I)=\mathbb{P}((t^{z_I}\Lambda_I\cap\Lambda)/t\Lambda)$, where $z_I$ is that unique integer such that 
$t\Lambda\subseteq t^{z_I}\Lambda_I$ and $\Lambda\nsubseteq t^{z_I}\Lambda_I$.
Note that $z_I>0$.
It is enough to show that $v_i\notin t^{z_I}\Lambda_I\subseteq t\Lambda_I$. 
Assume by contradiction that $v_i\in t\Lambda_I$. So $v_i=tu$, where $u\in K^3$ satisfies $\|u,v_j,v_k\|\geq\|v_1,v_j,v_k\|$ (see Lemma~\ref{latticeusingglobalnormalizationisstable}). 
Hence
\[
\|u,v_j,v_k\|\geq\|v_1,v_j,v_k\|=\|v_i,v_j,v_k\|=\|tu,v_j,v_k\|=1+\|u,v_j,v_k\|,
\]
which is impossible.
Next we claim that $W_{\Lambda}(\Lambda_I)$ is a point. If not, then both $\overline{a}_1(0),\overline{a}_i(0)\in\mathbb{P}(\Lambda)_\Bbbk$  lie outside the line $W_{\Lambda}(\Lambda_I)$. On the other hand,
the limits $\overline{a}_1(0),\overline{a}_i(0)\in\mathbb{P}(\Lambda_I)_\Bbbk$ are distinct because  $\Lambda_I$ stabilizes $a_1,a_i,a_j,a_k$. 
Analyzing both 
possible diagrams $\mathbb{P}(\Lambda)_\Bbbk\leftarrow\mathbb{P}(\Lambda,\Lambda_I)_\Bbbk\rightarrow\mathbb{P}(\Lambda_I)_\Bbbk$ of Figure~\ref{diagramsfortwolatticesandinducedlinearspaces},
we arrive at a contradiction.
Since $W_{\Lambda}(\Lambda_I)$ is a point, by Lemma~\ref{latticerestrictionpoint} it is equal to $\overline{a}_h(0)$ for some $h\in[n]\setminus\{1\}$. Now repeat the same argument above with $i$ replaced by $h$ to find a second induced linear subspace which is a point distinct from $\overline{a}_1(0)$ and $W_{\Lambda}(\Lambda_I)$.
By Lemma~\ref{criterionforlinescontracted}, this implies that $\pi|_{\widetilde{C}}$ is birational onto its image.
\end{proof}

\begin{lemma}
\label{latticerestrictionpoint}
Let $\mathbf{a}=(a_1,\ldots,a_n)\in(\mathbb{P}^2)^n(K)$ be in general linear position and let $[L_0]\in\mathfrak{B}_3^0$ be arbitrary. Let $[L]\in\Sigma_\mathbf{a}$ be a stable lattice class and assume that the linear subspace $W_{L_0}(L)\subseteq\mathbb{P}(L_0)_\Bbbk$ is a point. Then there exist distinct $i,j\in[n]$ such that
\[
\overline{a}_i(0)=W_{L_0}(L)=\overline{a}_j(0).
\]
\end{lemma}

\begin{proof}
Up to relabeling the indices, assume that $a_1,\ldots,a_4$ is the quadruple stabilized by $[L]$. We show that at least two among $\overline{a}_1(0),\ldots,\overline{a}_4(0)\in\mathbb{P}(L_0)_\Bbbk$ equal $W_{L_0}(L)$. Let $\Sigma=\{[L_0],[L]\}$ and consider the diagram $\mathbb{P}(L_0)_\Bbbk\leftarrow\mathbb{P}(\Sigma)_\Bbbk\rightarrow\mathbb{P}(L)_\Bbbk$.
Since $W_{L_0}(L)$ is a point, by Remark~\ref{allabouttwolattices},
$\mathbb{P}(\Sigma)_\Bbbk$ is isomorphic to the gluing of $\mathbb{P}^2$ and $\mathbb{F}_1$ along a line and the exceptional divisor respectively. The morphism $\mathbb{P}(\Sigma)_\Bbbk\rightarrow\mathbb{P}(L_0)_\Bbbk$ is the blow up at the point $W_{L_0}(L)$, and the morphism $\mathbb{P}(\Sigma)_\Bbbk\rightarrow\mathbb{P}(L)_\Bbbk$ contracts $\mathbb{F}_1$ along the ruling. Therefore, if no two points among $\overline{a}_1(0),\ldots,\overline{a}_4(0)\in\mathbb{P}(L_0)_\Bbbk$ are equal to $W_{L_0}(L)$, then  $\overline{a}_1(0),\ldots,\overline{a}_4(0)\in\mathbb{P}(L)_\Bbbk$ would not be in general linear position.
\end{proof}

\begin{lemma}
\label{alignedlimitsstablelattice}
Let $\mathbf{a}=(a_1,\ldots,a_n)\in(\mathbb{P}^2)^n(K)$ be in general linear position and let $[L_0]\in\mathfrak{B}_3^0$ be arbitrary. If three limit points $\overline{a}_i(0),\overline{a}_j(0),\overline{a}_k(0)\in\mathbb{P}(L_0)_\Bbbk$ are distinct and contained in a line $\ell$, then any $[L]\in\Sigma_\mathbf{a}$ stabilizing $a_i,a_j,a_k$ has the property that $W_{L_0}(L)=\ell$.
\end{lemma}

\begin{proof}
Let $[L]\in\Sigma_\mathbf{a}$ be any stable lattice class stabilizing a quadruple that includes $a_i,a_j,a_k$. Define $\Sigma=\{[L_0],[L]\}$ and consider the diagram $\mathbb{P}(L_0)_\Bbbk\leftarrow\mathbb{P}(\Sigma)_\Bbbk\rightarrow\mathbb{P}(L)_\Bbbk$, where in $\mathbb{P}(L_0)_\Bbbk$ the three limit points $\overline{a}_i(0),\overline{a}_j(0),\overline{a}_k(0)$ are aligned, but the same limits in $\mathbb{P}(L)_\Bbbk$ are in general linear position.
By Remark~\ref{allabouttwolattices}, there are two possibilities for $\mathbb{P}(\Sigma)_\Bbbk$. Considering all the possible ways the three aligned limit points in $\mathbb{P}(L_0)_\Bbbk$ can be related with respect to $W_{L_0}(L)$, and considering that these limits are in general linear position in $\mathbb{P}(L)_\Bbbk$, we see that $W_{L_0}(L)$ is a line and $\overline{a}_i(0),\overline{a}_j(0),\overline{a}_k(0)\in W_{L_0}(L)$.
\end{proof}


\section{Universal Mustafin join for arcs}
\label{UniversalMustafin}


\begin{example}\label{familyoverBn}
We start with  a counter-example to the claim \cite[\S4, Proposition~1]{GP91}
that $\overline{\mathbf{B}}_n$ is the moduli space of Mustafin joins for arcs.
Define arcs $\mathbf{a},\mathbf{b}\colon\Spec(K)\rightarrow\mathbf{B}_6$,
\[
\mathbf{a}(t)=([1:0:0],[0:1:0],[0:0:1],[1:1:1],[1:t:t^2],[t:1:t^2]),
\]
\[
\mathbf{b}(t)=([1:0:0],[0:1:0],[0:0:1],[1:1:1],[1:t^2:t^3],[t:1:t^3]).
\]

We claim that the limit points $\overline{\mathbf{a}}(0),\overline{\mathbf{b}}(0)\in\overline{\mathbf{B}}_6$ are equal. 
Recall from \S\ref{explicitclosedembedding} that $\overline{\mathbf{B}}_6$ can be embedded in $\prod_{\mathcal{Q}_6'}\mathbb{P}^2$, where $\mathcal{Q}_6'$ is the set of ordered quintuples $v=(v_1,\ldots,v_5)$ of distinct indices in $\{1,\ldots,6\}$ such that $v_1<\ldots<v_4$. We claim that for all $30$ quintuples $v\in\mathcal{Q}_6'$,
\[
\lim_{t\rightarrow0}\psi_{a_{v_1}(t),\ldots,a_{v_4}(t)}(a_{v_5}(t))=\lim_{t\rightarrow0}\psi_{b_{v_1}(t),\ldots,b_{v_4}(t)}(b_{v_5}(t)),
\]
where the morphism $\psi$ is defined in \S\ref{explicitclosedembedding}. 
The limits are computed in Table~\ref{30crossratios}.

Next we compute $\mathbb{P}(\Sigma_\mathbf{a})_\Bbbk$ and $\mathbb{P}(\Sigma_\mathbf{b})_\Bbbk$. The respective stable lattices are given by
\begin{align*}
\Sigma_\mathbf{a}=\{&[e_1R+e_2R+e_3R],[te_1R+e_2R+t^2e_3R],[e_1R+te_2R+t^2e_3R],\\
&[e_1R+e_2R+t^2e_3R],[e_1R+te_2R+te_3R],[te_1R+e_2R+te_3R]\},
\end{align*}
\begin{align*}
\Sigma_\mathbf{b}=\{&[e_1R+e_2R+e_3R],[te_1R+e_2R+t^3e_3R],[e_1R+t^2e_2R+t^3e_3R],\\
&[te_1R+e_2R+te_3R],[e_1R+t^2e_2R+t^2e_3R],[e_1R+e_2R+t^3e_3R]\}.
\end{align*}
We compute the central fibers $\mathbb{P}(\Sigma_\mathbf{a})_\Bbbk\not\cong\mathbb{P}(\Sigma_\mathbf{b})_\Bbbk$ using 
Remark~\ref{degenerationfrompicture} (see Figure~\ref{spidersforaandb}). The surfaces are illustrated in Figure~\ref{GPcounterexample}, where $\mathbb{P}(\Sigma_\mathbf{a})_\Bbbk$ is on the left and $\mathbb{P}(\Sigma_\mathbf{b})_\Bbbk$ is on the right.

\begin{table}[hbtp]
\centering
\begin{tabular}{c|c||c|c}
\toprule
Quintuple & Limit in $\mathbb{P}^2$ & Quintuple & Limit in $\mathbb{P}^2$\\
\hline
$12345$ & $[1:0:0]$ & $13465$ & $[1:0:0]$\\
$12346$ & $[0:1:0]$ & $13562$ & $[1:1:1]$\\
$12354$ & $[0:0:1]$ & $13564$ & $[0:1:0]$\\
$12356$ & $[0:1:0]$ & $14562$ & $[1:1:1]$\\
$12364$ & $[0:0:1]$ & $14563$ & $[0:1:0]$\\
$12365$ & $[1:0:0]$ & $23451$ & $[1:1:1]$\\
$12453$ & $[0:0:1]$ & $23456$ & $[1:0:0]$\\
$12456$ & $[0:1:0]$ & $23461$ & $[0:1:1]$\\
$12463$ & $[0:0:1]$ & $23465$ & $[0:1:1]$\\
$12465$ & $[1:0:0]$ & $23561$ & $[0:0:1]$\\
$12563$ & $[1:0:1]$ & $23564$ & $[0:1:0]$\\
$12564$ & $[1:0:1]$ & $24561$ & $[0:0:1]$\\
$13452$ & $[0:1:1]$ & $24563$ & $[0:1:0]$\\
$13456$ & $[0:1:1]$ & $34561$ & $[0:0:1]$\\
$13462$ & $[1:1:1]$ & $34562$ & $[1:1:1]$\\
\bottomrule
\end{tabular}
\caption{Coordinates of the limit $\overline{\mathbf{a}}(0)=\overline{\mathbf{b}}(0)\in\overline{\mathbf{B}}_6\subseteq\prod_{Q_6'}\mathbb{P}^2$. 
}
\label{30crossratios}
\end{table}

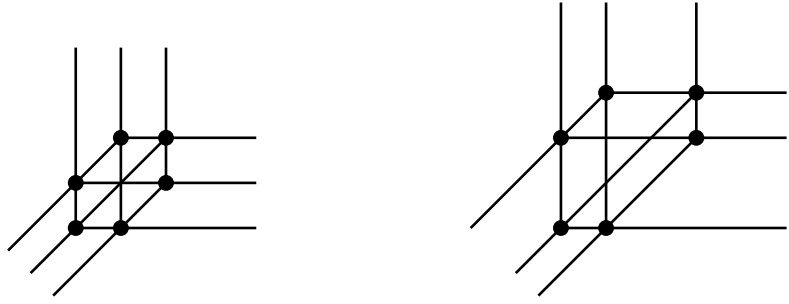
\begin{figure}[hbtp]
\begin{tikzpicture}[scale=0.6]
	\draw[line width=1pt] (0,0) -- (1,0);
	\draw[line width=1pt] (0,0) -- (0,1);
	\draw[line width=1pt] (1,0) -- (1+3,0);
	\draw[line width=1pt] (1,0) -- (1,2);
	\draw[line width=1pt] (0,1) -- (2,1);
	\draw[line width=1pt] (0,1) -- (0,1+3);
	\draw[line width=1pt] (1,2) -- (1+1,2);
	\draw[line width=1pt] (1,2) -- (1,2+2);
	\draw[line width=1pt] (1,2) -- (1-2.5,2-2.5);
	\draw[line width=1pt] (2,1) -- (2+2,1);
	\draw[line width=1pt] (2,1) -- (2,1+1);
	\draw[line width=1pt] (2,1) -- (2-2.5,1-2.5);
	\draw[line width=1pt] (2,2) -- (2+2,2);
	\draw[line width=1pt] (2,2) -- (2,2+2);
	\draw[line width=1pt] (2,2) -- (2-3,2-3);

	\fill (0,0) circle (5pt);
	\fill (1,0) circle (5pt);
	\fill (0,1) circle (5pt);
	\fill (1,2) circle (5pt);
	\fill (2,1) circle (5pt);
	\fill (2,2) circle (5pt);
\end{tikzpicture}
\hspace{1in}
\begin{tikzpicture}[scale=0.6]
	\draw[line width=1pt] (3,3) -- (3+2,3);
	\draw[line width=1pt] (3,3) -- (3,3+2);
	\draw[line width=1pt] (3,3) -- (3-4,3-4);
	\draw[line width=1pt] (3,2) -- (3+2,2);
	\draw[line width=1pt] (3,2) -- (3,2+1);
	\draw[line width=1pt] (3,2) -- (3-3.5,2-3.5);
	\draw[line width=1pt] (1,3) -- (1+2,3);
	\draw[line width=1pt] (1,3) -- (1,3+2);
	\draw[line width=1pt] (1,3) -- (1-3,3-3);
	\draw[line width=1pt] (0,2) -- (0,2+3);
	\draw[line width=1pt] (0,2) -- (0+3,2);
	\draw[line width=1pt] (0,2) -- (0,2);
	\draw[line width=1pt] (1,0) -- (1,0+3);
	\draw[line width=1pt] (1,0) -- (1+4,0);
	\draw[line width=1pt] (0,0) -- (0+1,0);
	\draw[line width=1pt] (0,0) -- (0,0+2);

	\fill (0,0) circle (5pt);
	\fill (1,0) circle (5pt);
	\fill (0,2) circle (5pt);
	\fill (1,3) circle (5pt);
	\fill (3,2) circle (5pt);
	\fill (3,3) circle (5pt);
\end{tikzpicture}
\caption{The lattice classes in $\Sigma_\mathbf{a}$ and $\Sigma_\mathbf{b}$ in the standard apartment. 
}
\label{spidersforaandb}
\end{figure}

\begin{figure}[hbtp]
\begin{tikzpicture}[scale=0.5]
	\draw[line width=1pt] (0,9) -- (9,9);
	\draw[line width=1pt] (9,0) -- (9,9);
	\draw[line width=1pt] (0,9) -- (9,0);

	\draw[line width=1pt] (3,6) -- (4.5,6);
	\draw[line width=1pt] (7.5,6) -- (9,6);
	\draw[line width=1pt] (6,4.5) -- (7.5,4.5);
	\draw[line width=1pt] (6,4.5) -- (6,3);
	\draw[line width=1pt] (7.5,4.5) -- (7.5,6);
	\draw[line width=1pt] (7.5,4.5) -- (9,3);
	\draw[line width=1pt] (4.5,6) -- (6,4.5);
	\draw[line width=1pt] (4.5,6) -- (4.5,7.5);
	\draw[line width=1pt] (4.5,7.5) -- (6,7.5);
	\draw[line width=1pt] (6,7.5) -- (7.5,6);
	\draw[line width=1pt] (6,9) -- (6,7.5);
	\draw[line width=1pt] (3,9) -- (4.5,7.5);

	\draw[line width=1pt] (10,9) -- (19,9);
	\draw[line width=1pt] (19,0) -- (19,9);
	\draw[line width=1pt] (10,9) -- (19,0);

	\draw[line width=1pt] (13,6) -- (14.5,6);
	\draw[line width=1pt] (17.5,6) -- (19,6);
	\draw[line width=1pt] (16,4.5) -- (17.5,4.5);
	\draw[line width=1pt] (16,4.5) -- (16,3);
	\draw[line width=1pt] (17.5,4.5) -- (17.5,6);
	\draw[line width=1pt] (17.5,4.5) -- (19,3);
	\draw[line width=1pt] (14.5,6) -- (16,4.5);
	\draw[line width=1pt] (14.5,6) -- (14.5,7.5);
	\draw[line width=1pt] (14.5,7.5) -- (16,7.5);
	\draw[line width=1pt] (16,7.5) -- (17.5,6);
	\draw[line width=1pt] (16,9) -- (16,7.5);
	\draw[line width=1pt] (13,9) -- (14.5,7.5);
	\draw[line width=1pt] (14.5,6) -- (16,6);
	\draw[line width=1pt] (16,6) -- (16,7.5);
	\draw[line width=1pt] (16,6) -- (17.5,4.5);

	\node at (5,4.9) {$\mathbb{F}_1$};
	\node at (15,4.9) {$\mathbb{F}_1$};

	\node at (8.4,4.9) {$\mathbb{F}_1$};
	\node at (18.4,4.9) {$\mathbb{F}_1$};

	\node at (5,8.3) {$\mathbb{F}_1$};
	\node at (15,8.3) {$\mathbb{F}_1$};

	\node at (3,7.5) {$\Bl_1\mathbb{F}_0$};
	\node at (13,7.5) {$\Bl_1\mathbb{F}_0$};

	\node at (7.7,7.5) {$\Bl_1\mathbb{F}_0$};
	\node at (17.7,7.5) {$\Bl_1\mathbb{F}_0$};

	\node at (7.7,2.8) {$\Bl_1\mathbb{F}_0$};
	\node at (17.7,2.8) {$\Bl_1\mathbb{F}_0$};

	\node at (15.3,6.7) {$\mathbb{F}_0$};
	\node at (16.8,6.1) {$\mathbb{F}_0$};
	\node at (16,5.3) {$\mathbb{F}_0$};

	\node at (6,6) {$\Bl_3\mathbb{P}^2$};
\end{tikzpicture}
\caption{Example of $\mathbb{P}(\Sigma_\mathbf{a})_\Bbbk$ and $\mathbb{P}(\Sigma_\mathbf{b})_\Bbbk$ with $\overline{\mathbf{a}}(0)=\overline{\mathbf{b}}(0)\in\overline{\mathbf{B}}_6$.}
\label{GPcounterexample}
\end{figure}
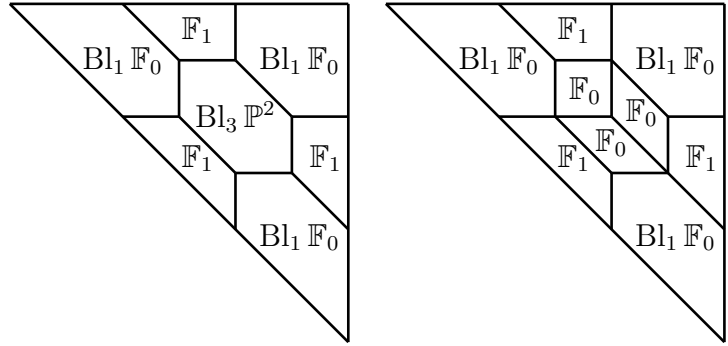
\end{example}

Let us also pinpoint a mistake in the proof of \cite[\S4, Proposition~1]{GP91}
and then explain how to construct the correct moduli space of Mustafin joins. Consider the diagram
\begin{equation}\label{sdgegefg}
\begin{aligned}
\begin{tikzpicture}[>=angle 90]
\matrix(a)[matrix of math nodes,
row sep=2em, column sep=2em,
text height=1.5ex, text depth=0.25ex]
{\mathbf{B}_n\times\mathbb{P}^2&\mathbf{B}_n\times(\mathbb{P}^2)^{\binom{n}{4}}\\
\mathbf{B}_n,&\\};
\path[right hook->] (a-1-1) edge node[above]{}(a-1-2);
\path[->] (a-1-1) edge node[left]{}(a-2-1);
\path[->] (a-1-2) edge node[right]{}(a-2-1);
\end{tikzpicture}
\end{aligned}
\end{equation}
where the horizontal map $\Psi$ sends 
$$((p_1,\ldots,p_n),p)\mapsto((p_1,\ldots,p_n),(\psi_{p_{v_1},\ldots,p_{v_4}}(p))_{v}).$$ (Recall from \S\ref{explicitclosedembedding} that, for any quadruple $v_1,\ldots,v_4$, $\psi_{p_{v_1},\ldots,p_{v_4}}$ is the unique element of $\PGL_3$ sending $p_{v_1},\ldots,p_{v_4}$ to the standard frame.) Let $\overline{\mathbf{F}}_n$ be the closure of the image of $\Psi$ in 
$\overline{\mathbf{B}}_n\times(\mathbb{P}^2)^{\binom{n}{4}}$ and consider
a morphism $\mu\colon\overline{\mathbf{F}}_n\to \overline{\mathbf{B}}_n$. It was assumed in \cite[\S4, Proposition~1]{GP91}
that formation of~$\mu$ commutes with arbitrary base-changes $S\to{\mathbf{B}}_n$.
However, this is wrong (Example~\ref{familyoverBn}), 
in particular $\mu$ is not flat. A  remedy is provided by the Grothendieck's universal flattening morphism as in \cite{Ray72}. Consider the morphism
$$\Phi\colon{\mathbf{B}}_n\to \overline{\mathbf{B}}_n\times\Hilb\left((\mathbb{P}^2)^{\binom{n}{4}}\right),$$
where the first component is the inclusion and the second component sends $(p_1,\ldots,p_n)\in{\mathbf{B}}_n$
to $\Psi((p_1,\ldots,p_n)\times\bP^2)$. By $\mathrm{Hilb}\left((\mathbb{P}^2)^{\binom{n}{4}}\right)$ we mean the connected component of the Hilbert scheme of closed subschemes in $(\mathbb{P}^2)^{\binom{n}{4}}$ which parametrizes the diagonally embedded $\mathbb{P}^2$ in $(\mathbb{P}^2)^{\binom{n}{4}}$. The correct moduli space is the closure of the image of $\Phi$ and the universal Mustafin join for arcs
is the pullback of the universal family of the Hilbert scheme. We~will analyze this construction in  detail, although
we will use the multigraded Hilbert scheme of $(\mathbb{P}^2)^{\binom{n}{4}}$
instead of the usual Hilbert scheme in order to have Proposition~\ref{sRGsrgwrH}.

\begin{definition}[\cite{HS04}]
For a commutative ring $k$, let $S=k[x_1,\ldots,x_n]$ be a polynomial ring with a grading by an abelian group~$A$
given by a semigroup homomorphism $\deg\colon\mathbb{N}^n\rightarrow A$. 
Fix a function 
$$h\colon A\rightarrow\mathbb{N}.$$ 
The {\em multigraded Hilbert scheme} $\mathbf{H}^h_S$ parametrizes all 
ideals in $S$ homogeneous with respect to $\deg$ and with Hilbert function $h$, i.e.~such that
$$\dim_k S_a/I_a=h(a)\quad \hbox{\rm for every}\quad a\in A.$$ 
For every $k$-algebra $R$, the set of $R$-points
$\mathbf{H}^h_S(R)$ is the set of homogeneous ideals $I\subseteq R\otimes_k S$
such that $(R\otimes S)_a/I_a$ is a locally free $R$-module of rank $h(a)$  for every~$a$.
The scheme $\mathbf{H}_S^h$ is quasi-projective over $k$, in fact 
projective if $1\in S$ is the only monomial of degree $0$.
\end{definition}

\begin{example}\label{srgsrgr}
We are interested in the multigraded Hilbert scheme of 
$(\bP^2)^N$.
Here 
$$S=\Bbbk[x_{ij}\mid i=1,2,3,\ j=1,\ldots,N],$$ 
where $\Bbbk$ is our algebraically closed base field and $[x_{1j}:x_{2j}:x_{3j}]$ are homogeneous
 coordinates of the $j$-th copy of~$\mathbb{P}^2$,
with the usual multigrading by~$\mathbb{Z}^N$. 
The numerical function $h$ is the Hilbert function of the diagonally 
 embedded~$\bP^2\hookrightarrow(\bP^2)^N$.
A detailed study of this case can be found in \cite{CS10}. 
The authors prove that $\mathbf{H}_S^h$ is connected and all ideals parametrized by it are radical and Cohen--Macaulay.
The morphism from $\mathbf{H}_S^h$ to the  Hilbert scheme of $(\bP^2)^N$
given by taking $\Proj(R\otimes_\Bbbk S)/I$ is injective on $\Bbbk$-points,
although it is not clear if it is a closed embedding. It follows that 
\end{example}

\begin{lemma}\label{sdfvsdfv}
The projection from $(\mathbb{P}^2)^{N}$ to $(\mathbb{P}^2)^{N-1}$ induces a morphism of multigraded 
Hilbert schemes  (with the Hilbert function of the diagonal). 
\end{lemma}

\begin{proof}
Let $S$ and $S'$ be multigraded coordinate rings of $(\mathbb{P}^2)^{N}$ and $(\mathbb{P}^2)^{N-1}$. The natural transformation of functors of points
takes a multihomogeneous ideal $I\subseteq R\otimes_\Bbbk S$ to the ideal of the projection,
which is $I\cap (R\otimes_\Bbbk S')$.
Its numerical function is $h'=h(a_1,\ldots,a_{N-1},0)$. 
\end{proof}


\begin{definition}
\label{definitionofthecorrectcompactification}
Let $\mathbf{H}_S^h$ be the multigraded Hilbert scheme
of $(\mathbb{P}^2)^{\binom{n}{4}}$ 
as in Example~\ref{srgsrgr}
 (with $N=\binom{n}{4}$).
The diagram \eqref{sdgegefg} induces an embedding 
$$\Phi:\,\mathbf{B}_n\hookrightarrow\overline{\mathbf{B}}_n\times\mathbf{H}_S^h.$$ 
Let $\XGP(3,n)$ be the Zariski closure of the image of $\Phi$. 
 Let $\overline{\mathcal{M}}\rightarrow\XGP(3,n)$ 
be the pullback of the 
universal family of the (usual, not multigraded) Hilbert scheme.
\end{definition}

\begin{theorem}
\label{indeedcorrectfamily}
The family $\overline{\mathcal{M}}$ is the universal Mustafin join for point configurations.
Concretely, take an arc $\mathbf{a}\colon\Spec(K)\rightarrow\mathbf{B}_n$ and its unique extension $\overline{\mathbf{a}}\colon\Spec(R)\rightarrow\XGP(3,n)$. Then $\overline{\mathbf{a}}^*\overline{\mathcal{M}}$ is isomorphic to the Mustafin join $\mathbb{P}(\Sigma_\mathbf{a})$, where $\Sigma_\mathbf{a}$ is the set of stable lattices with respect to $\mathbf{a}$.
\end{theorem}

\begin{proof}
Recall that $\mathbb{P}(\Sigma_\mathbf{a})\rightarrow\Spec(R)$ is flat and proper, and that $\mathbb{P}(\Sigma_\mathbf{a})$ is the Zariski closure of $\Spec(K)\times\mathbb{P}^2$ inside $\prod\limits_{L\in\Sigma_\mathbf{a}}\mathbb{P}(L)$. Therefore, we have the following commutative diagram:
\begin{center}
\begin{tikzpicture}[>=angle 90]
\matrix(a)[matrix of math nodes,
row sep=2em, column sep=2em,
text height=1.5ex, text depth=0.25ex]
{\mathbb{P}(\Sigma_\mathbf{a})&\Spec(R)\times(\mathbb{P}^2)^{\binom{n}{4}}\\
\Spec(R).&\\};
\path[right hook->] (a-1-1) edge node[above]{}(a-1-2);
\path[->] (a-1-1) edge node[left]{}(a-2-1);
\path[->] (a-1-2) edge node[right]{}(a-2-1);
\end{tikzpicture}
\end{center}
The horizontal map comes from the fact that for each quadruple, there is a unique stable lattice $L\in\Sigma_\mathbf{a}$ stabilizing it. 
The diagram  induces a morphism from $\Spec(R)$
to the (usual) Hilbert scheme of $(\mathbb{P}^2)^{\binom{n}{4}}$.
By its universal property, we have that $\mathbb{P}(\Sigma_\mathbf{a})\cong\overline{\mathbf{a}}^*\overline{\mathcal{M}}$.
\end{proof}

\begin{remark}
The corresponding $R$-point of $\mathbf{H}^h_S$, i.e.~the multi-homogeneous ideal of $\mathbb{P}(\Sigma_\mathbf{a})$,
can be computed as in Remark~\ref{equationsmustafinvariety}.
\end{remark}

\begin{corollary}
We have a birational morphism $\XGP(3,n)\to\overline{\mathbf{B}}_n$.
In particular, we have a morphism of normalizations $\XGP(3,n)^\nu\to\overline{\mathbf{B}}_n^\nu\cong\overline{\mathbf{X}}(3,n)^\nu$.
\end{corollary}

\begin{proof}
This follows from Definition~\ref{definitionofthecorrectcompactification}. As $\overline{\mathbf{X}}_{\GP}(3,n)\subseteq\overline{\mathbf{B}}_n\times\mathbf{H}_S^h$, we have a morphism $\overline{\mathbf{X}}_{\GP}(3,n)\rightarrow\overline{\mathbf{B}}_n$ given by the restriction of the projection onto the first factor. This morphism is birational as it restricts to the identity on $\mathbf{B}_n$.
\end{proof}


\begin{theorem}\label{sgsrgRWGwrhr}
We have an isomorphism $\XGP(3,5)\cong \overline{\mathbf{B}}_5\cong\overline{\mathrm{M}}_{0,5}$.
\end{theorem}

\begin{proof}
Since $\overline{\mathbf{B}}_5\cong\overline{\mathrm{M}}_{0,5}$ is normal (in fact smooth), it suffices to exhibit
a family of Mustafin joins for point configurations over $\overline{\mathrm{M}}_{0,5}$ inducing a morphism to $\Hilb\left((\mathbb{P}^2)^{\binom{5}{4}}\right)$.

The space $\overline{\mathrm{M}}_{0,5}$ is isomorphic to Hassett's moduli space $\overline{\mathrm{M}}_{0,(\frac{1}{2})^5}$ \cite{Has03}.
This gives a smooth conic bundle $\pi\colon\overline{\mathcal{C}}\rightarrow\overline{\mathrm{M}}_{0,5}$
with sections $s_1,\ldots,s_5$ such that
at most two 
sections are equal on each fiber.
The relative anti-canonical divisor $-K_\pi$ induces an embedding of $\overline{\mathcal{C}}$ into the $\mathbb{P}^2$-bundle $\overline{\mathcal{P}}'''$ with an associated vector bundle the pushforward of $-K_\pi$.
So we have a $\mathbb{P}^2$-bundle $\overline{\mathcal{P}}'''\rightarrow\overline{\mathrm{M}}_{0,5}$ 
with $5$ sections. Its fibers over a general point of a $1$-stratum (resp.~over a $0$-stratum) of $\overline{\mathrm{M}}_{0,5}$ are illustrated in the left picture of Figure~\ref{modificationsfibersinfamily1} (resp.~the left picture of Figure~\ref{modificationsfibersinfamily2}).

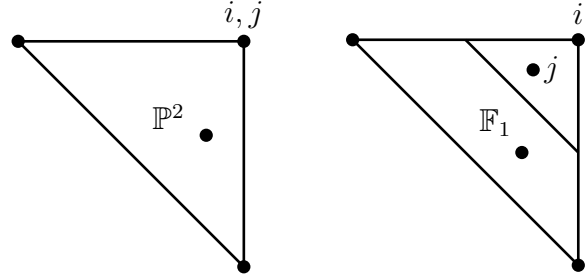
\begin{figure}[hbtp]
\centering
\begin{tabular}{>{\centering\arraybackslash}m{4cm}>{\centering\arraybackslash}m{4cm}}
\begin{tikzpicture}[scale=0.5]

	\draw[line width=1pt] (0,6) -- (6,6);
	\draw[line width=1pt] (6,0) -- (6,6);
	\draw[line width=1pt] (0,6) -- (6,0);
	\fill (0,6) circle (5pt);
	\fill (6,0) circle (5pt);
	\fill (6,6) circle (5pt);
	\fill (5,3.5) circle (5pt);

	\node at (6,6.7) {$i,j$};
	\node at (4,4) {$\mathbb{P}^2$};

\end{tikzpicture}
&
\begin{tikzpicture}[scale=0.5]

	\draw[line width=1pt] (0,6) -- (6,6);
	\draw[line width=1pt] (6,0) -- (6,6);
	\draw[line width=1pt] (0,6) -- (6,0);
	\draw[line width=1pt] (3,6) -- (6,3);
	\fill (0,6) circle (5pt);
	\fill (6,0) circle (5pt);
	\fill (6,6) circle (5pt);
	\fill (4.5,3) circle (5pt);
	\fill (4.8,5.2) circle (5pt);

	\node at (6,6.7) {$i$};
	\node at (5.3,5.2) {$j$};
	\node at (3.8,3.8) {$\mathbb{F}_1$};

\end{tikzpicture}
\end{tabular}
\caption{Fibers of families over a general point of the $1$-strata.}
\label{modificationsfibersinfamily1}
\end{figure}

\begin{figure}[hbtp]
\centering
\begin{tabular}{>{\centering\arraybackslash}m{3.5cm}>{\centering\arraybackslash}m{3.5cm}>{\centering\arraybackslash}m{3.5cm}>{\centering\arraybackslash}m{3.5cm}}
\begin{tikzpicture}[scale=0.4]

	\draw[line width=1pt] (0,6) -- (6,6);
	\draw[line width=1pt] (6,0) -- (6,6);
	\draw[line width=1pt] (0,6) -- (6,0);
	\fill (0,6) circle (5pt);
	\fill (6,0) circle (5pt);
	\fill (6,6) circle (5pt);

	\node at (0.5,6.7) {$i,j$};
	\node at (7,0) {$k,\ell$};
	\node at (4,4) {$\mathbb{P}^2$};

\end{tikzpicture}
&
\begin{tikzpicture}[scale=0.4]

	\draw[line width=1pt] (0,6) -- (6,6);
	\draw[line width=1pt] (6,0) -- (6,6);
	\draw[line width=1pt] (0,6) -- (6,0);
	\draw[line width=1pt] (2,4) -- (2,6);
	\draw[line width=1pt] (4,2) -- (6,2);
	\fill (0,6) circle (5pt);
	\fill (1.2,5.4) circle (5pt);
	\fill (6,0) circle (5pt);
	\fill (5.4,1.2) circle (5pt);
	\fill (6,6) circle (5pt);

	\node at (0,6.7) {$i$};
	\node at (1.6,5.4) {$j$};
	\node at (6.7,0) {$k$};
	\node at (5.05,1.6) {$\ell$};
	\node at (4.2,4.2) {$\Bl_2\mathbb{P}^2$};

\end{tikzpicture}
&
\begin{tikzpicture}[scale=0.4]

	\draw[line width=1pt] (0,6) -- (6,6);
	\draw[line width=1pt] (6,0) -- (6,6);
	\draw[line width=1pt] (0,6) -- (6,0);
	\draw[line width=1pt] (3,6) -- (3,4.8);
	\draw[line width=1pt] (3,4.8) -- (1.2,4.8);
	\draw[line width=1pt] (6,3) -- (4.8,3);
	\draw[line width=1pt] (4.8,3) -- (4.8,1.2);
	\draw[dashed,line width=1pt] (3,4.8) -- (4.8,3);
	\fill (0,6) circle (5pt);
	\fill (1.7,5.4) circle (5pt);
	\fill (6,0) circle (5pt);
	\fill (5.4,1.7) circle (5pt);
	\fill (6,6) circle (5pt);

	\node at (0,6.7) {$i$};
	\node at (2.1,5.4) {$j$};
	\node at (6.7,0) {$k$};
	\node at (5.1,2.2) {$\ell$};
	\node at (3.5,3.5) {$\mathbb{F}_1$};

\end{tikzpicture}
&
\begin{tikzpicture}[scale=0.4]

	\draw[line width=1pt] (0,6) -- (6,6);
	\draw[line width=1pt] (6,0) -- (6,6);
	\draw[line width=1pt] (0,6) -- (6,0);
	\draw[line width=1pt] (4,6) -- (4,2);
	\draw[line width=1pt] (6,4) -- (2,4);
	\fill (0,6) circle (5pt);
	\fill (2.5,5) circle (5pt);
	\fill (6,0) circle (5pt);
	\fill (5,2.5) circle (5pt);
	\fill (6,6) circle (5pt);

	\node at (0,6.7) {$i$};
	\node at (3,5) {$j$};
	\node at (6.7,0) {$k$};
	\node at (5.1,3.2) {$\ell$};
	\node at (5,5) {$\mathbb{F}_0$};

\end{tikzpicture}
\end{tabular}
\caption{Fibers of families 
over the $0$-strata.}
\label{modificationsfibersinfamily2}
\end{figure}
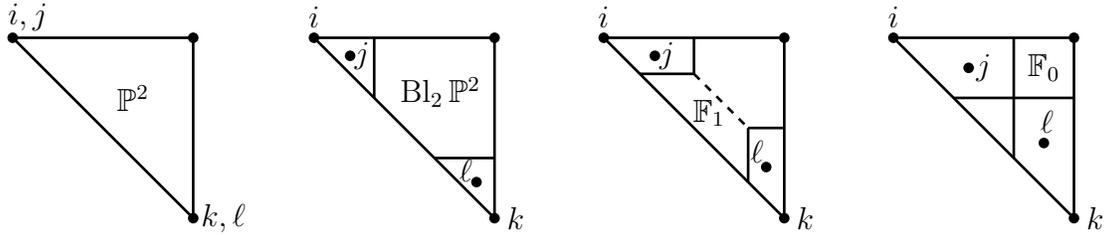

For distinct $i,j\in\{1,\ldots,5\}$, the curves
$C_{ij}=\mathrm{im}(s_i)\cap\mathrm{im}(s_j)$ are smooth and pairwise disjoint. Let 
$\overline{\mathcal{P}}''$ be the blow up of $\overline{\mathcal{P}}'''$ along these curves.
The fibers of $\overline{\mathcal{P}}''\rightarrow\overline{\mathrm{M}}_{0,5}$ over a general point of a $1$-stratum (resp.~over a $0$-stratum) of $\overline{\mathrm{M}}_{0,5}$ are illustrated in the right picture of Figure~\ref{modificationsfibersinfamily1} (resp.~the second picture from the left in Figure~\ref{modificationsfibersinfamily2}).

Given distinct $i,j,k,\ell\in\{1,\ldots,5\}$, let $p\in\overline{\mathrm{M}}_{0,5}$ be the $0$-stratum such that $s_i(p)=s_j(p),s_k(p)=s_\ell(p)$, and let $L_{ij,k\ell}\subseteq\overline{\mathcal{P}}'''$ be the line in the $\bP^2$ over $p$ passing through these two points. 
Let $\overline{\mathcal{P}}'$ be the blow up of $\overline{\mathcal{P}}''$ along the strict transforms of the lines $L_{ij,k\ell}$. 
This blow up affects the fibers over the $0$-strata, and the fiber $X$ of $\overline{\mathcal{P}}'\rightarrow\overline{\mathrm{M}}_{0,5}$ over a $0$-stratum
is illustrated in the third picture from the left in Figure~\ref{modificationsfibersinfamily2}.
The $(-1)$-curve $E\subseteq X$ on the $\mathbb{F}_1$ component is also a $(-1)$-curve
in $\Bl_2\mathbb{P}^2$ (this is the curve dashed in Figure~\ref{modificationsfibersinfamily2}). 
We claim that $E$ can be contracted.
Up to relabeling, we can assume $s_1(p)=s_2(p)$ and $s_3(p)=s_4(p)$.  Endow the family $\mathbb{P}_{\mathrm{M}_{0,5}}^2\rightarrow\mathrm{M}_{0,5}$ with the relative effective Cartier divisor $\mathcal{D}\subseteq\mathbb{P}_{\mathrm{M}_{0,5}}$ giving on the general fiber the line arrangement shown on the left of Figure~\ref{divisorforcontraction}. The Zariski closure $\overline{\mathcal{D}}\subseteq\overline{\mathcal{P}}'$ specializes to $X$ giving the line arrangement $D_p$ shown on the right of Figure~\ref{divisorforcontraction}.

\begin{figure}[hbtp]
\begin{tikzpicture}[scale=0.5]

	\draw[line width=1pt] (5,5) [partial ellipse=0:360:6cm and 3cm];

	\fill (5,8) circle (5pt);
	\fill (-0.5,3.8) circle (5pt);
	\fill (0.6,3) circle (5pt);
	\fill (10.5,3.8) circle (5pt);
	\fill (9.4,3) circle (5pt);

	\draw[line width=1pt,red] (5,8) -- (-0.5,3.8);
	\draw[line width=1pt,red] (5,8) -- (10.5,3.8);
	\draw[line width=1pt,red] (-0.5,3.8) -- (9.4,3);
	\draw[line width=1pt,red] (0.6,3) -- (10.5,3.8);
	\draw[line width=1pt,red] (-2,4.89) -- (2.1,1.9);
	\draw[line width=1pt,red] (7.9,1.9) -- (12,4.89);

	\node at (5,8.6) {$5$};
	\node at (-0.9,3.4) {$1$};
	\node at (0.1,2.5) {$2$};
	\node at (11,3.5) {$3$};
	\node at (9.5,2.5) {$4$};

	\draw[line width=1pt,red] (14,8) -- (21,8);
	\draw[line width=1pt,red] (21,1) -- (21,8);
	\draw[line width=1pt] (14,8) -- (21,1);
	\draw[line width=1pt] (17.5,8) -- (17.5,6.6);
	\draw[line width=1pt] (19.54,4.6) -- (21,4.6);
	\draw[line width=1pt] (17.54,6.6) -- (15.4,6.6);
	\draw[line width=1pt] (19.54,4.6) -- (19.54,2.45);
	\draw[line width=1pt] (17.5,6.64) -- (19.54,4.6);

	\draw[line width=1pt,red] (17.5,7.3) -- (21,7.3);
	\draw[line width=1pt,red] (17.5,7.3) -- (14,8);
	\draw[line width=1pt,red] (20.3,4.6) -- (20.3,8);
	\draw[line width=1pt,red] (20.3,4.6) -- (21,1);
	\draw[line width=1pt,red] (16.2,6.6) -- (19.5,4);
	\draw[line width=1pt,red] (16.8,6.6) -- (19.5,3.4);
	\draw[line width=1pt,red] (16.2,6.6) -- (16.2,8);
	\draw[line width=1pt,red] (16.8,6.6) -- (16.8,8);
	\draw[line width=1pt,red] (19.5,4) -- (21,4);
	\draw[line width=1pt,red] (19.5,3.4) -- (21,3.4);

\end{tikzpicture}
\caption{In red, a fiber of $\overline{\mathcal{D}}$ over 
a point in $\mathrm{M}_{0,5}$ (left) and a $0$-stratum (right).} 
\label{divisorforcontraction}
\end{figure}
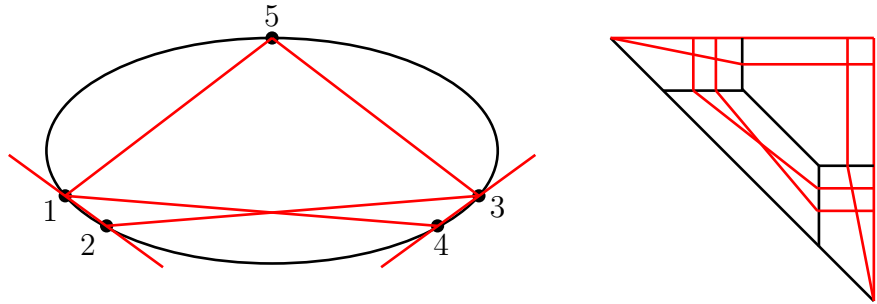

The pair $(X,D_p)$ is semi-log canonical and $K_X+D_p$ is nef but not ample:
its intersection with $E$ is $0$. Applying the relative minimal model program to $(\overline{\mathcal{P}}',K+\overline{\mathcal{D}})\rightarrow\overline{\mathrm{M}}_{0,5}$ in the neighborhood of $p$ gives a morphism contracting exactly $E$ over an open neighborhood of~$p$. Similarly, we can construct the contractions of the remaining $14$ curves of type $E$ and glue them to obtain a small contraction $\overline{\mathcal{P}}'\rightarrow\overline{\mathcal{P}}$. A priori, $\overline{\mathcal{P}}$ is an algebraic space, but, as we are going to see later, it is actually a projective variety.
The fibers of $\gamma\colon\overline{\mathcal{P}}\rightarrow\overline{\mathrm{M}}_{0,5}$ over the $0$-strata are described by the rightmost picture in Figure~\ref{modificationsfibersinfamily2}. We denote by $\widetilde{s}_1,\ldots,\widetilde{s}_5\colon\overline{\mathrm{M}}_{0,5}\rightarrow\overline{\mathcal{P}}$ the strict transforms of the sections $s_1,\ldots,s_5$.

To define a morphism $\overline{\mathrm{M}}_{0,5}\rightarrow\overline{\mathbf{X}}_{\GP}(3,5)$, we first construct a commutative diagram
\begin{center}
\begin{tikzpicture}[>=angle 90]
\matrix(a)[matrix of math nodes,
row sep=2em, column sep=2em,
text height=1.5ex, text depth=0.25ex]
{\overline{\mathcal{P}}&\overline{\mathrm{M}}_{0,5}\times(\mathbb{P}^2)^{\binom{5}{4}}\\
\overline{\mathrm{M}}_{0,5}.&\\};
\path[right hook->] (a-1-1) edge node[above]{}(a-1-2);
\path[->] (a-1-1) edge node[left]{$\gamma$}(a-2-1);
\path[->] (a-1-2) edge node[below right]{$\pi_1$}(a-2-1);
\end{tikzpicture}
\end{center}
We claim that the rational map $\varphi\colon\overline{\mathcal{P}}\dashrightarrow\overline{\mathrm{M}}_{0,5}\times(\mathbb{P}^2)^{\binom{5}{4}}$, defined over $\mathrm{M}_{0,5}=\mathbf{B}_5$ in Definition~\ref{definitionofthecorrectcompactification}, 
 is regular. In other words, each of the $\binom{5}{4}$ maps 
from $\overline{\mathcal{P}}$ to $\mathbb{P}^2$ is regular. 
Let $C\subseteq\overline{\mathrm{M}}_{0,5}$ be the open part of a $1$-stratum, e.g.~the one corresponding to the degeneration where $s_1(C)=s_2(C)$. The fibers of $\gamma$ over $C$ are illustrated at the top of Figure~\ref{contractionsfibersover1strata}. 

\begin{figure}[hbtp]
\begin{tikzpicture}[scale=0.5]

	\draw[line width=1pt] (0,6) -- (6,6);
	\draw[line width=1pt] (6,6) -- (6,0);
	\draw[line width=1pt] (6,0) -- (0,6);
	\draw[line width=1pt] (3,6) -- (6,3);
	\fill (6,0) circle (5pt);
	\fill (0,6) circle (5pt);
	\fill (6,6) circle (5pt);
	\fill (4,3) circle (5pt);
	\fill (4.5,5.5) circle (5pt);

	\node at (6.5,6) {$1$};
	\node at (5,5.1) {$2$};
	\node at (-.5,6) {$3$};
	\node at (6.5,0) {$5$};
	\node at (4.6,3.2) {$4$};

	\draw[->,line width=1pt] (2,2) -- (1,1);
	\draw[line width=1pt] (0,0) -- (-3,0);
	\draw[line width=1pt] (-3,0) -- (0,-3);
	\draw[line width=1pt] (0,0) -- (0,-3);
	\fill (0,0) circle (5pt);
	\fill (-3,0) circle (5pt);
	\fill (0,-3) circle (5pt);
	\fill (-1,-1) circle (5pt);
	\node at (0.9,0) {$1,2$};
	\node at (-3.5,0) {$3$};
	\node at (0.5,-3) {$5$};
	\node at (-0.5,-1) {$4$};

	\draw[->,line width=1pt] (7,2) -- (8,1);
	\draw[line width=1pt] (12,0) -- (9,0);
	\draw[line width=1pt] (9,0) -- (12,-3);
	\draw[line width=1pt] (12,0) -- (12,-3);
	\fill (12,0) circle (5pt);
	\fill (9,0) circle (5pt);
	\fill (12,-3) circle (5pt);
	\fill (10.5,-0.5) circle (5pt);
	\fill (11,-2) circle (5pt);
	\node at (12.9,0) {$1$};
	\node at (8.5,0) {$3$};
	\node at (12.5,-3) {$5$};
	\node at (11.2,-0.5) {$2$};
	\node at (11.5,-1.7) {$4$};

	\node at (1.1,2.1) {$c_1$};
	\node at (7.9,2.1) {$c_2$};
\end{tikzpicture}
\caption{Effect of the contractions $c_1,c_2$ on the fiber $\gamma^{-1}(p)$ for $p$ in the interior of a $1$-stratum of $\overline{\mathrm{M}}_{0,5}$.}
\label{contractionsfibersover1strata}
\end{figure}
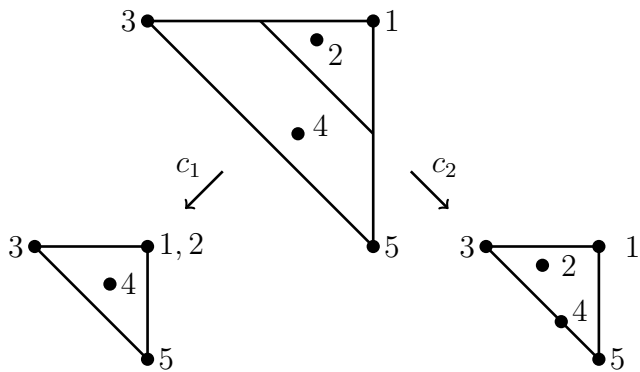

Let $U=\mathrm{M}_{0,5}\cup C$ be an open neighborhood of $C$ in $\overline{\mathrm{M}}_{0,5}$. We have contractions $c_1$ and~$c_2$ defined on $\gamma^{-1}(U)$ such that, over $C$, $c_1$ fiberwisely collapses the $\mathbb{P}^2$ component to a point, and $c_2$ contracts $\mathbb{F}_1$ along the ruling. 
The effect of these contractions on the fibers over $C$ is represented in Figure~\ref{contractionsfibersover1strata}.
These contractions are constructed by applying the relative $K+D$ minimal model program for an appropriate divisor $D$ on $\gamma^{-1}(U)$. Namely, on the fiber $\mathbb{P}^2$ over a point in $\mathrm{M}_{0,5}$, we consider the lines $\overline{13},\overline{14},\overline{35},\overline{45}$ (resp. $\overline{12},\overline{13},\overline{14},\overline{15}$), and the divisor $D$ is their closure in $\gamma^{-1}(U)$.
The induced contraction is $c_1$ (resp. $c_2$).  We define $\overline{\varphi}_U\colon\gamma^{-1}(U)\rightarrow U\times(\mathbb{P}^2)^{\binom{5}{4}}$ as follows. Given $x\in\gamma^{-1}(U)$, let $\overline{\varphi}_U(x)=(\gamma(x),(q_i)_{i=1}^5)$, where the $q_i$ are defined as follows. For exactly one $j\in\{1,2\}$, the four points $c_j(\widetilde{s}_1(x)),\ldots,\widehat{c_j(\widetilde{s}_i(x))},\ldots,c_j(\widetilde{s}_5(x))\in\mathbb{P}^2$ are in general linear position. So define $q_i$ to be the image of $c_j(x)$ under the unique projective linear transformation sending $c_j(\widetilde{s}_1(x)),\ldots,\widehat{c_j(\widetilde{s}_i(x))},\ldots,c_j(\widetilde{s}_5(x))\in\mathbb{P}^2$ in standard position.

Let $p$ be a $0$-stratum of $\overline{\mathrm{M}}_{0,5}$, for example
suppose it corresponds to a degeneration where $s_1(p)=s_2(p)$ and $s_3(p)=s_4(p)$. The fiber $\gamma^{-1}(p)$ is in the top right corner of Figure~\ref{contractionsfibersover0strata}.
\begin{figure}[hbtp]
\begin{tikzpicture}[scale=0.5]

	\draw[line width=1pt] (0,6) -- (6,6);
	\draw[line width=1pt] (6,6) -- (6,0);
	\draw[line width=1pt] (6,0) -- (0,6);
	\draw[line width=1pt] (4.5,6) -- (4.5,1.5);
	\draw[line width=1pt] (1.5,4.5) -- (6,4.5);
	\fill (6,0) circle (5pt);
	\fill (0,6) circle (5pt);
	\fill (6,6) circle (5pt);
	\fill (5.6,3) circle (5pt);
	\fill (3,5.6) circle (5pt);

	\node at (6.5,6) {$5$};
	\node at (3.2,5.1) {$2$};
	\node at (-.5,6) {$1$};
	\node at (6.5,0) {$3$};
	\node at (5.1,3.2) {$4$};

	\draw[->,line width=1pt] (-1,5) -- (-3,5);
	\draw[line width=1pt] (-4,6) -- (-7,6);
	\draw[line width=1pt] (-7,6) -- (-4,3);
	\draw[line width=1pt] (-4,6) -- (-4,3);
	\fill (-7,6) circle (5pt);
	\fill (-4,6) circle (5pt);
	\fill (-4,3) circle (5pt);
	\fill (-5,5) circle (5pt);
	\node at (-3.5,6) {$5$};
	\node at (-4.4,4.8) {$2$};
	\node at (-7.5,6) {$1$};
	\node at (-3.1,3) {$3,4$};

	\draw[->,line width=1pt] (2,2) -- (1,1);
	\draw[line width=1pt] (0,0) -- (-3,0);
	\draw[line width=1pt] (-3,0) -- (0,-3);
	\draw[line width=1pt] (0,0) -- (0,-3);
	\fill (0,0) circle (5pt);
	\fill (-3,0) circle (5pt);
	\fill (0,-3) circle (5pt);
	\fill (0,-1.5) circle (5pt);
	\fill (-1.5,0) circle (5pt);
	\node at (0.5,0) {$5$};
	\node at (-1.2,-0.5) {$2$};
	\node at (-3.5,0) {$1$};
	\node at (0.5,-3) {$3$};
	\node at (0.5,-1.5) {$4$};

	\draw[->,line width=1pt] (5,-1) -- (5,-3);
	\draw[line width=1pt] (6,-4) -- (3,-4);
	\draw[line width=1pt] (3,-4) -- (6,-7);
	\draw[line width=1pt] (6,-4) -- (6,-7);
	\fill (6,-4) circle (5pt);
	\fill (3,-4) circle (5pt);
	\fill (6,-7) circle (5pt);
	\fill (5,-5) circle (5pt);
	\node at (6.5,-4) {$5$};
	\node at (2.9,-3.5) {$1,2$};
	\node at (6.5,-7) {$3$};
	\node at (4.5,-4.8) {$4$};

	\node at (-2,5.5) {$c_3$};
	\node at (1.1,2.1) {$c_1$};
	\node at (5.6,-2) {$c_2$};

\end{tikzpicture}
\caption{Contractions $c_1,c_2,c_3$ on the fiber $\gamma^{-1}(p)$. 
}
\label{contractionsfibersover0strata}
\end{figure}
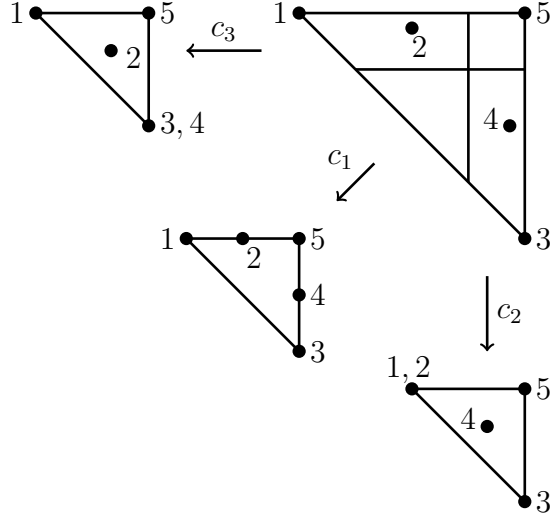
The points $\widetilde{s}_1(p),\widetilde{s}_2(p)$ are contained in one copy of $\mathbb{F}_1$, $\widetilde{s}_3(p),\widetilde{s}_4(p)$ are contained in the other copy of $\mathbb{F}_1$, and $\widetilde{s}_5(p)\in\mathbb{F}_0$. Let $U\subseteq\overline{\mathrm{M}}_{0,5}$ be an open neighborhood of $p$ avoiding all the strata of $\overline{\mathrm{M}}_{0,5}$ not specializing to $p$. Using the relative minimal model program as above, we can construct three contractions of $\gamma^{-1}(U)$ which we label $c_1$ (use the lines $\overline{14},\overline{15},\overline{23},\overline{35}$), $c_2$ (use the lines $\overline{14},\overline{15},\overline{34},\overline{35}$), and $c_3$ (use the lines $\overline{12},\overline{15},\overline{23},\overline{35}$).
See Figure~\ref{contractionsfibersover0strata}
for the effect of these contractions on the fiber over $p$.
The fibers over the adjacent $1$-strata 
are contracted accordingly. We define $\overline{\varphi}_U(x)=(\gamma(x),(q_i)_{i=1}^5)$ as in the previous case.

The morphisms $\overline{\varphi}_U$ glue giving a morphism 
$\overline{\varphi}\colon\overline{\mathcal{P}}\rightarrow\overline{\mathrm{M}}_{0,5}\times(\mathbb{P}^2)^{\binom{5}{4}}$, 
which is an embedding. 
Indeed, by Nakayama's lemma, it suffices to show that the restriction to each
fiber is a closed embedding into $(\mathbb{P}^2)^{\binom{5}{4}}$, which follows from the fact that the fibers of $\gamma$ are central fibers of Mustafin joins (see Remark~\ref{allabouttwolattices}).
Let $\XGP'(3,5)$ be the Zariski closure of the image of $\mathbf{B}_5$ in $\overline{\mathbf{B}}_5\times\Hilb\left((\mathbb{P}^2)^{\binom{5}{4}}\right)$.
The morphism $\overline{\varphi}$ induces a birational morphism $\overline{\mathrm{M}}_{0,5}\rightarrow\XGP'(3,5)$.
Since we also have a birational morphism $\XGP'(3,5)\rightarrow\overline{\mathbf{B}}_5\cong\overline{\mathrm{M}}_{0,5}$, we have  $\XGP'(3,5)\cong\overline{\mathrm{M}}_{0,5}$.
By~\cite[Proposition~3.1]{CS10}, the morphism $\XGP(3,5)\to\XGP'(3,5)$
is bijective. Since $\overline{\mathrm{M}}_{0,5}$ is normal,
$\XGP(3,5)\cong\XGP'(3,5)\cong\overline{\mathrm{M}}_{0,5}$
by Zariski's Main Theorem.
\end{proof}


\begin{proposition}\label{sRGsrgwrH}
For each $i=1,\ldots,n$, there exists a forgetful morphism
\[
\overline{\mathbf{X}}_{\GP}(3,n)\rightarrow\overline{\mathbf{X}}_{\GP}(3,n-1)
\]
extending the obvious forgetful map ${\mathbf{B}}_n\to {\mathbf{B}}_{n-1}$ that forgets the $i$-th point.
\end{proposition}

\begin{proof}
Say $i=n$. Following the definition of $\XGP(3,n)$ as the closure of ${\mathbf{B}}_n$ in the product
of $\overline{\mathbf{B}}_n$ and the multigraded Hilbert scheme and
the definition of $\overline{\mathbf{B}}_n$ as the closure of ${\mathbf{B}}_n$ in $(\mathbb{P}^1)^{\mathcal{Q}_n}$, 
we need to show two things.
Firstly, the projection from $(\mathbb{P}^1)^{\mathcal{Q}_n}$ to  $(\mathbb{P}^1)^{\mathcal{Q}_{n-1}}$
induces the forgetful morphism $\overline{\mathbf{B}}_n\to \overline{\mathbf{B}}_{n-1}$, which is clear.
Secondly, the projection from $(\mathbb{P}^2)^{\binom{n}{4}}$
to $(\mathbb{P}^2)^{\binom{n-1}{4}}$ induces a morphism of multigraded 
Hilbert schemes (with the Hilbert function of the diagonal), which is Lemma~\ref{sdfvsdfv}.
\end{proof}

\begin{definition}
Let $I\in\binom{[n]}{5}$, $i\in I$. Consider the composition of forgetful morphisms 
$$\XGP(3,n)\rightarrow\XGP(3,I)\cong\overline{\mathrm{M}}_{0,5},$$ 
and denote by $f_{I,i}$ the composition of this map with the $i$-th Kapranov's map $\overline{\mathrm{M}}_{0,5}\rightarrow\mathbb{P}^2$.

Given $k\in\{1,\ldots,n\}$, let $\binom{[n]}{4}_k\subseteq\binom{[n]}{4}$ be the subset of quadruples $J=\{j_1<\ldots<j_4\}$ containing $k$. Then define $\chi_k\colon\binom{[n]}{4}_k\rightarrow\mathbb{P}^2$ to be the following function:
\begin{displaymath}
\chi_k(J)=\left\{ \begin{array}{l}
[1:0:0]~\textrm{if}~k=j_1\\
{[}0:1:0{]}~\textrm{if}~k=j_2\\
{[}0:0:1{]}~\textrm{if}~k=j_3\\
{[}1:1:1{]}~\textrm{if}~k=j_4.
\end{array} \right.
\end{displaymath}

Finally, for $k=1,\ldots,n$, consider the morphism 
$$s_k\colon\overline{\mathbf{X}}_{\GP}(3,n)\rightarrow\overline{\mathbf{X}}_{\GP}(3,n)\times(\mathbb{P}^2)^{\binom{n}{4}}$$ defined as
\[
s_k=\id_{\overline{\mathbf{X}}_{\GP}(3,n)}\times\prod_{\substack{J\in\binom{[n]}{4},\\k\notin J}}f_{J\cup\{k\},k}\times\prod_{\substack{J\in\binom{[n]}{4},\\k\in J}}\chi_k(J).
\]
\end{definition}

\begin{theorem}
The morphisms $s_1,\ldots,s_n$ are sections $\overline{\mathbf{X}}_{\GP}(3,n)\rightarrow\overline{\mathcal{M}}$ of the universal Mustafin join of point configurations. They 
give  $n$ smooth distinct points on each fiber.
\end{theorem}

\begin{proof}
Since we already know that $s_k$ maps the interior of $\overline{\mathbf{X}}_{\GP}(3,n)$ into $\overline{\mathcal{M}}$, 
the image of $s_k$ is contained in $\overline{\mathcal{M}}$
by continuity.
Let $x\in\overline{\mathbf{X}}_{\GP}(3,n)$. Let $\mathbf{a}\colon\Spec(K)\rightarrow\mathbf{B}_n$ be an arc such that the unique extension $\overline{\mathbf{a}}\colon\Spec(R)\rightarrow\overline{\mathbf{X}}_{\GP}(3,n)$ satisfies $\overline{\mathbf{a}}(0)=x$. Then $\overline{\mathcal{M}}\rightarrow\overline{\mathbf{X}}_{\GP}(3,n)$ and $s_1,\ldots,s_k$ pullback to $\Spec(R)$ giving the Mustafin join $\mathbb{P}(\Sigma_\mathbf{a})\rightarrow\Spec(R)$ with the sections $\overline{a}_1,\ldots,\overline{a}_n$. The points $\overline{a}_1(0),\ldots,\overline{a}_n(0)\in\mathbb{P}(\Sigma_\mathbf{a})_\Bbbk$ are distinct by Lemma~\ref{disjointsectionsmustafinjoin} and smooth by Theorem~\ref{sectionssmoothcentralfiber}, implying the theorem.
\end{proof}

Finally, we give a criterion to establish which  
$n$-pointed degenerations of $\mathbb{P}^2$ arise as 
fibers of the universal Mustafin join 
$\overline{\mathcal{M}}\to\overline{\mathbf{X}}_{\GP}(3,n)$  of point configurations.

\begin{lemma}
\label{npointedcentralfibermustafinjoinnocalculation}
Let $X$ be a projective surface 
with smooth points $p_1,\ldots,p_n\in X$ such that
\begin{enumerate}

\item There exists $\Sigma=\{[L_1],\ldots,[L_m]\}\subseteq\mathfrak{B}_3^0$ such that $X\cong\mathbb{P}(\Sigma)_\Bbbk$ (recall that if $\Sigma$ is contained in an apartment, then $\mathbb{P}(\Sigma)_\Bbbk$ can be easily computed using Remark~\ref{degenerationfrompicture});

\item There exists a surjective map $\binom{[n]}{4}\rightarrow\Sigma$ associating to a quadruple of distinct indices $i,j,k,h\in[n]$ a lattice class $[L_r]\in\Sigma$ such that the images of the points $p_i,p_j,p_k,p_h$ under the composition $X\cong\mathbb{P}(\Sigma)_\Bbbk\to\mathbb{P}(L_r)_\Bbbk$ are in general linear position.

\end{enumerate}
Then there exists $\mathbf{a}\in(\mathbb{P}^2)^n(K)$ in general linear position such that $X\cong\mathbb{P}(\Sigma_\mathbf{a})_\Bbbk$, and $p_1,\ldots,p_n$ correspond respectively to $\overline{a}_1(0),\ldots,\overline{a}_n(0)$ under this isomorphism.
\end{lemma}

\begin{proof}
Since $p_1,\ldots,p_n$ are smooth points of $X$ (and hence smooth points of $\mathbb{P}(\Sigma)_\Bbbk$), by the infinitesimal lifting property we can find $\mathbf{a}\in(\mathbb{P}^2)^n(K)$ in general linear position such that $\overline{a}_1(0),\ldots,\overline{a}_n(0)\in\mathbb{P}(\Sigma)_\Bbbk$ correspond respectively to $p_1,\ldots,p_n$ under the isomorphism $\mathbb{P}(\Sigma)_\Bbbk\cong X$. By the definition of stable lattices and by our assumptions, we have that $\Sigma_\mathbf{a}=\Sigma$.
\end{proof}

\begin{example}
\label{case8comesfromanarc}
Let $(X;p_1,\ldots,p_6)$ be the $6$-pointed reducible surface $\#8$ in Table~\ref{tbl:GPdegenerationspart1}. We~claim that it arises as $(\mathbb{P}(\Sigma_\mathbf{a})_\Bbbk;\overline{a}_1(0),\ldots,\overline{a}_6(0))$ for some arc $\mathbf{a}\in(\mathbb{P}^2)^6(K)$ in general linear position.
Lemma~\ref{npointedcentralfibermustafinjoinnocalculation}~(1) is satisfied by choosing $\Sigma=\{[L_1],[L_2],[L_3]\}$, where
\[
L_1=e_1R+e_2R+e_3R,~L_2=te_1R+e_2R+te_3R,~L_3=e_1R+t^2e_2R+te_3R.
\]
Figure~\ref{GPdegenerationincase8} shows that the assumption~(2) is satisfied as well.
\end{example}

\begin{figure}[hbtp]
\begin{tikzpicture}[scale=0.5]

	\draw[line width=1pt] (0,6) -- (6,6);
	\draw[line width=1pt] (6,6) -- (6,0);
	\draw[line width=1pt] (6,0) -- (0,6);
	\draw[line width=1pt] (4.5,6) -- (4.5,1.5);
	\draw[line width=1pt] (1.5,4.5) -- (6,4.5);
	\fill[gray!50,nearly transparent] (1.5,4.5) -- (4.5,4.5) -- (4.5,1.5) -- cycle;
	\fill[gray!50,nearly transparent] (1.5,4.5) -- (4.5,4.5) -- (4.5,6) -- (0,6) -- cycle;
	\fill[gray!50,nearly transparent] (4.5,1.5) -- (4.5,4.5) -- (6,4.5) -- (6,0) -- cycle;
	\fill (6,0) circle (5pt);
	\fill (0,6) circle (5pt);
	\fill (6,6) circle (5pt);
	\fill (5.6,3) circle (5pt);
	\fill (1.5,5.5) circle (5pt);
	\fill (3,5.6) circle (5pt);

	\node at (6.5,6) {$1$};
	\node at (3.2,5.1) {$2$};
	\node at (1.7,5) {$3$};
	\node at (-.5,6) {$4$};
	\node at (6.5,0) {$5$};
	\node at (5.1,3.2) {$6$};

	\draw[->,line width=1pt] (-1,5) -- (-3,5);
	\draw[line width=1pt] (-4,6) -- (-7,6);
	\draw[line width=1pt] (-7,6) -- (-4,3);
	\draw[line width=1pt] (-4,6) -- (-4,3);
	\fill[gray!50,nearly transparent] (-4,6) -- (-7,6) -- (-4,3) -- cycle;
	\fill (-7,6) circle (5pt);
	\fill (-4,6) circle (5pt);
	\fill (-4,3) circle (5pt);
	\fill (-5.7,5.3) circle (5pt);
	\fill (-4.4,5.4) circle (5pt);
	\node at (-3.5,6) {$1$};
	\node at (-4.4,4.8) {$2$};
	\node at (-5.3,5) {$3$};
	\node at (-7.5,6) {$4$};
	\node at (-3.1,3) {$5,6$};

	\draw[->,line width=1pt] (2,2) -- (1,1);
	\draw[line width=1pt] (0,0) -- (-3,0);
	\draw[line width=1pt] (-3,0) -- (0,-3);
	\draw[line width=1pt] (0,0) -- (0,-3);
	\fill[gray!50,nearly transparent] (0,0) -- (-3,0) -- (0,-3) -- cycle;
	\fill (0,0) circle (5pt);
	\fill (-3,0) circle (5pt);
	\fill (0,-3) circle (5pt);
	\fill (0,-1.5) circle (5pt);
	\fill (-1,0) circle (5pt);
	\fill (-2,0) circle (5pt);
	\node at (0.5,0) {$1$};
	\node at (-0.8,-0.5) {$2$};
	\node at (-1.8,-0.5) {$3$};
	\node at (-3.5,0) {$4$};
	\node at (0.5,-3) {$5$};
	\node at (0.5,-1.5) {$6$};

	\draw[->,line width=1pt] (5,-1) -- (5,-3);
	\draw[line width=1pt] (6,-4) -- (3,-4);
	\draw[line width=1pt] (3,-4) -- (6,-7);
	\draw[line width=1pt] (6,-4) -- (6,-7);
	\fill[gray!50,nearly transparent] (6,-4) -- (3,-4) -- (6,-7) -- cycle;
	\fill (6,-4) circle (5pt);
	\fill (3,-4) circle (5pt);
	\fill (6,-7) circle (5pt);
	\fill (5,-5) circle (5pt);
	\node at (6.5,-4) {$1$};
	\node at (2.9,-3.5) {$2,3,4$};
	\node at (6.5,-7) {$5$};
	\node at (4.5,-4.8) {$6$};

	\node at (1,2.5) {\#8};

\end{tikzpicture}
\caption{Maps to the primary components in Example~\ref{case8comesfromanarc}.}
\label{GPdegenerationincase8}
\end{figure}
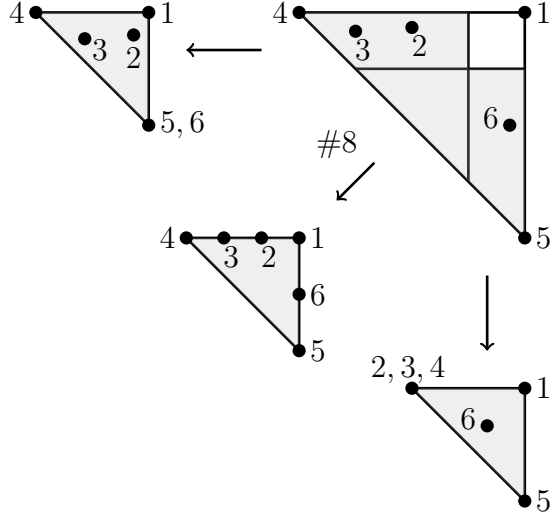


\section{Analogue of the Losev--Manin space for moduli of points in \texorpdfstring{$\mathbb{P}^2$}{Lg}}\label{sGsgsGsehsRH}

The Grothendieck--Knudsen moduli space $\overline{\mathrm{M}}_{0,n}=\overline{\mathbf{X}}(2,n)$ 
of stable rational curves has a toric analogue,
the Losev--Manin space \cite{LM00,CT20}, 
which parametrizes chains of rational curves with 
$m$ {\em light} and 
two {\em heavy} points (one at each end of the chain).
The light points can coincide with each other but not with the nodes or the heavy points.
The Losev--Manin space is a toric variety associated with the permutohedron.
In this and the next sections we will discuss
analogues of the Losev--Manin space for moduli spaces of points 
and lines 
in~$\bP^2$. We~start with a toric analogue of the Kapranov Chow quotient space $\overline{\mathbf{X}}(3,n)=(\mathbb{P}^2)^n/\!/\PGL_3$, which will be the Chow quotient 
$(\mathbb{P}^2)^m/\!/\mathbb{G}_m^2$ for $n=m+3$.


Denote by $Y_\mathcal{F}$ the toric variety associated to a fan $\mathcal{F}$. Let $T\subseteq Y_\mathcal{F}$ be the dense open torus. By \cite{KSZ91},
the Chow quotient $Y_\mathcal{F}/\!/H$ by a subtorus $H\subseteq T$
is a toric variety (not necessarily normal) of the quotient torus $T/H$. The~associated fan of the normalization
is called the {\em quotient fan}.
We recall its description in our case: $Y_\mathcal{F}=(\mathbb{P}^2)^m$ and $H$ is the dense open torus in $\mathbb{P}^2$ acting diagonally on $(\mathbb{P}^2)^m$.
The Chow quotient parametrizes translates $(p_1^{-1},\ldots,p_m^{-1})\cdot\mathbb{P}^2\subseteq(\mathbb{P}^2)^m$ of the diagonally embedded $\bP^2\subseteq(\mathbb{P}^2)^m$
and their limits in the Chow variety, where $(p_1,\ldots,p_m)\in(\mathbb{G}_m^2)^m$.

Let $N$ be the lattice of $1$-parameter subgroups of $H$ and let $\mathcal{P}$ be the fan of $\mathbb{P}^2$ in $N_\mathbb{R}\cong\mathbb{R}^2$. The product fan $\mathcal{P}^m$ in $N_\mathbb{R}^m$ gives the toric variety $(\mathbb{P}^2)^m$. In what follows, we identify $N_\mathbb{R}^m$ with~$\mathbb{R}^{2m}$ and we view $\mathbb{R}^2$ as diagonally embedded in $\mathbb{R}^{2m}$. Given $v\in\mathbb{R}^{2m}$, define
\[
\mathcal{P}_{v}^m=\{\sigma\in\mathcal{P}^m\mid\sigma\cap(v+\mathbb{R}^2)\neq\emptyset\}.
\]
Vectors $v,w\in\mathbb{R}^{2m}$ are called {\em equivalent} provided $\mathcal{P}_{v}^m=\mathcal{P}_{w}^m$. The closures of equivalence classes are rational convex cones in $\mathbb{R}^{2m}$  invariant under translations by $\mathbb{R}^2$. 
The images of these cones under the quotient map
\[
q\colon\mathbb{R}^{2m}\rightarrow\mathbb{R}^{2m}/\mathbb{R}^2
\]
form the fan $\mathcal{Q}_m$ called the \emph{quotient fan}.
We call $Y_{\mathcal{Q}_m}$
the {\em toric Kapranov space}. 
It is 
a normalization of the Chow quotient and the (multigraded) Hilbert quotient of 
$(\mathbb{P}^2)^m$ by~$H$ (the closure of $T/H$
in the (multigraded) Hilbert scheme of $(\mathbb{P}^2)^m$),
see Example~\ref{srgsrgr}.

We would like to describe cones in $\mathcal{Q}_m$
explicitly.

\begin{definition}
\label{combinatoriallyequivalentsubdivisions}
Given $(v,v')\in\mathbb{R}^2$, a {\em spider} $S(v,v')$ is the union of three rays in~$\mathbb{R}^2$:
\[
S(v,v')=\{(-v+t,-v')\mid t\in\mathbb{R}_{\geq0}\}\cup\{(-v,-v'+t)\mid t\in\mathbb{R}_{\geq0}\}\cup\{(-v-t,-v'-t)\mid t\in\mathbb{R}_{\geq0}\}.
\]
It induces a fan $\Sigma(v,v')$ centered at $(-v,-v')$ with seven cones of various dimensions.

Let $v=(v_1,v_1',\ldots,v_m,v_m')\in\mathbb{R}^{2m}$. The union of the spiders $\cup_{i=1}^mS(v_i,v_i')$ subdivides $\mathbb{R}^2$ into the union of $0$, $1$, and $2$-dimensional convex cells (some of these cells are unbounded). We~denote this convex subdivision of $\mathbb{R}^2$ by $\mathcal{S}(v)$.

Let $v,w\in\mathbb{R}^{2m}$. We say that $\mathcal{S}(v)$ and $\mathcal{S}(w)$ are \emph{combinatorially equivalent} if for all choices of cones $\sigma_i\in\Sigma(v_i,v_i'),\tau_i\in\Sigma(w_i,w_i')$ such that $\sigma_i+(v_i,v_i')=\tau_i+(w_i,w_i')$, we have  $\cap_{i=1}^m\sigma_i\neq\emptyset$ if and only if $\cap_{i=1}^m\tau_i\neq\emptyset$. 
We denote by $[\mathcal{S}(v)]$ the equivalence class of $\mathcal{S}(v)$. 
If~$[\mathcal{S}(v)]=[\mathcal{S}(w)]$, then we say that two cells in $\mathcal{S}(v)$ and $\mathcal{S}(w)$ are \emph{corresponding} if they can be written as $\cap_{i=1}^m\sigma_i$ and $\cap_{i=1}^m\tau_i$ where $\sigma_i+(v_i,v_i')=\tau_i+(w_i,w_i')$ for all $i$ (see Table~\ref{exampleofcombinatoriallyequivalentsubdivision}).
\end{definition}

\begin{table}[hbtp]
\centering
\begin{tabular}{|>{\centering\arraybackslash}m{5cm}|>{\centering\arraybackslash}m{5cm}|}
\hline
\vspace{.1in}
\begin{tikzpicture}[scale=0.5]

	\draw[line width=1pt] (0,0) -- (0+5,0);
	\draw[line width=1pt] (0,0) -- (0,0+5);
	\draw[line width=1pt] (0,0) -- (0-1,0-1);

	\draw[line width=1pt] (4,2) -- (4+1,2);
	\draw[line width=1pt] (4,2) -- (4,2+3);
	\draw[line width=1pt] (4,2) -- (4-3,2-3);

	\draw[line width=1pt] (2,4) -- (2+3,4);
	\draw[line width=1pt] (2,4) -- (2,4+1);
	\draw[line width=1pt] (2,4) -- (2-3,4-3);

	\fill (0,0) circle (5pt);
	\fill (4,2) circle (5pt);
	\fill (2,4) circle (5pt);

	\node at (-0.5,3) {$1$};
	\node at (1,4) {$2$};
	\node at (3,4.7) {$3$};
	\node at (4.6,4.7) {$4$};
	\node at (4.6,3) {$5$};
	\node at (4,1) {$6$};
	\node at (3,-0.7) {$7$};
	\node at (0.5,-0.7) {$8$};
	\node at (-0.5,0.5) {$9$};
	\node at (2,2) {$10$};

\end{tikzpicture}
&\vspace{.1in}
\begin{tikzpicture}[scale=0.5]

	\draw[line width=1pt] (0,0) -- (0+7,0);
	\draw[line width=1pt] (0,0) -- (0,0+7);
	\draw[line width=1pt] (0,0) -- (0-1,0-1);

	\draw[line width=1pt] (6,2) -- (6+1,2);
	\draw[line width=1pt] (6,2) -- (6,2+5);
	\draw[line width=1pt] (6,2) -- (6-3,2-3);

	\draw[line width=1pt] (4,6) -- (4+3,6);
	\draw[line width=1pt] (4,6) -- (4,6+1);
	\draw[line width=1pt] (4,6) -- (4-5,6-5);

	\fill (0,0) circle (5pt);
	\fill (6,2) circle (5pt);
	\fill (4,6) circle (5pt);

	\node at (-0.5,3) {$1$};
	\node at (2,6) {$2$};
	\node at (5,6.7) {$3$};
	\node at (6.6,6.7) {$4$};
	\node at (6.6,4) {$5$};
	\node at (6,1) {$6$};
	\node at (5,-0.7) {$7$};
	\node at (1.5,-0.7) {$8$};
	\node at (-0.5,0.5) {$9$};
	\node at (3.5,3) {$10$};

\end{tikzpicture}
\\
\hline
\vspace{.1in}
\begin{tikzpicture}[scale=0.5]

	\draw[line width=1pt] (0,0) -- (0+4,0);
	\draw[line width=1pt] (0,0) -- (0,0+2);
	\draw[line width=1pt] (0,0) -- (0-1,0-1);

	\draw[line width=1pt] (1,-2) -- (1+3,-2);
	\draw[line width=1pt] (1,-2) -- (1,-2+4);
	\draw[line width=1pt] (1,-2) -- (1-1,-2-1);

	\draw[line width=1pt] (3,1) -- (3+1,1);
	\draw[line width=1pt] (3,1) -- (3,1+1);
	\draw[line width=1pt] (3,1) -- (3-4,1-4);

	\fill (0,0) circle (5pt);
	\fill (1,-2) circle (5pt);
	\fill (3,1) circle (5pt);

\end{tikzpicture}
&\vspace{.1in}
\begin{tikzpicture}[scale=0.5]

	\draw[line width=1pt] (0,0) -- (0+4,0);
	\draw[line width=1pt] (0,0) -- (0,0+3);
	\draw[line width=1pt] (0,0) -- (0-1,0-1);

	\draw[line width=1pt] (2,-1) -- (2+2,-1);
	\draw[line width=1pt] (2,-1) -- (2,-1+4);
	\draw[line width=1pt] (2,-1) -- (2-1,-1-1);

	\draw[line width=1pt] (3,2) -- (3+1,2);
	\draw[line width=1pt] (3,2) -- (3,2+1);
	\draw[line width=1pt] (3,2) -- (3-4,2-4);

	\fill (0,0) circle (5pt);
	\fill (2,-1) circle (5pt);
	\fill (3,2) circle (5pt);

\end{tikzpicture}
\\
\hline
\end{tabular}
\caption{In the top row, an example of combinatorially equivalent subdivisions $\mathcal{S}(v)$ and~ $\mathcal{S}(w)$.  
Corresponding $2$-cells are labeled with the same number. In~the second row, an example of not combinatorially equivalent subdivisions.}
\label{exampleofcombinatoriallyequivalentsubdivision}
\end{table}

\begin{lemma}
\label{preliminaryfibersfamilyoverchowquotient}
Intersections of cones $\sigma\in\mathcal{P}_v^m$ with $v+\mathbb{R}^2\cong\mathbb{R}^2$ give the subdivision~$\mathcal{S}(v)$. In~particular, 
if $v,w\in\mathbb{R}^{2m}$  then
$[\mathcal{S}(v)]=[\mathcal{S}(w)]$ if and only if $[v],[w]\in\mathbb{R}^{2m}/\mathbb{R}^2$ lie in the relative interior of 
the same cone of the quotient fan $\mathcal{Q}_m$.
\end{lemma}

\begin{proof}
Let $C\subseteq\mathbb{R}^2$ be an arbitrary cell in $\mathcal{S}(v)$. Then for $i=1,\ldots,m$ there exist cones $\sigma_i\in\Sigma(v_i,v_i')$ such that $C=\cap_{i=1}^m\sigma_i$. Note that
\[
\tau=((v_1,v_1')+\sigma_1)\times\cdots\times((v_m,v_m')+\sigma_m)
\]
is a cone in the product fan $\mathcal{P}_v^m$. It is clear that
$
v+C=(v+\mathbb{R}^2)\cap\tau
$.
\end{proof}


\begin{table}[hbtp]
\centering
\begin{tabular}{|>{\centering\arraybackslash}m{2.3cm}|>{\centering\arraybackslash}m{2.3cm}|>{\centering\arraybackslash}m{2.3cm}|>{\centering\arraybackslash}m{2.3cm}|>{\centering\arraybackslash}m{2.3cm}|>{\centering\arraybackslash}m{2.3cm}|}
\hline
a&b&c&d&e&f
\\
\hline
\vspace{.1in}
\begin{tikzpicture}[scale=0.4]

	\fill (0,0) circle (5pt);

	\draw[line width=1pt] (0,0) -- (0+1,0);
	\draw[line width=1pt] (0,0) -- (0,0+1);
	\draw[line width=1pt] (0,0) -- (0-1,0-1);

\end{tikzpicture}
&\vspace{.1in}
\begin{tikzpicture}[scale=0.4]

	\fill (0,0) circle (5pt);
	\fill (1,1) circle (5pt);

	\draw[line width=1pt] (0,0) -- (0+2,0);
	\draw[line width=1pt] (0,0) -- (0,0+2);
	\draw[line width=1pt] (0,0) -- (0-1,0-1);

	\draw[line width=1pt] (1,1) -- (1+1,1);
	\draw[line width=1pt] (1,1) -- (1,1+1);
	\draw[line width=1pt] (1,1) -- (1-2,1-2);

\end{tikzpicture}
&\vspace{.1in}
\begin{tikzpicture}[scale=0.4]

	\fill (0,0) circle (5pt);
	\fill (2,1) circle (5pt);

	\draw[line width=1pt] (0,0) -- (0-1,0-1);
	\draw[line width=1pt] (0,0) -- (0+3,0);
	\draw[line width=1pt] (0,0) -- (0,0+2);

	\draw[line width=1pt] (2,1) -- (2+1,1);
	\draw[line width=1pt] (2,1) -- (2,1+1);
	\draw[line width=1pt] (2,1) -- (2-2,1-2);

\end{tikzpicture}
&\vspace{.1in}
\begin{tikzpicture}[scale=0.4]

	\fill (0,0) circle (5pt);
	\fill (0,1) circle (5pt);
	\fill (1,0) circle (5pt);

	\draw[line width=1pt] (0,0) -- (0-1,0-1);
	\draw[line width=1pt] (0,0) -- (0+2,0);
	\draw[line width=1pt] (0,0) -- (0,0+2);

	\draw[line width=1pt] (0,1) -- (0+2,1);
	\draw[line width=1pt] (0,1) -- (0-1,1-1);

	\draw[line width=1pt] (1,0) -- (1,0+2);
	\draw[line width=1pt] (1,0) -- (1-1,0-1);

\end{tikzpicture}
&\vspace{.1in}
\begin{tikzpicture}[scale=0.4]

	\fill (0,0) circle (5pt);
	\fill (1,1) circle (5pt);
	\fill (2,2) circle (5pt);

	\draw[line width=1pt] (0,0) -- (0+3,0);
	\draw[line width=1pt] (0,0) -- (0,0+3);

	\draw[line width=1pt] (1,1) -- (1+2,1);
	\draw[line width=1pt] (1,1) -- (1,1+2);

	\draw[line width=1pt] (2,2) -- (2+1,2);
	\draw[line width=1pt] (2,2) -- (2,2+1);
	\draw[line width=1pt] (2,2) -- (2-3,2-3);

\end{tikzpicture}
&\vspace{.1in}
\begin{tikzpicture}[scale=0.4]

	\fill (0,0) circle (5pt);
	\fill (1,1) circle (5pt);
	\fill (1,0) circle (5pt);

	\draw[line width=1pt] (1,1) -- (1-2,1-2);
	\draw[line width=1pt] (1,1) -- (1+1,1);

	\draw[line width=1pt] (1,0) -- (1-1,0-1);
	\draw[line width=1pt] (1,0) -- (1+1,0);
	\draw[line width=1pt] (1,0) -- (1,0+2);

	\draw[line width=1pt] (0,0) -- (0+1,0);
	\draw[line width=1pt] (0,0) -- (0,0+2);

\end{tikzpicture}
\\
\hline
g&h&i&j&k&l
\\
\hline
\vspace{.1in}
\begin{tikzpicture}[scale=0.4]

	\fill (0,0) circle (5pt);
	\fill (2,1) circle (5pt);
	\fill (1,2) circle (5pt);

	\draw[line width=1pt] (0,0) -- (0+3,0);
	\draw[line width=1pt] (0,0) -- (0,0+3);
	\draw[line width=1pt] (0,0) -- (0-1,0-1);

	\draw[line width=1pt] (1,2) -- (1+2,2);
	\draw[line width=1pt] (1,2) -- (1,2+1);
	\draw[line width=1pt] (1,2) -- (1-2,2-2);

	\draw[line width=1pt] (2,1) -- (2+1,1);
	\draw[line width=1pt] (2,1) -- (2,1+2);
	\draw[line width=1pt] (2,1) -- (2-2,1-2);

\end{tikzpicture}
&\vspace{.1in}
\begin{tikzpicture}[scale=0.4]

	\fill (0,0) circle (5pt);
	\fill (1,1) circle (5pt);
	\fill (3,2) circle (5pt);

	\draw[line width=1pt] (0,0) -- (0+4,0);
	\draw[line width=1pt] (0,0) -- (0,0+3);

	\draw[line width=1pt] (1,1) -- (1,1+2);
	\draw[line width=1pt] (1,1) -- (1+3,1);
	\draw[line width=1pt] (1,1) -- (1-2,1-2);

	\draw[line width=1pt] (3,2) -- (3+1,2);
	\draw[line width=1pt] (3,2) -- (3,2+1);
	\draw[line width=1pt] (3,2) -- (3-3,2-3);

\end{tikzpicture}
&\vspace{.1in}
\begin{tikzpicture}[scale=0.4]

	\fill (0,0) circle (5pt);
	\fill (1,0) circle (5pt);
	\fill (2,1) circle (5pt);

	\draw[line width=1pt] (2,1) -- (2-2,1-2);
	\draw[line width=1pt] (2,1) -- (2+1,1);
	\draw[line width=1pt] (2,1) -- (2,1+1);

	\draw[line width=1pt] (1,0) -- (1,0+2);

	\draw[line width=1pt] (0,0) -- (0,0+2);
	\draw[line width=1pt] (0,0) -- (0+3,0);
	\draw[line width=1pt] (0,0) -- (0-1,0-1);

\end{tikzpicture}
&\vspace{.1in}
\begin{tikzpicture}[scale=0.4]

	\fill (1,0) circle (5pt);
	\fill (0,1) circle (5pt);
	\fill (2,1) circle (5pt);

	\draw[line width=1pt] (1,0) -- (1+2,0);
	\draw[line width=1pt] (1,0) -- (1,0+2);

	\draw[line width=1pt] (0,1) -- (0+3,1);
	\draw[line width=1pt] (0,1) -- (0,1+1);
	\draw[line width=1pt] (0,1) -- (0-1,1-1);

	\draw[line width=1pt] (2,1) -- (2,1+1);
	\draw[line width=1pt] (2,1) -- (2-2,1-2);

\end{tikzpicture}
&\vspace{.1in}
\begin{tikzpicture}[scale=0.4]

	\fill (0,0) circle (5pt);
	\fill (0,1) circle (5pt);
	\fill (2,1) circle (5pt);

	\draw[line width=1pt] (2,1) -- (2-2,1-2);
	\draw[line width=1pt] (2,1) -- (2,1+1);

	\draw[line width=1pt] (0,1) -- (0-1,1-1);
	\draw[line width=1pt] (0,1) -- (0+3,1);

	\draw[line width=1pt] (0,0) -- (0-1,0-1);
	\draw[line width=1pt] (0,0) -- (0+3,0);
	\draw[line width=1pt] (0,0) -- (0,0+2);

\end{tikzpicture}
&\vspace{.1in}
\begin{tikzpicture}[scale=0.4]

	\fill (0,1) circle (5pt);
	\fill (1,0) circle (5pt);
	\fill (3,2) circle (5pt);

	\draw[line width=1pt] (3,2) -- (3-3,2-3);
	\draw[line width=1pt] (3,2) -- (3+1,2);
	\draw[line width=1pt] (3,2) -- (3,2+1);

	\draw[line width=1pt] (1,0) -- (1-1,0-1);
	\draw[line width=1pt] (1,0) -- (1,0+3);
	\draw[line width=1pt] (1,0) -- (1+3,0);

	\draw[line width=1pt] (0,1) -- (0-1,1-1);
	\draw[line width=1pt] (0,1) -- (0,1+2);
	\draw[line width=1pt] (0,1) -- (0+4,1);

\end{tikzpicture}
\\
\hline
m&n&o&p&q&r
\\
\hline
\vspace{.1in}
\begin{tikzpicture}[scale=0.4]

	\fill (0,0) circle (5pt);
	\fill (2,1) circle (5pt);
	\fill (3,2) circle (5pt);

	\draw[line width=1pt] (0,0) -- (0-1,0-1);
	\draw[line width=1pt] (0,0) -- (0+4,0);
	\draw[line width=1pt] (0,0) -- (0,0+3);

	\draw[line width=1pt] (3,2) -- (3-3,2-3);
	\draw[line width=1pt] (3,2) -- (3+1,2);
	\draw[line width=1pt] (3,2) -- (3,2+1);

	\draw[line width=1pt] (2,1) -- (2+2,1);
	\draw[line width=1pt] (2,1) -- (2,1+2);

\end{tikzpicture}
&\vspace{.1in}
\begin{tikzpicture}[scale=0.4]

	\fill (0,2) circle (5pt);
	\fill (1,0) circle (5pt);
	\fill (2,1) circle (5pt);

	\draw[line width=1pt] (2,1) -- (2-2,1-2);
	\draw[line width=1pt] (2,1) -- (2+1,1);
	\draw[line width=1pt] (2,1) -- (2,1+2);

	\draw[line width=1pt] (0,2) -- (0-1,2-1);
	\draw[line width=1pt] (0,2) -- (0+3,2);
	\draw[line width=1pt] (0,2) -- (0,2+1);

	\draw[line width=1pt] (1,0) -- (1+2,0);
	\draw[line width=1pt] (1,0) -- (1,0+3);

\end{tikzpicture}
&\vspace{.1in}
\begin{tikzpicture}[scale=0.4]

	\fill (0,1) circle (5pt);
	\fill (1,0) circle (5pt);
	\fill (3,1) circle (5pt);

	\draw[line width=1pt] (0,1) -- (0,1+1);
	\draw[line width=1pt] (0,1) -- (0+4,1);
	\draw[line width=1pt] (0,1) -- (0-1,1-1);

	\draw[line width=1pt] (1,0) -- (1,0+2);
	\draw[line width=1pt] (1,0) -- (1+3,0);
	\draw[line width=1pt] (1,0) -- (1-1,0-1);

	\draw[line width=1pt] (3,1) -- (3,1+1);
	\draw[line width=1pt] (3,1) -- (3-2,1-2);

\end{tikzpicture}
&\vspace{.1in}
\begin{tikzpicture}[scale=0.4]

	\fill (0,0) circle (5pt);
	\fill (0,2) circle (5pt);
	\fill (2,1) circle (5pt);

	\draw[line width=1pt] (0,0) -- (0-1,0-1);
	\draw[line width=1pt] (0,0) -- (0+3,0);
	\draw[line width=1pt] (0,0) -- (0,0+3);

	\draw[line width=1pt] (2,1) -- (2-2,1-2);
	\draw[line width=1pt] (2,1) -- (2+1,1);
	\draw[line width=1pt] (2,1) -- (2,1+2);

	\draw[line width=1pt] (0,2) -- (0-1,2-1);
	\draw[line width=1pt] (0,2) -- (0+3,2);

\end{tikzpicture}
&\vspace{.1in}
\begin{tikzpicture}[scale=0.4]

	\fill (0,1) circle (5pt);
	\fill (2,2) circle (5pt);
	\fill (1,0) circle (5pt);

	\draw[line width=1pt] (0,1) -- (0-1,1-1);
	\draw[line width=1pt] (0,1) -- (0+3,1);
	\draw[line width=1pt] (0,1) -- (0,1+2);

	\draw[line width=1pt] (1,0) -- (1-1,0-1);
	\draw[line width=1pt] (1,0) -- (1+2,0);
	\draw[line width=1pt] (1,0) -- (1,0+3);

	\draw[line width=1pt] (2,2) -- (2,2+1);
	\draw[line width=1pt] (2,2) -- (2+1,2);
	\draw[line width=1pt] (2,2) -- (2-3,2-3);

\end{tikzpicture}
&\vspace{.1in}
\begin{tikzpicture}[scale=0.4]

	\fill (0,1) circle (5pt);
	\fill (1,0) circle (5pt);
	\fill (2,4) circle (5pt);

	\draw[line width=1pt] (0,1) -- (0-1,1-1);
	\draw[line width=1pt] (0,1) -- (0+3,1);
	\draw[line width=1pt] (0,1) -- (0,1+4);

	\draw[line width=1pt] (1,0) -- (1-1,0-1);
	\draw[line width=1pt] (1,0) -- (1+2,0);
	\draw[line width=1pt] (1,0) -- (1,0+5);

	\draw[line width=1pt] (2,4) -- (2-3,4-3);
	\draw[line width=1pt] (2,4) -- (2+1,4);
	\draw[line width=1pt] (2,4) -- (2,4+1);

\end{tikzpicture}
\\
\hline
s&t&u&&&
\\
\hline
\vspace{.1in}
\begin{tikzpicture}[scale=0.4]

	\fill (0,1) circle (5pt);
	\fill (1,3) circle (5pt);
	\fill (2,0) circle (5pt);

	\draw[line width=1pt] (0,1) -- (0-1,1-1);
	\draw[line width=1pt] (0,1) -- (0+3,1);
	\draw[line width=1pt] (0,1) -- (0,1+3);

	\draw[line width=1pt] (1,3) -- (1-2,3-2);
	\draw[line width=1pt] (1,3) -- (1+2,3);
	\draw[line width=1pt] (1,3) -- (1,3+1);

	\draw[line width=1pt] (2,0) -- (2-1,0-1);
	\draw[line width=1pt] (2,0) -- (2+1,0);
	\draw[line width=1pt] (2,0) -- (2,0+4);

\end{tikzpicture}
&\vspace{.1in}
\begin{tikzpicture}[scale=0.4]

	\fill (0,0) circle (5pt);
	\fill (1,2) circle (5pt);
	\fill (2,4) circle (5pt);

	\draw[line width=1pt] (0,0) -- (0-1,0-1);
	\draw[line width=1pt] (0,0) -- (0+3,0);
	\draw[line width=1pt] (0,0) -- (0,0+5);

	\draw[line width=1pt] (1,2) -- (1-2,2-2);
	\draw[line width=1pt] (1,2) -- (1+2,2);
	\draw[line width=1pt] (1,2) -- (1,2+3);

	\draw[line width=1pt] (2,4) -- (2-3,4-3);
	\draw[line width=1pt] (2,4) -- (2+1,4);
	\draw[line width=1pt] (2,4) -- (2,4+1);

\end{tikzpicture}
&\vspace{.1in}
\begin{tikzpicture}[scale=0.4]

	\fill (-1,0) circle (5pt);
	\fill (2,2) circle (5pt);
	\fill (1,-1) circle (5pt);

	\draw[line width=1pt] (-1,0) -- (-1-1,0-1);
	\draw[line width=1pt] (-1,0) -- (-1+4,0);
	\draw[line width=1pt] (-1,0) -- (-1,0+3);

	\draw[line width=1pt] (1,-1) -- (1-1,-1-1);
	\draw[line width=1pt] (1,-1) -- (1+2,-1);
	\draw[line width=1pt] (1,-1) -- (1,-1+4);

	\draw[line width=1pt] (2,2) -- (2,2+1);
	\draw[line width=1pt] (2,2) -- (2+1,2);
	\draw[line width=1pt] (2,2) -- (2-4,2-4);

\end{tikzpicture}
&\vspace{.1in}
&\vspace{.1in}
&\vspace{.1in}
\\
\hline
\end{tabular}
\caption{All  combinatorial types of Mustafin triangles (up to symmetries) enumerating all cones in $\mathcal{Q}_3$. The table is from \cite[Figure 6]{CHSW11}.}
\label{triplesoflatticesinapartment}
\end{table}

\begin{definition}
\label{fanforfamilyoverchowquotient}
Next we define the family over $Y_{\mathcal{Q}_m}$.
Let $\mathcal{R}_m$ be the fan in $\mathbb{R}^{2m}$ 
with cones 
$\eta=q^{-1}(\xi)\cap\tau$ for  $\xi\in\mathcal{Q}_m$ and $\tau\in\mathcal{P}^m$.
\end{definition}

\begin{proposition}
\label{flatfamilywithreducedfibersontoricchowquotientP2}
We have 
a commutative diagram of morphisms of toric varieties:
\begin{center}
\begin{tikzpicture}[>=angle 90]
\matrix(a)[matrix of math nodes,
row sep=2em, column sep=2em,
text height=1.5ex, text depth=0.25ex]
{Y_{\mathcal{R}_m}&Y_{\mathcal{Q}_m}\times(\mathbb{P}^2)^m\\
Y_{\mathcal{Q}_m}.&\\};
\path[right hook->] (a-1-1) edge node[]{}(a-1-2);
\path[->] (a-1-1) edge node[]{}(a-2-1);
\path[->] (a-1-2) edge node[]{}(a-2-1);
\end{tikzpicture}
\end{center}
The top row is an embedding, and the toric morphism $Y_{\mathcal{R}_m}\rightarrow Y_{\mathcal{Q}_m}$ is flat with reduced fibers.
\end{proposition}

\begin{proof}
Consider the injective linear map given by
$$
\mathbb{R}^{2m}\hookrightarrow(\mathbb{R}^{2m}/\mathbb{R}^2)\times\mathbb{R}^{2m},
\quad
v\mapsto([v],v).
$$
Then $\mathcal{R}_m$ is the restriction to $\mathbb{R}^{2m}$ of the product fan $\mathcal{Q}_m\times\mathcal{P}^m$.
By \cite[Theorem~2.1.4]{Mol21}, 
in order to show that $Y_{\mathcal{R}_m}\rightarrow Y_{\mathcal{Q}_m}$ is flat with reduced fibers,
we need to check that
\begin{enumerate}

\item Every cone $\eta\in\mathcal{R}_m$ surjects onto a cone $\sigma\in\mathcal{Q}_m$;

\item Whenever $q(\eta)=\sigma$, the lattice points in $\eta$ surject onto the lattice points in $\sigma$.
\end{enumerate}

Let $\eta\in\mathcal{R}_m$. Then 
$\eta=q^{-1}(\xi)\cap\tau$ for some cones $\xi\in\mathcal{Q}_m$ and $\tau\in\mathcal{P}^m$. Thus
$
q(\eta)=\xi\cap q(\tau)
$.
By \cite[\S1]{KSZ91}, the quotient fan $\mathcal{Q}_m$ is
the common refinement of the images in $\mathbb{R}^{2m}/\mathbb{R}^2$ of cones in $\mathcal{P}^m$. For this reason, we have that $\xi\cap q(\tau)$ is a cone in $\mathcal{Q}_m$, proving part (1).

For part (2), let $\eta\in\mathcal{R}_m$ and let $[w]\in q(\eta)$ be a lattice point, which means that we can assume $w$ has integral coordinates. 
Write $\eta=q^{-1}(\xi)\cap\tau$ for  $\xi\in\mathcal{Q}_m$,  $\tau\in\mathcal{P}^m$, 
implying that $[w]\in\xi$ and $[w]\in q(\tau)$. Thus, $w\in q^{-1}(\xi)$ and $w+(x,x',\ldots,x,x')\in\tau$ for some $(x,x')\in\mathbb{R}^2$.

Write $\tau=\tau_1\times\cdots\times\tau_m$. 
Then $(x,x')\in\tau_i-(w_i,w_i')$ for all~$i$. So $(x,x')$ is a point in the cell $C=\cap_{i=1}^m(\tau_i-(w_i,w_i'))$ of the subdivision $\mathcal{S}(w)$ of $\mathbb{R}^2$. Notice that $C$ always contains a point $(y,y')$ with integral coordinates. This is because, by the geometry of the fan $\mathcal{P}$, $C$ is a convex polytope with integral vertices and sides with slopes $0,1$, or vertical. 

It follows that $w+(y,y',\ldots,y,y')\in\tau$
and therefore  $w+(y,y',\ldots,y,y')\in q^{-1}(\xi)\cap\tau=\eta$ is a lattice point mapping to $[w]$.
\end{proof}

\begin{remark}
One could adapt the above results for $Y^m$, where $Y$ is an arbitrary  projective toric surface, except for Proposition~\ref{flatfamilywithreducedfibersontoricchowquotientP2}, which 
used specific properties of the fan of~$\mathbb{P}^2$. The~conclusion of Proposition~\ref{flatfamilywithreducedfibersontoricchowquotientP2} holds for any projective toric surface such that the slopes of the rays of its associated fan are in the set $\{0,1,\infty\}$, i.e.~for a toric del Pezzo surface.
\end{remark}

\begin{lemma}
\label{droppingspiderspreservesequivalence}
For $i\in\{1,\ldots,m\}$, we have 
``forgetful'' toric morphisms of toric varieties 
$$Y_{\mathcal{Q}_m}\to Y_{\mathcal{Q}_{m-1}}$$
induced by 
linear maps $p_i\colon\mathbb{R}^{2m}\rightarrow\mathbb{R}^{2m-2}$,
$
v\mapsto(v_1,v_1',\ldots,v_{i-1},v_{i-1}',v_{i+1},v_{i+1}',\ldots,v_m,v_m')
$.
\end{lemma}

\begin{proof}
Let $v,w\in\mathbb{R}^{2m}$. We claim that if $[\mathcal{S}(v)]=[\mathcal{S}(w)]$
then $[\mathcal{S}(p_i(v))]=[\mathcal{S}(p_i(w))]$ for all $i$. But 
this is immediate from Definition~\ref{combinatoriallyequivalentsubdivisions}: for $j\in\{1,\ldots,m\}\setminus\{i\}$, consider cones $\sigma_j\in\Sigma(v_j,v_j'),\tau_j\in\Sigma(w_j,w_j')$ such that $\sigma_j+(v_j,v_j')=\tau_j+(w_j,w_j')$. Assume that $\cap_{j\neq i}\sigma_j\neq\emptyset$. We want to show that $\cap_{j\neq i}\tau_j\neq\emptyset$. There exists $\sigma_i\in\Sigma(v_i,v_i')$ such that $\cap_i\sigma_i\neq\emptyset$ because $\Sigma(v_i,v_i')$ covers the entire $\mathbb{R}^2$. Let $\tau_i\in\Sigma(w_i,w_i')$ such that $\sigma_i+(v_i,v_i')=\tau_i+(w_i,w_i')$. Since $[\mathcal{S}(v)]=[\mathcal{S}(w)]$, we know that also $\cap_j\tau_j\neq\emptyset$, so $\cap_{j\neq i}\tau_j\neq\emptyset$, proving what we want.
\end{proof}

Next we give a modular interpretation of $Y_{\mathcal{Q}_m}$ 
as the moduli space of reducible surfaces with marked points.
The torus $(\mathbb{G}_m^2)^m$ is a configuration space of $m$
points 
$p_1,\ldots,p_m\in\mathbb{G}_m^2$. 
These {\em light} points are allowed to coincide.
If we add three {\em heavy} points 
$e_1,e_2,e_3$, the quotient torus
$(\mathbb{G}_m^2)^m/\mathbb{G}_m^2$ becomes isomorphic to the $\PGL_3$-orbit space
for $m+3$ points $$e_1,e_2,e_3,p_1,\ldots,p_m\in\mathbb{P}^2$$
such that any of the $m$ 
quadruples of the form $(p_i,e_1,e_2,e_3)$ is in linearly general position.
To~reduce to the torus action, we normalize $e_1,e_2,e_3$ to be the standard basis  vectors.
Given $K$-points $p_1,\ldots,p_m\in\mathbb{G}_m^2(K)$, 
we have $m$ lattices $L_i$ for $i=1,\ldots,m$ that stabilize quadruples 
of the form $(e_1,e_2,e_3,p_i)$ (some of these lattices can coincide). Notice that these lattices are in the same apartment that corresponds to the basis $e_1,e_2,e_3$.

\begin{proposition}\label{fibersfamilyfromspiders}
The family $Y_{\mathcal{R}_m}\to Y_{\mathcal{Q}_m}$ is the pullback
of the universal family of the Hilbert scheme of $({\mathbb P}^2)^m$.
Let $\overline{\mathbf{a}}\colon\Spec(R)\rightarrow Y_{\mathcal{Q}_m}$ be a morphism such that $\overline{\mathbf{a}}(\Spec K)$
is contained in the quotient torus. Then 
$\overline{\mathbf{a}}^*Y_{\mathcal{R}_m}$ is isomorphic to the Mustafin join $\mathbb{P}(L_1,\ldots,L_m)$.
\end{proposition}

\begin{proof}
The first part follows from Proposition~\ref{flatfamilywithreducedfibersontoricchowquotientP2}.
For the second part, write $\overline{\mathbf{a}}(\Spec K)=(a_1,\ldots,a_m)\mod \mathbb{G}_m^2$,
an element of $(\mathbb{G}_m^2)^m/\mathbb{G}_m^2(K)$. 
We have that 
$(\overline{\mathbf{a}}|_{\Spec K})^*Y_{\mathcal{R}_m}\cong\mathbb{P}_K^2$ is embedded in $(\mathbb{P}_K^2)^m$ by the diagonal embedding of $\mathbb{P}_K^2$ followed by the automorphism of $(\mathbb{P}_K^2)^m$ induced by the action of $(a_1^{-1},\ldots,a_m^{-1})$. The morphism $Y_{\mathcal{R}_m}\rightarrow Y_{\mathcal{Q}_m}$ is flat by Proposition~\ref{flatfamilywithreducedfibersontoricchowquotientP2}, so $\overline{\mathbf{a}}^*Y_{\mathcal{R}_m}$ equals the Zariski closure of $\mathbb{P}_K^2$ in $(\mathbb{P}_R^2)^m$ under the above embedding. This shows that $\overline{\mathbf{a}}^*Y_{\mathcal{R}_m}$ is equal to the Mustafin join of the lattices $L_i=a_{i1}e_1R+a_{i2}e_2R+a_{i3}e_3R$, where we view $a_i$ as the point in $\mathbb{P}^2$ given by $[a_{i1}:a_{i2}:a_{i3}]$.
\end{proof}

\begin{definition}\label{wrgawrgarhgar}
Let $\lambda_i$ to be the map of lattices $\mathbb{Z}^{2m}/\mathbb{Z}^2\rightarrow\mathbb{Z}^{2m}$ given by
\[
[v]\mapsto(v_1-v_i,v_1'-v_i',\ldots,v_m-v_i,v_m'-v_i').
\]
Each of these maps is a section of the quotient map 
$q\colon\mathbb{Z}^{2m}\to \mathbb{Z}^{2m}/\mathbb{Z}^2$.
\end{definition}

\begin{proposition}\label{qwrgawrhwarh}\label{awrgasrgsrh}
The family
$Y_{{\mathcal{R}}_m}\rightarrow Y_{{\mathcal{Q}}_m}$
admits $m$ ``light'' sections $\ell_1,\ldots,\ell_m$, which are toric morphisms
induced by the linear maps $\lambda_1,\ldots,\lambda_m$, and 
three ``heavy'' constant sections  $t_1,t_2,t_3$
given by $t_i=Y_{{\mathcal{Q}}_m}\times \{e_i\}\times\ldots\times\{e_i\}\subseteq Y_{\mathcal{Q}_m}\times(\mathbb{P}^2)^m$
for $i=1,2,3$ (see Proposition~\ref{flatfamilywithreducedfibersontoricchowquotientP2}).
\end{proposition}

\begin{proof}
We need to show that
each of the maps $\lambda_i$
induces a map of fans from ${\mathcal{Q}}_m$ to ${\mathcal{R}}_m$.
Suppose $[v],[w]\in\mathbb{R}^{2m}/\mathbb{R}^2$ 
are contained in the relative interior of 
a cone $\sigma$ of ${\mathcal{Q}}_m$.
In~order to show that $\lambda_i(\sigma)$ 
is contained in a cone of ${\mathcal{R}}_m$,
it suffices to show that $\lambda_i([v]),\lambda_i([w])$ are contained in the same cone of $\mathcal{P}^m$,
i.e.~that $(v_j-v_i,v_j'-v_i')$ and $(w_j-w_i,w_j'-w_i')$
are contained in the same cone of $\mathcal{P}$ for every $j$.
Note that the cone of $\mathcal{P}$ containing 
$(v_j-v_i,v_j'-v_i')$ (resp.~$(w_j-w_i,w_j'-w_i')$)
is completely determined by the relative position of the spiders
$S(v_i,v_i')$ and $S(v_j,v_j')$ (resp.~
$S(w_i,w_i')$ and $S(w_j,w_j')$).
Since $\mathcal{S}(v)$ and~ $\mathcal{S}(w)$ are combinatorially equivalent, the claim follows.
\end{proof}

\begin{corollary}\label{rgwRGsrgSRH}
Fibers of the family $Y_{\mathcal{R}_m}\rightarrow Y_{\mathcal{Q}_m}$ 
can be computed as in Remark~\ref{degenerationfrompicture}. A~spider decomposition $S(v)$ induces a regular mixed polyhedral subdivision of  $m\Delta_2$.
The~corresponding fiber is a broken toric variety
with irreducible components that correspond to $2$-cells of the subdivision of $m\Delta_2$.
Three heavy sections correspond to vertices of $m\Delta_2$.
Every light section~$\ell_i$ lies in 
the irreducible component that corresponds to the spider $S(v_i,v_i')$.
\end{corollary}

\begin{proof}
After our discussion so far, the only claim needing proof is the one about light sections. We can reduce this claim to considering a one-parameter family $\mathbf{a}\colon\Spec(R)\rightarrow Y_{\mathcal{Q}_m}$ with $\mathbf{a}(\Spec(K))$ contained in the quotient torus. Let us prove the statement for $\ell_1$, which, without loss of generality, we can assume equal to $[1:1:1]$. In particular, we have that $L_1=e_1R+e_2R+e_3R$. Let us rescale the points $[a_{j1}:a_{j2}:a_{j3}]$, $j\geq2$, so that $\val(a_{j1}),\val(a_{j2}),\val(a_{j3})\leq0$, with equality achieved at least once for each $j$. Then $v=(1,1,1)\in L_j\setminus tL_j$ for all $j=1,\ldots,m$. Moreover, $L_1$ is contained in the intersection $L_1\cap\ldots\cap L_m$. Therefore, by Lemma~\ref{smoothsectionifminimum} we have that the limit of $\ell_i$ in the central fiber of the Mustafin join $\mathbb{P}(L_1,\ldots,L_m)$ is a smooth point contained in the primary component corresponding to $L_1$, which is the one corresponding to the $i$-th spider.
\end{proof}

\begin{example}\label{fvqergqebrq}
Let us consider the arc
\[
\mathbf{a}=(e_1=[1:0:0],e_2=[0:1:0],e_3=[0:0:1],a_1=[1:1:1],a_2=[t^2:1:t]).
\]
We studied the $5$-pointed degeneration of $\mathbb{P}^2$ parametrized by $\overline{\mathbf{a}}(0)\in\overline{\mathbf{X}}_{\GP}(3,5)$ in Example~\ref{seconddegeneration5points}. We now compute the $5$-pointed fiber of $Y_{\mathcal{R}_2}\rightarrow Y_{\mathcal{Q}_2}$ over the limit point $\overline{\mathbf{a}}(0)\in Y_{\mathcal{Q}_2}$. By Proposition~\ref{fibersfamilyfromspiders}, such fiber is the central fiber of the Mustafin join of the following two lattices:
\[
L_1=e_1R+e_2R+e_3R,~L_2=t^2e_1R+e_2R+te_3R.
\]
We can compute $\mathbb{P}(L_1,L_2)_\Bbbk$ using Remark~\ref{degenerationfrompicture}, and the result is pictured in Figure~\ref{degeneration5points2parametbyYQ2}. Note that the limit of $a_i$ lies in the primary component corresponding to $L_i$ by Corollary~\ref{rgwRGsrgSRH}.
\end{example}

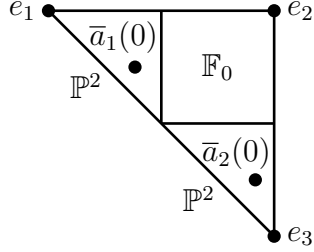
\begin{figure}[hbtp]
\begin{tikzpicture}[scale=0.5]
	\draw[line width=1pt] (0,6) -- (6,6);
	\draw[line width=1pt] (6,0) -- (6,6);
	\draw[line width=1pt] (0,6) -- (6,0);
	\draw[line width=1pt] (3,6) -- (3,3);
	\draw[line width=1pt] (6,3) -- (3,3);

	\node at (1,4) {$\mathbb{P}^2$};

	\node at (4,1) {$\mathbb{P}^2$};

	\node at (4.5,4.5) {$\mathbb{F}_0$};

	\fill (0,6) circle (5pt);
	\fill (6,0) circle (5pt);
	\fill (6,6) circle (5pt);
	\fill (2.3,4.5) circle (5pt);
	\fill (5.5,1.5) circle (5pt);

	\node at (-0.7,6) {$e_1$};
	\node at (6.7,6) {$e_2$};
	\node at (6.7,0) {$e_3$};
	\node at (2,5.3) {$\overline{a}_1(0)$};
	\node at (5,2.2) {$\overline{a}_2(0)$};
\end{tikzpicture}
\caption{Fiber of $Y_{\mathcal{R}_2}\rightarrow Y_{\mathcal{Q}_2}$ over $\overline{\mathbf{a}}(0)$ in Example~\ref{fvqergqebrq}.}
\label{degeneration5points2parametbyYQ2}
\end{figure}


\section{Toric analogue of  \texorpdfstring{$\overline{\mathbf{X}}_{\GP}(3,n)$}{Lg}}\label{asrgasrhadrhd}

We will use the toric Kapranov spaces $Y_{\mathcal{Q}_m}$
in \S\ref{openpatchesfromtoricvarieties}
to construct open patches of the Kapranov space $\overline{\mathbf{X}}(3,n)$ that  cover
an open {\em planar locus} with toroidal singularities.
Its~preimage  in the Gerrizen--Piwek space $\overline{\mathbf{X}}_{\GP}(3,n)$
will be described using toric Gerrizen--Piwek spaces 
$Y_{\widetilde{\mathcal{Q}}_m}$ introduced in this section. 
While the same arc gives different 
fibers of the family $Y_{\mathcal{R}_m}$ 
(see ~Example~\ref{fvqergqebrq})
and the Gerrizen--Piwek family  
(see~Example~\ref{seconddegeneration5points}), 
under some conditions, which we investigate in this and the following sections, the Gerrizen--Piwek fiber can be reconstructed
from the fiber in $Y_{\mathcal{R}_m}$.
To~motivate the construction, write $n=m+3$ and take an arc
$$a=(e_1,e_2,e_3,a_1,\ldots,a_m)\in(\mathbb{P}^2)^n(K)$$ 
in linearly general position.
Consider an apartment $A$ that corresponds to the basis $e_1,e_2,e_3$.
By Remark~\ref{degenerationfrompicture}, a lattice  
$t^\alpha e_1R+t^\beta e_2R+t^\gamma e_3R$ in the apartment corresponds to a lattice point $(-a,-b)=(\gamma-\alpha,\gamma-\beta)\in\mathbb{R}^2$. 
It is the $0$-cell of the spider $S(a,b)$. 

\begin{definition}\label{dualspider} 
The {\em dual spider} is defined as follows:
\[
S(a,b)^\vee=\{(-a-t,-b)\mid t\in\mathbb{R}_{\geq0}\}\cup\{(-a,-b-t)\mid t\in\mathbb{R}_{\geq0}\}\cup\{(-a+t,-b+t)\mid t\in\mathbb{R}_{\geq0}\}.
\]
\end{definition}

\begin{figure}[hbtp]
\begin{tikzpicture}[scale=0.5]

	\draw[line width=1pt] (0,0) -- (0+7,0);
	\draw[line width=1pt] (0,0) -- (0,0+5);
	\draw[line width=1pt] (0,0) -- (0-3,0-3);

	\draw[dashed,line width=1pt] (0,0) -- (0-3,0);
	\draw[dashed,line width=1pt] (0,0) -- (0,0-3);
	\draw[dashed,line width=1pt] (0,0) -- (0+5,0+5);

	\draw[line width=1pt] (4,2) -- (4+3,2);
	\draw[line width=1pt] (4,2) -- (4,2+3);
	\draw[line width=1pt] (4,2) -- (4-5,2-5);

	\draw[dashed,line width=1pt] (4,2) -- (4-7,2);
	\draw[dashed,line width=1pt] (4,2) -- (4,2-5);
	\draw[dashed,line width=1pt] (4,2) -- (4+3,2+3);

	\fill (0,0) circle (5pt);
	\fill (4,2) circle (5pt);
	\fill[red] (2,2) circle (5pt);

	\node at (8.5,0) {$S(a,b)$};
	\node at (8.5,2) {$S(c,d)$};

\end{tikzpicture}
\caption{Intersection point of dashed 
dual spiders $S(a,b)^\vee$ and $S(c,d)^\vee$.}
\label{inducedstablelatticewithdualspiders}
\end{figure}
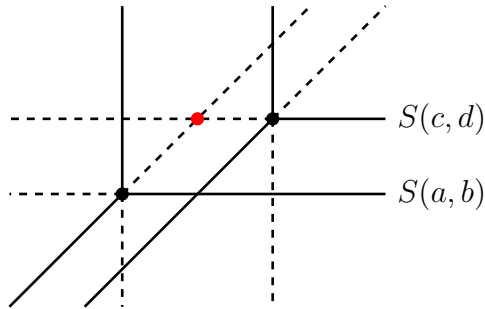

\begin{lemma}\label{wEGsegsEGew}
Consider the stable lattices $L_i$ (resp.~$\bar L_{\alpha\beta,ij}$) stabilizing quadruples $(e_1,e_2,e_3,a_i)$ for $i=1,\ldots,m$ (resp.~$(e_\alpha,e_\beta,a_i,a_j)$ for $\alpha,\beta=1,2,3$, $i,j=1,\ldots,m$). The lattices $L_i$ belong to the apartment $A$ and are given by the integral points $(-v_i,-v_i')$ in $\mathbb{R}^2$. A lattice $\bar L_{\alpha\beta,ij}$ does not have to be in the apartment~$A$, but if it does, then it is the lattice $L_{\alpha\beta,ij}$ that corresponds to the integral point $(-x_{\alpha\beta,ij},-x_{\alpha\beta,ij}')$ from Table~\ref{howtodecidetheredspider}. If this point is different from $(-v_i,-v_i')$ and $(-v_j,-v_j')$ then it is the intersection point $(-a,-b)$ of the dual spiders $S(v_i,v_i')^\vee$ and $S(v_j,v_j')^\vee$.
\end{lemma}

\begin{proof}
Let $i,j\in\{1,\ldots,m\}$ be distinct indices and consider the stable lattices $L_i,L_j$. All their possible reciprocal positions as points $(-v_i,-v_i'),(-v_j,-v_j')$ in $\mathbb{R}^2$ are listed in Table~\ref{howtodecidetheredspider}. Let us determine the integral points $(x_{\alpha\beta,ij},x_{\alpha\beta,ij}')$.

Assume that the reciprocal position of $L_i,L_j$ is as shown in the first entry of the first row of Table~\ref{howtodecidetheredspider}. Then the limits of the points $p_1,p_2,p_3,a_i,a_j$ in the central fiber of the Mustafin join $\mathbb{P}(L_i,L_j)$ are shown in the left side of Figure~\ref{determinelatticeLalphabetaij}. We can then directly compute that
\[
(x_{12,ij},x_{12,ij}')=(v_j,v_j'),~(x_{13,ij},x_{13,ij}')=(v_i,v_i'),~(x_{23,ij},x_{23,ij}')=(v_j,v_j').
\]
An analogous argument applies to the remaining cases in the first two rows of Table~\ref{howtodecidetheredspider}.

Now assume that the reciprocal position of $L_i,L_j$ is as shown in the first entry of the third row of Table~\ref{howtodecidetheredspider} instead. Define $L_{ij}$ as the lattice corresponding to the intersection point $(-a,-b)$ of the dual spiders $S(v_i,v_i')^\vee$ and $S(v_j,v_j')^\vee$. Then the limits of the points $p_1,p_2,p_3,a_i,a_j$ in the central fiber of the Mustafin join $\mathbb{P}(L_i,L_j,L_{ij})$ can be found in the right side of Figure~\ref{determinelatticeLalphabetaij}. From this we can directly compute that
\[
(x_{12,ij},x_{12,ij}')=(v_i,v_i'),~(x_{13,ij},x_{13,ij}')=(a,b),~(x_{23,ij},x_{23,ij}')=(v_j,v_j').
\]
The remaining cases in the last two rows of Table~\ref{howtodecidetheredspider} are handled in an analogous way.
\end{proof}

\begin{figure}[hbtp]
\begin{tikzpicture}[scale=0.5]

	\draw[line width=1pt] (0,6) -- (6,6);
	\draw[line width=1pt] (6,6) -- (6,0);
	\draw[line width=1pt] (6,0) -- (0,6);
	\draw[line width=1pt] (3,6) -- (6,3);
	\fill (6,0) circle (5pt);
	\fill (0,6) circle (5pt);
	\fill (6,6) circle (5pt);
	\fill (4,3) circle (5pt);
	\fill (4.5,5.5) circle (5pt);

	\node at (6.7,6) {$e_2$};
	\node at (5.2,4.9) {{\small$\overline{a}_j(0)$}};
	\node at (-.7,6) {$e_1$};
	\node at (6.7,0) {$e_3$};
	\node at (4.1,3.7) {{\small$\overline{a}_i(0)$}};
	\node at (4.5,6.5) {$L_j$};
	\node at (2,3) {$L_i$};

	\draw[->,line width=1pt] (2,2) -- (1,1);
	\draw[line width=1pt] (0,0) -- (-3,0);
	\draw[line width=1pt] (-3,0) -- (0,-3);
	\draw[line width=1pt] (0,0) -- (0,-3);
	\fill (0,0) circle (5pt);
	\fill (-3,0) circle (5pt);
	\fill (0,-3) circle (5pt);
	\fill (-1,-1) circle (5pt);
	\node at (0.9,0) {$e_2,$};
	\node at (2.4,0) {{\small$\overline{a}_j(0)$}};
	\node at (-3.7,0) {$e_1$};
	\node at (0.7,-3) {$e_3$};
	\node at (-1.4,-0.5) {{\small$\overline{a}_i(0)$}};

	\draw[->,line width=1pt] (7,2) -- (8,1);
	\draw[line width=1pt] (12,0) -- (9,0);
	\draw[line width=1pt] (9,0) -- (12,-3);
	\draw[line width=1pt] (12,0) -- (12,-3);
	\fill (12,0) circle (5pt);
	\fill (9,0) circle (5pt);
	\fill (12,-3) circle (5pt);
	\fill (10.5,-0.5) circle (5pt);
	\fill (11,-2) circle (5pt);
	\node at (12.8,0) {$e_2$};
	\node at (8.3,0) {$e_1$};
	\node at (12.7,-3) {$e_3$};
	\node at (11.2,-1.1) {{\small$\overline{a}_j(0)$}};
	\node at (9.8,-2.3) {{\small$\overline{a}_i(0)$}};

	\node at (6.5,-7) {};

\end{tikzpicture}
\hspace{-0.2in}
\begin{tikzpicture}[scale=0.5]

	\draw[line width=1pt] (0,6) -- (6,6);
	\draw[line width=1pt] (6,6) -- (6,0);
	\draw[line width=1pt] (6,0) -- (0,6);
	\draw[line width=1pt] (4.5,6) -- (4.5,1.5);
	\draw[line width=1pt] (1.5,4.5) -- (6,4.5);
	\fill (6,0) circle (5pt);
	\fill (0,6) circle (5pt);
	\fill (6,6) circle (5pt);
	\fill (5.6,3) circle (5pt);
	\fill (3,5.6) circle (5pt);

	\node at (6.7,6) {$e_2$};
	\node at (3.2,5) {{\small$\overline{a}_i(0)$}};
	\node at (-0.7,6) {$e_1$};
	\node at (6.7,0) {$e_3$};
	\node at (5.4,2.3) {{\small$\overline{a}_j(0)$}};
	\node at (2,6.5) {$L_i$};
	\node at (6.7,1.5) {$L_j$};
	\node at (2.3,2.8) {$L_{ij}$};

	\draw[->,line width=1pt] (-1,5) -- (-3,5);
	\draw[line width=1pt] (-4,6) -- (-7,6);
	\draw[line width=1pt] (-7,6) -- (-4,3);
	\draw[line width=1pt] (-4,6) -- (-4,3);
	\fill (-7,6) circle (5pt);
	\fill (-4,6) circle (5pt);
	\fill (-4,3) circle (5pt);
	\fill (-5,5) circle (5pt);
	\node at (-3.3,6) {$e_2$};
	\node at (-5.3,5.5) {{\small$\overline{a}_i(0)$}};
	\node at (-7.7,6) {$e_1$};
	\node at (-3,3) {$e_3,$};
	\node at (-1.5,3) {{\small$\overline{a}_j(0)$}};

	\draw[->,line width=1pt] (2,2) -- (1,1);
	\draw[line width=1pt] (0,0) -- (-3,0);
	\draw[line width=1pt] (-3,0) -- (0,-3);
	\draw[line width=1pt] (0,0) -- (0,-3);
	\fill (0,0) circle (5pt);
	\fill (-3,0) circle (5pt);
	\fill (0,-3) circle (5pt);
	\fill (0,-1.5) circle (5pt);
	\fill (-1.5,0) circle (5pt);
	\node at (0.7,0) {$e_2$};
	\node at (-1.4,0.6) {{\small$\overline{a}_i(0)$}};
	\node at (-3.7,0) {$e_1$};
	\node at (0.7,-3) {$e_3$};
	\node at (1.2,-1.5) {{\small$\overline{a}_j(0)$}};

	\draw[->,line width=1pt] (5,-1) -- (5,-3);
	\draw[line width=1pt] (6,-4) -- (3,-4);
	\draw[line width=1pt] (3,-4) -- (6,-7);
	\draw[line width=1pt] (6,-4) -- (6,-7);
	\fill (6,-4) circle (5pt);
	\fill (3,-4) circle (5pt);
	\fill (6,-7) circle (5pt);
	\fill (5,-5) circle (5pt);
	\node at (6.7,-4) {$e_2$};
	\node at (2.6,-3.5) {$e_1,$};
	\node at (4,-3.5) {{\small$\overline{a}_i(0)$}};
	\node at (6.7,-7) {$e_3$};
	\node at (4.6,-4.5) {{\small$\overline{a}_j(0)$}};

\end{tikzpicture}
\caption{In the proof of Lemma~\ref{wEGsegsEGew}, the central fibers of the Mustafin joins $\mathbb{P}(L_i,L_j),\mathbb{P}(L_i,L_j,L_{ij})$ and the limits of $e_1,e_2,e_3,a_i,a_j$.}
\label{determinelatticeLalphabetaij}
\end{figure}
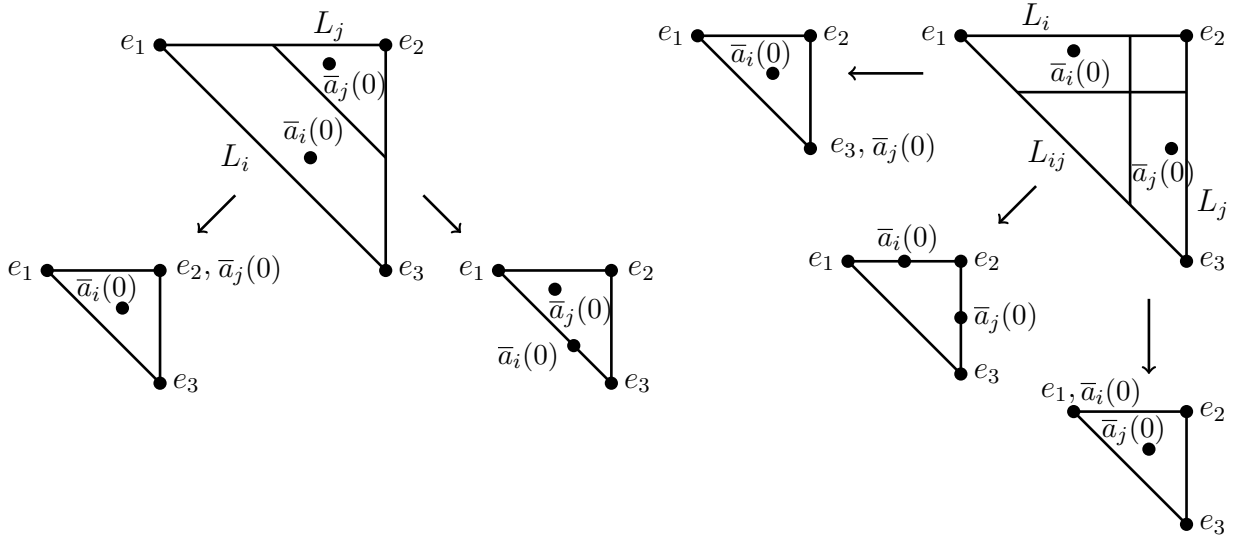

\begin{table}[hbtp]
\centering
\begin{tabular}{|>{\centering\arraybackslash}m{5cm}|>{\centering\arraybackslash}m{5cm}|>{\centering\arraybackslash}m{5cm}|}
\hline
\vspace{.1in}
\begin{tikzpicture}[scale=0.5]

	\draw[line width=1pt] (0,0) -- (0+2,0);
	\draw[line width=1pt] (0,0) -- (0,0+2);
	\draw[line width=1pt] (0,0) -- (0-1,0-1);

	\draw[line width=1pt] (1,1) -- (1+1,1);
	\draw[line width=1pt] (1,1) -- (1,1+1);
	\draw[line width=1pt] (1,1) -- (1-2,1-2);

	\fill (0,0) circle (5pt);
	\fill (1,1) circle (5pt);

	\node at (-0.5,0.5) {$i$};
	\node at (1.5,1.5) {$j$};

\end{tikzpicture}
&\vspace{.1in}
\begin{tikzpicture}[scale=0.5]

	\draw[line width=1pt] (0,0) -- (0+1,0);
	\draw[line width=1pt] (0,0) -- (0,0+1);
	\draw[line width=1pt] (0,0) -- (0-1,0-1);

	\draw[line width=1pt] (0,-1) -- (0+1,-1);
	\draw[line width=1pt] (0,-1) -- (0,-1+1);
	\draw[line width=1pt] (0,-1) -- (0-1,-1-1);

	\fill (0,0) circle (5pt);
	\fill (0,-1) circle (5pt);

	\node at (-0.5,0.5) {$i$};
	\node at (0.5,-1.5) {$j$};

\end{tikzpicture}
&\vspace{.1in}
\begin{tikzpicture}[scale=0.5]

	\draw[line width=1pt] (0,0) -- (0+1,0);
	\draw[line width=1pt] (0,0) -- (0,0+1);
	\draw[line width=1pt] (0,0) -- (0-1,0-1);

	\draw[line width=1pt] (-1,0) -- (-1+1,0);
	\draw[line width=1pt] (-1,0) -- (-1,0+1);
	\draw[line width=1pt] (-1,0) -- (-1-1,0-1);

	\fill (0,0) circle (5pt);
	\fill (-1,0) circle (5pt);

	\node at (0.5,-0.5) {$i$};
	\node at (-1.5,0.5) {$j$};

\end{tikzpicture}
\\
\hline
$(x_{12,ij},x_{12,ij}')=(v_j,v_j')$
$(x_{13,ij},x_{13,ij}')=(v_i,v_i')$
$(x_{23,ij},x_{23,ij}')=(v_j,v_j')$
&
$(x_{12,ij},x_{12,ij}')=(v_i,v_i')$
$(x_{13,ij},x_{13,ij}')=(v_j,v_j')$
$(x_{23,ij},x_{23,ij}')=(v_j,v_j')$
&
$(x_{12,ij},x_{12,ij}')=(v_j,v_j')$
$(x_{13,ij},x_{13,ij}')=(v_j,v_j')$
$(x_{23,ij},x_{23,ij}')=(v_i,v_i')$
\\
\hline
\vspace{.1in}
\begin{tikzpicture}[scale=0.5]

	\draw[line width=1pt] (0,0) -- (0+2,0);
	\draw[line width=1pt] (0,0) -- (0,0+2);
	\draw[line width=1pt] (0,0) -- (0-1,0-1);

	\draw[line width=1pt] (1,1) -- (1+1,1);
	\draw[line width=1pt] (1,1) -- (1,1+1);
	\draw[line width=1pt] (1,1) -- (1-2,1-2);

	\fill (0,0) circle (5pt);
	\fill (1,1) circle (5pt);

	\node at (-0.5,0.5) {$j$};
	\node at (1.5,1.5) {$i$};

\end{tikzpicture}
&\vspace{.1in}
\begin{tikzpicture}[scale=0.5]

	\draw[line width=1pt] (0,0) -- (0+1,0);
	\draw[line width=1pt] (0,0) -- (0,0+1);
	\draw[line width=1pt] (0,0) -- (0-1,0-1);

	\draw[line width=1pt] (0,-1) -- (0+1,-1);
	\draw[line width=1pt] (0,-1) -- (0,-1+1);
	\draw[line width=1pt] (0,-1) -- (0-1,-1-1);

	\fill (0,0) circle (5pt);
	\fill (0,-1) circle (5pt);

	\node at (-0.5,0.5) {$j$};
	\node at (0.5,-1.5) {$i$};

\end{tikzpicture}
&\vspace{.1in}
\begin{tikzpicture}[scale=0.5]

	\draw[line width=1pt] (0,0) -- (0+1,0);
	\draw[line width=1pt] (0,0) -- (0,0+1);
	\draw[line width=1pt] (0,0) -- (0-1,0-1);

	\draw[line width=1pt] (-1,0) -- (-1+1,0);
	\draw[line width=1pt] (-1,0) -- (-1,0+1);
	\draw[line width=1pt] (-1,0) -- (-1-1,0-1);

	\fill (0,0) circle (5pt);
	\fill (-1,0) circle (5pt);

	\node at (0.5,-0.5) {$j$};
	\node at (-1.5,0.5) {$i$};

\end{tikzpicture}
\\
\hline
$(x_{12,ij},x_{12,ij}')=(v_i,v_i')$
$(x_{13,ij},x_{13,ij}')=(v_i,v_i')$
$(x_{23,ij},x_{23,ij}')=(v_j,v_j')$
&
$(x_{12,ij},x_{12,ij}')=(v_j,v_j')$
$(x_{13,ij},x_{13,ij}')=(v_i,v_i')$
$(x_{23,ij},x_{23,ij}')=(v_i,v_i')$
&
$(x_{12,ij},x_{12,ij}')=(v_i,v_i')$
$(x_{13,ij},x_{13,ij}')=(v_i,v_i')$
$(x_{23,ij},x_{23,ij}')=(v_j,v_j')$
\\
\hline
\vspace{.1in}
\begin{tikzpicture}[scale=0.5]

	\draw[line width=1pt] (0,0) -- (0+2,0);
	\draw[line width=1pt] (0,0) -- (0,0+1);
	\draw[line width=1pt] (0,0) -- (0-1,0-1);

	\draw[line width=1pt] (1,-1) -- (1+1,-1);
	\draw[line width=1pt] (1,-1) -- (1,-1+2);
	\draw[line width=1pt] (1,-1) -- (1-1,-1-1);

	\draw[red,line width=1pt] (0,-1) -- (0+1,-1);
	\draw[red,line width=1pt] (0,-1) -- (0,-1+1);
	\draw[red,line width=1pt] (0,-1) -- (0-1,-1-1);

	\fill (0,0) circle (5pt);
	\fill (1,-1) circle (5pt);
	\fill[red] (0,-1) circle (5pt);

	\node at (-0.5,0.5) {$i$};
	\node at (1.5,-1.5) {$j$};

\end{tikzpicture}
&\vspace{.1in}
\begin{tikzpicture}[scale=0.5]

	\draw[line width=1pt] (0,0) -- (0+2,0);
	\draw[line width=1pt] (0,0) -- (0,0+3);
	\draw[line width=1pt] (0,0) -- (0-1,0-1);

	\draw[line width=1pt] (1,2) -- (1+1,2);
	\draw[line width=1pt] (1,2) -- (1,2+1);
	\draw[line width=1pt] (1,2) -- (1-2,2-2);

	\draw[red,line width=1pt] (1,1) -- (1+1,1);
	\draw[red,line width=1pt] (1,1) -- (1,1+1);
	\draw[red,line width=1pt] (1,1) -- (1-1,1-1);

	\fill (0,0) circle (5pt);
	\fill (1,2) circle (5pt);
	\fill[red] (1,1) circle (5pt);

	\node at (0.5,-0.5) {$i$};
	\node at (1.5,2.5) {$j$};

\end{tikzpicture}
&\vspace{.1in}
\begin{tikzpicture}[scale=0.5]

	\draw[line width=1pt] (0,0) -- (0+3,0);
	\draw[line width=1pt] (0,0) -- (0,0+2);
	\draw[line width=1pt] (0,0) -- (0-1,0-1);

	\draw[line width=1pt] (2,1) -- (2+1,1);
	\draw[line width=1pt] (2,1) -- (2,1+1);
	\draw[line width=1pt] (2,1) -- (2-2,1-2);

	\draw[red,line width=1pt] (1,1) -- (1+1,1);
	\draw[red,line width=1pt] (1,1) -- (1,1+1);
	\draw[red,line width=1pt] (1,1) -- (1-1,1-1);

	\fill (0,0) circle (5pt);
	\fill (2,1) circle (5pt);
	\fill[red] (1,1) circle (5pt);

	\node at (-0.5,0.5) {$i$};
	\node at (2.5,1.5) {$j$};

\end{tikzpicture}
\\
\hline
$(x_{12,ij},x_{12,ij}')=(v_i,v_i')$
$(x_{13,ij},x_{13,ij}')=(a,b)$
$(x_{23,ij},x_{23,ij}')=(v_j,v_j')$
&
$(x_{12,ij},x_{12,ij}')=(v_j,v_j')$
$(x_{13,ij},x_{13,ij}')=(v_i,v_i')$
$(x_{23,ij},x_{23,ij}')=(a,b)$
&
$(x_{12,ij},x_{12,ij}')=(a,b)$
$(x_{13,ij},x_{13,ij}')=(v_i,v_i')$
$(x_{23,ij},x_{23,ij}')=(v_j,v_j')$
\\
\hline
\vspace{.1in}
\begin{tikzpicture}[scale=0.5]

	\draw[line width=1pt] (0,0) -- (0+1,0);
	\draw[line width=1pt] (0,0) -- (0,0+2);
	\draw[line width=1pt] (0,0) -- (0-1,0-1);

	\draw[line width=1pt] (-1,1) -- (-1+2,1);
	\draw[line width=1pt] (-1,1) -- (-1,1+1);
	\draw[line width=1pt] (-1,1) -- (-1-1,1-1);

	\draw[red,line width=1pt] (-1,0) -- (-1+1,0);
	\draw[red,line width=1pt] (-1,0) -- (-1,0+1);
	\draw[red,line width=1pt] (-1,0) -- (-1-1,0-1);

	\fill (0,0) circle (5pt);
	\fill (-1,1) circle (5pt);
	\fill[red] (-1,0) circle (5pt);

	\node at (0.5,-0.5) {$i$};
	\node at (-1.5,1.5) {$j$};

\end{tikzpicture}
&\vspace{.1in}
\begin{tikzpicture}[scale=0.5]

	\draw[line width=1pt] (0,0) -- (0+1,0);
	\draw[line width=1pt] (0,0) -- (0,0+1);
	\draw[line width=1pt] (0,0) -- (0-2,0-2);

	\draw[line width=1pt] (-1,-2) -- (-1+2,-2);
	\draw[line width=1pt] (-1,-2) -- (-1,-2+3);
	\draw[line width=1pt] (-1,-2) -- (-1-1,-2-1);

	\draw[red,line width=1pt] (0,-1) -- (0+1,-1);
	\draw[red,line width=1pt] (0,-1) -- (0,-1+1);
	\draw[red,line width=1pt] (0,-1) -- (0-1,-1-1);

	\fill (0,0) circle (5pt);
	\fill (-1,-2) circle (5pt);
	\fill[red] (0,-1) circle (5pt);

	\node at (0.5,0.5) {$i$};
	\node at (-0.5,-2.5) {$j$};

\end{tikzpicture}
&\vspace{.1in}
\begin{tikzpicture}[scale=0.5]

	\draw[line width=1pt] (0,0) -- (0+1,0);
	\draw[line width=1pt] (0,0) -- (0,0+1);
	\draw[line width=1pt] (0,0) -- (0-2,0-2);

	\draw[line width=1pt] (-2,-1) -- (-2+3,-1);
	\draw[line width=1pt] (-2,-1) -- (-2,-1+2);
	\draw[line width=1pt] (-2,-1) -- (-2-1,-1-1);

	\draw[red,line width=1pt] (-1,0) -- (-1+1,0);
	\draw[red,line width=1pt] (-1,0) -- (-1,0+1);
	\draw[red,line width=1pt] (-1,0) -- (-1-1,0-1);

	\fill (0,0) circle (5pt);
	\fill (-2,-1) circle (5pt);
	\fill[red] (-1,0) circle (5pt);

	\node at (0.5,0.5) {$i$};
	\node at (-2.5,-0.5) {$j$};

\end{tikzpicture}
\\
\hline
$(x_{12,ij},x_{12,ij}')=(v_j,v_j')$
$(x_{13,ij},x_{13,ij}')=(a,b)$
$(x_{23,ij},x_{23,ij}')=(v_i,v_i')$
&
$(x_{12,ij},x_{12,ij}')=(a,b)$
$(x_{13,ij},x_{13,ij}')=(v_j,v_j')$
$(x_{23,ij},x_{23,ij}')=(v_i,v_i')$
&
$(x_{12,ij},x_{12,ij}')=(a,b)$
$(x_{13,ij},x_{13,ij}')=(v_i,v_i')$
$(x_{23,ij},x_{23,ij}')=(v_j,v_j')$
\\
\hline
\end{tabular}
\caption{ The red dot (intersection of dual spiders) is in position $(-a,-b)\in\mathbb{R}^2$.}
\label{howtodecidetheredspider}
\end{table}


\begin{definition}\label{refinementusingKapranovspiders}
We let 
$$N=\binom{3}{2}\binom{m}{2}.$$
Let $\pi\colon\mathbb{R}^{2m+2N}\rightarrow\mathbb{R}^{2m}$ be the projection map
and let $\iota\colon\mathbb{R}^{2m}\to \mathbb{R}^{2m+2N}$ be
the function
that
sends a vector $v$  to the vector $(v,x)$ such that 
the components $x_{\alpha\beta,ij}, x_{\alpha\beta,ij}'$
of the vector $x$ are given in Table~\ref{howtodecidetheredspider}. 
The induced maps $\mathbb{R}^{2m+2N}/\mathbb{R}^{2}\rightarrow\mathbb{R}^{2m}/\mathbb{R}^{2}$ and 
$\mathbb{R}^{2m}/\mathbb{R}^{2}\to\mathbb{R}^{2m+2N}/\mathbb{R}^{2}$
are well-defined and we denote them by $\bar\pi$ and $\bar\iota$,
respectively.
We define
\[
\widetilde{\mathcal{S}}(v)=\mathcal{S}(\iota(v)),
\]
see Figure~\ref{exampleofrefinementusingKapranovspiders} for an example.
The subdivision $\widetilde{\mathcal{S}}(v)$ refines $\mathcal{S}(v)$.

\begin{figure}[hbtp]
\begin{tikzpicture}[scale=0.5]

	\draw[line width=1pt] (0,0) -- (0+7,0);
	\draw[line width=1pt] (0,0) -- (0,0+4);
	\draw[line width=1pt] (0,0) -- (0-2,0-2);

	\draw[line width=1pt] (3,2) -- (3+4,2);
	\draw[line width=1pt] (3,2) -- (3,2+2);
	\draw[line width=1pt] (3,2) -- (3-4,2-4);

	\draw[line width=1pt] (5,-1) -- (5+2,-1);
	\draw[line width=1pt] (5,-1) -- (5,-1+5);
	\draw[line width=1pt] (5,-1) -- (5-1,-1-1);

	\fill (0,0) circle (5pt);
	\fill (3,2) circle (5pt);
	\fill (5,-1) circle (5pt);

\end{tikzpicture}
\hspace{1in}
\begin{tikzpicture}[scale=0.5]

	\draw[line width=1pt] (0,0) -- (0+7,0);
	\draw[line width=1pt] (0,0) -- (0,0+4);
	\draw[line width=1pt] (0,0) -- (0-2,0-2);

	\draw[line width=1pt] (3,2) -- (3+4,2);
	\draw[line width=1pt] (3,2) -- (3,2+2);
	\draw[line width=1pt] (3,2) -- (3-4,2-4);

	\draw[line width=1pt] (5,-1) -- (5+2,-1);
	\draw[line width=1pt] (5,-1) -- (5,-1+5);
	\draw[line width=1pt] (5,-1) -- (5-1,-1-1);

	\draw[red,line width=1pt] (2,2) -- (2+1,2);
	\draw[red,line width=1pt] (2,2) -- (2,2+2);
	\draw[red,line width=1pt] (2,2) -- (2-2,2-2);

	\draw[red,line width=1pt] (0,-1) -- (0+5,-1);
	\draw[red,line width=1pt] (0,-1) -- (0,-1+1);

	\draw[red,line width=1pt] (3,-1) -- (3,-1+3);
	\draw[red,line width=1pt] (3,-1) -- (3-1,-1-1);

	\fill (0,0) circle (5pt);
	\fill (3,2) circle (5pt);
	\fill (5,-1) circle (5pt);
	\fill[red] (2,2) circle (5pt);
	\fill[red] (0,-1) circle (5pt);
	\fill[red] (3,-1) circle (5pt);

\end{tikzpicture}
\caption{Example of $\mathcal{S}(v)$ (on the left) and its refinement $\widetilde{\mathcal{S}}(v)$ (on the right).}
\label{exampleofrefinementusingKapranovspiders}
\end{figure}
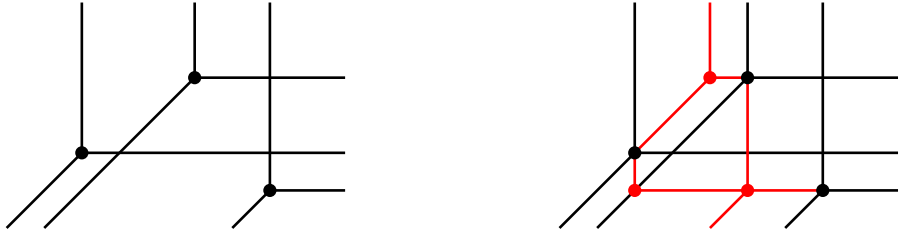
\end{definition}

We note that the image of $\iota$ consists of entire cones
of the fan $\mathcal{Q}_{m+N}$:

\begin{lemma}
\label{localleftinverseofpi}
Let $v\in\mathbb{R}^{2m},z\in\mathbb{R}^{2m+2N}$. If $[\mathcal{S}(\iota(v))]=[\mathcal{S}(z)]$, then $z=\iota(\pi(z))$.
\end{lemma}

\begin{proof}
Write $\iota(v)=(v,x),z=(\pi(z),y)$, where
\[
x=(\ldots,x_{\alpha\beta,ij},x_{\alpha\beta,ij}',\ldots),~y=(\ldots,y_{\alpha\beta,ij},y_{\alpha\beta,ij}',\ldots).
\]
Recall that the pairs $(x_{\alpha\beta,ij},x_{\alpha\beta,ij}')$ are defined according to Definition~\ref{refinementusingKapranovspiders}. Fix distinct indices $\alpha,\beta\in[3]$, $i,j\in[m]$, and for simplicity of notation set $v_{ij}=(v_i,v_i',v_j,v_j'),z_{ij}=(z_i,z_i',z_j,z_j')$. Note that the hypothesis $[\mathcal{S}(\iota(v))]=[\mathcal{S}(z)]$ implies that
\begin{equation}
\label{combinatorialequivalence3spiders}
[\mathcal{S}(v_{ij},x_{\alpha\beta,ij},x_{\alpha\beta,ij}')]=[\mathcal{S}(z_{ij},y_{\alpha\beta,ij},y_{\alpha\beta,ij}')].
\end{equation}
Based on the definition of $\iota(v)$, we have the following three cases on $(x_{\alpha\beta,ij},x_{\alpha\beta,ij}')$:

If $(x_{\alpha\beta,ij},x_{\alpha\beta,ij}')=(v_i,v_i')$, then we show that $(y_{\alpha\beta,ij},y_{\alpha\beta,ij}')=(z_i,z_i')$. Up to symmetries, we have two possibilities for the reciprocal position of the points $(-v_i,-v_i')$, $(-v_j,-v_j')$ in $\mathbb{R}^2$.
Say we are in the first case of the first row of Table~\ref{howtodecidetheredspider}. Consider the cones $\sigma_i=\{(-v_i,-v_i')\}\in\Sigma(v_i,v_i')$, $\sigma_j\in\Sigma(v_j,v_j')$, $\sigma=\{(-x_{\alpha\beta,ij},-x_{\alpha\beta,ij}')\}\in\Sigma(x_{\alpha\beta,ij},x_{\alpha\beta,ij}')$ such that $\sigma_i\cap\sigma_j\cap\sigma=\{(-v_i,-v_i')\}$ and $\sigma_j$ is a ray. Consider the corresponding cones $\tau_i=\{(-z_i,-z_i')\}\in\Sigma(z_i,z_i')$, $\tau_j\in\Sigma(z_j,z_j')$, $\tau=\{(-y_{\alpha\beta,ij},-y_{\alpha\beta,ij}')\}\in\Sigma(y_{\alpha\beta,ij},y_{\alpha\beta,ij}')$. Then also $\tau_i\cap\tau_j\cap\tau\neq\emptyset$ by \eqref{combinatorialequivalence3spiders}, and hence this intersection has to be equal to $\{(-z_i,-z_i')\}$, showing that $(y_{\alpha\beta,ij},y_{\alpha\beta,ij}')=(z_i,z_i')$.
The second possibility for the reciprocal position of $(-v_i,-v_i')$, $(-v_j,-v_j')$, up to symmetry, is the first case of the third row of Table~\ref{howtodecidetheredspider}. Here we use the same argument as in the previous case, but we take $\sigma_j$ to be $2$-dimensional instead.

If $(x_{\alpha\beta,ij},x_{\alpha\beta,ij}')=(v_j,v_j')$, then we show that $(y_{\alpha\beta,ij},y_{\alpha\beta,ij}')=(z_j,z_j')$ by an analogous argument to the one above.

Finally, assume $(x_{\alpha\beta,ij},x_{\alpha\beta,ij}')$ is the unique intersection point $(-a,-b)$ of the dual spiders $S(v_i,v_i')^\vee,S(v_j,v_j')^\vee$. Consider the cones given by $\sigma_i\in\Sigma(v_i,v_i')$, $\sigma_j\in\Sigma(v_j,v_j')$, $\sigma=\{(-x_{\alpha\beta,ij},-x_{\alpha\beta,ij}')\}\in\Sigma(x_{\alpha\beta,ij},x_{\alpha\beta,ij}')$ such that $\sigma_i\cap\sigma_j\cap\sigma=\{(-a,-b)\}$ and $\sigma_i,\sigma_j$ are rays. Consider the corresponding cones $\tau_i\in\Sigma(z_i,z_i')$, $\tau_j\in\Sigma(z_j,z_j')$, $\tau=\{(-y_{\alpha\beta,ij},-y_{\alpha\beta,ij}')\}\in\Sigma(y_{\alpha\beta,ij},y_{\alpha\beta,ij}')$. Then  $\tau_i\cap\tau_j\cap\tau\neq\emptyset$ by \eqref{combinatorialequivalence3spiders}, and hence $(-y_{\alpha\beta,ij},-y_{\alpha\beta,ij}')$ is the unique intersection point of the dual spiders $S(z_i,z_i')^\vee,S(z_j,z_j')^\vee$.
\end{proof}

\begin{lemdef}\label{qergaerhaerh}
There is a unique complete 
fan $\widetilde{\mathcal{Q}}_m$ in $\mathbb{R}^{2m}/\mathbb{R}^2$ such that
$\bar\iota$ is linear on every cone $\sigma$ of 
$\widetilde{\mathcal{Q}}_m$ and $\bar\iota(\sigma)$ is a cone
of the quotient fan $\mathcal{Q}_{m+N}$ in $\mathbb{R}^{2m+2N}/\mathbb{R}^2$.
Vectors $[v],[w]\in \mathbb{R}^{2m}/\mathbb{R}^{2}$ 
lie in the relative interior of the same cone $\sigma$ of the fan $\widetilde{\mathcal{Q}}_m$ if and only if $[\widetilde{\mathcal{S}}(v)]=[\widetilde{\mathcal{S}}(w)]$.
The collection of cones
$\bar\iota(\sigma)$ is a subfan $\mathcal{Q}_{m+N}'$ 
of $\mathcal{Q}_{m+N}$
of 
cones corresponding to subdivisions $\mathcal{S}(z)$, $z\in\mathbb{R}^{2m+2N}$, 
such that
$\mathcal{S}(z)=\widetilde{\mathcal{S}}(\pi(z))$.
\end{lemdef}

\begin{proof}
Given $v\in\mathbb{R}^{2m}$ and two distinct indices $i,j\in[m]$, then the possible reciprocal positions for $(-v_i,-v_i'),(-v_j,-v_j')\in\mathbb{R}^2$, if they are distinct points, are pictured in Table~\ref{howtodecidetheredspider}. Each possibility is described by a set of inequalities and/or equalities. 
Therefore, given $v\in\mathbb{R}^{2m}$, for each pair of distinct indices $i,j\in[m]$ we have a set of inequalities and/or equalities, and considering all these together defines a convex cone $\sigma_v\subseteq\mathbb{R}^{2m}$. (Clearly, if $w$ is in the relative interior of $\sigma_v$, then $\sigma_w=\sigma_v$.) We denote by $\mathcal{W}$ the collection of all these cones in~$\mathbb{R}^{2m}$. We have that $\iota$ is linear on each cone in $\mathcal{W}$ because each $(x_{\alpha\beta,ij},x_{\alpha\beta,ij}')$ is linear by its definition. Note that given $v\in\mathbb{R}^{2m}$ and $(a,b)\in\mathbb{R}^2$, then $v$ and $v+(a,b,\ldots,a,b)$ satisfy the same set of inequalities. Hence, the collection $q(\mathcal{W})$ of images of cones in $\mathcal{W}$ under the quotient $q\colon\mathbb{R}^{2m}\rightarrow\mathbb{R}^{2m}/\mathbb{R}^2$ is a collection of cones on which $\overline{\iota}$ is linear.

We now define $\mathcal{Q}_{m+N}'$ to be the collection of cones $\sigma\in\mathcal{Q}_{m+N}$ such that $\overline{\iota}([v])$ is in the relative interior of $\sigma$ for some $[v]\in\mathbb{R}^{2m}/\mathbb{R}^2$. Note that a cone $\sigma\in\mathcal{Q}_{m+N}'$ is entirely contained in the image of $\overline{\iota}$ by Lemma~\ref{localleftinverseofpi}. Moreover, $\overline{\iota}$ is linear on $(\overline{\iota})^{-1}(\sigma)$. To prove this, it is enough to show the following. Let $\overline{\iota}([v])$ be in the relative interior of $\sigma$ and let $[z]$ be any other point in the relative interior of $\sigma$. We want to show that $\overline{\iota}$ is linear on the $2$-dimensional cone generated by $\overline{\pi}([z])$ and $[v]$. Equivalently, $\iota$ is linear on the $2$-dimensional cone generated by $\pi(z)$ and $v$. But this is clear because $z=\iota(\pi(z))$ by Lemma~\ref{localleftinverseofpi}, and $\pi(z),v$ belong to the same cone of 
$\mathcal{W}$ because $[\mathcal{S}(z)]=[\mathcal{S}(\iota(v))]$.

Next, we observe that $\mathcal{Q}_{m+N}'$ is a fan. Consider $\sigma\in\mathcal{Q}_{m+N}'$, and let $\varphi\subseteq\sigma$ be a face. By the argument above, $\sigma$ is contained in the image of $\overline{\iota}$, hence also $\varphi$ is, so we can certainly find $[v]\in\mathbb{R}^{2m}/\mathbb{R}^2$ such that $\overline{\iota}([v])$ is in the relative interior of $\varphi$. Combined with the fact that $\varphi\in\mathcal{Q}_{m+N}$ because $\mathcal{Q}_{m+N}$ is a fan, we can conclude that $\varphi\in\mathcal{Q}_{m+N}'$.

We then define $\widetilde{\mathcal{Q}}_m$ to be the collection of preimages of cones in $\mathcal{Q}_{m+N}'$ under $\overline{\iota}$. By the discussion above, $\widetilde{\mathcal{Q}}_m$ is a fan, and it is clearly complete. 

We conclude with the combinatorial interpretation of the cones in $\widetilde{\mathcal{Q}}_m$. Let $[v],[w]\in\mathbb{R}^{2m}/\mathbb{R}^2$ be in the relative interior of the same cone $(\overline{\iota})^{-1}(\sigma)\in\widetilde{\mathcal{Q}}_m$, where $\sigma\in\mathcal{Q}_{m+N}'$. Then $\overline{\iota}([v]),\overline{\iota}([w])$ are in the relative interior of $\sigma$, implying that $[\mathcal{S}(\iota(v))]=[\mathcal{S}(\iota(w))]$. Conversely, assume that $[\mathcal{S}(\iota(v))]=[\mathcal{S}(\iota(w))]$. This means that $\overline{\iota}([v]),\overline{\iota}([w])$ are in the relative interior of the same cone $\sigma\in\mathcal{Q}_{m+N}$, which by definition is also a cone of $\mathcal{Q}_{m+N}'$. Hence, $[v],[w]\in(\overline{\iota})^{-1}(\sigma)\in\widetilde{\mathcal{Q}}_m$, proving the last statement.
\end{proof}

\begin{definition}
The variety $Y_{\widetilde{\mathcal{Q}}_m}$ is called the \emph{toric Gerritzen--Piwek space}. 
\end{definition}

\begin{proposition}
\label{combinatorialequivalencetildesubdivisionsanddroppingspiders}
There is a birational toric  morphism
$
Y_{\widetilde{\mathcal{Q}}_m}\rightarrow Y_{\mathcal{Q}_m}
$, a~``forgetful'' toric  morphism 
$Y_{\widetilde{\mathcal{Q}}_m}\rightarrow Y_{\widetilde{\mathcal{Q}}_{[m]\setminus\{i\}}}$
for each $i=1,\ldots,m$ and  
a commutative diagram
\begin{center}
\begin{tikzpicture}[>=angle 90]
\matrix(a)[matrix of math nodes,
row sep=2em, column sep=2em,
text height=1.5ex, text depth=0.25ex]
{
Y_{\widetilde{\mathcal{Q}}_m}  &  Y_{\widetilde{\mathcal{Q}}_{[m]\setminus\{i\}}}\\
Y_{{\mathcal{Q}}_m}    &  Y_{{\mathcal{Q}}_{[m]\setminus\{i\}}}.\\
};
\path[->] (a-1-1) edge node[above]{}(a-1-2);
\path[->] (a-1-1) edge node[left]{}(a-2-1);
\path[->] (a-2-1) edge node[below]{}(a-2-2);
\path[->] (a-1-2) edge node[right]{}(a-2-2);
\end{tikzpicture}
\end{center}
\end{proposition}

\begin{proof}
We claim that the fan $\widetilde{\mathcal{Q}}_m$
refines the fan ${\mathcal{Q}}_m$.
We need to show that 
if $v,w\in\mathbb{R}^{2m}$ and 
$
[\widetilde{\mathcal{S}}(v)]=[\widetilde{\mathcal{S}}(w)]$
then $[\mathcal{S}(v)]=[\mathcal{S}(w)]$.
But this is a consequence of Lemma~\ref{droppingspiderspreservesequivalence}. 

For the second statement, 
let $v,w\in\mathbb{R}^{2m}$ and consider the projections $p_i\colon\mathbb{R}^{2m}\rightarrow\mathbb{R}^{2m-2}$ for $i=1,\ldots,m$. We claim that 
if $\widetilde{\mathcal{S}}(v)$ and $\widetilde{\mathcal{S}}(w)$ are combinatorially equivalent, then $\widetilde{\mathcal{S}}(p_i(v))$ and $\widetilde{\mathcal{S}}(p_i(w))$ are combinatorial equivalent. Write $\widetilde{\mathcal{S}}(v)=\mathcal{S}(v,x),\widetilde{\mathcal{S}}(w)=\mathcal{S}(w,y)$ for $x,y\in\mathbb{R}^{\binom{m}{2}}$. Notice that $\widetilde{\mathcal{S}}(p_i(v))$ is obtained from $\mathcal{S}(v,x)$ by dropping the pairs $(v_i,v_i')$ and $(x_{ij},x_{ij}')$ for $j\neq i$, and similarly for $\widetilde{\mathcal{S}}(p_i(w))$. It remains to apply
Lemma~\ref{droppingspiderspreservesequivalence}.
\end{proof}


\begin{table}[hbtp]
\centering
\begin{tabular}{|>{\centering\arraybackslash}m{2.3cm}|>{\centering\arraybackslash}m{2.3cm}|>{\centering\arraybackslash}m{2.3cm}|>{\centering\arraybackslash}m{2.3cm}|>{\centering\arraybackslash}m{2.3cm}|>{\centering\arraybackslash}m{2.3cm}|}
\hline
a&b&c&d&e&f
\\
\hline
\vspace{.1in}
\begin{tikzpicture}[scale=0.4]

	\fill (0,0) circle (5pt);

	\draw[line width=1pt] (0,0) -- (0+1,0);
	\draw[line width=1pt] (0,0) -- (0,0+1);
	\draw[line width=1pt] (0,0) -- (0-1,0-1);

\end{tikzpicture}
&\vspace{.1in}
\begin{tikzpicture}[scale=0.4]

	\fill (0,0) circle (5pt);
	\fill (1,1) circle (5pt);

	\draw[line width=1pt] (0,0) -- (0+2,0);
	\draw[line width=1pt] (0,0) -- (0,0+2);
	\draw[line width=1pt] (0,0) -- (0-1,0-1);

	\draw[line width=1pt] (1,1) -- (1+1,1);
	\draw[line width=1pt] (1,1) -- (1,1+1);
	\draw[line width=1pt] (1,1) -- (1-2,1-2);

\end{tikzpicture}
&\vspace{.1in}
\begin{tikzpicture}[scale=0.4]

	\fill (0,0) circle (5pt);
	\fill (2,1) circle (5pt);
	\fill[red] (1,1) circle (5pt);

	\draw[line width=1pt] (0,0) -- (0-1,0-1);
	\draw[line width=1pt] (0,0) -- (0+3,0);
	\draw[line width=1pt] (0,0) -- (0,0+2);

	\draw[line width=1pt] (2,1) -- (2+1,1);
	\draw[line width=1pt] (2,1) -- (2,1+1);
	\draw[line width=1pt] (2,1) -- (2-2,1-2);

	\draw[red,line width=1pt] (1,1) -- (1+1,1);
	\draw[red,line width=1pt] (1,1) -- (1,1+1);
	\draw[red,line width=1pt] (1,1) -- (1-1,1-1);

\end{tikzpicture}
&\vspace{.1in}
\begin{tikzpicture}[scale=0.4]

	\fill (0,0) circle (5pt);
	\fill (0,1) circle (5pt);
	\fill (1,0) circle (5pt);

	\draw[line width=1pt] (0,0) -- (0-1,0-1);
	\draw[line width=1pt] (0,0) -- (0+2,0);
	\draw[line width=1pt] (0,0) -- (0,0+2);

	\draw[line width=1pt] (0,1) -- (0+2,1);
	\draw[line width=1pt] (0,1) -- (0-1,1-1);

	\draw[line width=1pt] (1,0) -- (1,0+2);
	\draw[line width=1pt] (1,0) -- (1-1,0-1);

\end{tikzpicture}
&\vspace{.1in}
\begin{tikzpicture}[scale=0.4]

	\fill (0,0) circle (5pt);
	\fill (1,1) circle (5pt);
	\fill (2,2) circle (5pt);

	\draw[line width=1pt] (0,0) -- (0+3,0);
	\draw[line width=1pt] (0,0) -- (0,0+3);

	\draw[line width=1pt] (1,1) -- (1+2,1);
	\draw[line width=1pt] (1,1) -- (1,1+2);

	\draw[line width=1pt] (2,2) -- (2+1,2);
	\draw[line width=1pt] (2,2) -- (2,2+1);
	\draw[line width=1pt] (2,2) -- (2-3,2-3);

\end{tikzpicture}
&\vspace{.1in}
\begin{tikzpicture}[scale=0.4]

	\fill (0,0) circle (5pt);
	\fill (1,1) circle (5pt);
	\fill (1,0) circle (5pt);

	\draw[line width=1pt] (1,1) -- (1-2,1-2);
	\draw[line width=1pt] (1,1) -- (1+1,1);

	\draw[line width=1pt] (1,0) -- (1-1,0-1);
	\draw[line width=1pt] (1,0) -- (1+1,0);
	\draw[line width=1pt] (1,0) -- (1,0+2);

	\draw[line width=1pt] (0,0) -- (0+1,0);
	\draw[line width=1pt] (0,0) -- (0,0+2);

\end{tikzpicture}
\\
\hline
g&h&i&j&k&l
\\
\hline
\vspace{.1in}
\begin{tikzpicture}[scale=0.4]

	\fill (0,0) circle (5pt);
	\fill (2,1) circle (5pt);
	\fill (1,2) circle (5pt);
	\fill[red] (1,1) circle (5pt);

	\draw[line width=1pt] (0,0) -- (0+3,0);
	\draw[line width=1pt] (0,0) -- (0,0+3);
	\draw[line width=1pt] (0,0) -- (0-1,0-1);

	\draw[line width=1pt] (1,2) -- (1+2,2);
	\draw[line width=1pt] (1,2) -- (1,2+1);
	\draw[line width=1pt] (1,2) -- (1-2,2-2);

	\draw[line width=1pt] (2,1) -- (2+1,1);
	\draw[line width=1pt] (2,1) -- (2,1+2);
	\draw[line width=1pt] (2,1) -- (2-2,1-2);

	\draw[red,line width=1pt] (1,1) -- (1+1,1);
	\draw[red,line width=1pt] (1,1) -- (1,1+1);
	\draw[red,line width=1pt] (1,1) -- (1-1,1-1);

\end{tikzpicture}
&\vspace{.1in}
\begin{tikzpicture}[scale=0.4]

	\fill (0,0) circle (5pt);
	\fill (1,1) circle (5pt);
	\fill (3,2) circle (5pt);
	\fill[red] (2,2) circle (5pt);

	\draw[line width=1pt] (0,0) -- (0+4,0);
	\draw[line width=1pt] (0,0) -- (0,0+3);

	\draw[line width=1pt] (1,1) -- (1,1+2);
	\draw[line width=1pt] (1,1) -- (1+3,1);
	\draw[line width=1pt] (1,1) -- (1-2,1-2);

	\draw[line width=1pt] (3,2) -- (3+1,2);
	\draw[line width=1pt] (3,2) -- (3,2+1);
	\draw[line width=1pt] (3,2) -- (3-3,2-3);

	\draw[red,line width=1pt] (2,2) -- (2+1,2);
	\draw[red,line width=1pt] (2,2) -- (2,2+1);
	\draw[red,line width=1pt] (2,2) -- (2-1,2-1);

\end{tikzpicture}
&\vspace{.1in}
\begin{tikzpicture}[scale=0.4]

	\fill (0,0) circle (5pt);
	\fill (1,0) circle (5pt);
	\fill (2,1) circle (5pt);
	\fill[red] (1,1) circle (5pt);

	\draw[line width=1pt] (2,1) -- (2-2,1-2);
	\draw[line width=1pt] (2,1) -- (2+1,1);
	\draw[line width=1pt] (2,1) -- (2,1+1);

	\draw[line width=1pt] (1,0) -- (1,0+2);

	\draw[line width=1pt] (0,0) -- (0,0+2);
	\draw[line width=1pt] (0,0) -- (0+3,0);
	\draw[line width=1pt] (0,0) -- (0-1,0-1);

	\draw[red,line width=1pt] (1,1) -- (1+1,1);
	\draw[red,line width=1pt] (1,1) -- (1-1,1-1);

\end{tikzpicture}
&\vspace{.1in}
\begin{tikzpicture}[scale=0.4]

	\fill (1,0) circle (5pt);
	\fill (0,1) circle (5pt);
	\fill (2,1) circle (5pt);
	\fill[red] (0,0) circle (5pt);

	\draw[line width=1pt] (1,0) -- (1+2,0);
	\draw[line width=1pt] (1,0) -- (1,0+2);

	\draw[line width=1pt] (0,1) -- (0+3,1);
	\draw[line width=1pt] (0,1) -- (0,1+1);
	\draw[line width=1pt] (0,1) -- (0-1,1-1);

	\draw[line width=1pt] (2,1) -- (2,1+1);
	\draw[line width=1pt] (2,1) -- (2-2,1-2);

	\draw[red,line width=1pt] (0,0) -- (0+1,0);
	\draw[red,line width=1pt] (0,0) -- (0,0+1);
	\draw[red,line width=1pt] (0,0) -- (0-1,0-1);

\end{tikzpicture}
&\vspace{.1in}
\begin{tikzpicture}[scale=0.4]

	\fill (0,0) circle (5pt);
	\fill (0,1) circle (5pt);
	\fill (2,1) circle (5pt);
	\fill[red] (1,1) circle (5pt);

	\draw[line width=1pt] (2,1) -- (2-2,1-2);
	\draw[line width=1pt] (2,1) -- (2,1+1);

	\draw[line width=1pt] (0,1) -- (0-1,1-1);
	\draw[line width=1pt] (0,1) -- (0+3,1);

	\draw[line width=1pt] (0,0) -- (0-1,0-1);
	\draw[line width=1pt] (0,0) -- (0+3,0);
	\draw[line width=1pt] (0,0) -- (0,0+2);

	\draw[red,line width=1pt] (1,1) -- (1-1,1-1);
	\draw[red,line width=1pt] (1,1) -- (1,1+1);

\end{tikzpicture}
&\vspace{.1in}
\begin{tikzpicture}[scale=0.4]

	\fill[red] (0,0) circle (5pt);
	\fill[red] (1,2) circle (5pt);
	\fill (0,1) circle (5pt);
	\fill (1,0) circle (5pt);
	\fill (3,2) circle (5pt);

	\draw[line width=1pt] (3,2) -- (3-3,2-3);
	\draw[line width=1pt] (3,2) -- (3+1,2);
	\draw[line width=1pt] (3,2) -- (3,2+1);

	\draw[line width=1pt] (1,0) -- (1-1,0-1);
	\draw[line width=1pt] (1,0) -- (1,0+3);
	\draw[line width=1pt] (1,0) -- (1+3,0);

	\draw[line width=1pt] (0,1) -- (0-1,1-1);
	\draw[line width=1pt] (0,1) -- (0,1+2);
	\draw[line width=1pt] (0,1) -- (0+4,1);

	\draw[red,line width=1pt] (1,2) -- (1-1,2-1);
	\draw[red,line width=1pt] (1,2) -- (1+2,2);

	\draw[red,line width=1pt] (0,0) -- (0-1,0-1);
	\draw[red,line width=1pt] (0,0) -- (0+1,0);
	\draw[red,line width=1pt] (0,0) -- (0,0+1);

\end{tikzpicture}
\\
\hline
m&n&o&p&q&r
\\
\hline
\vspace{.1in}
\begin{tikzpicture}[scale=0.4]

	\fill (0,0) circle (5pt);
	\fill[red] (1,1) circle (5pt);
	\fill[red] (2,2) circle (5pt);
	\fill (2,1) circle (5pt);
	\fill (3,2) circle (5pt);

	\draw[line width=1pt] (0,0) -- (0-1,0-1);
	\draw[line width=1pt] (0,0) -- (0+4,0);
	\draw[line width=1pt] (0,0) -- (0,0+3);

	\draw[line width=1pt] (3,2) -- (3-3,2-3);
	\draw[line width=1pt] (3,2) -- (3+1,2);
	\draw[line width=1pt] (3,2) -- (3,2+1);

	\draw[line width=1pt] (2,1) -- (2+2,1);
	\draw[line width=1pt] (2,1) -- (2,1+2);

	\draw[red,line width=1pt] (1,1) -- (1+1,1);
	\draw[red,line width=1pt] (1,1) -- (1,1+2);

	\draw[red,line width=1pt] (2,2) -- (2-2,2-2);
	\draw[red,line width=1pt] (2,2) -- (2+1,2);

\end{tikzpicture}
&\vspace{.1in}
\begin{tikzpicture}[scale=0.4]

	\fill[red] (0,0) circle (5pt);
	\fill[red] (0,1) circle (5pt);
	\fill (0,2) circle (5pt);
	\fill (1,0) circle (5pt);
	\fill (2,1) circle (5pt);

	\draw[line width=1pt] (2,1) -- (2-2,1-2);
	\draw[line width=1pt] (2,1) -- (2+1,1);
	\draw[line width=1pt] (2,1) -- (2,1+2);

	\draw[line width=1pt] (0,2) -- (0-1,2-1);
	\draw[line width=1pt] (0,2) -- (0+3,2);
	\draw[line width=1pt] (0,2) -- (0,2+1);

	\draw[line width=1pt] (1,0) -- (1+2,0);
	\draw[line width=1pt] (1,0) -- (1,0+3);

	\draw[red,line width=1pt] (0,0) -- (0-1,0-1);
	\draw[red,line width=1pt] (0,0) -- (0+1,0);
	\draw[red,line width=1pt] (0,0) -- (0,0+2);

	\draw[red,line width=1pt] (0,1) -- (0-1,1-1);
	\draw[red,line width=1pt] (0,1) -- (0+2,1);

\end{tikzpicture}
&\vspace{.1in}
\begin{tikzpicture}[scale=0.4]

	\fill[red] (0,0) circle (5pt);
	\fill (0,1) circle (5pt);
	\fill (1,0) circle (5pt);
	\fill (3,1) circle (5pt);
	\fill[red] (2,1) circle (5pt);

	\draw[line width=1pt] (0,1) -- (0,1+1);
	\draw[line width=1pt] (0,1) -- (0+4,1);
	\draw[line width=1pt] (0,1) -- (0-1,1-1);

	\draw[line width=1pt] (1,0) -- (1,0+2);
	\draw[line width=1pt] (1,0) -- (1+3,0);
	\draw[line width=1pt] (1,0) -- (1-1,0-1);

	\draw[line width=1pt] (3,1) -- (3,1+1);
	\draw[line width=1pt] (3,1) -- (3-2,1-2);

	\draw[red,line width=1pt] (0,0) -- (0,0+1);
	\draw[red,line width=1pt] (0,0) -- (0+1,0);
	\draw[red,line width=1pt] (0,0) -- (0-1,0-1);

	\draw[red,line width=1pt] (2,1) -- (2,1+1);
	\draw[red,line width=1pt] (2,1) -- (2-1,1-1);

\end{tikzpicture}
&\vspace{.1in}
\begin{tikzpicture}[scale=0.4]

	\fill (0,0) circle (5pt);
	\fill[red] (0,1) circle (5pt);
	\fill (0,2) circle (5pt);
	\fill[red] (1,1) circle (5pt);
	\fill (2,1) circle (5pt);

	\draw[line width=1pt] (0,0) -- (0-1,0-1);
	\draw[line width=1pt] (0,0) -- (0+3,0);
	\draw[line width=1pt] (0,0) -- (0,0+3);

	\draw[line width=1pt] (2,1) -- (2-2,1-2);
	\draw[line width=1pt] (2,1) -- (2+1,1);
	\draw[line width=1pt] (2,1) -- (2,1+2);

	\draw[line width=1pt] (0,2) -- (0-1,2-1);
	\draw[line width=1pt] (0,2) -- (0+3,2);

	\draw[red,line width=1pt] (0,1) -- (0-1,1-1);
	\draw[red,line width=1pt] (0,1) -- (0+2,1);

	\draw[red,line width=1pt] (1,1) -- (1-1,1-1);
	\draw[red,line width=1pt] (1,1) -- (1,1+2);

\end{tikzpicture}
&\vspace{.1in}
\begin{tikzpicture}[scale=0.4]

	\fill[red] (0,0) circle (5pt);
	\fill (0,1) circle (5pt);
	\fill[red] (1,2) circle (5pt);
	\fill (2,2) circle (5pt);
	\fill[red] (2,1) circle (5pt);
	\fill (1,0) circle (5pt);

	\draw[red,line width=1pt] (0,0) -- (0,0+1);
	\draw[red,line width=1pt] (0,0) -- (0+1,0);

	\draw[line width=1pt] (0,1) -- (0-1,1-1);
	\draw[line width=1pt] (0,1) -- (0+3,1);
	\draw[line width=1pt] (0,1) -- (0,1+2);

	\draw[line width=1pt] (1,0) -- (1-1,0-1);
	\draw[line width=1pt] (1,0) -- (1+2,0);
	\draw[line width=1pt] (1,0) -- (1,0+3);

	\draw[line width=1pt] (2,2) -- (2,2+1);
	\draw[line width=1pt] (2,2) -- (2+1,2);
	\draw[line width=1pt] (2,2) -- (2-3,2-3);

	\draw[red,line width=1pt] (1,2) -- (1-1,2-1);
	\draw[red,line width=1pt] (1,2) -- (1+1,2);

	\draw[red,line width=1pt] (2,1) -- (2-1,1-1);
	\draw[red,line width=1pt] (2,1) -- (2,1+1);

\end{tikzpicture}
&\vspace{.1in}
\begin{tikzpicture}[scale=0.4]

	\fill[red] (0,0) circle (5pt);
	\fill (0,1) circle (5pt);
	\fill[red] (2,1) circle (5pt);
	\fill[red] (2,3) circle (5pt);
	\fill (1,0) circle (5pt);
	\fill (2,4) circle (5pt);

	\draw[line width=1pt] (0,1) -- (0-1,1-1);
	\draw[line width=1pt] (0,1) -- (0+3,1);
	\draw[line width=1pt] (0,1) -- (0,1+4);

	\draw[line width=1pt] (1,0) -- (1-1,0-1);
	\draw[line width=1pt] (1,0) -- (1+2,0);
	\draw[line width=1pt] (1,0) -- (1,0+5);

	\draw[line width=1pt] (2,4) -- (2-3,4-3);
	\draw[line width=1pt] (2,4) -- (2+1,4);
	\draw[line width=1pt] (2,4) -- (2,4+1);

	\draw[red,line width=1pt] (0,0) -- (0,0+1);
	\draw[red,line width=1pt] (0,0) -- (0+1,0);
	\draw[red,line width=1pt] (0,0) -- (0-1,0-1);

	\draw[red,line width=1pt] (2,1) -- (2,1+3);
	\draw[red,line width=1pt] (2,1) -- (2-1,1-1);

	\draw[red,line width=1pt] (2,3) -- (2+1,3);
	\draw[red,line width=1pt] (2,3) -- (2-2,3-2);

\end{tikzpicture}
\\
\hline
s&t&u&&&
\\
\hline
\vspace{.1in}
\begin{tikzpicture}[scale=0.4]

	\fill[red] (0,0) circle (5pt);
	\fill (0,1) circle (5pt);
	\fill[red] (1,2) circle (5pt);
	\fill (1,3) circle (5pt);
	\fill[red] (1,0) circle (5pt);
	\fill (2,0) circle (5pt);

	\draw[line width=1pt] (0,1) -- (0-1,1-1);
	\draw[line width=1pt] (0,1) -- (0+3,1);
	\draw[line width=1pt] (0,1) -- (0,1+3);

	\draw[line width=1pt] (1,3) -- (1-2,3-2);
	\draw[line width=1pt] (1,3) -- (1+2,3);
	\draw[line width=1pt] (1,3) -- (1,3+1);

	\draw[line width=1pt] (2,0) -- (2-1,0-1);
	\draw[line width=1pt] (2,0) -- (2+1,0);
	\draw[line width=1pt] (2,0) -- (2,0+4);

	\draw[red,line width=1pt] (0,0) -- (0-1,0-1);
	\draw[red,line width=1pt] (0,0) -- (0+2,0);
	\draw[red,line width=1pt] (0,0) -- (0,0+1);

	\draw[red,line width=1pt] (1,0) -- (1-1,0-1);
	\draw[red,line width=1pt] (1,0) -- (1,0+3);

	\draw[red,line width=1pt] (1,2) -- (1-1,2-1);
	\draw[red,line width=1pt] (1,2) -- (1+2,2);

\end{tikzpicture}
&\vspace{.1in}
\begin{tikzpicture}[scale=0.4]

	\fill (0,0) circle (5pt);
	\fill (1,2) circle (5pt);
	\fill (2,4) circle (5pt);
	\fill[red] (1,1) circle (5pt);
	\fill[red] (2,2) circle (5pt);
	\fill[red] (2,3) circle (5pt);

	\draw[line width=1pt] (0,0) -- (0-1,0-1);
	\draw[line width=1pt] (0,0) -- (0+3,0);
	\draw[line width=1pt] (0,0) -- (0,0+5);

	\draw[line width=1pt] (1,2) -- (1-2,2-2);
	\draw[line width=1pt] (1,2) -- (1+2,2);
	\draw[line width=1pt] (1,2) -- (1,2+3);

	\draw[line width=1pt] (2,4) -- (2-3,4-3);
	\draw[line width=1pt] (2,4) -- (2+1,4);
	\draw[line width=1pt] (2,4) -- (2,4+1);

	\draw[red,line width=1pt] (1,1) -- (1+2,1);
	\draw[red,line width=1pt] (1,1) -- (1,1+1);

	\draw[red,line width=1pt] (2,2) -- (2-2,2-2);
	\draw[red,line width=1pt] (2,2) -- (2,2+2);

	\draw[red,line width=1pt] (2,3) -- (2-1,3-1);
	\draw[red,line width=1pt] (2,3) -- (2+1,3);

\end{tikzpicture}
&\vspace{.1in}
\begin{tikzpicture}[scale=0.4]

	\fill[red] (-1,-1) circle (5pt);
	\fill (-1,0) circle (5pt);
	\fill[red] (1,2) circle (5pt);
	\fill (2,2) circle (5pt);
	\fill[red] (2,0) circle (5pt);
	\fill (1,-1) circle (5pt);

	\draw[line width=1pt] (-1,0) -- (-1-1,0-1);
	\draw[line width=1pt] (-1,0) -- (-1+4,0);
	\draw[line width=1pt] (-1,0) -- (-1,0+3);

	\draw[line width=1pt] (1,-1) -- (1-1,-1-1);
	\draw[line width=1pt] (1,-1) -- (1+2,-1);
	\draw[line width=1pt] (1,-1) -- (1,-1+4);

	\draw[line width=1pt] (2,2) -- (2,2+1);
	\draw[line width=1pt] (2,2) -- (2+1,2);
	\draw[line width=1pt] (2,2) -- (2-4,2-4);

	\draw[red,line width=1pt] (-1,-1) -- (-1,-1+1);
	\draw[red,line width=1pt] (-1,-1) -- (-1+2,-1);

	\draw[red,line width=1pt] (1,2) -- (1-2,2-2);
	\draw[red,line width=1pt] (1,2) -- (1+1,2);

	\draw[red,line width=1pt] (2,0) -- (2-1,0-1);
	\draw[red,line width=1pt] (2,0) -- (2,0+2);

\end{tikzpicture}
&\vspace{.1in}
&\vspace{.1in}
&\vspace{.1in}
\\
\hline
\end{tabular}
\caption{Decompositions $\widetilde{\mathcal{S}}(v)$
obtained from decompositions $\mathcal{S}(v)$ of Table~\ref{triplesoflatticesinapartment}.}
\label{triplesoflatticesinapartmentandcorrespondingredspiders}
\end{table}

\begin{table}[hbtp]
\centering
\begin{tabular}{|>{\centering\arraybackslash}m{2.4cm}>{\centering\arraybackslash}m{3.4cm}|>{\centering\arraybackslash}m{2.4cm}>{\centering\arraybackslash}m{3.4cm}|}
\hline
\vspace{.1in}
\begin{tikzpicture}[scale=0.4]

	\fill (0,0) circle (5pt);
	\fill (2,1) circle (5pt);
	\fill (1,2) circle (5pt);
	\fill[red] (1,1) circle (5pt);

	\draw[line width=1pt] (0,0) -- (0+3,0);
	\draw[line width=1pt] (0,0) -- (0,0+3);
	\draw[line width=1pt] (0,0) -- (0-1,0-1);

	\draw[line width=1pt] (1,2) -- (1+2,2);
	\draw[line width=1pt] (1,2) -- (1,2+1);
	\draw[line width=1pt] (1,2) -- (1-2,2-2);

	\draw[line width=1pt] (2,1) -- (2+1,1);
	\draw[line width=1pt] (2,1) -- (2,1+2);
	\draw[line width=1pt] (2,1) -- (2-2,1-2);

	\draw[red,line width=1pt] (1,1) -- (1+1,1);
	\draw[red,line width=1pt] (1,1) -- (1,1+1);
	\draw[red,line width=1pt] (1,1) -- (1-1,1-1);

\end{tikzpicture}
&\vspace{.1in}
\begin{tikzpicture}[scale=0.4]

	\fill (0,0) circle (5pt);
	\fill (3,2) circle (5pt);
	\fill (1,3) circle (5pt);
	\fill[red] (1,1) circle (5pt);
	\fill[red] (2,2) circle (5pt);
	\fill[red] (1,2) circle (5pt);

	\draw[line width=1pt] (0,0) -- (0-1,0-1);
	\draw[line width=1pt] (0,0) -- (0+4,0);
	\draw[line width=1pt] (0,0) -- (0,0+4);

	\draw[line width=1pt] (3,2) -- (3-3,2-3);
	\draw[line width=1pt] (3,2) -- (3+1,2);
	\draw[line width=1pt] (3,2) -- (3,2+2);

	\draw[line width=1pt] (1,3) -- (1-2,3-2);
	\draw[line width=1pt] (1,3) -- (1+3,3);
	\draw[line width=1pt] (1,3) -- (1,3+1);

	\draw[red,line width=1pt] (2,2) -- (2-2,2-2);
	\draw[red,line width=1pt] (2,2) -- (2+1,2);
	\draw[red,line width=1pt] (2,2) -- (2,2+2);

	\draw[red,line width=1pt] (1,1) -- (1+3,1);
	\draw[red,line width=1pt] (1,1) -- (1,1+2);

	\draw[red,line width=1pt] (1,2) -- (1-2,2-2);
	\draw[red,line width=1pt] (1,2) -- (1+1,2);

\end{tikzpicture}
&\vspace{.1in}
\begin{tikzpicture}[scale=0.4]

	\fill (0,0) circle (5pt);
	\fill (1,0) circle (5pt);
	\fill (2,1) circle (5pt);
	\fill[red] (1,1) circle (5pt);

	\draw[line width=1pt] (2,1) -- (2-2,1-2);
	\draw[line width=1pt] (2,1) -- (2+1,1);
	\draw[line width=1pt] (2,1) -- (2,1+1);

	\draw[line width=1pt] (1,0) -- (1,0+2);

	\draw[line width=1pt] (0,0) -- (0,0+2);
	\draw[line width=1pt] (0,0) -- (0+3,0);
	\draw[line width=1pt] (0,0) -- (0-1,0-1);

	\draw[red,line width=1pt] (1,1) -- (1+1,1);
	\draw[red,line width=1pt] (1,1) -- (1-1,1-1);

\end{tikzpicture}
&\vspace{.1in}
\begin{tikzpicture}[scale=0.4]

	\fill (0,0) circle (5pt);
	\fill (2,0) circle (5pt);
	\fill (3,1) circle (5pt);
	\fill[red] (1,1) circle (5pt);

	\draw[line width=1pt] (0,0) -- (0+4,0);
	\draw[line width=1pt] (0,0) -- (0,0+2);
	\draw[line width=1pt] (0,0) -- (0-1,0-1);

	\draw[line width=1pt] (3,1) -- (3-2,1-2);
	\draw[line width=1pt] (3,1) -- (3,1+1);
	\draw[line width=1pt] (3,1) -- (3+1,1);

	\draw[red,line width=1pt] (1,1) -- (1-1,1-1);
	\draw[red,line width=1pt] (1,1) -- (1+2,1);
	\draw[red,line width=1pt] (1,1) -- (1,1+1);

	\draw[line width=1pt] (2,0) -- (2,0+2);

\end{tikzpicture}
\\
\hline
\begin{tikzpicture}[scale=0.4]

	\fill[red] (0,0) circle (5pt);
	\fill[red] (1,2) circle (5pt);
	\fill (0,1) circle (5pt);
	\fill (1,0) circle (5pt);
	\fill (3,2) circle (5pt);

	\draw[line width=1pt] (3,2) -- (3-3,2-3);
	\draw[line width=1pt] (3,2) -- (3+1,2);
	\draw[line width=1pt] (3,2) -- (3,2+1);

	\draw[line width=1pt] (1,0) -- (1-1,0-1);
	\draw[line width=1pt] (1,0) -- (1,0+3);
	\draw[line width=1pt] (1,0) -- (1+3,0);

	\draw[line width=1pt] (0,1) -- (0-1,1-1);
	\draw[line width=1pt] (0,1) -- (0,1+2);
	\draw[line width=1pt] (0,1) -- (0+4,1);

	\draw[red,line width=1pt] (1,2) -- (1-1,2-1);
	\draw[red,line width=1pt] (1,2) -- (1+2,2);

	\draw[red,line width=1pt] (0,0) -- (0-1,0-1);
	\draw[red,line width=1pt] (0,0) -- (0+1,0);
	\draw[red,line width=1pt] (0,0) -- (0,0+1);

\end{tikzpicture}
&\vspace{.1in}
\begin{tikzpicture}[scale=0.4]

	\fill[red] (0,0) circle (5pt);
	\fill[red] (1,2) circle (5pt);
	\fill (0,1) circle (5pt);
	\fill (2,0) circle (5pt);
	\fill (4,2) circle (5pt);

	\draw[line width=1pt] (4,2) -- (4-3,2-3);
	\draw[line width=1pt] (4,2) -- (4+1,2);
	\draw[line width=1pt] (4,2) -- (4,2+1);

	\draw[line width=1pt] (2,0) -- (2-1,0-1);
	\draw[line width=1pt] (2,0) -- (2,0+3);
	\draw[line width=1pt] (2,0) -- (2+3,0);

	\draw[line width=1pt] (0,1) -- (0-1,1-1);
	\draw[line width=1pt] (0,1) -- (0,1+2);
	\draw[line width=1pt] (0,1) -- (0+5,1);

	\draw[red,line width=1pt] (1,2) -- (1-1,2-1);
	\draw[red,line width=1pt] (1,2) -- (1+3,2);
	\draw[red,line width=1pt] (1,2) -- (1,2+1);

	\draw[red,line width=1pt] (0,0) -- (0-1,0-1);
	\draw[red,line width=1pt] (0,0) -- (0+2,0);
	\draw[red,line width=1pt] (0,0) -- (0,0+1);

\end{tikzpicture}
&\vspace{.1in}
\begin{tikzpicture}[scale=0.4]

	\fill (0,0) circle (5pt);
	\fill[red] (1,1) circle (5pt);
	\fill[red] (2,2) circle (5pt);
	\fill (2,1) circle (5pt);
	\fill (3,2) circle (5pt);

	\draw[line width=1pt] (0,0) -- (0-1,0-1);
	\draw[line width=1pt] (0,0) -- (0+4,0);
	\draw[line width=1pt] (0,0) -- (0,0+3);

	\draw[line width=1pt] (3,2) -- (3-3,2-3);
	\draw[line width=1pt] (3,2) -- (3+1,2);
	\draw[line width=1pt] (3,2) -- (3,2+1);

	\draw[line width=1pt] (2,1) -- (2+2,1);
	\draw[line width=1pt] (2,1) -- (2,1+2);

	\draw[red,line width=1pt] (1,1) -- (1+1,1);
	\draw[red,line width=1pt] (1,1) -- (1,1+2);

	\draw[red,line width=1pt] (2,2) -- (2-2,2-2);
	\draw[red,line width=1pt] (2,2) -- (2+1,2);

\end{tikzpicture}
&\vspace{.1in}
\begin{tikzpicture}[scale=0.4]

	\fill (0,0) circle (5pt);
	\fill[red] (1,1) circle (5pt);
	\fill[red] (2,2) circle (5pt);
	\fill (3,1) circle (5pt);
	\fill (4,2) circle (5pt);

	\draw[line width=1pt] (0,0) -- (0-1,0-1);
	\draw[line width=1pt] (0,0) -- (0+4,0);
	\draw[line width=1pt] (0,0) -- (0,0+3);

	\draw[line width=1pt] (4,2) -- (4-3,2-3);
	\draw[line width=1pt] (4,2) -- (4+1,2);
	\draw[line width=1pt] (4,2) -- (4,2+1);

	\draw[line width=1pt] (3,1) -- (3+2,1);
	\draw[line width=1pt] (3,1) -- (3,1+2);

	\draw[red,line width=1pt] (1,1) -- (1+2,1);
	\draw[red,line width=1pt] (1,1) -- (1,1+2);

	\draw[red,line width=1pt] (2,2) -- (2-2,2-2);
	\draw[red,line width=1pt] (2,2) -- (2+2,2);
	\draw[red,line width=1pt] (2,2) -- (2,2+1);

\end{tikzpicture}
\\
\hline
\begin{tikzpicture}[scale=0.4]

	\fill[red] (0,0) circle (5pt);
	\fill (0,1) circle (5pt);
	\fill[red] (2,1) circle (5pt);
	\fill[red] (2,3) circle (5pt);
	\fill (1,0) circle (5pt);
	\fill (2,4) circle (5pt);

	\draw[line width=1pt] (0,1) -- (0-1,1-1);
	\draw[line width=1pt] (0,1) -- (0+3,1);
	\draw[line width=1pt] (0,1) -- (0,1+4);

	\draw[line width=1pt] (1,0) -- (1-1,0-1);
	\draw[line width=1pt] (1,0) -- (1+2,0);
	\draw[line width=1pt] (1,0) -- (1,0+5);

	\draw[line width=1pt] (2,4) -- (2-3,4-3);
	\draw[line width=1pt] (2,4) -- (2+1,4);
	\draw[line width=1pt] (2,4) -- (2,4+1);

	\draw[red,line width=1pt] (0,0) -- (0,0+1);
	\draw[red,line width=1pt] (0,0) -- (0+1,0);
	\draw[red,line width=1pt] (0,0) -- (0-1,0-1);

	\draw[red,line width=1pt] (2,1) -- (2,1+3);
	\draw[red,line width=1pt] (2,1) -- (2-1,1-1);

	\draw[red,line width=1pt] (2,3) -- (2+1,3);
	\draw[red,line width=1pt] (2,3) -- (2-2,3-2);

\end{tikzpicture}
&\vspace{.1in}
\begin{tikzpicture}[scale=0.4]

	\fill[red] (0,-1) circle (5pt);
	\fill (0,1) circle (5pt);
	\fill[red] (2,0) circle (5pt);
	\fill[red] (2,3) circle (5pt);
	\fill (1,-1) circle (5pt);
	\fill (2,4) circle (5pt);

	\draw[line width=1pt] (0,1) -- (0-1,1-1);
	\draw[line width=1pt] (0,1) -- (0+3,1);
	\draw[line width=1pt] (0,1) -- (0,1+4);

	\draw[line width=1pt] (1,-1) -- (1-1,-1-1);
	\draw[line width=1pt] (1,-1) -- (1+2,-1);
	\draw[line width=1pt] (1,-1) -- (1,-1+6);

	\draw[line width=1pt] (2,4) -- (2-3,4-3);
	\draw[line width=1pt] (2,4) -- (2+1,4);
	\draw[line width=1pt] (2,4) -- (2,4+1);

	\draw[red,line width=1pt] (0,-1) -- (0,-1+2);
	\draw[red,line width=1pt] (0,-1) -- (0+1,-1);
	\draw[red,line width=1pt] (0,-1) -- (0-1,-1-1);

	\draw[red,line width=1pt] (2,0) -- (2,0+4);
	\draw[red,line width=1pt] (2,0) -- (2+1,0);
	\draw[red,line width=1pt] (2,0) -- (2-1,0-1);

	\draw[red,line width=1pt] (2,3) -- (2+1,3);
	\draw[red,line width=1pt] (2,3) -- (2-2,3-2);

\end{tikzpicture}
&\vspace{.1in}
\begin{tikzpicture}[scale=0.4]

	\fill (0,0) circle (5pt);
	\fill (1,2) circle (5pt);
	\fill (2,4) circle (5pt);
	\fill[red] (1,1) circle (5pt);
	\fill[red] (2,2) circle (5pt);
	\fill[red] (2,3) circle (5pt);

	\draw[line width=1pt] (0,0) -- (0-1,0-1);
	\draw[line width=1pt] (0,0) -- (0+3,0);
	\draw[line width=1pt] (0,0) -- (0,0+5);

	\draw[line width=1pt] (1,2) -- (1-2,2-2);
	\draw[line width=1pt] (1,2) -- (1+2,2);
	\draw[line width=1pt] (1,2) -- (1,2+3);

	\draw[line width=1pt] (2,4) -- (2-3,4-3);
	\draw[line width=1pt] (2,4) -- (2+1,4);
	\draw[line width=1pt] (2,4) -- (2,4+1);

	\draw[red,line width=1pt] (1,1) -- (1+2,1);
	\draw[red,line width=1pt] (1,1) -- (1,1+1);

	\draw[red,line width=1pt] (2,2) -- (2-2,2-2);
	\draw[red,line width=1pt] (2,2) -- (2,2+2);

	\draw[red,line width=1pt] (2,3) -- (2-1,3-1);
	\draw[red,line width=1pt] (2,3) -- (2+1,3);

\end{tikzpicture}
&\vspace{.1in}
\begin{tikzpicture}[scale=0.4]

	\fill (0,0) circle (5pt);
	\fill (1,2) circle (5pt);
	\fill (3,6) circle (5pt);
	\fill[red] (1,1) circle (5pt);
	\fill[red] (3,3) circle (5pt);
	\fill[red] (3,4) circle (5pt);

	\draw[line width=1pt] (0,0) -- (0-1,0-1);
	\draw[line width=1pt] (0,0) -- (0+4,0);
	\draw[line width=1pt] (0,0) -- (0,0+7);

	\draw[line width=1pt] (1,2) -- (1-2,2-2);
	\draw[line width=1pt] (1,2) -- (1+3,2);
	\draw[line width=1pt] (1,2) -- (1,2+5);

	\draw[line width=1pt] (3,6) -- (3-4,6-4);
	\draw[line width=1pt] (3,6) -- (3+1,6);
	\draw[line width=1pt] (3,6) -- (3,6+1);

	\draw[red,line width=1pt] (1,1) -- (1+3,1);
	\draw[red,line width=1pt] (1,1) -- (1,1+1);

	\draw[red,line width=1pt] (3,3) -- (3-3,3-3);
	\draw[red,line width=1pt] (3,3) -- (3,3+3);
	\draw[red,line width=1pt] (3,3) -- (3+1,3);

	\draw[red,line width=1pt] (3,4) -- (3-2,4-2);
	\draw[red,line width=1pt] (3,4) -- (3+1,4);

\end{tikzpicture}
\\
\hline
\multicolumn{4}{|c|}{
\begin{tikzpicture}[scale=0.4]

	\fill[red] (0,0) circle (5pt);
	\fill (0,1) circle (5pt);
	\fill[red] (1,2) circle (5pt);
	\fill (2,2) circle (5pt);
	\fill[red] (2,1) circle (5pt);
	\fill (1,0) circle (5pt);

	\draw[red,line width=1pt] (0,0) -- (0,0+1);
	\draw[red,line width=1pt] (0,0) -- (0+1,0);

	\draw[line width=1pt] (0,1) -- (0-1,1-1);
	\draw[line width=1pt] (0,1) -- (0+3,1);
	\draw[line width=1pt] (0,1) -- (0,1+2);

	\draw[line width=1pt] (1,0) -- (1-1,0-1);
	\draw[line width=1pt] (1,0) -- (1+2,0);
	\draw[line width=1pt] (1,0) -- (1,0+3);

	\draw[line width=1pt] (2,2) -- (2,2+1);
	\draw[line width=1pt] (2,2) -- (2+1,2);
	\draw[line width=1pt] (2,2) -- (2-3,2-3);

	\draw[red,line width=1pt] (1,2) -- (1-1,2-1);
	\draw[red,line width=1pt] (1,2) -- (1+1,2);

	\draw[red,line width=1pt] (2,1) -- (2-1,1-1);
	\draw[red,line width=1pt] (2,1) -- (2,1+1);

\end{tikzpicture}
\hspace{0.5in}
\begin{tikzpicture}[scale=0.4]

	\fill[red] (0,0) circle (5pt);
	\fill (0,1) circle (5pt);
	\fill[red] (1,2) circle (5pt);
	\fill (3,2) circle (5pt);
	\fill[red] (3,1) circle (5pt);
	\fill (2,0) circle (5pt);

	\draw[red,line width=1pt] (0,0) -- (0,0+1);
	\draw[red,line width=1pt] (0,0) -- (0+2,0);
	\draw[red,line width=1pt] (0,0) -- (0-1,0-1);

	\draw[line width=1pt] (0,1) -- (0-1,1-1);
	\draw[line width=1pt] (0,1) -- (0+3,1);
	\draw[line width=1pt] (0,1) -- (0,1+2);

	\draw[line width=1pt] (2,0) -- (2-1,0-1);
	\draw[line width=1pt] (2,0) -- (2+2,0);
	\draw[line width=1pt] (2,0) -- (2,0+3);

	\draw[line width=1pt] (3,2) -- (3,2+1);
	\draw[line width=1pt] (3,2) -- (3+1,2);
	\draw[line width=1pt] (3,2) -- (3-3,2-3);

	\draw[red,line width=1pt] (1,2) -- (1-1,2-1);
	\draw[red,line width=1pt] (1,2) -- (1+2,2);
	\draw[red,line width=1pt] (1,2) -- (1,2+1);

	\draw[red,line width=1pt] (3,1) -- (3-1,1-1);
	\draw[red,line width=1pt] (3,1) -- (3,1+1);

\end{tikzpicture}
\hspace{0.5in}
\begin{tikzpicture}[scale=0.4]

	\fill[red] (0,0) circle (5pt);
	\fill (0,1) circle (5pt);
	\fill[red] (3,4) circle (5pt);
	\fill (5,4) circle (5pt);
	\fill[red] (5,3) circle (5pt);
	\fill (2,0) circle (5pt);

	\draw[red,line width=1pt] (0,0) -- (0,0+1);
	\draw[red,line width=1pt] (0,0) -- (0+2,0);
	\draw[red,line width=1pt] (0,0) -- (0-1,0-1);

	\draw[line width=1pt] (0,1) -- (0-1,1-1);
	\draw[line width=1pt] (0,1) -- (0+6,1);
	\draw[line width=1pt] (0,1) -- (0,1+4);

	\draw[line width=1pt] (2,0) -- (2-1,0-1);
	\draw[line width=1pt] (2,0) -- (2+4,0);
	\draw[line width=1pt] (2,0) -- (2,0+5);

	\draw[line width=1pt] (5,4) -- (5,4+1);
	\draw[line width=1pt] (5,4) -- (5+1,4);
	\draw[line width=1pt] (5,4) -- (5-5,4-5);

	\draw[red,line width=1pt] (3,4) -- (3-3,4-3);
	\draw[red,line width=1pt] (3,4) -- (3+2,4);
	\draw[red,line width=1pt] (3,4) -- (3,4+1);

	\draw[red,line width=1pt] (5,3) -- (5-3,3-3);
	\draw[red,line width=1pt] (5,3) -- (5,3+1);
	\draw[red,line width=1pt] (5,3) -- (5+1,3);

\end{tikzpicture}
}
\\
\hline
\multicolumn{4}{|c|}{
\begin{tikzpicture}[scale=0.4]

	\fill[red] (-1,-1) circle (5pt);
	\fill (-1,0) circle (5pt);
	\fill[red] (1,2) circle (5pt);
	\fill (2,2) circle (5pt);
	\fill[red] (2,0) circle (5pt);
	\fill (1,-1) circle (5pt);

	\draw[line width=1pt] (-1,0) -- (-1-1,0-1);
	\draw[line width=1pt] (-1,0) -- (-1+4,0);
	\draw[line width=1pt] (-1,0) -- (-1,0+3);

	\draw[line width=1pt] (1,-1) -- (1-1,-1-1);
	\draw[line width=1pt] (1,-1) -- (1+2,-1);
	\draw[line width=1pt] (1,-1) -- (1,-1+4);

	\draw[line width=1pt] (2,2) -- (2,2+1);
	\draw[line width=1pt] (2,2) -- (2+1,2);
	\draw[line width=1pt] (2,2) -- (2-4,2-4);

	\draw[red,line width=1pt] (-1,-1) -- (-1,-1+1);
	\draw[red,line width=1pt] (-1,-1) -- (-1+2,-1);

	\draw[red,line width=1pt] (1,2) -- (1-2,2-2);
	\draw[red,line width=1pt] (1,2) -- (1+1,2);

	\draw[red,line width=1pt] (2,0) -- (2-1,0-1);
	\draw[red,line width=1pt] (2,0) -- (2,0+2);

\end{tikzpicture}
\hspace{0.5in}
\begin{tikzpicture}[scale=0.4]

	\fill[red] (-1,-1) circle (5pt);
	\fill (-1,0) circle (5pt);
	\fill[red] (1,2) circle (5pt);
	\fill (3,2) circle (5pt);
	\fill[red] (3,0) circle (5pt);
	\fill (2,-1) circle (5pt);

	\draw[line width=1pt] (-1,0) -- (-1-1,0-1);
	\draw[line width=1pt] (-1,0) -- (-1+4,0);
	\draw[line width=1pt] (-1,0) -- (-1,0+3);

	\draw[line width=1pt] (2,-1) -- (2-1,-1-1);
	\draw[line width=1pt] (2,-1) -- (2+2,-1);
	\draw[line width=1pt] (2,-1) -- (2,-1+4);

	\draw[line width=1pt] (3,2) -- (3,2+1);
	\draw[line width=1pt] (3,2) -- (3+1,2);
	\draw[line width=1pt] (3,2) -- (3-4,2-4);

	\draw[red,line width=1pt] (-1,-1) -- (-1,-1+1);
	\draw[red,line width=1pt] (-1,-1) -- (-1+3,-1);
	\draw[red,line width=1pt] (-1,-1) -- (-1-1,-1-1);

	\draw[red,line width=1pt] (1,2) -- (1-2,2-2);
	\draw[red,line width=1pt] (1,2) -- (1+2,2);
	\draw[red,line width=1pt] (1,2) -- (1,2+1);

	\draw[red,line width=1pt] (3,0) -- (3-1,0-1);
	\draw[red,line width=1pt] (3,0) -- (3,0+2);
	\draw[red,line width=1pt] (3,0) -- (3+1,0);

\end{tikzpicture}
\hspace{0.5in}
\begin{tikzpicture}[scale=0.4]

	\fill[red] (-1,-1) circle (5pt);
	\fill (-1,0) circle (5pt);
	\fill[red] (3,4) circle (5pt);
	\fill (5,4) circle (5pt);
	\fill[red] (5,2) circle (5pt);
	\fill (2,-1) circle (5pt);

	\draw[line width=1pt] (-1,0) -- (-1-1,0-1);
	\draw[line width=1pt] (-1,0) -- (-1+7,0);
	\draw[line width=1pt] (-1,0) -- (-1,0+5);

	\draw[line width=1pt] (2,-1) -- (2-1,-1-1);
	\draw[line width=1pt] (2,-1) -- (2+4,-1);
	\draw[line width=1pt] (2,-1) -- (2,-1+6);

	\draw[line width=1pt] (5,4) -- (5,4+1);
	\draw[line width=1pt] (5,4) -- (5+1,4);
	\draw[line width=1pt] (5,4) -- (5-6,4-6);

	\draw[red,line width=1pt] (-1,-1) -- (-1,-1+1);
	\draw[red,line width=1pt] (-1,-1) -- (-1+3,-1);
	\draw[red,line width=1pt] (-1,-1) -- (-1-1,-1-1);

	\draw[red,line width=1pt] (3,4) -- (3-4,4-4);
	\draw[red,line width=1pt] (3,4) -- (3+2,4);
	\draw[red,line width=1pt] (3,4) -- (3,4+1);

	\draw[red,line width=1pt] (5,2) -- (5-3,2-3);
	\draw[red,line width=1pt] (5,2) -- (5,2+2);
	\draw[red,line width=1pt] (5,2) -- (5+1,2);

\end{tikzpicture}
}
\\
\hline
\end{tabular}
\caption{
``Stretching'' the decompositions $\mathcal{S}(v)$ 
in Table~\ref{triplesoflatticesinapartmentandcorrespondingredspiders}
creates further combinatorial types of $\widetilde{\mathcal{S}}(v)$
producing further cones in $\widetilde{\mathcal{Q}}_3$.}
\label{stretchesofspiders}
\end{table}

The next step is to define a refinement $\widetilde{\mathcal{R}}_m$ of the fan $\mathcal{R}_m$ of Definition~\ref{fanforfamilyoverchowquotient}.

\begin{lemdef}
Let $\widetilde{\mathcal{R}}_m$ be the collection of cones in $\mathbb{R}^{2m}$ in the form
\[
\pi(Q^{-1}(\sigma)\cap\tau)\quad\hbox{\rm for}\quad
\sigma\in\mathcal{Q}_{m+N}',
\   \tau\in\mathcal{P}^{m+N},
\]
where
$Q\colon\mathbb{R}^{2m+2N}\rightarrow\mathbb{R}^{2m+2N}/\mathbb{R}^2$
is the quotient map.
This collection is a fan and 
the quotient map $q\colon\mathbb{R}^{2m}\rightarrow\mathbb{R}^{2m}/\mathbb{R}^2$ induces a map of fans between $\widetilde{\mathcal{R}}_m$ and $\widetilde{\mathcal{Q}}_m$. 
\end{lemdef}

\begin{proof}
Take a cone $\eta=\pi(Q^{-1}(\sigma)\cap\tau)\in\widetilde{\mathcal{R}}_m$. Then $\pi|_{\eta}$ is bijective with the inverse $\iota|_{\pi(\eta)}$.
Since cones of the form $Q^{-1}(\sigma)\cap\tau$ form a fan, it follows that $\widetilde{\mathcal{R}}_m$ is a fan as well. Next,
\[
q(\eta)=q(\pi(Q^{-1}(\sigma)\cap\tau))=\overline{\pi}(Q(Q^{-1}(\sigma)\cap\tau))=\overline{\pi}(\sigma\cap Q(\tau))\subseteq\overline{\pi}(\sigma),
\]where $\overline{\pi}(\sigma)\in\widetilde{\mathcal{Q}}_m$ by definition. Hence $q$ induces a map of fans.
\end{proof}

\begin{lemma}\label{wrgawrgarsg}
Let $N=\binom{3}{2}\binom{m}{2}$. We have a commutative diagram of toric morphisms
\begin{center}
\begin{tikzpicture}[>=angle 90]
\matrix(a)[matrix of math nodes,
row sep=2em, column sep=2em,
text height=1.5ex, text depth=0.25ex]
{
Y_{\mathcal{R'}_{m+N}}  &  Y_{\widetilde{\mathcal{R}}_m}\\
Y_{\mathcal{Q'}_{m+N}}  &  Y_{\widetilde{\mathcal{Q}}_m},\\
};
\path[->] (a-1-1) edge node[above]{$y_\pi$}(a-1-2);
\path[->] (a-1-1) edge node[left]{$y_Q$}(a-2-1);
\path[->] (a-2-1) edge node[below]{$y_{\bar\pi}$}(a-2-2);
\path[->] (a-1-2) edge node[right]{$y_q$}(a-2-2);
\end{tikzpicture}
\end{center}
where $y_\pi$ and $y_{\bar\pi}$ are quotients by
free ${\mathbb G}_m^{2N}$-action.
The morphism $y_q$ is flat with reduced~fibers.
\end{lemma}

\begin{proof}
The morphisms $y_\pi$ and $y_{\bar\pi}$ are induced by linear maps $\pi$ and $\bar\pi$
of fans which have piece-wise linear inverses $\iota$ and $\bar\iota$.
This proves the first part because every affine toric chart of the variety on the left
is isomorphic (not canonically) to the product of the corresponding affine toric chart of the variety on the right with ${\mathbb G}_m^{2N}$.
Since $y_Q$ is flat with reduced fibers, the same is true for $y_q$ by smooth descent.
\end{proof}


\begin{proposition}\label{,amNEBFjahevfaeV}
The family $Y_{\widetilde{\mathcal{R}}_m}\rightarrow Y_{\widetilde{\mathcal{Q}}_m}$ over the toric Gerrizen--Piwek space
admits ``light'' sections
$\ell_1,\ldots,\ell_m$ induced by linear maps $\lambda_1,\ldots,\lambda_m$ of Definition~\ref{wrgawrgarhgar}
and  ``heavy'' constant sections $t_1,t_2,t_3$ defined as in Proposition~\ref{awrgasrgsrh}.
\end{proposition}

\begin{proof}
The proof is the same as for Proposition~\ref{qwrgawrhwarh}.
\end{proof}


\section{Projective duality and the quotient fans \texorpdfstring{$\mathcal{Q}_m$}{Lg} and \texorpdfstring{$\mathcal{Q}_m^\vee$}{Lg}}
\label{projdualityandquotfan}

In the construction of the quotient fan $\mathcal{Q}_m$ and its interpretation in terms of subdivisions of $\mathbb{R}^2$, we could have used 
the fan $\mathcal{P}^\vee=-\mathcal{P}$
of the dual projective plane $(\mathbb{P}^2)^\vee$ 
in place of $\mathcal{P}$.
As in \S\ref{sGsgsGsehsRH},
we define the \emph{dual quotient fan} $\mathcal{Q}_m^\vee$. So we have the isomorphism
\[
((\mathbb{P}^2)^\vee)^m/\!/\;H\cong Y_{\mathcal{Q}_m^\vee},
\]
where $H\subseteq\mathbb{P}^2$ is the dense torus. The cones in $\mathcal{Q}_m^\vee$ can be interpreted combinatorially 
as in Lemma~\ref{preliminaryfibersfamilyoverchowquotient}
in terms of polyhedral subdivisions $\mathcal{S}(v)^\vee$ of $\mathbb{R}^2$ induced by dual spiders
 $S(v_i,v_i')^\vee$ for $i=1,\ldots,m$ 
 (see Definition~\ref{dualspider}).
 Notice that the homomorphism $\mathbb{R}^{2m}/\mathbb{R}^2$ sending $[v]$ to $-[v]$ maps the cones in $\mathcal{Q}_m$ to cones in $\mathcal{Q}_m^\vee$ (for an explicit example, see Figure~\ref{projectivedualityconesquotientfan}). This induces a natural toric isomorphism between the Chow quotients $((\mathbb{P}^2)^\vee)^m/\!/\;H$ and $(\mathbb{P}^2)^m/\!/\;H$.

\begin{figure}[hbtp]
\begin{tikzpicture}[scale=0.5]

	\draw[line width=1pt] (0,0) -- (0-5,0);
	\draw[line width=1pt] (0,0) -- (0,0-5);
	\draw[line width=1pt] (0,0) -- (0+1,0+1);

	\draw[line width=1pt] (-4,-2) -- (-4-1,-2);
	\draw[line width=1pt] (-4,-2) -- (-4,-2-3);
	\draw[line width=1pt] (-4,-2) -- (-4+3,-2+3);

	\draw[line width=1pt] (-2,-4) -- (-2-3,-4);
	\draw[line width=1pt] (-2,-4) -- (-2,-4-1);
	\draw[line width=1pt] (-2,-4) -- (-2+3,-4+3);

	\fill (0,0) circle (5pt);
	\fill (-4,-2) circle (5pt);
	\fill (-2,-4) circle (5pt);

\end{tikzpicture}
\hspace{1in}
\begin{tikzpicture}[scale=0.5]

	\draw[line width=1pt] (0,0) -- (0+5,0);
	\draw[line width=1pt] (0,0) -- (0,0+5);
	\draw[line width=1pt] (0,0) -- (0-1,0-1);

	\draw[line width=1pt] (4,2) -- (4+1,2);
	\draw[line width=1pt] (4,2) -- (4,2+3);
	\draw[line width=1pt] (4,2) -- (4-3,2-3);

	\draw[line width=1pt] (2,4) -- (2+3,4);
	\draw[line width=1pt] (2,4) -- (2,4+1);
	\draw[line width=1pt] (2,4) -- (2-3,4-3);

	\fill (0,0) circle (5pt);
	\fill (4,2) circle (5pt);
	\fill (2,4) circle (5pt);

\end{tikzpicture}
\caption{On the left, subdivision of $\mathbb{R}^2$ given by a cone in $\mathcal{Q}_3^\vee$, and on the right its image under the homomorphism $[v]\mapsto-[v]$ given by a cone in~$\mathcal{Q}_3$.}
\label{projectivedualityconesquotientfan}
\end{figure}
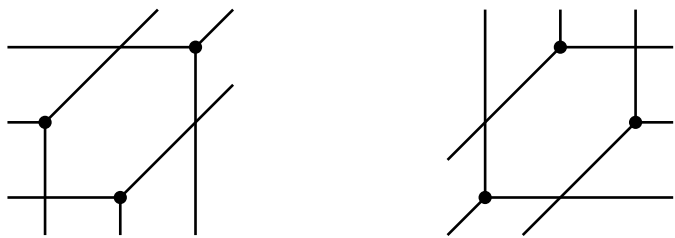

A  surprising fact is that 
the fan $\widetilde{\mathcal{Q}}_m$ refines 
not only $\mathcal{Q}_m$ (Proposition~\ref{combinatorialequivalencetildesubdivisionsanddroppingspiders}) but also 
$\mathcal{Q}_m^\vee$.
Before we can prove this fact in Theorem~\ref{tildefanrefinesdualfan}, we need a little preparation.

\begin{definition}
Given $[v]\in\mathbb{R}^{2m}/\mathbb{R}^2$, for all distinct $i,j\in[m]$, let $\delta_{ij}\in\Sigma(v_i,v_i')$ be the minimal cone that contains $(-v_j,-v_j')$. We call the collection of cones $\{\delta_{ij}\}_{1\leq i,j\leq m,i\neq j}$ the \emph{cone data of $v$}. If $\{\delta_{ij}\}$ and $\{\eta_{ij}\}$ are the cone data attached to two subdivisions $\mathcal{S}(v),\mathcal{S}(w)$, then we say that $\{\delta_{ij}\}$ and $\{\eta_{ij}\}$ are \emph{equivalent} provided for all distinct $i,j\in[m]$,
\[
\delta_{ij}+(v_i,v_i')=\eta_{ij}+(w_i,w_i').
\]
\end{definition}

\begin{lemma}
\label{combinatoriallyequivalentimpliesequivalentconedata}
Let $[v],[w]\in\mathbb{R}^{2m}/\mathbb{R}^2$ and consider the respective cone data $\{\delta_{ij}\}$ and $\{\eta_{ij}\}$. If $\mathcal{S}(v)$ and $\mathcal{S}(w)$ are combinatorially equivalent, then $\{\delta_{ij}\}$ and $\{\eta_{ij}\}$ are equivalent.
\end{lemma}

\begin{proof}
Assume that $\mathcal{S}(v)$ and $\mathcal{S}(w)$ are combinatorially equivalent. By Lemma~\ref{droppingspiderspreservesequivalence}, the subdivisions obtained by dropping all the spiders other than the $i$-th and the $j$-th, are also equivalent. This means precisely that the cone data are equivalent.
\end{proof}

The next proposition analyzes the converse of the the previous lemma when $m=3$.

\begin{proposition}
\label{combinatoriallyequivalentequivalentequivalentconedatam=3}
Let $[v],[w]\in\mathbb{R}^{6}/\mathbb{R}^2$ and suppose their cone data $\{\delta_{ij}\}$ and $\{\eta_{ij}\}$ are equivalent. 
Then $\mathcal{S}(v)$ is combinatorially equivalent to $\mathcal{S}(w)$ unless they are 
combinatorially equivalent to (q) and (u) in Table~\ref{triplesoflatticesinapartment}, respectively (or vice versa). 
The subdivisions (q) and (u) have equivalent cone data
but are not combinatorially equivalent.
\end{proposition}

\begin{proof}
Going through all the cases in
Table~\ref{triplesoflatticesinapartment},
which we reproduced from \cite[Figure~6]{CHSW11},
shows that the combinatorial equivalence class 
of the plane subdivision
is uniquely determined by the cone data 
except for the subdivisions (q) and (u).
\end{proof}

Using the fans $\Sigma(v_i,v_i')^\vee$ instead of $\Sigma(v_i,v_i')$, we can associate the \emph{dual cone data} $\{\delta_{ij}^\vee\}$. Analogous results for $\mathcal{S}(v)^\vee$ and $\{\delta_{ij}^\vee\}$ hold. We notice the following property.

\begin{lemma}
\label{equivalenttildesubdivisionimpliesequivalentdualconedata}
Let $[v],[w]\in\mathbb{R}^{2m}/\mathbb{R}^2$.
If $\widetilde{\mathcal{S}}(v)$ is combinatorially equivalent to $\widetilde{\mathcal{S}}(w)$, then 
the dual cone data $\{\delta_{ij}^\vee\}$ and $\{\eta_{ij}^\vee\}$ are equivalent.
\end{lemma}

\begin{proof}
Given a pair of distinct indices, we want to show that
\[
\delta_{ij}^\vee+(v_i,v_i')=\eta_{ij}^\vee+(w_i,w_i').
\]
Let $\delta_{ij}\in\Sigma(v_i,v_i')$ be the minimal cone containing $(-v_j,-v_j')$. If $\delta_{ij}$ has dimension at most~$1$, then we can immediately determine $\delta_{ij}^\vee\in\Sigma(v_i,v_i')^\vee$, and in each case it holds that $\delta_{ij}^\vee+(v_i,v_i')=\eta_{ij}^\vee+(w_i,w_i')$.
Therefore, assume that $\delta_{ij}$ is $2$-dimensional, and without loss of generality assume that $\delta_{ij}$ is the cone whose two rays form a right angle (an analogous argument applies to the other $2$-dimensional cones). Then $\delta_{ij}^\vee$ can be  $\alpha$, $\beta$ or $\alpha\cap\beta$ in Figure~\ref{possibleconesdependingonpositionofvjvj'}, 
\begin{figure}[hbtp]
\begin{tikzpicture}[scale=0.5]

	\draw[red,line width=1pt] (0,0) -- (0-2,0);
	\draw[red,line width=1pt] (0,0) -- (0,0-2);
	\draw[red,line width=1pt] (0,0) -- (0+2,0+2);

	\fill (0,0) circle (5pt);

	\node[red] at (0,1.5) {$\alpha$};
	\node[red] at (1.5,0) {$\beta$};

\end{tikzpicture}
\caption{The fan $\Sigma(v_i,v_i')^\vee$ and two of its maximal cones.}
\label{possibleconesdependingonpositionofvjvj'}
\end{figure}
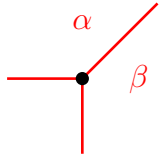
and this is uniquely determined by the position of $(-v_j,-v_j')$ with respect to $(-x_{ij},-x_{ij})$ (see Definition~\ref{refinementusingKapranovspiders}). We have
three possible cases, which are pictured in Figure~\ref{possiblepositionsofvjvj'}:
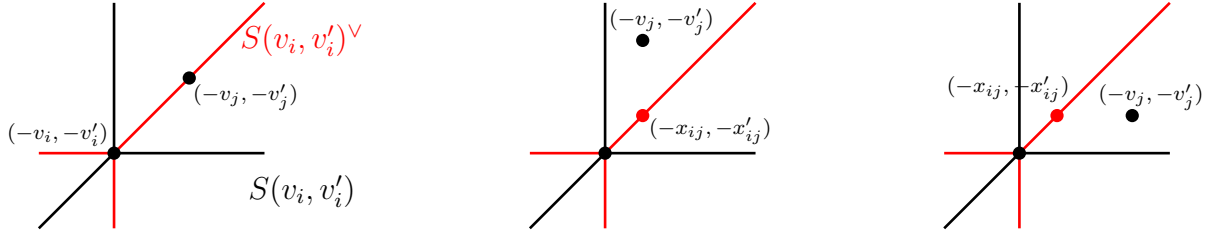
\begin{figure}[hbtp]
\begin{tikzpicture}[scale=0.5]

	\draw[line width=1pt] (0,0) -- (0+4,0);
	\draw[line width=1pt] (0,0) -- (0,0+4);
	\draw[line width=1pt] (0,0) -- (0-2,0-2);
	\draw[red,line width=1pt] (0,0) -- (0-2,0);
	\draw[red,line width=1pt] (0,0) -- (0,0-2);
	\draw[red,line width=1pt] (0,0) -- (0+4,0+4);

	\fill (0,0) circle (5pt);
	\fill (2,2) circle (5pt);

	\node[red] at (5,3) {$S(v_i,v_i')^\vee$};
	\node at (5,-1) {$S(v_i,v_i')$};
	\node at (3.5,1.5) {{\tiny$(-v_j,-v_j')$}};
	\node at (-1.5,0.5) {{\tiny$(-v_i,-v_i')$}};

\end{tikzpicture}
\hspace{0.7in}
\begin{tikzpicture}[scale=0.5]

	\draw[line width=1pt] (0,0) -- (0+4,0);
	\draw[line width=1pt] (0,0) -- (0,0+4);
	\draw[line width=1pt] (0,0) -- (0-2,0-2);
	\draw[red,line width=1pt] (0,0) -- (0-2,0);
	\draw[red,line width=1pt] (0,0) -- (0,0-2);
	\draw[red,line width=1pt] (0,0) -- (0+4,0+4);

	\fill (0,0) circle (5pt);
	\fill[red] (1,1) circle (5pt);
	\fill (1,3) circle (5pt);

	\node at (1.5,3.5) {{\tiny$(-v_j,-v_j')$}};
	\node at (2.8,0.6) {{\tiny$(-x_{ij},-x_{ij}')$}};

\end{tikzpicture}
\hspace{0.7in}
\begin{tikzpicture}[scale=0.5]

	\draw[line width=1pt] (0,0) -- (0+4,0);
	\draw[line width=1pt] (0,0) -- (0,0+4);
	\draw[line width=1pt] (0,0) -- (0-2,0-2);
	\draw[red,line width=1pt] (0,0) -- (0-2,0);
	\draw[red,line width=1pt] (0,0) -- (0,0-2);
	\draw[red,line width=1pt] (0,0) -- (0+4,0+4);

	\fill (0,0) circle (5pt);
	\fill[red] (1,1) circle (5pt);
	\fill (3,1) circle (5pt);

	\node at (3.5,1.5) {{\tiny$(-v_j,-v_j')$}};
	\node at (-0.3,1.7) {{\tiny$(-x_{ij},-x_{ij}')$}};

\end{tikzpicture}
\caption{Possible positions of $(-v_j,-v_j')$.}
\label{possiblepositionsofvjvj'}
\end{figure}
\begin{enumerate}

\item $(x_{ij},x_{ij}')=(v_j,v_j')$ if and only if $(-v_j,-v_j')$ is contained in the relative interior of a ray in $\Sigma(v_i,v_i')^\vee$. If this is the case, then we have that $\delta_{ij}=\alpha\cap\beta$.

\item If $-x_{ij}=-v_j$ and $-x_{ij}'<-v_j'$, then $\delta_{ij}=\alpha$.

\item Lastly, if $-x_{ij}<-v_j$ and $-x_{ij}'=-v_j'$, then $\delta_{ij}=\beta$.

\end{enumerate}
In each case, $\delta_{ij}^\vee+(v_i,v_i')=\eta_{ij}^\vee+(w_i,w_i')$.
\end{proof}

\begin{theorem}
\label{tildefanrefinesdualfan}
The fan $\widetilde{\mathcal{Q}}_m$ refines the dual quotient fan $\mathcal{Q}_m^\vee$.
\end{theorem}

\begin{proof}

Let $[v],[w]\in\mathbb{R}^{2m}/\mathbb{R}^2$ be two vectors lying in the relative interior of the same cone of $\widetilde{\mathcal{Q}}_m$, which means that the polyhedral subdivisions $\widetilde{\mathcal{S}}(v)$ and $\widetilde{\mathcal{S}}(w)$ are combinatorially equivalent. We want to show that $\mathcal{S}(v)^\vee$ and $\mathcal{S}(w)^\vee$ are combinatorially equivalent as well.

If $m=1$ then this is  true because $\mathcal{S}(v)^\vee$ and $\mathcal{S}(w)^\vee$ are equal up to a translation in $\mathbb{R}^2$. If $m=2$ then 
the combinatorial equivalence class $[\widetilde{\mathcal{S}}(v)]$ uniquely determines $[\mathcal{S}(v)]$, which uniquely determines $[\mathcal{S}(v)^\vee]$. So if $[\widetilde{\mathcal{S}}(v)]=[\widetilde{\mathcal{S}}(w)]$, then $[\mathcal{S}(v)^\vee]=[\mathcal{S}(w)^\vee]$.

For $m\geq3$ we argue by induction. Let $m=3$. By Lemma~\ref{equivalenttildesubdivisionimpliesequivalentdualconedata}, the dual cone data $\{\delta_{ij}^\vee\},\{\eta_{ij}^\vee\}$ associated to $[v],[w]$ respectively are equivalent.
 Therefore, by Proposition~\ref{combinatoriallyequivalentequivalentequivalentconedatam=3} adapted to the dual case,  $\mathcal{S}(v)^\vee$ is equivalent to $\mathcal{S}(w)^\vee$ provided we are away from the exceptional case in the first row of Table~\ref{exceptionalcaseinprooftildefanrefinesdualfan}. However, this case can not occur: 
 $\widetilde{\mathcal{S}}(v)$ and $\widetilde{\mathcal{S}}(w)$ are not combinatorially equivalent as we can see in the second row of Table~\ref{exceptionalcaseinprooftildefanrefinesdualfan}.

\begin{table}[hbtp]
\centering
\begin{tabular}{|>{\centering\arraybackslash}m{5cm}|>{\centering\arraybackslash}m{5cm}|}
\hline
\vspace{.1in}
\begin{tikzpicture}[scale=0.5]

	\draw[line width=1pt] (0,0) -- (0-1,0);
	\draw[line width=1pt] (0,0) -- (0,0-1);
	\draw[line width=1pt] (0,0) -- (0+4,0+4);

	\draw[line width=1pt] (1,3) -- (1-2,3);
	\draw[line width=1pt] (1,3) -- (1,3-4);
	\draw[line width=1pt] (1,3) -- (1+1,3+1);

	\draw[line width=1pt] (3,2) -- (3-4,2);
	\draw[line width=1pt] (3,2) -- (3,2-3);
	\draw[line width=1pt] (3,2) -- (3+1,2+1);

	\fill (0,0) circle (5pt);
	\fill (1,3) circle (5pt);
	\fill (3,2) circle (5pt);

	\node at (5,1.5) {$\mathcal{S}(v)^\vee$};

\end{tikzpicture}
&\vspace{.1in}
\begin{tikzpicture}[scale=0.5]

	\draw[line width=1pt] (0,0) -- (0-1,0);
	\draw[line width=1pt] (0,0) -- (0,0-1);
	\draw[line width=1pt] (0,0) -- (0+3,0+3);

	\draw[line width=1pt] (2,1) -- (2-3,1);
	\draw[line width=1pt] (2,1) -- (2,1-2);
	\draw[line width=1pt] (2,1) -- (2+1,1+1);

	\draw[line width=1pt] (1,2) -- (1-2,2);
	\draw[line width=1pt] (1,2) -- (1,2-3);
	\draw[line width=1pt] (1,2) -- (1+1,2+1);

	\fill (0,0) circle (5pt);
	\fill (2,1) circle (5pt);
	\fill (1,2) circle (5pt);

	\node at (5,1.5) {$\mathcal{S}(w)^\vee$};

\end{tikzpicture}
\\
\hline
\vspace{.1in}
\begin{tikzpicture}[scale=0.5]

	\draw[line width=1pt] (0,0) -- (0+4,0);
	\draw[line width=1pt] (0,0) -- (0,0+4);
	\draw[line width=1pt] (0,0) -- (0-1,0-1);

	\draw[line width=1pt] (1,3) -- (1+3,3);
	\draw[line width=1pt] (1,3) -- (1,3+1);
	\draw[line width=1pt] (1,3) -- (1-2,3-2);

	\draw[line width=1pt] (3,2) -- (3+1,2);
	\draw[line width=1pt] (3,2) -- (3,2+2);
	\draw[line width=1pt] (3,2) -- (3-3,2-3);

	\draw[red,line width=1pt] (1,1) -- (1+3,1);
	\draw[red,line width=1pt] (1,1) -- (1,1+2);
	\draw[red,line width=1pt] (1,1) -- (1-1,1-1);

	\draw[red,line width=1pt] (1,2) -- (1+1,2);
	\draw[red,line width=1pt] (1,2) -- (1,2+1);
	\draw[red,line width=1pt] (1,2) -- (1-2,2-2);

	\draw[red,line width=1pt] (2,2) -- (2+1,2);
	\draw[red,line width=1pt] (2,2) -- (2,2+2);
	\draw[red,line width=1pt] (2,2) -- (2-2,2-2);

	\fill (0,0) circle (5pt);
	\fill (1,3) circle (5pt);
	\fill (3,2) circle (5pt);

	\fill[red] (1,1) circle (5pt);
	\fill[red] (1,2) circle (5pt);
	\fill[red] (2,2) circle (5pt);

	\node at (5,1.5) {$\widetilde{\mathcal{S}}(v)$};

\end{tikzpicture}
&\vspace{.1in}
\begin{tikzpicture}[scale=0.5]

	\draw[line width=1pt] (0,0) -- (0+3,0);
	\draw[line width=1pt] (0,0) -- (0,0+3);
	\draw[line width=1pt] (0,0) -- (0-1,0-1);

	\draw[line width=1pt] (2,1) -- (2+1,1);
	\draw[line width=1pt] (2,1) -- (2,1+2);
	\draw[line width=1pt] (2,1) -- (2-2,1-2);

	\draw[line width=1pt] (1,2) -- (1+2,2);
	\draw[line width=1pt] (1,2) -- (1,2+1);
	\draw[line width=1pt] (1,2) -- (1-2,2-2);

	\draw[red,line width=1pt] (1,1) -- (1+1,1);
	\draw[red,line width=1pt] (1,1) -- (1,1+1);
	\draw[red,line width=1pt] (1,1) -- (1-1,1-1);

	\fill (0,0) circle (5pt);
	\fill (2,1) circle (5pt);
	\fill (1,2) circle (5pt);
	\fill[red] (1,1) circle (5pt);

	\node at (5,1.5) {$\widetilde{\mathcal{S}}(w)$};

\end{tikzpicture}
\\
\hline
\end{tabular}
\caption{Exceptional case for $m=3$ in the proof of Theorem~\ref{tildefanrefinesdualfan}.}
\label{exceptionalcaseinprooftildefanrefinesdualfan}
\end{table}

Now assume the conclusion for $m-1\geq3$ and let us prove it for $m$. Consider cones $\sigma_i^\vee\in\Sigma(v_i,v_i')^\vee,\tau_i^\vee\in\Sigma(w_i,w_i')^\vee$ such that $\sigma_i^\vee+(v_i,v_i')=\tau_i^\vee+(w_i,w_i')$ and assume that $\cap_{i=1}^m\sigma_i^\vee\neq\emptyset$. We want to show that $\cap_{i=1}^m\tau_i^\vee$ is also nonempty, which, by Helly's theorem, is equivalent to $\tau_{i_1}^\vee\cap\tau_{i_2}^\vee\cap\tau_{i_3}^\vee\neq\emptyset$ for every choice of  distinct indices $i_1,i_2,i_3\in[m]$. Since $m>3$, it will be enough to show that $\cap_{i\neq j}\tau_i^\vee\neq\emptyset$ for each choice of index $j\in[m]$. So consider the projection map $p_j\colon\mathbb{R}^{2m}\rightarrow\mathbb{R}^{2m-2}$. By Proposition~\ref{combinatorialequivalencetildesubdivisionsanddroppingspiders} we have that $\widetilde{\mathcal{S}}(p_j(v))$ is combinatorially equivalent to $\widetilde{\mathcal{S}}(p_j(w))$. So by the inductive assumption we have that $\mathcal{S}^\vee(p_j(v))$ is combinatorially equivalent to $\mathcal{S}(p_j(w))^\vee$. This means that $\cap_{i\neq j}\tau_i^\vee$ is nonempty because $\emptyset\neq\cap_{i=1}^m\sigma_i^\vee\subseteq\cap_{i\neq j}\sigma_i^\vee$, concluding the proof.
\end{proof}


\section{Planar locus in \texorpdfstring{$\overline{\mathbf{X}}(3,n)$}{Lg}
and \texorpdfstring{$\overline{\mathbf{X}}_{\GP}(3,n)$}{Lg}}\label{openpatchesfromtoricvarieties}

Given rational weights $\mathbf{b}=(b_1,\ldots,b_n)$, $0<b_i\leq1$, satisfying $\sum_{i=1}^nb_i>3$, by \cite{Ale15} there exists a projective moduli space $\overline{\mathbf{M}}_{\mathbf{b}}(3,n)$ parametrizing stable pairs $(\mathbb{P}^2,\sum_{i=1}^nb_iL_i)$, where $L_i\subseteq\mathbb{P}^2$ are lines, and their stable degenerations. For $\mathbf{b}=(1,\ldots,1)$, $\overline{\mathbf{M}}_{\mathbf{b}}(3,n)=\overline{\mathbf{X}}(3,n)$.
On the other hand, choose weights 
such that three of them, say $b_i,b_j,b_k$, are equal to $1$ and the remaining weights are equal to $\epsilon\ll1$. By \cite[Example 9.6]{Ale08}, we have an isomorphism $\overline{\mathbf{M}}_{\mathbf{b}}(3,n)\cong ((\mathbb{P}^2)^\vee)^{n-3}/\!/\;H$, and so its 
normalization is isomorphic to $Y_{\mathcal{Q}_{n-3}^\vee}$. Setting $m=n-3$, we obtain a commutative diagram of birational maps and morphisms:
\begin{center}
\begin{tikzpicture}[>=angle 90]
\matrix(a)[matrix of math nodes,
row sep=2em, column sep=2em,
text height=1.5ex, text depth=0.75ex]
{\overline{\mathbf{X}}_{\GP}(3,m+3)^\nu&&Y_{\widetilde{\mathcal{Q}}_m}\\
\overline{\mathbf{X}}(3,m+3)^\nu&\overline{\mathbf{M}}_{\mathbf{b}}(3,n)^\nu&Y_{\mathcal{Q}_m^\vee},\\};
\path[dashed,->] (a-1-1) edge node[above]{$A_{ijk}$}(a-1-3);
\path[->] (a-1-1) edge node[]{}(a-2-1);
\path[->] (a-1-3) edge node[]{}(a-2-3);
\path[->] (a-2-1) edge node[above]{$\rho_{ijk}$}(a-2-2);
\path[->] (a-2-2) edge node[above]{$\cong$}(a-2-3);
\end{tikzpicture}
\end{center}
where the morphism $Y_{\widetilde{\mathcal{Q}}_m}\rightarrow Y_{\mathcal{Q}_m^\vee}$ exists by Theorem~\ref{tildefanrefinesdualfan}. 
In what follows we focus on the case $(i,j,k)=(1,2,3)$. The same results hold for other cases after permuting the labels. The commutative diagram is a diagram of isomorphisms over $\mathbf{B}_{m+3}\cong {\mathbf{X}}(3,n)$.


\begin{definition}
\label{definitionofU123GPVandVtildetoricpatch}
Let $P=\Delta(3,m+3)$ and consider Lafforgue's toric variety $\mathcal{A}^P$. Let $\mathcal{A}_{123}^P$ the union of the toric strata of $\mathcal{A}^P$ corresponding to regular matroid polytope subdivisions of~$P$ whose maximal dimensional polytopes contain the vertex $e_1+e_2+e_3$, where $e_1,\ldots,e_{m+3}$ is the canonical basis of $\mathbb{R}^{m+3}$. Notice that $\mathcal{A}_{123}^P$ is an open subset because if a polyhedral subdivision has a maximal polytope not containing $e_1+e_2+e_3$, then any refinement of it does, so the complement is closed. This induces an open subset $$\mathbf{U}^{123}\subseteq\overline{\mathbf{X}}(3,m+3)^\nu$$
and its preimage 
$\mathbf{U}_{\GP}^{123}\subseteq\overline{\mathbf{X}}_{\GP}(3,m+3)^\nu$.
We have the following commutative diagram:
\begin{center}
\begin{tikzpicture}[>=angle 90]
\matrix(a)[matrix of math nodes,
row sep=2em, column sep=2em,
text height=1.5ex, text depth=0.75ex]
{\mathbf{U}_{\GP}^{123}&Y_{\widetilde{\mathcal{Q}}_m}\\
\mathbf{U}^{123}&Y_{\mathcal{Q}_m^\vee}.\\};
\path[dashed,->] (a-1-1) edge node[above]{$A_{123}$}(a-1-2);
\path[->] (a-1-1) edge node[]{}(a-2-1);
\path[->] (a-1-2) edge node[]{}(a-2-2);
\path[->] (a-2-1) edge node[above]{$B_{123}$}(a-2-2);
\end{tikzpicture}
\end{center}
The \emph{planar locus} $\mathbf{U}(3,n)\subseteq\overline{\mathbf{X}}(3,n)^\nu$
(resp.~$\mathbf{U}_{\GP}(3,n)\subseteq\overline{\mathbf{X}}_{\GP}(3,n)^\nu$) is the union of $\mathbf{U}^{ijk}$ (resp.~$\mathbf{U}_{\GP}^{ijk}$) for all possible triples $i,j,k$. 
\end{definition}

\begin{theorem}
\label{reductionmorphismisoonU123}
The  map $A_{123}$ is regular and both maps $A_{123}$
and $B_{123}$ are open embeddings.
The planar loci 
$\mathbf{U}(3,n)\subseteq\overline{\mathbf{X}}(3,n)^\nu$
and $\mathbf{U}_{\GP}(3,n)\subseteq\overline{\mathbf{X}}_{\GP}(3,n)^\nu$
have toroidal singularities.
\end{theorem}

\begin{proof}
We start with $B_{123}$.
Let $(\mathcal{X},\sum\limits_{i=1}^{m+3}\mathcal{L}_i)$ be the 
family of stable pairs over $\overline{\mathbf{X}}(3,m+3)$ and take its fiber 
$(X,\sum\limits_{i=1}^{m+3}L_i)$ over a point in $\mathbf{U}^{123}$.
By \cite[\S8]{KT06}, $X$ is a stable toric surface with the 
stable toric boundary given by $L_1+L_2+L_3$.
In particular, $K_X+L_1+L_2+L_3\sim0$.
Concretely, we can realize $(X,\sum\limits_{i=1}^{m+3}L_i)$ as the central fiber
of the pullback of the family of stable pairs to $\Spec R$ via
some arc $\overline{\mathbf{a}}\colon\Spec R\to \overline{\mathbf{X}}(3,m+3)$
such that $\overline{\mathbf{a}}(\Spec K)\in {\mathbf{X}}(3,m+3)(K)$.
Since all maximal-dimensional polytopes in the matroid decomposition share a vertex $e_1+e_2+e_3$, all stable lattices for the arc are contained 
in the corresponding apartment \cite[\S8]{KT06},
and therefore $X$ is a stable toric surface that corresponds to a regular mixed polyhedral subdivision of
the dual simplex $m\Delta_2^\vee$.
The broken lines $L_1,\ldots,L_n$ 
intersect pairwise at smooth points of $X$  \cite[Theorem~1.1]{HKT06}.

It follows that the pair
 $(X,\sum\limits_{i=1}^3L_i+\sum\limits_{i=4}^{m+3}\epsilon L_i)$ 
is semi-log canonical 
and the $\mathbb Q$-line bundle 
\[
K_X+\sum_{i=1}^3L_i+\sum_{i=4}^{m+3}\epsilon L_i\sim\epsilon\cdot\sum_{i=4}^{m+3}L_i
\]
is ample because 
$K_X+\sum\limits_{i=1}^{m+3}L_i\sim\sum\limits_{i=4}^{m+3}L_i$ is ample.
Therefore, the restriction of the family $(\mathcal{X},\sum\limits_{i=1}^{m+3}\mathcal{L}_i)$ to $\mathbf{U}^{123}$
embeds the latter in $\overline{\mathbf{M}}_{\mathbf{b}}(3,n)$ as an open subset.

Next, we consider $A_{123}$.
Recall that $\overline{\mathbf{X}}_{\GP}(3,n)$ 
is defined as the closure of $\mathbf{B}_{n}$ in the multigraded Hilbert scheme of $({\mathbb P}^2)^{\binom{n}{4}}$.
Write $n=3+m$, $N=\binom{3}{2}\binom{m}{2}$
and $m+N+M=\binom{3+m}{4}$. 

The pullback to $\overline{\mathbf{X}}_{\GP}(3,n)$ of the universal family of the (ordinary) Hilbert scheme is a flat family $\overline{\mathcal{M}}$
with  $n$ sections. 
For any point $x\in\overline{\mathbf{X}}_{\GP}(3,n)$,
the fiber $\overline{\mathcal{M}}_x$
is a reduced surface $(X;p_1,\ldots,p_n)$ with $n$ smooth marked points.
For every quadruple $I=\{i_1,\ldots,i_4\}\subseteq[n]$,
the morphism $X\to\bP^2$ induced by the projection onto the $I$-th
component of $({\mathbb P}^2)^{\binom{n}{4}}$ sends  $p_{i_1},\ldots,p_{i_4}$ to the points $[1:0:0]$, 
$[0:1:0]$, $[0:0:1]$, and $[1:1:1]$. The surface $(X;p_1,\ldots,p_n)$
is the special fiber of the Mustafin join 
${\mathbb P}(\Sigma_{\mathbf{a}})$ for any arc 
${\mathbf{a}}\colon\Spec K\to\mathbf{B}_n$
with a given limit point $\overline{\mathbf{a}}(0)=x\in\overline{\mathbf{X}}_{\GP}(3,n)$.

The open subset $\mathbf{U}^{123}_{\GP}$ parametrizes surfaces
such that $\Sigma_{\mathbf{a}}$ is 
contained in the apartment that corresponds to $a_1,a_2,a_3$ (for any arc converging to $x$).

\begin{lemma}
The restriction of the projection $({\mathbb P}^2)^{m+N+M}\to
({\mathbb P}^2)^{m+N}$
to every marked surface $(X;p_1,\ldots,p_n)$ parametrized by $\mathbf{U}^{123}_{\GP}$ is an isomorphism. Moreover,
the restriction to $\mathbf{U}^{123}_{\GP}$ of the induced
morphism from the multigraded Hilbert scheme 
$$\mathbf{H}\left(({\mathbb P}^2)^{m+N+M}\right)\to
\mathrm{Hilb}\left(({\mathbb P}^2)^{m+N}\right)$$ 
to the ordinary Hilbert scheme is bijective onto its image,
let $\mathbf V$ be its normalization.
\end{lemma}

\begin{proof}
To prove the first statement, we  choose an arc ${\mathbf a}$ and represent $(X;p_1,\ldots,p_n)$ as the special fiber of the corresponding Mustafin join.
Since all stable lattices are contained in the apartment that
corresponds to sections $a_1,a_2,a_3$, the set
$\Sigma_{{\mathbf a}}$ is the set of lattices $L_i$, $L_{\alpha\beta,ij}$
in the notation of Lemma~\ref{wEGsegsEGew}.
Thus $X$ is isomorphic to a subscheme in $({\mathbb P}^2)^{m+N}$ via projection. To prove the second statement, it suffices to prove that
the morphisms $X\to\bP^2_I$ parametrized by quadruples $I$ indexed by $M$
are uniquely determined by the morphisms $X\to\bP^2_J$ parametrized by quadruples $J$ indexed by $m+N$. Since $\Sigma_\mathbf{a}$ is contained in the apartment that corresponds to the sections $a_1,a_2,a_3$, for every quadruple $I=\{i_1,\ldots,i_4\}$ indexed by $M$, there exists a quadruple $J$ indexed by $m+N$ such that the image of the quadruple $p_{i_1},\ldots,p_{i_4}$ in ${\mathbb P}^2_J$
is in linearly general position.  Thus the morphism $X\to{\mathbb P}^2_I$
is a composition of the morphism $X\to{\mathbb P}^2_J$ and an 
isomorphism ${\mathbb P}^2_J\to{\mathbb P}^2_I$
uniquely determined by the marked surface $(X;p_1,\ldots,p_n)$ and its morphism to ${\mathbb P}^2_J$.
\end{proof}

At this point we can forget about ${\mathbf U}^{123}_{\GP}$ and work with $\mathbf V$
instead, which is a normalization of the partial compactification of ${\mathbf B}_n$ in the Hilbert scheme of 
$({\mathbb P}^2)^{m+N}$ parametrizing limits such that all stable lattices are in the apartment that corresponds to the first three sections (for some choice of an arc).
The pullback $\mathcal{X}\to\mathbf V$ of the universal family of the Hilbert scheme 
is smooth along $n$ disjoint sections giving the same fibers $(X;p_1,\ldots,p_n)$
as before. Composing sections with projections gives morphisms 
$${\mathcal P}_1^I,\ldots,{\mathcal P}_n^I\colon\mathbf V\to\bP^2$$
for every quadruple $I$ of the form $(1,2,3,i)$ (first type) or $(\alpha,\beta,i,j)$ (second type). Note that the points ${\mathcal P}_1^I(v),{\mathcal P}_2^I(v),{\mathcal P}_3^I(v)$ are linearly independent for any $v\in \mathbf V$. Moreover, for quadruples of the first type
they are the standard points $e_1,e_2,e_3$. However, for quadruples of the second type
only ${\mathcal P}_\alpha^I(v)$ and ${\mathcal P}_\beta^I(v)$ are standard points
$e_1$ and $e_2$, since by definition (see the Introduction) ${\mathcal P}_i^I(v)=e_3$ 
(and ${\mathcal P}_j^I(v)=e_4=[1:1:1]$).
For a fixed $v$ and a quadruple $I$ of the second type, 
we can find an automorphism of $\mathbb P^2$ that sends $e_1$ to $e_\alpha$,
$e_2$ to $e_\beta$ and ${\mathcal P}_\gamma^I(v)$ to $e_\gamma$, where 
$\{\alpha,\beta,\gamma\}=\{1,2,3\}$. This automorphism is determined uniquely up to the $\mathbb G_m^2$-action. This gives a $\mathbb G_m^{2N}$-torsor $\psi\colon\widetilde{\mathbf V}\to {\mathbf V}$ and an isomorphism
$$\Psi\colon\widetilde{\mathbf V}\times({\mathbb P}^2)^{m+N}\to \widetilde{\mathbf V}\times({\mathbb P}^2)^{m+N},$$
which is defined to be the identity on the first factor. Let $\widetilde{\mathcal{X}}\to\widetilde{\mathbf V}$ be the family $\Psi(\psi^*\mathcal{X})$.

Since $\mathbf B_n$ is an open subset of an algebraic torus
$\mathbb G_m^{2m-2}$ and $\psi$ is a torsor, $\psi^{-1}(\mathbf B_n)$ is an open subset of an algebraic torus
$\mathbb G_m^{2(m+N)-2}$. The morphism of $\psi^{-1}(\mathbf B_n)$ into the Hilbert scheme of $({\mathbb P}^2)^{m+N}$ given by the family $\widetilde{\mathcal{X}}$ agrees with the embedding of $\mathbb G_m^{2(m+N)-2}$ into this Hilbert scheme as in \S\ref{sGsgsGsehsRH}. Indeed, note that the first three sections of $\widetilde{\mathcal{X}}$ are constant sections given by the  $\mathbb G_m^2$-invariant points in $\mathbb P^2$, thus every fiber of $\widetilde{\mathcal{X}}$ over $\psi^{-1}(\mathbf B_n)$
is a $\mathbb P^2$ embedded into $(\mathbb P^2)^{m+N}$ by torus translates.
Therefore, $\widetilde{\mathbf V}$ is a partial compactification
of $\psi^{-1}(\mathbf B_n)$ in the toric Kapranov space $Y_{{\mathcal Q}_{m+N}}$.
By Lemma~\ref{wEGsegsEGew}
and Lemma--Definition~\ref{qergaerhaerh}, $\widetilde{\mathbf V}$~is an open subset of 
$Y_{\mathcal{Q}_{m+N}'}$ and therefore, as claimed, $\mathbf V=\widetilde{\mathbf V}/\mathbb G_m^{2N}$ is isomorphic to an open subset of
$Y_{\mathcal{Q}_{m+N}'}/\mathbb G_m^{2N}=Y_{\widetilde {\mathcal Q}_{m}}$ by Lemma~\ref{wrgawrgarsg}.
\end{proof}


\section{Study of \texorpdfstring{$\overline{\mathbf{X}}_{\GP}(3,6)$}{Lg}}

In this section we 
explicitly describe the compactification $\overline{\mathbf{X}}_{\GP}(3,6)$
using the theory developed in \S\ref{sGsgsGsehsRH}--\ref{openpatchesfromtoricvarieties}. 
We classify the $6$-pointed degenerations parametrized by $\overline{\mathbf{X}}_{\GP}(3,6)$, 
show their relation to the corresponding line arrangement degenerations parametrized by Kapranov's Chow quotient $\overline{\mathbf{X}}(3,6)$, and finally prove that $\overline{\mathbf{X}}_{\GP}(3,6)^\nu$ is a tropical compactification of $\mathbf{X}(3,6)$ given by a specific polyhedral subdivision
of the tropical Grassmannian $\Sigma(3,6)$.



\begin{proposition}
Up to permuting the labels, the degenerations parametrized by the planar locus $\mathbf{U}_{\GP}(3,6)$ are listed in Tables~\ref{tbl:GPdegenerationspart1}, \ref{tbl:GPdegenerationspart2}, and \ref{tbl:furtherGPdegenerations}.
\end{proposition}

\begin{proof}
We first classify all cones in the quotient fan $\mathcal Q_3$,
i.e.~all possible decompositions $\mathcal{S}(v)$ for integral $v\in(\mathbb{R}^2)^3$. 
All the possibilities are computed in \cite[Figure 6]{CHSW11}, and we list them in Table~\ref{triplesoflatticesinapartment}. 
Then, we need to classify all cones in the fan $\widetilde{\mathcal Q}_3$, i.e.~all possible subdivisions $\widetilde{\mathcal{S}}(v)$, which are listed in Tables~\ref{triplesoflatticesinapartmentandcorrespondingredspiders} and \ref{stretchesofspiders}. From these, we can determine the corresponding central fibers $X$ of Mustafin joins by Remark~\ref{degenerationfrompicture}.

Next, we need to put markings $p_1,\ldots,p_6$ on $X$. First, we choose three labeled points $p_i,p_j,p_k$ and put them in the three vertices of the degeneration (we do not always choose just $p_1,p_2,p_3$ for the reason explained in Remark~\ref{labelingofthedegenerationsaccordingtokapranovdegenerations}). These three points are in general linear position in all primary components. By Proposition~\ref{,amNEBFjahevfaeV}, the remaining ``light'' points are arranged 
according to Proposition~\ref{rgwRGsrgSRH}, so that the $i$-the point
belongs to the component that corresponds to $(v_i,v_i')$. 
As a guide, case $\#8$ was already worked out in Example~\ref{case8comesfromanarc}.
The final result is depicted in Tables~\ref{tbl:GPdegenerationspart1}, \ref{tbl:GPdegenerationspart2}, and \ref{tbl:furtherGPdegenerations}.
\end{proof}

\begin{remark}
\label{labelingofthedegenerationsaccordingtokapranovdegenerations}
The numbering of the degenerations in Tables~\ref{tbl:GPdegenerationspart1}, \ref{tbl:GPdegenerationspart2}, and \ref{tbl:furtherGPdegenerations} 
is determined as follows. Given $\mathbf{a}\in(\mathbb{P}^2)^6(K)$ in general linear position over $K$, we have $x=\overline{\mathbf{a}}(0)\in\overline{\mathbf{X}}_{\GP}(3,6)$. Interpreting $\mathbf{a}$ as a one-parameter family of lines in $(\mathbb{P}^2)^\vee$ instead, we obtain another limit $y\in\overline{\mathbf{X}}(3,6)$, and $x\mapsto y$ under the morphism $\overline{\mathbf{X}}_{\GP}(3,6)\rightarrow\overline{\mathbf{X}}(3,6)$.
Recall that the boundary of Kapranov's compactification $\overline{\mathbf{X}}(3,6)$ is stratified according to the regular matroid polytope subdivisions of the hypersimplex $\Delta(3,6)$. These are listed up to $S_6$-action in \cite[Table 4.4]{Ale15}. Each boundary stratum 
corresponds to degenerate KSBA stable pairs listed in \cite[Figures 5.11, 5.12, and 5.13]{Ale15}, which we reproduce with permission of the author in Figures~\ref{Alexeevdegenerations6linesC}, \ref{Alexeevdegenerations6linesA}, and \ref{Alexeevdegenerations6linesB} (notice that the degeneration $\#16$ corresponds to the matroid subdivision no. 16 in \cite[Table 4.4]{Ale15} with (3) $x_{1234}\leq2$, $x_{12}\leq1$ instead).

We claim that a degeneration of six points in $\mathbb{P}^2$ labeled by $\#N$ in our tables is mapped to the interior of the Kapranov's stratum $\#N$. To establish this matching, let us first look at an example. Assume $y$ belongs to the interior of the Kapranov's stratum corresponding to the matroid polytope subdivision no.~$8$ in \cite[Table 4.4]{Ale15}. This gives rise to the degenerate line arrangements in $\mathbb{P}^2$ in the first row of Table~\ref{tbl:matroidtiling8degenerations}. 
\begin{table}[hbtp]
\centering
\begin{tabular}{|>{\centering\arraybackslash}m{4cm}|>{\centering\arraybackslash}m{4cm}|>{\centering\arraybackslash}m{4cm}|}
\hline
\begin{tikzpicture}[scale=0.5]

	\draw[line width=1pt] (-3,0) -- (3,0);
	\draw[line width=1pt] (-3,1) -- (-1,-1);
	\draw[line width=1pt] (-3,-1) -- (-1,1);
	\draw[line width=1pt] (-2,1) -- (-2,-1);
	\draw[line width=1pt] (3,1) -- (1,-1);
	\draw[line width=1pt] (3,-1) -- (1,1);

	\node at (0,0.5) {$1$};
	\node at (-3,1.5) {$2$};
	\node at (-2,1.5) {$3$};
	\node at (-1,1.5) {$4$};
	\node at (1,1.5) {$5$};
	\node at (3,1.5) {$6$};

\end{tikzpicture}
&
\begin{tikzpicture}[scale=0.5]

	\draw[line width=1pt] (-3,0) -- (3,0);
	\draw[line width=1pt] (-3,0.2) -- (3,0.2);

	\node at (0,0.7) {$5,6$};

\end{tikzpicture}
&
\begin{tikzpicture}[scale=0.5]

	\draw[line width=1pt] (-3,0) -- (3,0);
	\draw[line width=1pt] (-3,0.2) -- (3,0.2);
	\draw[line width=1pt] (-3,0.4) -- (3,0.4);

	\node at (0,0.9) {$2,3,4$};

\end{tikzpicture}
\\
\hline
\begin{tikzpicture}[scale=0.5]

	\fill (-3,0) circle (5pt);
	\fill (-2,0) circle (5pt);
	\fill (-1,0) circle (5pt);
	\fill (0,0) circle (5pt);
	\fill (0,-1) circle (5pt);
	\fill (0,-2) circle (5pt);

	\node at (0,0.6) {$1$};
	\node at (-1,0.6) {$2$};
	\node at (-2,0.6) {$3$};
	\node at (-3,0.6) {$4$};
	\node at (0.5,-1) {$5$};
	\node at (0.5,-2) {$6$};

\end{tikzpicture}
&
\begin{tikzpicture}[scale=0.5]

	\fill (0,0) circle (5pt);

	\node at (0,0.6) {$5,6$};

\end{tikzpicture}
&
\begin{tikzpicture}[scale=0.5]

	\fill (0,0) circle (5pt);

	\node at (0,0.6) {$2,3,4$};

\end{tikzpicture}
\\
\hline
\end{tabular}
\caption{In the first row, degenerate line arrangements corresponding to three matroid polytopes in subdivision no.~8 in \cite[Table 4.4]{Ale15}. In the second row, their respective projectively dual arrangements from Figure~\ref{GPdegenerationincase8}.}
\label{tbl:matroidtiling8degenerations}
\end{table}
The projective dual of these line degenerations are in the second row of the same table. Now, the marking of $(\mathbb{P}(\Sigma_\mathbf{a})_\Bbbk;\overline{a}_1(0),\ldots,\overline{a}_6(0))$ parametrized by $x$ must be such that on the primary components the image of the markings are given in the second row of Table~\ref{tbl:matroidtiling8degenerations}. Such degeneration is given in Figure~\ref{GPdegenerationincase8}. One can inspect case by case that the degenerations listed in Tables~\ref{tbl:GPdegenerationspart1} and \ref{tbl:GPdegenerationspart2} match bijectively the matroid subdivisions in \cite[Table 4.4]{Ale15}. We deal with the Kapranov's stratum $\#7$ separately in Lemma~\ref{non-planardegenerationgp7}. Lastly, in Table~\ref{tbl:furtherGPdegenerations}, the numbering $\#12.1,\#15.1$, and so on correspond to the stretches of spiders in Table~\ref{stretchesofspiders}.
\end{remark}


Next we describe the boundary stratification of $\overline{\mathbf{X}}(3,6)$ and construct its explicit blow up
$\widehat{\mathbf{X}}(3,6)$ based on the analysis of the toroidal morphism
of planar loci ${\mathbf{U}}_{\GP}(3,6)\to {\mathbf{U}}(3,6)$.
Then we will add an ad hoc analysis
of the non-planar locus to prove 
$\widehat{\mathbf{X}}(3,6)\cong\overline{\mathbf{X}}_{\GP}(3,6)^\nu$.
The space $\mathbf{X}(3,6)=G^0(3,6)/\mathbb{G}_m^5$ is a very affine variety, i.e.~a closed subset of the torus
\[
T=\mathbb{G}_m^{\binom{6}{3}-1}/\mathbb{G}_m^{5}.
\]
It was proved in \cite{Lux08} that Kapranov's compactification $\overline{\mathbf{X}}(3,6)$ can be obtained as the Zariski closure of $\mathbf{X}(3,6)$ inside the toric variety $Y_{\Sigma(3,6)}$ with dense open subtorus $T$, where the fan $\Sigma(3,6)$ is the tropical Grassmannian of Speyer--Sturmfels in \cite{SS04}. Moreover, $\overline{\mathbf{X}}(3,6)$ is \emph{tropical}, which means that $\overline{\mathbf{X}}(3,6)$ is proper and the multiplication map
\[
T\times\overline{\mathbf{X}}(3,6)\rightarrow Y_{\Sigma(3,6)}
\]
is faithfully flat (\cite[Definition~1.1]{Tev07}). Even better,
Luxton proves in \cite{Lux08} that 
the multiplication map is smooth (and so $\mathbf{X}(3,6)$ is \emph{sch\"on} \cite[Definition~1.3]{Tev07}).
Adopting the notation in \cite{Lux08}, the fan $\Sigma(3,6)$ has three types of rays, which correspond to boundary divisors of $\overline{\mathbf{X}}(3,6)$. In parentheses, we give the corresponding notation from \cite{SS04}.

\begin{enumerate}
\item For each triple of distinct indices $i,j,k\in[6]$, we have a ray $(ijk)$ (rays of type E).

\item For each pair of distinct indices $i,j\in[6]$, we have a ray $(ij)$ (rays of type F).

\item For each partition $\{\{i,j\},\{k,\ell\},\{m,n\}\}$ of $[6]$, we have two distinct rays denoted by $(ij)(k\ell)(mn)$ and $(ij)(mn)(k\ell)$ (these are the rays of type G in \cite{SS04}).
\end{enumerate}

There is a bijective correspondence between the boundary strata of  $\overline{\mathbf{X}}(3,6)$ and the toric strata of $Y_{\Sigma(3,6)}$, i.e.~the cones of $\Sigma(3,6)$. Explicitly, up to $S_6$-action, consider the labeling of the strata of $\overline{\mathbf{X}}(3,6)$ given in \cite{Ale15}. Then this correspondence is given in Table~\ref{correspondencekapranov'sstrataandcones}.

\begin{table}[hbtp]
\centering
\begin{tabular}{|>{\centering\arraybackslash}m{0.9cm}|>{\centering\arraybackslash}m{5cm}||>{\centering\arraybackslash}m{0.9cm}|>{\centering\arraybackslash}m{8cm}|}
\hline
$\#1$ & $(123)$
&
$\#14$ & $(12)(34)(56),(356)$
\\
\hline
$\#2$ &$(123),(456)$
&
$\#15$ &$(12)(34)(56),(12),(56)$
\\
\hline
$\#3$& $(124),(456)$
&
$\#16$& $(12),(56),(124)$
\\
\hline
$\#4$& $(456),(124),(135)$
&
$\#17$& $(12)(34)(56),(125),(12)$
\\
\hline
$\#5$& $(56)$
&
$\#18$& $(12)(34)(56),(356),(12)$
\\
\hline
$\#6$& $(56),(123)$
&
$\#19$& $(12)(34)(56),(125),(356)$
\\
\hline
$\#7$& $(123),(145),(246),(356)$
&
$\#20$& $(12)(34)(56),(12)(56)(34),(12),(34),(56)$
\\
\hline
$\#8$ &$(56),(156)$
&
$\#21$ &$(12)(34)(56),(34),(56),(256)$
\\
\hline
$\#9$ &$(123),(56),(156)$
&
$\#22$& $(12)(34)(56),(34),(256),(346)$
\\
\hline
$\#10$& $(234),(56),(156)$
&
$\#23$ &$(34),(56),(134),(256)$
\\
\hline
$\#11$ &$(12)(34)(56)$
&
$\#24$ &$(34),(56),(134),(156)$
\\
\hline
$\#12$ &$(56),(34)$
&
$\#25$& $(12)(34)(56),(234),(456),(125)$
\\
\hline
$\#13$& $(12)(34)(56),(56)$
&&

\\
\hline
\end{tabular}
\caption{Correspondence between strata of $\overline{\mathbf{X}}(3,6)$ and cones of $\Sigma(3,6)$.}
\label{correspondencekapranov'sstrataandcones}
\end{table}


Note that only the cone \#20 (of type FFFGG) is not simplicial. These cones are known as the \emph{bipyramid cones}: each one can be obtained by taking the cone over an appropriate triangular bipyramid, which is the gluing of two tetrahedra along a common facet.


\begin{definition}
\label{explicitdescriptionoftherefinementinducedbysplittingbypiramid}
Let $\sigma\in\Sigma(3,6)$ be a bipyramid cone. 
We write $\sigma=\langle e_1,e_2,e_3,e_4,e_4'\rangle$,
where $e_1,e_2,e_3$ are the vertices of the common base of the two pyramids, and split it into the following $12$-cones
\[
\sigma_{ijk}=\langle e_i,e_i+e_j,e_i+e_j+e_k,e_4\rangle,~\sigma_{ijk}'=\langle e_i,e_i+e_j,e_i+e_j+e_k,e_4'\rangle,
\]
where $\{i,j,k\}=\{1,2,3\}$. 
By also splitting the cones $\#21$, $\#23$, and $\#24$
as illustrated in Table~\ref{subdivisionofthebipyramid}, we obtain a refinement $\widehat{\Sigma}(3,6)$ of the fan $\Sigma(3,6)$. We define $\widehat{\mathbf{X}}(3,6)$ to be the closure of $\mathbf{X}(3,6)$ inside the smooth toric variety $Y_{\widehat{\Sigma}(3,6)}$. We have a commutative diagram

\begin{center}
\begin{tikzpicture}[>=angle 90]
\matrix(a)[matrix of math nodes,
row sep=2em, column sep=2em,
text height=1.5ex, text depth=0.25ex]
{\widehat{\mathbf{X}}(3,6)&Y_{\widehat{\Sigma}(3,6)}\\
\overline{\mathbf{X}}(3,6)&Y_{\Sigma(3,6)}.\\};
\path[right hook->] (a-1-1) edge node[above]{}(a-1-2);
\path[->] (a-1-1) edge node[left]{}(a-2-1);
\path[right hook->] (a-2-1) edge node[below]{}(a-2-2);
\path[->] (a-1-2) edge node[right]{}(a-2-2);
\end{tikzpicture}
\end{center}
Since $\widehat{\Sigma}(3,6)$ refines $\Sigma(3,6)$, it follows by \cite[Proposition~2.5]{Tev07} that $\widehat{\mathbf{X}}(3,6)$ is a tropical compactification of $\mathbf{X}(3,6)$ and its morphism to
$\overline{\mathbf{X}}(3,6)$ is toroidal.
\end{definition}

\begin{table}[hbtp]
\centering
\begin{tabular}{|>{\centering\arraybackslash}m{6.5cm}|>{\centering\arraybackslash}m{6.5cm}|}
\hline
\begin{tikzpicture}[scale=0.7]

	\draw[line width=1pt] (0,0) -- (5,1);
	\draw[dashed,line width=1pt] (5,1) -- (1.5,2);
	\draw[dashed,line width=1pt] (1.5,2) -- (0,0);
	\draw[dashed,line width=1pt] (1.5,2) -- (6.5/3,4.5);
	\draw[line width=1pt] (0,0) -- (6.5/3,4.5);
	\draw[line width=1pt] (5,1) -- (6.5/3,4.5);
	\draw[dashed,line width=1pt] (1.5,2) -- (6.5/3,-2.5);
	\draw[line width=1pt] (0,0) -- (6.5/3,-2.5);
	\draw[line width=1pt] (5,1) -- (6.5/3,-2.5);

	\draw[dotted,line width=1pt] (6.5/3,4.5) -- (6.5/3,-2.5);
	\draw[dotted,line width=1pt] (0,0) -- (6.5/2,3/2);
	\draw[dotted,line width=1pt] (1.5,2) -- (5/2,1/2);
	\draw[dotted,line width=1pt] (5,1) -- (1.5/2,2/2);

	\draw[dotted,line width=1pt] (6.5/3,4.5) -- (6.5/2,3/2);
	\draw[dotted,line width=1pt] (6.5/3,4.5) -- (5/2,1/2);
	\draw[dotted,line width=1pt] (6.5/3,4.5) -- (1.5/2,2/2);

	\draw[dotted,line width=1pt] (6.5/3,-2.5) -- (6.5/2,3/2);
	\draw[dotted,line width=1pt] (6.5/3,-2.5) -- (5/2,1/2);
	\draw[dotted,line width=1pt] (6.5/3,-2.5) -- (1.5/2,2/2);

	\node at (2.3,5) {$(ij)(k\ell)(mn)$};
	\node at (2.3,-3) {$(ij)(mn)(k\ell)$};
	\node at (-0.5,0) {$(ij)$};
	\node at (5.6,1) {$(k\ell)$};
	\node at (0.5,2.2) {$(mn)$};

\end{tikzpicture}
&
\begin{tikzpicture}[scale=0.7]

	\draw[line width=1pt] (0,0) -- (5,1);
	\draw[dashed,line width=1pt] (5,1) -- (1.5,2);
	\draw[dashed,line width=1pt] (1.5,2) -- (0,0);
	\draw[dashed,line width=1pt] (1.5,2) -- (6.5/3,4.5);
	\draw[line width=1pt] (0,0) -- (6.5/3,4.5);
	\draw[line width=1pt] (5,1) -- (6.5/3,4.5);

	\draw[dotted,line width=1pt] (1.5,2) -- (5/2,1/2);

	\draw[dotted,line width=1pt] (6.5/3,4.5) -- (5/2,1/2);

	\node at (2.3,5) {$(ijm)$};
	\node at (2.4,2.2) {$(k\ell m)$};
	\node at (5.5,1) {$(ij)$};
	\node at (-0.6,0) {$(k\ell)$};

\end{tikzpicture}
\\
\hline
\begin{tikzpicture}[scale=0.7]

	\draw[line width=1pt] (0,0) -- (5,1);
	\draw[dashed,line width=1pt] (5,1) -- (1.5,2);
	\draw[dashed,line width=1pt] (1.5,2) -- (0,0);
	\draw[dashed,line width=1pt] (1.5,2) -- (6.5/3,4.5);
	\draw[line width=1pt] (0,0) -- (6.5/3,4.5);
	\draw[line width=1pt] (5,1) -- (6.5/3,4.5);

	\draw[dotted,line width=1pt] (1.5,2) -- (5/2,1/2);

	\draw[dotted,line width=1pt] (6.5/3,4.5) -- (5/2,1/2);

	\node at (2.3,5) {$(ijm)$};
	\node at (2.3,2.2) {$(k\ell n)$};
	\node at (5.5,1) {$(ij)$};
	\node at (-0.6,0) {$(k\ell)$};

\end{tikzpicture}
&
\begin{tikzpicture}[scale=0.7]

	\draw[line width=1pt] (0,0) -- (5,1);
	\draw[dashed,line width=1pt] (5,1) -- (1.5,2);
	\draw[dashed,line width=1pt] (1.5,2) -- (0,0);
	\draw[dashed,line width=1pt] (1.5,2) -- (6.5/3,4.5);
	\draw[line width=1pt] (0,0) -- (6.5/3,4.5);
	\draw[line width=1pt] (5,1) -- (6.5/3,4.5);

	\draw[dotted,line width=1pt] (1.5,2) -- (5/2,1/2);

	\draw[dotted,line width=1pt] (6.5/3,4.5) -- (5/2,1/2);

	\node at (2.3,5) {$(ij)(k\ell)(mn)$};
	\node at (5.6,1) {$(k\ell)$};
	\node at (-0.7,0) {$(mn)$};
	\node at (2.4,2.2) {$(jmn)$};

\end{tikzpicture}
\\
\hline
\end{tabular}
\caption{
The refinement $\widehat{\Sigma}(3,6)$ of the tropical Grassmannian.}
\label{subdivisionofthebipyramid}
\end{table}

\begin{remark}
Kapranov's compactification $\overline{\mathbf{X}}(3,6)$ is singular at exactly $15$ isolated points each locally isomorphic to the cone over $\mathbb{P}^1\times\mathbb{P}^2$ (\cite[Theorem 4.2.4]{Lux08}). The 
$15$ singular points are the strata that correspond to the bipyramid cones in $\Sigma(3,6)$. These singularities admit two small resolutions related by a flop: let $\overline{\mathbf{X}}_1(3,6)$ (resp. $\overline{\mathbf{X}}_2(3,6)$) be the small resolution  with an exceptional locus isomorphic to $\mathbb{P}^1$ (resp. $\mathbb{P}^2$). There is also the blow up $\overline{\mathbf{X}}'(3,6)$ of the cone singularities with exceptional divisor $\mathbb{P}^1\times\mathbb{P}^2$.
All these compactifications are  tropical and correspond to refinements of $\Sigma(3,6)$: $\Sigma_1(3,6)$ is obtained by splitting each bipyramid into its two pyramids, $\Sigma_2(3,6)$ is obtained by adding to each bipyramid the segment that joins the two opposite vertices, and finally, $\Sigma'(3,6)$ is the common refinement of $\Sigma_1(3,6)$ and $\Sigma_2(3,6)$. The space $\widehat{\mathbf{X}}(3,6)$ is obtained from $\overline{\mathbf{X}}'(3,6)$ by blowing up the codimension $2$ strata of type $\#12$. Summarizing, we have the following commutative diagram:
\begin{center}
\begin{tikzpicture}[>=angle 90]
\matrix(a)[matrix of math nodes,
row sep=2em, column sep=2em,
text height=2.0ex, text depth=0.25ex]
{&\widehat{\mathbf{X}}(3,6)&\\
&\overline{\mathbf{X}}'(3,6)&\\
\overline{\mathbf{X}}_1(3,6)&&\overline{\mathbf{X}}_2(3,6)\\
&\overline{\mathbf{X}}(3,6).&\\};
\path[->] (a-1-2) edge node[]{}(a-2-2);
\path[->] (a-2-2) edge node[]{}(a-3-1);
\path[->] (a-2-2) edge node[]{}(a-3-3);
\path[dashed,->] (a-3-1) edge node[above]{flop}(a-3-3);
\path[->] (a-3-1) edge node[]{}(a-4-2);
\path[->] (a-3-3) edge node[]{}(a-4-2);
\end{tikzpicture}
\end{center}
\end{remark}


\begin{proof}[Proof of Theorem~\ref{zh}]
Consider the birational map $f\colon\widehat{\mathbf{X}}(3,6)\dashrightarrow\overline{\mathbf{X}}_{\GP}(3,6)^\nu$, which
is an isomorphism over ${\mathbf{X}}(3,6)$. We first show that $f$ is an   isomorphism over the open planar locus $\mathbf{U}_{\GP}(3,6)\subseteq\overline{\mathbf{X}}_{\GP}(3,6)^\nu$. Since $\mathbf{U}_{\GP}(3,6)$ is the union of the open subsets $\mathbf{U}_{\GP}^{ijk}$ for all triples of distinct indices,
it suffices to show that $f$ is an isomorphism over~ $\mathbf{U}_{\GP}^{123}$.
Recall from \S\ref{openpatchesfromtoricvarieties} that  $\mathbf{U}^{123}\subseteq\overline{\mathbf{X}}(3,6)$ is isomorphic to the open subset $\mathbf{V}\subseteq Y_{\mathcal{Q}_3^\vee}$. Let $\widehat{\mathbf{U}}^{123}\subseteq\widehat{\mathbf{X}}(3,6)$ be the preimage of $\mathbf{U}^{123}$ under the morphism $\widehat{\mathbf{X}}(3,6)\rightarrow\overline{\mathbf{X}}(3,6)$. 
Finally, let $\widehat{\mathbf{V}}\subseteq Y_{\widetilde{\mathcal{Q}}_3}$ be the open subset $A_{123}(\mathbf{U}_{\GP}^{123})$ as in Definition~\ref{definitionofU123GPVandVtildetoricpatch}.
We claim that $f$ induces an isomorphism $\widehat{\mathbf{U}}_{\GP}^{123}
\cong\widehat{\mathbf{V}}$, so that we have a commutative diagram
\begin{center}
\begin{tikzpicture}[>=angle 90]
\matrix(a)[matrix of math nodes,
row sep=2em, column sep=2em,
text height=2.0ex, text depth=0.5ex]
{Y_{\widehat{\Sigma}(3,6)}&\widehat{\mathbf{U}}^{123}&\widehat{\mathbf{V}}&Y_{\widetilde{\mathcal{Q}}_3}\\
Y_{\Sigma(3,6)}&\mathbf{U}^{123}&\mathbf{V}&Y_{\mathcal{Q}_3^\vee}.\\};
\path[left hook->] (a-1-2) edge node[]{}(a-1-1);
\path[left hook->] (a-2-2) edge node[]{}(a-2-1);
\path[->] (a-1-2) edge node[above]{$\cong$}(a-1-3);
\path[->] (a-2-2) edge node[above]{$\cong$}(a-2-3);
\path[right hook->] (a-1-3) edge node[]{}(a-1-4);
\path[right hook->] (a-2-3) edge node[]{}(a-2-4);
\path[->] (a-1-1) edge node[]{}(a-2-1);
\path[->] (a-1-2) edge node[]{}(a-2-2);
\path[->] (a-1-3) edge node[]{}(a-2-3);
\path[->] (a-1-4) edge node[]{}(a-2-4);
\end{tikzpicture}
\end{center}

Since both spaces are obtained by toroidal blow ups of strata
of the same variety $\mathbf{U}^{123}$, it suffices to prove that the strata match.
We compare the refinements 
$\widehat{\Sigma}(3,6)\preccurlyeq\Sigma(3,6)$ and $\widetilde{\mathcal{Q}}_3\preccurlyeq\mathcal{Q}_3^\vee$.
The maximal cones in $\Sigma(3,6)$ which are split in $\widehat{\Sigma}(3,6)$ and that correspond to points in $\mathbf{U}^{123}$ are the cones $\#20$, $\#24$, $\#21$ in Table~\ref{correspondencekapranov'sstrataandcones}.
Under the isomorphism $\mathbf{U}^{123}\rightarrow\mathbf{V}$, these cones correspond to the cones of $\mathcal{Q}_3^\vee$ in the left column of Table~\ref{conesinrefinedquotientfancorrespondingtosplittingbipyramid}. 
The  cones in $\Sigma(3,6)$ are split as in Definition~\ref{explicitdescriptionoftherefinementinducedbysplittingbypiramid}, which matches the refinement given by the maximal cones in $\widetilde{\mathcal{Q}}_3$ in the right column of Table~\ref{conesinrefinedquotientfancorrespondingtosplittingbipyramid}. 
The three thickened segment lengths $x_1,x_2,x_3$ satisfy $x_1>x_2>x_3$. Different choices of inequality yield a total of $12$ different cones in $\widetilde{\mathcal{Q}}_3$, which match the splitting in $\widehat{\Sigma}(3,6)$ of the bipyramid cone. This implies that $\widehat{\mathbf{U}}^{123}\cong\widehat{\mathbf{V}}$. 

\begin{table}[hbtp]
\centering
\begin{tabular}{|>{\centering\arraybackslash}m{2.8cm}|>{\centering\arraybackslash}m{4.8cm}>{\centering\arraybackslash}m{4.8cm}|}
\hline
\vspace{0.1in}
\begin{tikzpicture}[scale=0.4]

	\draw[line width=1pt] (0,0) -- (0-1,0);
	\draw[line width=1pt] (0,0) -- (0,0-2);
	\draw[line width=1pt] (0,0) -- (0+2,0+2);

	\draw[line width=1pt] (1,-1) -- (1-2,-1);
	\draw[line width=1pt] (1,-1) -- (1,-1-1);
	\draw[line width=1pt] (1,-1) -- (1+2,-1+2);

	\draw[line width=1pt] (2,1) -- (2-3,1);
	\draw[line width=1pt] (2,1) -- (2,1-3);
	\draw[line width=1pt] (2,1) -- (2+1,1+1);

	\fill (0,0) circle (5pt);
	\fill (1,-1) circle (5pt);
	\fill (2,1) circle (5pt);

\end{tikzpicture}
&\vspace{0.1in}
\begin{tikzpicture}[scale=0.4]

	\draw[red,line width=1pt] (0,0) -- (0+4,0);
	\draw[red,line width=1pt] (0,0) -- (0,0+1);
	\draw[red,line width=1pt] (0,0) -- (0-1,0-1);

	\draw[line width=1pt] (0,1) -- (0+7,1);
	\draw[line width=1pt] (0,1) -- (0,1+4);
	\draw[line width=1pt] (0,1) -- (0-1,1-1);

	\draw[line width=1pt] (4,0) -- (4+1,0);
	\draw[line width=1pt] (4,0) -- (4,0+5);
	\draw[line width=1pt] (4,0) -- (4-1,0-1);

	\draw[red,line width=1pt] (6,2) -- (6+1,2);
	\draw[red,line width=1pt] (6,2) -- (6,2+2);
	\draw[red,line width=1pt] (6,2) -- (6-2,2-2);

	\draw[line width=1pt] (6,4) -- (6+1,4);
	\draw[line width=1pt] (6,4) -- (6,4+1);
	\draw[line width=1pt] (6,4) -- (6-5,4-5);

	\draw[red,line width=1pt] (3,4) -- (3+3,4);
	\draw[red,line width=1pt] (3,4) -- (3,4+1);
	\draw[red,line width=1pt] (3,4) -- (3-3,4-3);

	\fill[red] (0,0) circle (5pt);
	\fill (0,1) circle (5pt);
	\fill (4,0) circle (5pt);
	\fill[red] (6,2) circle (5pt);
	\fill (6,4) circle (5pt);
	\fill[red] (3,4) circle (5pt);

	\node at (1.5,1.5) {$x_1$};
	\draw[line width=2pt] (0,1) -- (0+3,1);
	\node at (4.6,3.5) {$x_2$};
	\draw[line width=2pt] (4,2) -- (4+2,2+2);
	\node at (3.4,0.5) {$x_3$};
	\draw[line width=2pt] (4,0) -- (4,0+1);

\end{tikzpicture}
&\vspace{0.1in}
\begin{tikzpicture}[scale=0.4]

	\draw[red,line width=1pt] (0,0) -- (0+4,0);
	\draw[red,line width=1pt] (0,0) -- (0,0+2);
	\draw[red,line width=1pt] (0,0) -- (0-1,0-1);

	\draw[line width=1pt] (0,2) -- (0+7,2);
	\draw[line width=1pt] (0,2) -- (0,2+3);
	\draw[line width=1pt] (0,2) -- (0-1,2-1);

	\draw[line width=1pt] (3,0) -- (3+4,0);
	\draw[line width=1pt] (3,0) -- (3,0+5);
	\draw[line width=1pt] (3,0) -- (3-1,0-1);

	\draw[red,line width=1pt] (6,3) -- (6+1,3);
	\draw[red,line width=1pt] (6,3) -- (6,3+2);
	\draw[red,line width=1pt] (6,3) -- (6-3,3-3);

	\draw[line width=1pt] (6,4) -- (6+1,4);
	\draw[line width=1pt] (6,4) -- (6,4+1);
	\draw[line width=1pt] (6,4) -- (6-5,4-5);

	\draw[red,line width=1pt] (2,4) -- (2+4,4);
	\draw[red,line width=1pt] (2,4) -- (2,4+1);
	\draw[red,line width=1pt] (2,4) -- (2-2,4-2);

	\fill[red] (0,0) circle (5pt);
	\fill (0,2) circle (5pt);
	\fill (3,0) circle (5pt);
	\fill[red] (6,3) circle (5pt);
	\fill (6,4) circle (5pt);
	\fill[red] (2,4) circle (5pt);

	\node at (1.5,1.5) {$x_1$};
	\draw[line width=2pt] (0,2) -- (0+3,2);
	\node at (4.6,3.5) {$x_2$};
	\draw[line width=2pt] (4,2) -- (4+2,2+2);
	\node at (3.7,0.5) {$x_3$};
	\draw[line width=2pt] (3,0) -- (3,0+1);

\end{tikzpicture}
\\
\hline
\begin{tikzpicture}[scale=0.4]

	\draw[line width=1pt] (0,2) -- (0-1,2);
	\draw[line width=1pt] (0,2) -- (0,2-3);
	\draw[line width=1pt] (0,2) -- (0+1,2+1);

	\draw[line width=1pt] (1,1) -- (1-2,1);
	\draw[line width=1pt] (1,1) -- (1,1-2);
	\draw[line width=1pt] (1,1) -- (1+2,1+2);

	\draw[line width=1pt] (2,0) -- (2-3,0);
	\draw[line width=1pt] (2,0) -- (2,0-1);
	\draw[line width=1pt] (2,0) -- (2+1,0+1);

	\fill (0,2) circle (5pt);
	\fill (1,1) circle (5pt);
	\fill (2,0) circle (5pt);

\end{tikzpicture}
&\vspace{0.1in}
\begin{tikzpicture}[scale=0.4]

	\draw[red,line width=1pt] (0,0) -- (0+1,0);
	\draw[red,line width=1pt] (0,0) -- (0,0+2);
	\draw[red,line width=1pt] (0,0) -- (0-1,0-1);

	\draw[red,line width=1pt] (1,0) -- (1+1,0);
	\draw[red,line width=1pt] (1,0) -- (1,0+2);
	\draw[red,line width=1pt] (1,0) -- (1-1,0-1);

	\draw[line width=1pt] (2,0) -- (2+1,0);
	\draw[line width=1pt] (2,0) -- (2,0+4);
	\draw[line width=1pt] (2,0) -- (2-1,0-1);

	\draw[line width=1pt] (0,3) -- (0+3,3);
	\draw[line width=1pt] (0,3) -- (0,3+1);
	\draw[line width=1pt] (0,3) -- (0-1,3-1);

	\draw[red,line width=1pt] (0,2) -- (0+1,2);
	\draw[red,line width=1pt] (0,2) -- (0,2+1);
	\draw[red,line width=1pt] (0,2) -- (0-1,2-1);

	\draw[line width=1pt] (1,2) -- (1+2,2);
	\draw[line width=1pt] (1,2) -- (1,2+2);
	\draw[line width=1pt] (1,2) -- (1-2,2-2);

	\fill[red] (0,0) circle (5pt);
	\fill[red] (1,0) circle (5pt);
	\fill (2,0) circle (5pt);
	\fill[red] (0,2) circle (5pt);
	\fill (0,3) circle (5pt);
	\fill (1,2) circle (5pt);

\end{tikzpicture}
&\vspace{0.1in}
\begin{tikzpicture}[scale=0.4]

	\draw[red,line width=1pt] (0,0) -- (0+2,0);
	\draw[red,line width=1pt] (0,0) -- (0,0+1);
	\draw[red,line width=1pt] (0,0) -- (0-1,0-1);

	\draw[red,line width=1pt] (2,0) -- (2+1,0);
	\draw[red,line width=1pt] (2,0) -- (2,0+1);
	\draw[red,line width=1pt] (2,0) -- (2-1,0-1);

	\draw[line width=1pt] (3,0) -- (3+1,0);
	\draw[line width=1pt] (3,0) -- (3,0+3);
	\draw[line width=1pt] (3,0) -- (3-1,0-1);

	\draw[line width=1pt] (0,2) -- (0+4,2);
	\draw[line width=1pt] (0,2) -- (0,2+1);
	\draw[line width=1pt] (0,2) -- (0-1,2-1);

	\draw[red,line width=1pt] (0,1) -- (0+2,1);
	\draw[red,line width=1pt] (0,1) -- (0,1+1);
	\draw[red,line width=1pt] (0,1) -- (0-1,1-1);

	\draw[line width=1pt] (2,1) -- (2+2,1);
	\draw[line width=1pt] (2,1) -- (2,1+2);
	\draw[line width=1pt] (2,1) -- (2-2,1-2);

	\fill[red] (0,0) circle (5pt);
	\fill[red] (2,0) circle (5pt);
	\fill (3,0) circle (5pt);
	\fill[red] (0,1) circle (5pt);
	\fill (0,2) circle (5pt);
	\fill (2,1) circle (5pt);

\end{tikzpicture}
\\
\hline
\begin{tikzpicture}[scale=0.4]

	\draw[line width=1pt] (0,3) -- (0-1,3);
	\draw[line width=1pt] (0,3) -- (0,3-4);
	\draw[line width=1pt] (0,3) -- (0+1,3+1);

	\draw[line width=1pt] (2,2) -- (2-3,2);
	\draw[line width=1pt] (2,2) -- (2,2-3);
	\draw[line width=1pt] (2,2) -- (2+1,2+1);

	\draw[line width=1pt] (1,0) -- (1-2,0);
	\draw[line width=1pt] (1,0) -- (1,0-1);
	\draw[line width=1pt] (1,0) -- (1+2,0+2);

	\fill (0,3) circle (5pt);
	\fill (2,2) circle (5pt);
	\fill (1,0) circle (5pt);

\end{tikzpicture}
&\vspace{0.1in}
\begin{tikzpicture}[scale=0.4]

	\draw[red,line width=1pt] (0,0) -- (0+1,0);
	\draw[red,line width=1pt] (0,0) -- (0,0+3);
	\draw[red,line width=1pt] (0,0) -- (0-1,0-1);

	\draw[line width=1pt] (1,0) -- (1+2,0);
	\draw[line width=1pt] (1,0) -- (1,0+5);
	\draw[line width=1pt] (1,0) -- (1-1,0-1);

	\draw[line width=1pt] (0,4) -- (0+3,4);
	\draw[line width=1pt] (0,4) -- (0,4+1);
	\draw[line width=1pt] (0,4) -- (0-1,4-1);

	\draw[red,line width=1pt] (0,3) -- (0+2,3);
	\draw[red,line width=1pt] (0,3) -- (0,3+1);
	\draw[red,line width=1pt] (0,3) -- (0-1,3-1);

	\draw[red,line width=1pt] (2,1) -- (2+1,1);
	\draw[red,line width=1pt] (2,1) -- (2,1+2);
	\draw[red,line width=1pt] (2,1) -- (2-1,1-1);

	\draw[line width=1pt] (2,3) -- (2+1,3);
	\draw[line width=1pt] (2,3) -- (2,3+2);
	\draw[line width=1pt] (2,3) -- (2-3,3-3);

	\fill[red] (0,0) circle (5pt);
	\fill[red] (0,3) circle (5pt);
	\fill (0,4) circle (5pt);
	\fill (1,0) circle (5pt);
	\fill[red] (2,1) circle (5pt);
	\fill (2,3) circle (5pt);

\end{tikzpicture}
&\vspace{0.1in}
\begin{tikzpicture}[scale=0.4]

	\draw[red,line width=1pt] (0,0) -- (0+2,0);
	\draw[red,line width=1pt] (0,0) -- (0,0+2);
	\draw[red,line width=1pt] (0,0) -- (0-1,0-1);

	\draw[red,line width=1pt] (0,2) -- (0+3,2);
	\draw[red,line width=1pt] (0,2) -- (0,2+1);
	\draw[red,line width=1pt] (0,2) -- (0-1,2-1);

	\draw[red,line width=1pt] (3,1) -- (3+1,1);
	\draw[red,line width=1pt] (3,1) -- (3,1+1);
	\draw[red,line width=1pt] (3,1) -- (3-1,1-1);

	\draw[line width=1pt] (2,0) -- (2+2,0);
	\draw[line width=1pt] (2,0) -- (2,0+4);
	\draw[line width=1pt] (2,0) -- (2-1,0-1);

	\draw[line width=1pt] (0,3) -- (0+4,3);
	\draw[line width=1pt] (0,3) -- (0,3+1);
	\draw[line width=1pt] (0,3) -- (0-1,3-1);

	\draw[line width=1pt] (3,2) -- (3+1,2);
	\draw[line width=1pt] (3,2) -- (3,2+2);
	\draw[line width=1pt] (3,2) -- (3-3,2-3);

	\fill[red] (0,0) circle (5pt);
	\fill[red] (0,2) circle (5pt);
	\fill (0,3) circle (5pt);
	\fill (2,0) circle (5pt);
	\fill[red] (3,1) circle (5pt);
	\fill (3,2) circle (5pt);

\end{tikzpicture}
\\
\hline
\end{tabular}
\caption{On the left, subdivisions of $\mathbb{R}^2$ giving cones in $\mathcal{Q}_3^\vee$ which correspond to points in $\mathbf{V}$. These cones are refined by $\widetilde{\mathcal{Q}}_3$ on the right.}
\label{conesinrefinedquotientfancorrespondingtosplittingbipyramid}
\end{table}
Next we study $f$ over the non-planar locus of $\overline{\mathbf{X}}_{\GP}(3,6)^\nu$. 
Let $x\in\overline{\mathbf{X}}_{\GP}(3,6)\setminus\mathbf{U}_{\GP}(3,6)$ and let  
$y\in\overline{\mathbf{X}}(3,6)$ be the image of $x$. Then $y$ is in the interior of the strata $\#7$, $\#2$, $\#10$, or $\#23$ -- these are the only degenerations in  Figures~\ref{Alexeevdegenerations6linesC}, \ref{Alexeevdegenerations6linesA}, and \ref{Alexeevdegenerations6linesB} which do not have a triple of lines in general linear position in each primary component.
Each stratum $\#7$ is an isolated point of the non-planar locus.
By the description of the fan $\widehat{\Sigma}(3,6)$,
the morphism $\widehat{\mathbf{X}}(3,6)\rightarrow\overline{\mathbf{X}}(3,6)$ is a local isomorphism over this stratum.
By~normality, it suffices to show that $\overline{\mathbf{X}}_{\GP}(3,6)$
contains a unique point over each stratum~$\#7$, which is proved in
Lemma~\ref{non-planardegenerationgp7}.

Next we study $f$ over the preimage of the closed strata $\#2=(ijk),(lmn)$ in $\overline{\mathbf{X}}(3,6)$,
which contains the strata $\#10$ and $\#23$ in its interior.
By \cite[Theorem~4.2.9]{Lux08}, the boundary strata $(ijk)$ are isomorphic to $\overline{\mathrm{M}}_{0,6}$
and each stratum $\#2=(ijk),(lmn)$ corresponds to a boundary stratum of $\overline{\mathrm{M}}_{0,6}$ isomorphic to $\mathbb{P}^1\times\mathbb{P}^1$, which we denote by $S$. The boundary divisors of type $(ij),(ik),(jk)$ restrict to $S$ giving three distinct parallel rulings, which are strata $\#10$.
The pairwise intersections of these rulings with the analogous 
restrictions of the boundary divisors of type $(\ell m),(\ell n),(mn)$ give nine $0$-dimensional boundary strata $\#23$. It follows from this analysis and the structure of the refinement $\widehat{\Sigma}(3,6)$ in Table~\ref{subdivisionofthebipyramid} that the preimage $\widehat{S}$ of $S$ under the blow up $\widehat{\mathbf{X}}(3,6)\rightarrow\overline{\mathbf{X}}(3,6)$ is $\Bl_9S$, the blow up of $S$ at the nine $0$-dimensional boundary strata $\#23$.
Let $S_{\GP}\subseteq\overline{\mathbf{X}}_{\GP}(3,6)$ be the preimage of~$S$.
By Lemma~\ref{type2stratainGPareblowupp1xp1at9pts},
the morphism $S_{\GP}\to S$ induces a bijective  normalization morphism $\widehat{S}\cong\Bl_9 S\to S_{\GP}$.
Take $x\in S_{\GP}$ 
and take any $\mathbf{a}\colon\Spec(K)\rightarrow\mathbf{B}_6$ with limit $x$. Denote by $z\in\widehat{S}\subseteq\widehat{\mathbf{X}}(3,6)$ the corresponding limit of $f^{-1}\circ\mathbf{a}$. 
In the proof of Lemma~\ref{type2stratainGPareblowupp1xp1at9pts} we showed 
that $x$ only depends on $z$ and not on the choice of $\mathbf{a}$.
This shows that $f$ is regular at $z$ by \cite[Theorem~7.3]{GG14}.
Since $f$ is finite and birational, $\overline{\mathbf{X}}_{\GP}(3,6)^\nu\cong\widehat{\mathbf{X}}(3,6)$.
\end{proof}



\begin{lemma}\label{non-planardegenerationgp7}
$\overline{\mathbf{X}}_{\GP}(3,6)$
contains a unique point over each stratum $\#7$
in $\overline{\mathbf{X}}(3,6)$.
\end{lemma}

\begin{proof}
Consider the following one-parameter family $\mathbf{a}(t)$ of six points in $\mathbb{P}^2$:
\begin{gather*}
a_1(t)=[1:0:0],~a_2(t)=[0:1:0],~a_3(t)=[0:0:1],~a_4(t)=[1:1:1],\\
a_5(t)=[c_1t:1:1+c_2t],~a_6(t)=[1+c_3t:1:c_4t],
\end{gather*}
where  $c_1,\ldots,c_4\in R\setminus\{0\}$. 
Let $\lambda_1,\ldots,\lambda_4$ be the valuations of $c_1,\ldots,c_4$, respectively.
The limit points for $t=0$ are the double points of the line arrangement in Figure~\ref{degeneration6pointstype7} given by:
\[
\ell_{126}=\{x_3=0\},~\ell_{145}=\{x_2=x_3\},~\ell_{235}=\{x_1=0\},~\ell_{346}=\{x_1=x_2\}.
\]

\begin{figure}[hbtp]
\begin{tikzpicture}[scale=0.7]

	\draw[line width=1pt] (-2,-1) -- (6,3);
	\draw[line width=1pt] (0,3) -- (8,-1);
	\draw[line width=1pt] (-1,-1) -- (4,4);
	\draw[line width=1pt] (2,4) -- (7,-1);

	\fill (0,0) circle (5pt);
	\fill (2,2) circle (5pt);
	\fill (4,2) circle (5pt);
	\fill (3,3) circle (5pt);
	\fill (3,1.5) circle (5pt);
	\fill (6,0) circle (5pt);

	\node at (1,-0.5) {$a_1(0)$};
	\node at (3,4) {$a_2(0)$};
	\node at (5,-0.5) {$a_3(0)$};
	\node at (3,0.8) {$a_4(0)$};
	\node at (5.5,2) {$a_5(0)$};
	\node at (0.5,2) {$a_6(0)$};

	\node at (-0.5,-1.5) {$\ell_{126}$};
	\node at (-2.5,-1.2) {$\ell_{145}$};
	\node at (7.5,-1.5) {$\ell_{235}$};
	\node at (9,-1.2) {$\ell_{346}$};

\end{tikzpicture}
\caption{Degeneration of six points in $\mathbb{P}^2$.}
\label{degeneration6pointstype7}
\end{figure}

The stable lattices and the corresponding stabilized quadruples are the following:
\begin{itemize}

\item $L_0=e_1R+e_2R+e_3R$ stabilizes $(1,2,3,4),(1,3,5,6),(2,4,5,6)$; 

\item $L_1=t^{\lambda_1+1}e_1R+e_2R+e_3R$ stabilizes $(1,2,3,5),(2,3,4,5),(2,3,5,6)$;

\item $L_2=e_1R+e_2R+t^{\lambda_4+1}e_3R$ stabilizes $(1,2,3,6),(1,2,4,6),(1,2,5,6)$;

\item $L_3=e_1R+t^{\lambda_2+1}e_2R+(e_2+e_3)R$ stabilizes $(1,2,4,5),(1,3,4,5),(1,4,5,6)$;

\item $L_4=t^{\lambda_3+1}e_1R+(e_1+e_2)R+e_3R$ stabilizes $(1,3,4,6),(2,3,4,6),(3,4,5,6)$. 
\end{itemize}

We now compute the $6$-pointed degenerate central fiber $\mathbb{P}(\Sigma_\mathbf{a})_\Bbbk$. Let us start by observing that the induced linear spaces on $\mathbb{P}(L_0)_\Bbbk$ are given by
\[
W_{L_0}(L_1)=\{x_1=0\},~W_{L_0}(L_2)=\{x_3=0\},~W_{L_0}(L_3)=\{x_2=x_3\},~W_{L_0}(L_4)=\{x_1=x_2\}.
\]
In particular the primary component in $\mathbb{P}(\Sigma_\mathbf{a})_\Bbbk$ corresponding to $L_0$ is isomorphic to $\mathbb{P}^2$. On the other hand, let $i,j$ be distinct nonzero indices. Then $W_{L_i}(L_0)$ is a point by Remark~\ref{allabouttwolattices} (note that $L_i\subseteq L_0$ and $W_{L_0}(L_i)$ is a line). Moreover $W_{L_i}(L_j)$ is a line again by Remark~\ref{allabouttwolattices} ($L_i\nsubseteq L_j$ and $L_j\nsubseteq L_i$). So the primary component $S_i$ in $\mathbb{P}(\Sigma_\mathbf{a})_\Bbbk$ corresponding to $L_i$ is isomorphic to $\mathbb{F}_1$. The secondary components of $\mathbb{P}(\Sigma_\mathbf{a})_\Bbbk$ can be understood as follows. If $i,j$ are distinct nonzero indices, then the central fiber $\mathbb{P}(L_0,L_i,L_j)_\Bbbk$ is as shown in Figure~\ref{degeneration5points2}, and there is a unique secondary component $T_{ij}$ isomorphic to $\mathbb{P}^1\times\mathbb{P}^1$. Note that the center of the blow up $\mathbb{P}(\Sigma_\mathbf{a})\rightarrow\mathbb{P}(L_0,L_i,L_j)$ does not intersect $T_{ij}$. In conclusion, $\mathbb{P}(\Sigma_\mathbf{a})_\Bbbk$ has six secondary components $T_{ij}$ isomorphic to $\mathbb{P}^1\times\mathbb{P}^1$, each one glued along two of its incident rulings to a fiber in each $S_i,S_j$. The components $S_1,\ldots,S_4$ are then glued along their exceptional curves to $\mathbb{P}^2$ along the lines $W_{L_0}(L_1),\ldots,W_{L_0}(L_4)$ respectively. The resulting surface $\mathbb{P}(\Sigma_\mathbf{a})_\Bbbk$ is pictured in Figure~\ref{degeneration6pointstype7GPlimit}.  
It is independent from $c_1,\ldots,c_4\in R\setminus\{0\}$.

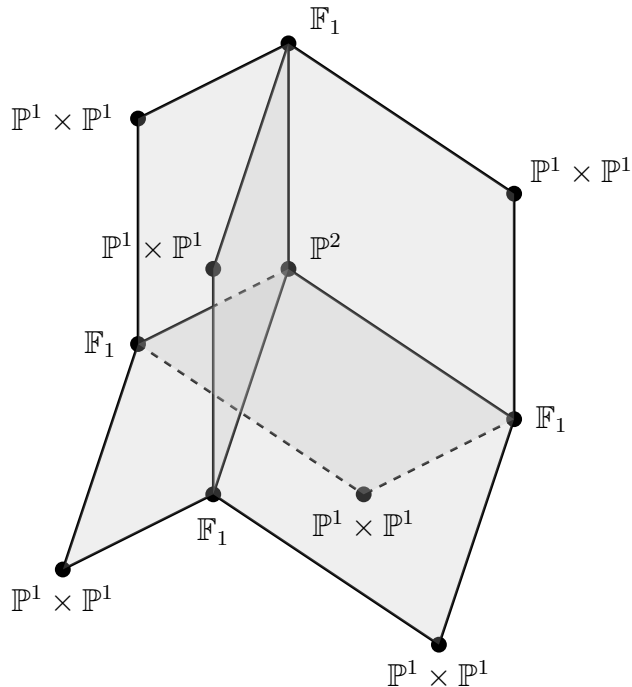
\begin{figure}[hbtp]
\begin{tikzpicture}[scale=1.0]

	\draw[line width=1pt] (0,0) -- (0,3);
	\draw[dashed,line width=1pt] (0,0) -- (-1,-0.5);
	\draw[line width=1pt] (-1,-0.5) -- (-2,-1);
	\draw[line width=1pt] (0,0) -- (3,-2);
	\draw[line width=1pt] (0,0) -- (-1,-3);
	\draw[line width=1pt] (0,3) -- (3,1);
	\draw[line width=1pt] (0,3) -- (-2,2);
	\draw[line width=1pt] (3,1) -- (3,-2);
	\draw[line width=1pt] (-2,2) -- (-2,-1);
	\draw[line width=1pt] (-2,-1) -- (-3,-4);
	\draw[line width=1pt] (-3,-4) -- (-1,-3);
	\draw[line width=1pt] (3,-2) -- (2,-5);
	\draw[line width=1pt] (-1,-3) -- (2,-5);
	\draw[dashed,line width=1pt] (-2,-1) -- (1,-3);
	\draw[dashed,line width=1pt] (3,-2) -- (1,-3);
	\draw[line width=1pt] (0,3) -- (-1,0);
	\draw[line width=1pt] (-1,0) -- (-1,-3);

	\fill (0,0) circle (3pt);

	\fill (0,3) circle (3pt);
	\fill (-2,-1) circle (3pt);
	\fill (3,-2) circle (3pt);
	\fill (-1,-3) circle (3pt);

	\fill (3,1) circle (3pt);
	\fill (-2,2) circle (3pt);
	\fill (-3,-4) circle (3pt);
	\fill (2,-5) circle (3pt);
	\fill (1,-3) circle (3pt);
	\fill (-1,0) circle (3pt);

	\fill[gray!50,nearly transparent] (0,0) -- (0,3) -- (3,1) -- (3,-2) -- cycle;
	\fill[gray!50,nearly transparent] (0,0) -- (0,3) -- (-2,2) -- (-2,-1) -- cycle;
	\fill[gray!50,nearly transparent] (0,0) -- (-2,-1) -- (-3,-4) -- (-1,-3) -- cycle;
	\fill[gray!50,nearly transparent] (0,0) -- (3,-2) -- (2,-5) -- (-1,-3) -- cycle;
	\fill[gray!50,nearly transparent] (0,0) -- (0,3) -- (-1,0) -- (-1,-3) -- cycle;
	\fill[gray!50,nearly transparent] (0,0) -- (3,-2) -- (1,-3) -- (-2,-1) -- cycle;

	\node at (0.5,0.3) {$\mathbb{P}^2$};
	\node at (0.5,3.3) {$\mathbb{F}_1$};
	\node at (3.5,-2) {$\mathbb{F}_1$};
	\node at (-1,-3.5) {$\mathbb{F}_1$};
	\node at (-2.5,-1) {$\mathbb{F}_1$};

	\node at (3.9,1.3) {$\mathbb{P}^1\times\mathbb{P}^1$};
	\node at (-3,2) {$\mathbb{P}^1\times\mathbb{P}^1$};
	\node at (2,-5.4) {$\mathbb{P}^1\times\mathbb{P}^1$};
	\node at (-3,-4.4) {$\mathbb{P}^1\times\mathbb{P}^1$};
	\node at (1,-3.4) {$\mathbb{P}^1\times\mathbb{P}^1$};
	\node at (-1.8,0.3) {$\mathbb{P}^1\times\mathbb{P}^1$};

\end{tikzpicture}
\caption{Dual complex of the degeneration discussed in Lemma~\ref{non-planardegenerationgp7}.}
\label{degeneration6pointstype7GPlimit}
\end{figure}

Let us determine the limits $\overline{a}_1(0),\ldots,\overline{a}_6(0)\in\mathbb{P}(\Sigma_\mathbf{a})_\Bbbk$. Under the birational morphism $\mathbb{P}(\Sigma_\mathbf{a})\rightarrow\mathbb{P}(L_0)$, the irreducible components of $\mathbb{P}(\Sigma_\mathbf{a})_\Bbbk$ mapping to $\ell_{126}$ (resp. $\ell_{145}$) are $S_2,T_{12},T_{23},T_{24}$ (resp. $S_3,T_{13},T_{23},T_{34}$). This implies that the limit point $\overline{a}_1^{L_0}(0)=\ell_{126}\cap\ell_{145}$ is on $T_{23}$. By the same reasoning, we have that
\[
\overline{a}_1(0)\in T_{23},~\overline{a}_2(0)\in T_{12},~\overline{a}_3(0)\in T_{14},~\overline{a}_4(0)\in T_{34},~\overline{a}_5(0)\in T_{13},~\overline{a}_6(0)\in T_{24}.
\]

Finally, we count the moduli associated to the choice of the six markings on the $T_{ij}$ components of $\mathbb{P}(\Sigma_\mathbf{a})_\Bbbk$. We show that we can find an appropriate automorphism of the surface $\mathbb{P}(\Sigma_\mathbf{a})_\Bbbk$ that fixes the six markings. Consider $S_1$, which is glued to $T_{12},T_{13},T_{14}$ along three rulings $f_{12},f_{13},f_{14}$ respectively. The automorphism group of $S_1$ that preserves the $(-1)$-section and the rulings $f_{12},f_{13},f_{14}$ is three dimensional, and we can choose one of such automorphisms $\varphi_1$ that fixes the images on $S_1$ of the limits $\overline{a}_2(0)\in T_{12},\overline{a}_3(0)\in T_{14},\overline{a}_5(0)\in T_{13}$. Consider the analogous automorphisms $\varphi_2,\varphi_3,\varphi_4$ for $S_2,S_3,S_4$ respectively. We can choose an appropriate $\psi_{ij}\in\Aut(T_{ij})$ acting on the two rulings in such a way that $\psi_{ij},\varphi_k,\id_{\mathbb{P}(L_0)_\Bbbk}$ glue to a global automorphism of $\mathbb{P}(\Sigma_\mathbf{a})_\Bbbk$ which fixes the six markings.
\end{proof}

Next we consider the non-planar locus in $\overline{\mathbf{X}}_{\GP}(3,6)^\nu$ dominating the closed stratum $\#2$ in $\overline{\mathbf{X}}(3,6)$. First we classify degenerations parametrized by it.

\begin{lemma}
The $6$-pointed degenerations of $\mathbb{P}^2$ parametrized by points 
in $\overline{\mathbf{X}}_{\GP}(3,6)$
over the (closed) stratum of type $\#2$ in $\overline{\mathbf{X}}(3,6)$ are depicted in figures $\#2,10,23$, $23.1$ in Table~\ref{tbl:nonplanarGPdegenerations}.
\end{lemma}

\begin{table}[hbtp]
\centering
\begin{tabular}{|>{\centering\arraybackslash}m{2.5cm}|>{\centering\arraybackslash}m{2.5cm}|>{\centering\arraybackslash}m{2.5cm}|>{\centering\arraybackslash}m{2.5cm}|}
\hline
\vspace{.1in}
\begin{tikzpicture}[scale=0.5]

	\fill (0,0) circle (5pt);
	\fill (0,1) circle (5pt);
	\fill (1,0) circle (5pt);

	\draw[line width=1pt] (0,1) -- (0-1,1);
	\draw[line width=1pt] (0,1) -- (0,1-2);
	\draw[line width=1pt] (0,1) -- (0+1,1+1);

	\draw[line width=1pt] (1,0) -- (1-2,0);
	\draw[line width=1pt] (1,0) -- (1,0-1);
	\draw[line width=1pt] (1,0) -- (1+1,0+1);

	\draw[line width=1pt] (0,0) -- (0+2,0+2);

\end{tikzpicture}
&\vspace{.1in}
\begin{tikzpicture}[scale=0.5]

	\fill (0,0) circle (5pt);
	\fill (0,1) circle (5pt);
	\fill (1,0) circle (5pt);

	\draw[line width=1pt] (0,0) -- (0-1,0-1);
	\draw[line width=1pt] (0,0) -- (0+2,0);
	\draw[line width=1pt] (0,0) -- (0,0+2);

	\draw[line width=1pt] (0,1) -- (0+2,1);
	\draw[line width=1pt] (0,1) -- (0-1,1-1);

	\draw[line width=1pt] (1,0) -- (1,0+2);
	\draw[line width=1pt] (1,0) -- (1-1,0-1);

\end{tikzpicture}
&\vspace{.1in}
\begin{tikzpicture}[scale=0.5]

	\fill (0,0) circle (5pt);
	\fill (1,1) circle (5pt);
	\fill (3,2) circle (5pt);
	\fill[red] (2,2) circle (5pt);

	\draw[line width=1pt] (0,0) -- (0-1,0);
	\draw[line width=1pt] (0,0) -- (0,0-1);
	\draw[line width=1pt] (0,0) -- (0+3,0+3);

	\draw[line width=1pt] (1,1) -- (1-2,1);
	\draw[line width=1pt] (1,1) -- (1,1-2);

	\draw[line width=1pt] (3,2) -- (3-3,2);
	\draw[line width=1pt] (3,2) -- (3,2-3);
	\draw[line width=1pt] (3,2) -- (3+1,2+1);

\end{tikzpicture}
&\vspace{.1in}
\begin{tikzpicture}[scale=0.5]

	\fill (0,0) circle (5pt);
	\fill (1,1) circle (5pt);
	\fill (3,2) circle (5pt);
	\fill[red] (2,2) circle (5pt);

	\draw[line width=1pt] (0,0) -- (0+4,0);
	\draw[line width=1pt] (0,0) -- (0,0+3);

	\draw[line width=1pt] (1,1) -- (1,1+2);
	\draw[line width=1pt] (1,1) -- (1+3,1);
	\draw[line width=1pt] (1,1) -- (1-2,1-2);

	\draw[line width=1pt] (3,2) -- (3+1,2);
	\draw[line width=1pt] (3,2) -- (3,2+1);
	\draw[line width=1pt] (3,2) -- (3-3,2-3);

	\draw[red,line width=1pt] (2,2) -- (2+1,2);
	\draw[red,line width=1pt] (2,2) -- (2,2+1);
	\draw[red,line width=1pt] (2,2) -- (2-1,2-1);

\end{tikzpicture}
\\
\hline
\multicolumn{2}{|c|}{{\small $\#2$}} & \multicolumn{2}{c|}{{\small $\#10$}}
\\
\hline
\vspace{.1in}
\begin{tikzpicture}[scale=0.5]

	\fill (0,0) circle (5pt);
	\fill (1,2) circle (5pt);
	\fill (2,4) circle (5pt);
	\fill[red] (1,1) circle (5pt);
	\fill[red] (2,2) circle (5pt);
	\fill[red] (2,3) circle (5pt);

	\draw[line width=1pt] (0,0) -- (0-1,0);
	\draw[line width=1pt] (0,0) -- (0,0-1);
	\draw[line width=1pt] (0,0) -- (0+3,0+3);

	\draw[line width=1pt] (1,2) -- (1-2,2);
	\draw[line width=1pt] (1,2) -- (1,2-3);
	\draw[line width=1pt] (1,2) -- (1+2,2+2);

	\draw[line width=1pt] (2,4) -- (2-3,4);
	\draw[line width=1pt] (2,4) -- (2,4-5);
	\draw[line width=1pt] (2,4) -- (2+1,4+1);

\end{tikzpicture}
&\vspace{.1in}
\begin{tikzpicture}[scale=0.5]

	\fill (0,0) circle (5pt);
	\fill (1,2) circle (5pt);
	\fill (2,4) circle (5pt);
	\fill[red] (1,1) circle (5pt);
	\fill[red] (2,2) circle (5pt);
	\fill[red] (2,3) circle (5pt);

	\draw[line width=1pt] (0,0) -- (0-1,0-1);
	\draw[line width=1pt] (0,0) -- (0+3,0);
	\draw[line width=1pt] (0,0) -- (0,0+5);

	\draw[line width=1pt] (1,2) -- (1-2,2-2);
	\draw[line width=1pt] (1,2) -- (1+2,2);
	\draw[line width=1pt] (1,2) -- (1,2+3);

	\draw[line width=1pt] (2,4) -- (2-3,4-3);
	\draw[line width=1pt] (2,4) -- (2+1,4);
	\draw[line width=1pt] (2,4) -- (2,4+1);

	\draw[red,line width=1pt] (1,1) -- (1+2,1);
	\draw[red,line width=1pt] (1,1) -- (1,1+1);

	\draw[red,line width=1pt] (2,2) -- (2-2,2-2);
	\draw[red,line width=1pt] (2,2) -- (2,2+2);

	\draw[red,line width=1pt] (2,3) -- (2-1,3-1);
	\draw[red,line width=1pt] (2,3) -- (2+1,3);

\end{tikzpicture}
&\vspace{.1in}
\begin{tikzpicture}[scale=0.4]

	\fill (0,0) circle (5pt);
	\fill (1,2) circle (5pt);
	\fill (3,5) circle (5pt);
	\fill[red] (1,1) circle (5pt);
	\fill[red] (3,3) circle (5pt);
	\fill[red] (3,4) circle (5pt);

	\draw[line width=1pt] (0,0) -- (0-1,0);
	\draw[line width=1pt] (0,0) -- (0,0-1);
	\draw[line width=1pt] (0,0) -- (0+4,0+4);

	\draw[line width=1pt] (1,2) -- (1-2,2);
	\draw[line width=1pt] (1,2) -- (1,2-3);
	\draw[line width=1pt] (1,2) -- (1+3,2+3);

	\draw[line width=1pt] (3,5) -- (3-4,5);
	\draw[line width=1pt] (3,5) -- (3,5-6);
	\draw[line width=1pt] (3,5) -- (3+1,5+1);

\end{tikzpicture}
&\vspace{.1in}
\begin{tikzpicture}[scale=0.4]

	\fill (0,0) circle (5pt);
	\fill (1,2) circle (5pt);
	\fill (3,6) circle (5pt);
	\fill[red] (1,1) circle (5pt);
	\fill[red] (3,3) circle (5pt);
	\fill[red] (3,4) circle (5pt);

	\draw[line width=1pt] (0,0) -- (0-1,0-1);
	\draw[line width=1pt] (0,0) -- (0+4,0);
	\draw[line width=1pt] (0,0) -- (0,0+7);

	\draw[line width=1pt] (1,2) -- (1-2,2-2);
	\draw[line width=1pt] (1,2) -- (1+3,2);
	\draw[line width=1pt] (1,2) -- (1,2+5);

	\draw[line width=1pt] (3,6) -- (3-4,6-4);
	\draw[line width=1pt] (3,6) -- (3+1,6);
	\draw[line width=1pt] (3,6) -- (3,6+1);

	\draw[red,line width=1pt] (1,1) -- (1+3,1);
	\draw[red,line width=1pt] (1,1) -- (1,1+1);

	\draw[red,line width=1pt] (3,3) -- (3-3,3-3);
	\draw[red,line width=1pt] (3,3) -- (3,3+3);
	\draw[red,line width=1pt] (3,3) -- (3+1,3);

	\draw[red,line width=1pt] (3,4) -- (3-2,4-2);
	\draw[red,line width=1pt] (3,4) -- (3+1,4);

\end{tikzpicture}
\\
\hline
\multicolumn{2}{|c|}{{\small $\#23$}} & \multicolumn{2}{c|}{{\small $\#23.1$}}
\\
\hline
\end{tabular}
\caption{Dual spiders for the degenerations $\#2,10,23$ of six lines in $(\mathbb{P}^2)^\vee$, and the corresponding spiders for the degenerations of six points $\mathbb{P}^2$.}
\label{stablelatticesinapartmentfortype2102323.1}
\end{table}

\begin{proof}
Let $x\in\overline{\mathbf{X}}_{\GP}(3,6)$ be a point over a point 
$y\in\overline{\mathbf{X}}(3,6)$
in the (closed) stratum $\#2$.
From the possible Kapranov's degenerations $\#2,\#10,\#23$ parametrized by $y$, which can be found in \cite[Figures~5.12 and 5.13]{Ale15}, we 
conclude that all stable lattices
are contained in one apartment.  For clarity, let us point out that this apartment
is not determined by a triple of sections as for degenerations parametrized by the planar locus.
Instead, the point $y$ is in the image of the planar locus in $\overline{\mathbf{X}}(3,7)$ with respect to the forgetful map.
By considering combinatorial types of the dual spiders giving the degenerations of $(\mathbb{P}^2)^\vee$ parametrized by~$y$ and the corresponding usual spiders, we compute the resulting degenerations of $\mathbb{P}^2$,
taking into account that stretching the spiders may yield different degenerations. This is summarized in Table~\ref{stablelatticesinapartmentfortype2102323.1}. 
For each pair, on the left we give dual spiders in one apartment for the degenerations $\#2,10,23$ of six lines in $(\mathbb{P}^2)^\vee$, and on the right the corresponding 
spiders for the degenerations of six points in $\mathbb{P}^2$.
Finally, we add to these degenerations of $\mathbb{P}^2$ six marked points as prescribed by Lemma~\ref{npointedcentralfibermustafinjoinnocalculation}. 
\end{proof}


\begin{lemma}
\label{type2stratainGPareblowupp1xp1at9pts}
Let $S\subseteq\overline{\mathbf{X}}(3,6)$ be a (closed) stratum of type  $\#2$ and let $S_{\GP}\subseteq\overline{\mathbf{X}}_{\GP}(3,6)$ be its preimage. The birational morphism 
$S_{\GP}\to S$ induces a bijective  normalization morphism $\Bl_9 S\to S_{\GP}$,
where $\Bl_9 S$ is the blow up of $S$ at the nine double points (strata $\#23$) of the arrangement of six rulings (strata $\#10$). 
\end{lemma}

\begin{proof}
Consider a general one-parameter family of 
six lines
\[
\mathbf{a}(t)=([1:0:0],[0:1:0],[0:0:1],[1:1:1],[1:c_1:c_2t],[1:1+c_3t:c_4])
\]
in $(\mathbb{P}^2)^\vee$ limiting to a point $y\in\overline{\mathbf{X}}(3,6)$
in the open stratum $\#2$. Here $c_1,\ldots,c_4\in R\setminus\{0\}$ and $c_1(0),c_4(0)\neq0,1$.
It can also be viewed as a general one-parameter family of
six points in~$\mathbb{P}^2$ with limit $x\in\overline{\mathbf{X}}_{\GP}(3,6)^\nu$.
We want to show that $x$ only depends on $y$, and not on the choice of the arc $\mathbf{a}$. 
The degeneration of $\mathbb{P}^2$ parametrized by $x$ is pictured in Table~\ref{tbl:nonplanarGPdegenerations}, $\#2$, and it only depends on the two cross-ratios we are about to describe. The limits in $\mathbb{P}^2$ of the six points in $\mathbf{a}$ lie on the two lines $x_3=0$ and $x_1-x_2=0$, which intersect at $[1:1:0]$. So consider the quadruple $[1:0:0]$, $[0:1:0]$, $[1:c_1(0):0]$, $[1:1:0]$ on the first line, and $[1:1:1]$, $[0:0:1]$, $[1:1:c_4(0)]$, $[1:1:0]$ on the second line. Their cross-ratios are $c_1(0)$ and $c_4(0)$ respectively. On the other hand, let $\ell_1,\ldots,\ell_6$ be the lines in $(\mathbb{P}^2)^\vee$ corresponding to the six points in $\mathbf{a}$. In the limit, $\ell_1,\ell_2,\ell_5$ (resp. $\ell_3,\ell_4,\ell_6$) pass through $[0:0:1]$ (resp. $[1:-1:0]$). Let $\ell$ be the line passing through $[0:0:1]$ and $[1:-1:0]$. The stable pair parametrized by $y\in\overline{\mathbf{X}}(3,6)$ is pictured in Figure~\ref{Alexeevdegenerations6linesA}, $\#2$ (note that the labels used for the lines in this proof do not match the ones in the figure), and the limits of $\ell_1,\ell_2,\ell_5,\ell$ and $\ell_3,\ell_4,\ell_6,\ell$ determine four points on each one of the two irreducible $\mathbb{P}^1$ in the double locus. The degeneration parametrized by $y\in\overline{\mathbf{X}}(3,6)$ only depends on the cross-ratios of these two quadruples, which are equal to $c_1(0)$ and $c_4(0)$ respectively. In conclusion, $x$ only depends on $y$, and not on the tangent direction of the arc limiting to $y$.

We apply the same strategy for a point $x\in\overline{\mathbf{X}}_{\GP}(3,6)^\nu$ 
in the preimage of  the interior of the stratum $\#10$. The general one-parameter family of six points in $\mathbb{P}^2$ with limit $x$
is
\[
\mathbf{a}(t)=([1:0:0],[0:1:0],[0:0:1],[1:1:1],[1:c_1:c_2t],[1:1+c_3t:c_4t]),
\]
where $c_1,\ldots,c_4\in R\setminus\{0\}$ and $c_1(0)\neq0,1$. Viewing this as a family of lines we obtain a limit $y\in\overline{\mathbf{X}}(3,6)$. The degeneration parametrized by $x$ depends on one cross-ratio, which is determined by $[1:0:0]$, $[0:1:0]$, $[1:c_1(0):0]$, $[1:1:0]$ on the line $x_3=0$ and equals~$c_1(0)$. Consider the corresponding degeneration of six lines parametrized by $y\in\overline{\mathbf{X}}(3,6)$, which is pictured in Figure~\ref{Alexeevdegenerations6linesA}, $\#10$. This degeneration is determined by the cross-ratio of four points on the gluing locus between $\mathbb{P}^1\times\mathbb{P}^1$ and $\mathbb{F}_1$. These four points are cut out by another curve in the double locus and the limits of $\ell_1,\ell_2,\ell_5$. Their cross-ratio is equal to~$c_1(0)$.

The general one-parameter family limiting to $y\in\overline{\mathbf{X}}(3,6)$ in the (closed) stratum $\#23$ is given by
\[
\mathbf{a}(t)=([1:0:0],[0:1:0],[0:0:1],[1:1:1],[1:1+c_1t:c_2t],[1:1+c_3t:c_4t]),
\]
where $c_1,\ldots,c_4\in R\setminus\{0\}$. Here we interpret the arc 
as a one-parameter family of lines $\ell_1,\ldots,\ell_6$ in $(\mathbb{P}^2)^\vee$.
Consider the cross-ratio morphisms $\gamma_1,\gamma_2:\,\overline{\mathbf{X}}(3,6)\rightarrow\mathbb{P}^1$ 
given by restricting the lines $\ell_1,\ell_2,\ell_5,\ell_6$ to $\ell_3$ and $\ell_3,\ell_4,\ell_5,\ell_6$ to $\ell_1$. Then $(\gamma_1,\gamma_2)$ induces a local isomorphism near $y$ between $S$ and $\mathbb{P}^1\times\mathbb{P}^1$. 
A simple calculation shows that the limits as $t\to0$ of the derivatives with respect to $t$ of $\gamma_1(\mathbf{a}(t))$ and $\gamma_2(\mathbf{a}(t))$ are equal to $c_3(0)-c_1(0)$ and $c_4(0)-c_2(0)$ respectively.
We claim that the limit as $t\to0$ of the arc $\mathbf{a}(t)$ in $\overline{\mathbf{X}}_{\GP}(3,6)$
is uniquely determined by
the point $[c_1(0)-c_3(0):c_4(0)-c_2(0)]\in\mathbb P^1$. In $\mathbb{P}^2$, consider the cross-ratio of the following four points on the line $x_2=0$: $[1:0:0],[0:0:1],[1:0:1]$, and the intersection of $x_2=0$ with the line spanned by $a_5(t),a_6(t)$, which is given by $[c_1-c_3:(c_4-c_2)+(c_1c_4-c_2c_3)t]$. Denote by $\beta$ its limit for $t\to0$.

This cross-ratio manifests itself in the $6$-pointed degeneration $X$ of $\mathbb{P}^2$ parametrized by the limit $x\in\overline{\mathbf{X}}_{\GP}(3,6)^\nu$ over the (closed) stratum $\#23$ in $\overline{\mathbf{X}}(3,6)$ as follows. Assume that the degenerate surface $X$ parametrized by $x$ is obtained by gluing five copies of $\mathbb{F}_1$, three copies of $\mathbb{P}^1\times\mathbb{P}^1$, and one copy of the blow up of $\mathbb{P}^2$ at two points, and that the limit points are as shown in Figure~\ref{reconstructingdegenerationnumber23}. Denote these surfaces by $F_1,\ldots,F_5,P_1,P_2,P_3,B$ respectively. On $F_1$ (resp. $F_5$) let $f_1,f_2$ (resp. $f_3,f_4$) be the rulings passing through the points marked by $1,2$ (resp. $3,4$). These induce other two rulings $f_1',f_2'$ (resp. $f_3',f_4'$) on $F_2$ (resp. $F_4$). The rulings $f_1',\ldots,f_4'$ intersect $F_3$ in four points $p_1,\ldots,p_4$, respectively. Let $r_5$ (resp. $r_6$) be the ruling in $P_1$ (resp. $P_3$) passing through the point marked by $5$ (resp. $6$) and intersecting $B$ into a point $q_5$ (resp. $q_6$). There is a unique line $\ell$ in $B\cong\Bl_2\mathbb{P}^2$ disjoint from the two exceptional divisors which passes through $q_5$ and $q_6$. This line $\ell$ intersects in a point $p$ the strict transform in $B$ of the line passing through the two points blown up. The point $p$ lies in the exceptional divisor of $F_3$, and it is contained in a unique ruling $f$. Now consider the line $\ell_{13}$ (resp. $\ell_{24}$) in $F_3$ spanned by $p_1$ and $p_3$ (resp. $p_2$ and $p_4$). On the line $\ell_{13}$ we have four distinct points $p_1,p_3,\ell_{13}\cap\ell_{24},\ell_{13}\cap f$, and denote their cross-ratio by $\beta'$ (this construction is reproduced in Figure~\ref{reconstructingdegenerationnumber23}). Note that $F_3$ is the primary component corresponding to the standard lattice $L_0=e_1R+e_2R+e_3R$ with respect to which the arc $\mathbf{a}(t)$ was given above, and $\beta$ was constructed inside $\mathbb{P}(L_0)_\Bbbk\cong\mathbb{P}^2$ in the same way we did for $\beta'$. Therefore, $\beta'=\beta\neq0,\infty$, and it characterizes the isomorphism type of the $6$-pointed degeneration $X$. The extreme cases $\beta=0,\infty$ correspond to the following further degenerations of $X$ in Table~\ref{tbl:nonplanarGPdegenerations}, $\#23.1$. The surface $F_3\subseteq X$ further breaks into the gluing of a surface $P\cong\mathbb{P}^1\times\mathbb{P}^1$, and a surface $T\cong\mathbb{P}^2$ along a ruling of $P$ and a line in $T$. Without loss of generality, assume that $P$ is glued with $F_2$. In this case the line $\ell_{13}$ breaks into two irreducible components, and the analogues of the points $p_1,p_3,\ell_{13}\cap\ell_{24},\ell_{13}\cap f$ split into two groups $\{p_1,\ell_{13}\cap f\},\{p_3,\ell_{13}\cap\ell_{24}\}$, where $p_1,\ell_{13}\cap f$ lie on the surface $P$ and $p_3,\ell_{13}\cap\ell_{24}$ on~$T$. This together with the other degeneration where $P$ is glued to $F_4$ instead, correspond, up to order, to the two cross-ratios $\beta=0,\infty$.

This shows that we have morphisms $\Bl_9S\rightarrow S_{\GP}\rightarrow S$,
where the first morphism is a bijective normalization.
\end{proof}

\begin{figure}[hbtp]
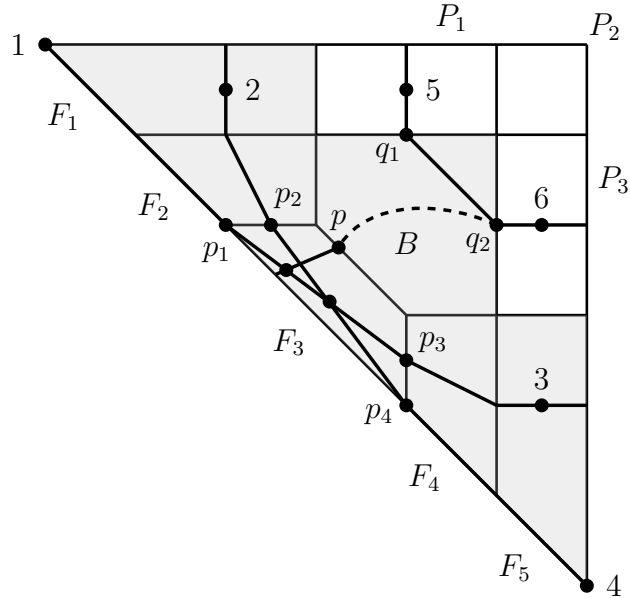


\caption{Cross-ratio $\beta$ in a degeneration of six points in $\mathbb{P}^2$ parametrized by a point in $\overline{\mathbf{X}}_{\GP}(3,6)^\nu$ over the (closed) stratum $\#23$ in $\overline{\mathbf{X}}(3,6)$.}
\label{reconstructingdegenerationnumber23}
\end{figure}

\begin{table}[hbtp]
\centering
\begin{tabular}{|>{\centering\arraybackslash}m{4cm}|>{\centering\arraybackslash}m{4cm}|>{\centering\arraybackslash}m{4cm}|}
\hline
\begin{tikzpicture}[scale=0.5]

	\draw[line width=1pt] (0,6) -- (6,6);
	\draw[line width=1pt] (6,6) -- (6,0);
	\draw[line width=1pt] (6,0) -- (0,6);
	\draw[line width=1pt] (4.5,6) -- (4.5,1.5);
	\draw[line width=1pt] (1.5,4.5) -- (6,4.5);
	\fill[gray!50,nearly transparent] (1.5,4.5) -- (4.5,4.5) -- (4.5,1.5) -- cycle;
	\fill[gray!50,nearly transparent] (1.5,4.5) -- (4.5,4.5) -- (4.5,6) -- (0,6) -- cycle;
	\fill[gray!50,nearly transparent] (4.5,1.5) -- (4.5,4.5) -- (6,4.5) -- (6,0) -- cycle;
	\fill (6,0) circle (5pt);
	\fill (0,6) circle (5pt);
	\fill (5.5,1.5) circle (5pt);
	\fill (5.6,3) circle (5pt);
	\fill (1.5,5.5) circle (5pt);
	\fill (3,5.6) circle (5pt);

	\node at (-.5,6) {$1$};
	\node at (1.7,5) {$2$};
	\node at (3.2,5.1) {$3$};
	\node at (6.5,0) {$4$};
	\node at (5,1.7) {$5$};
	\node at (5.1,3.2) {$6$};

	\node at (1,2) {\#2};

\end{tikzpicture}
&
Figure~\ref{degeneration6pointstype7GPlimit}

$\#7$
&
\begin{tikzpicture}[scale=0.5]

	\draw[line width=1pt] (0,6) -- (6,6);
	\draw[line width=1pt] (6,6) -- (6,0);
	\draw[line width=1pt] (6,0) -- (0,6);
	\draw[line width=1pt] (18/4,6) -- (18/4,6/4);
	\draw[line width=1pt] (12/4,6) -- (12/4,12/4);
	\draw[line width=1pt] (6,6-6/4) -- (6/4,6-6/4);
	\fill[gray!50,nearly transparent] (6/4,6-6/4) -- (6,6-6/4) -- (6,0) -- cycle;
	\fill[gray!50,nearly transparent] (0,6) -- (12/4,6) -- (12/4,6-6/4) -- (6/4,6-6/4) --cycle;
	\fill (6,0) circle (5pt);
	\fill (0,6) circle (5pt);
	\fill (5.5,3.5) circle (5pt);
	\fill (1.8,5.2) circle (5pt);
	\fill (15/4,5.2) circle (5pt);
	\fill (5.5,1.8) circle (5pt);

	\node at (5,3.5) {$4$};
	\node at (2.3,5.5) {$5$};
	\node at (-.5,6) {$6$};
	\node at (4.2,5.5) {$1$};
	\node at (5,1.8) {$2$};
	\node at (6.5,0) {$3$};

	\node at (1,2) {\#10};

\end{tikzpicture}
\\
\hline
\end{tabular}
\caption{Non-planar degenerations parametrized by $\overline{\mathbf{X}}_{\GP}(3,6)$.}
\label{tbl:nonplanarGPdegenerations}
\end{table}

\begin{figure}[hbtp]
\centering
\includegraphics[scale=0.50,valign=t]{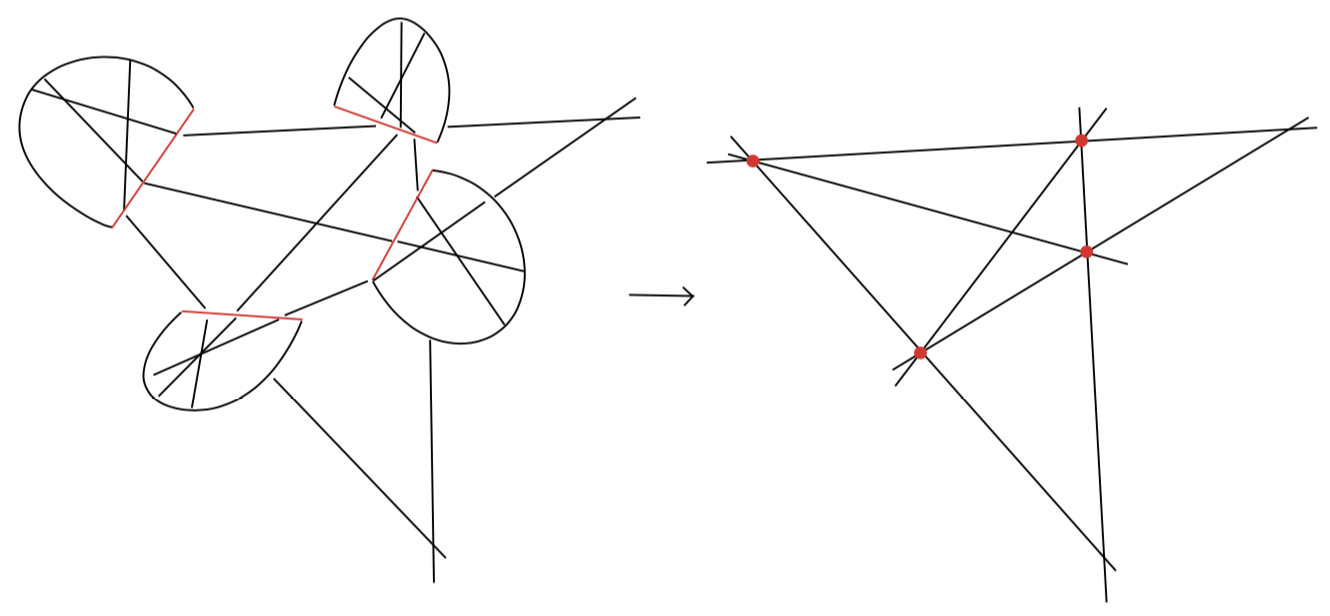}
\caption{(From \cite[Figure~5.11]{Ale15} with permission of the author.)}
\label{Alexeevdegenerations6linesC}
\end{figure}

\begin{table}[hbtp]
\centering
\begin{tabular}{|>{\centering\arraybackslash}m{4cm}|>{\centering\arraybackslash}m{4cm}|>{\centering\arraybackslash}m{4cm}|}
\hline
\begin{tikzpicture}[scale=0.5]

	\draw[line width=1pt] (0,6) -- (6,6);
	\draw[line width=1pt] (6,6) -- (6,0);
	\draw[line width=1pt] (6,0) -- (0,6);
	\draw[line width=1pt] (3,6) -- (6,3);
	\fill[gray!50,nearly transparent] (0,6) -- (6,6) -- (6,0) -- cycle;
	\fill (6,0) circle (5pt);
	\fill (6,6) circle (5pt);
	\fill (0,6) circle (5pt);
	\fill (5.5,5) circle (5pt);
	\fill (4.8,5.2) circle (5pt);
	\fill (4,3.5) circle (5pt);

	\node at (6.5,6) {$1$};
	\node at (5.5,4.4) {$2$};
	\node at (4.3,5.6) {$3$};
	\node at (-.5,6) {$4$};
	\node at (3.5,3.5) {$5$};
	\node at (6.5,0) {$6$};

	\node at (1,2) {\#1};

\end{tikzpicture}
&
\begin{tikzpicture}[scale=0.5]

	\draw[line width=1pt] (0,6) -- (6,6);
	\draw[line width=1pt] (6,6) -- (6,0);
	\draw[line width=1pt] (6,0) -- (0,6);
	\draw[line width=1pt] (4.5,6) -- (4.5,1.5);
	\draw[line width=1pt] (1.5,4.5) -- (6,4.5);
	\fill[gray!50,nearly transparent] (1.5,4.5) -- (4.5,4.5) -- (4.5,1.5) -- cycle;
	\fill[gray!50,nearly transparent] (1.5,4.5) -- (4.5,4.5) -- (4.5,6) -- (0,6) -- cycle;
	\fill[gray!50,nearly transparent] (4.5,1.5) -- (4.5,4.5) -- (6,4.5) -- (6,0) -- cycle;
	\fill (6,0) circle (5pt);
	\fill (0,6) circle (5pt);
	\fill (6,6) circle (5pt);
	\fill (5.6,3) circle (5pt);
	\fill (3,5.6) circle (5pt);
	\fill (3.5,3.5) circle (5pt);

	\node at (-.5,6) {$1$};
	\node at (3.2,5.1) {$2$};
	\node at (3.8,3) {$3$};
	\node at (6.5,6) {$4$};
	\node at (5.1,3.2) {$5$};
	\node at (6.5,0) {$6$};

	\node at (1,2) {\#3};

\end{tikzpicture}
&
\begin{tikzpicture}[scale=0.5]

	\draw[line width=1pt] (0,6) -- (6,6);
	\draw[line width=1pt] (6,6) -- (6,0);
	\draw[line width=1pt] (6,0) -- (0,6);
	\draw[line width=1pt] (6/4,6-6/4) -- (6,6-6/4);
	\draw[line width=1pt] (6-6/4,6/4) -- (6-6/4,6);
	\draw[line width=1pt] (6/4,6) -- (6,6/4);
	\fill[gray!50,nearly transparent] (6/4,6) -- (6-6/4,6) -- (6-6/4,6-6/4) -- (6-12/4,6-6/4) -- cycle;
	\fill[gray!50,nearly transparent] (6,6/4) -- (6,6-6/4) -- (6-6/4,6-6/4) -- (6-6/4,6-12/4) -- cycle;
	\fill[gray!50,nearly transparent] (6/4,6-6/4) -- (12/4,6-6/4) -- (18/4,6-12/4) -- (18/4,6/4) -- cycle;
	\fill[gray!50,nearly transparent] (12/4,6-6/4) -- (18/4,6-6/4) -- (18/4,6-12/4) -- cycle;
	\fill (6,0) circle (5pt);
	\fill (0,6) circle (5pt);
	\fill (6,6) circle (5pt);
	\fill (3.3,5.3) circle (5pt);
	\fill (5.3,3.3) circle (5pt);
	\fill (3.3,3.3) circle (5pt);

	\node at (-.5,6) {$1$};
	\node at (3.8,5.3) {$2$};
	\node at (3.8,3.1) {$3$};
	\node at (6.5,6) {$4$};
	\node at (6.5,0) {$5$};
	\node at (5.3,3.9) {$6$};

	\node at (1,2) {\#4};

\end{tikzpicture}
\\
\hline
\end{tabular}
\caption{Planar degenerations parametrized by $\overline{\mathbf{X}}_{\GP}(3,6)$.}
\label{tbl:GPdegenerationspart1}
\end{table}

\begin{table}[hbtp]
\centering
\begin{tabular}{|>{\centering\arraybackslash}m{4cm}|>{\centering\arraybackslash}m{4cm}|>{\centering\arraybackslash}m{4cm}|}
\hline
\begin{tikzpicture}[scale=0.5]

	\draw[line width=1pt] (0,6) -- (6,6);
	\draw[line width=1pt] (6,6) -- (6,0);
	\draw[line width=1pt] (6,0) -- (0,6);
	\draw[line width=1pt] (12/4,6) -- (12/4,6-6/4);
	\draw[line width=1pt] (12/4,6-6/4) -- (18/4,6-12/4);
	\draw[line width=1pt] (18/4,6-12/4) -- (6,6-12/4);
	\draw[line width=1pt] (6-12/4,6-6/4) -- (6/4,6-6/4);
	\draw[line width=1pt] (6-6/4,3) -- (6,6/4);
	\draw[line width=1pt] (6-6/4,3) -- (3,3);
	\fill[gray!50,nearly transparent] (0,6) -- (3,6) -- (3,6-6/4) -- (6/4,6-6/4) -- cycle;
	\fill[gray!50,nearly transparent] (3,3) -- (3+6/4,3) -- (6,3-6/4) -- (6,0) -- cycle;
	\fill[gray!50,nearly transparent] (3+6/4,3) -- (6,3) -- (6,3-6/4) -- cycle;
	\fill[gray!50,nearly transparent] (3,6) -- (6,6) -- (6,3) -- (6-6/4,3) -- (3,6-6/4) -- cycle;
	\fill (6,0) circle (5pt);
	\fill (0,6) circle (5pt);
	\fill (6,6) circle (5pt);
	\fill (2,5.2) circle (5pt);
	\fill (5,5) circle (5pt);
	\fill (5,1.7) circle (5pt);

	\node at (-.5,6) {$1$};
	\node at (2.5,5.5) {$2$};
	\node at (6.5,6) {$3$};
	\node at (4.5,4.5) {$4$};
	\node at (6.5,0) {$5$};
	\node at (4.6,2.2) {$6$};

	\node at (1,2) {\#13};

\end{tikzpicture}
&
\begin{tikzpicture}[scale=0.5]

	\draw[line width=1pt] (0,6) -- (6,6);
	\draw[line width=1pt] (6,6) -- (6,0);
	\draw[line width=1pt] (6,0) -- (0,6);
	\draw[line width=1pt] (6/4,6) -- (6,6/4);
	\draw[line width=1pt] (3,6) -- (18/4,6-6/4);
	\draw[line width=1pt] (18/4,6-6/4) -- (6,6-6/4);
	\draw[line width=1pt] (18/4,6-6/4) -- (18/4,6/4);
	\fill[gray!50,nearly transparent] (0,6) -- (3,6) -- (18/4,6-6/4) -- (18/4,6/4) -- cycle;
	\fill[gray!50,nearly transparent] (3,6) -- (6,6) -- (6,18/4) -- (6-6/4,18/4) -- cycle;
	\fill[gray!50,nearly transparent] (18/4,6-6/4) -- (6,6-6/4) -- (6,6/4) -- (6-6/4,3) -- cycle;
	\fill (6,0) circle (5pt);
	\fill (0,6) circle (5pt);
	\fill (6,6) circle (5pt);
	\fill (5,5.3) circle (5pt);
	\fill (5.3,3.5) circle (5pt);
	\fill (2.5,4.2) circle (5pt);

	\node at (6.5,6) {$1$};
	\node at (4.5,5.4) {$2$};
	\node at (6.5,0) {$3$};
	\node at (5.7,3.2) {$4$};
	\node at (-.5,6) {$5$};
	\node at (3,3.9) {$6$};

	\node at (1,2) {\#14};

\end{tikzpicture}
&
\begin{tikzpicture}[scale=0.5]

	\draw[line width=1pt] (0,6) -- (6,6);
	\draw[line width=1pt] (6,6) -- (6,0);
	\draw[line width=1pt] (6,0) -- (0,6);
	\draw[line width=1pt] (6/5,6-6/5) -- (18/5,6-6/5);
	\draw[line width=1pt] (18/5,6-6/5) -- (18/5,6);
	\draw[line width=1pt] (12/5,6-6/5) -- (12/5,6);
	\draw[line width=1pt] (18/5,6-6/5) -- (24/5,6-12/5);
	\draw[line width=1pt] (24/5,6-12/5) -- (6,6-12/5);
	\draw[line width=1pt] (24/5,6-12/5) -- (6,6-18/5);
	\draw[line width=1pt] (24/5,6-12/5) -- (18/5,6-12/5);
	\draw[line width=1pt] (18/5,6-12/5) -- (12/5,6-6/5);
	\draw[line width=1pt] (18/5,6-12/5) -- (18/5,6-18/5);
	\fill[gray!50,nearly transparent] (18/5,6) -- (6,6) -- (6,6-18/5) -- (18/5,6-6/5) -- cycle;
	\fill[gray!50,nearly transparent] (0,6) -- (12/5,6) -- (12/5,6-6/5) -- (6/5,6-6/5) -- cycle;
	\fill[gray!50,nearly transparent] (6/5,6-6/5) -- (12/5,6-6/5) -- (18/5,6-12/5) -- (18/5,6-18/5) -- cycle;
	\fill[gray!50,nearly transparent] (18/5,12/5) -- (18/5,18/5) -- (24/5,18/5) -- (6,12/5) -- (6,0) -- cycle;
	\fill (6,0) circle (5pt);
	\fill (0,6) circle (5pt);
	\fill (6,6) circle (5pt);
	\fill (5,2) circle (5pt);
	\fill (5,5) circle (5pt);
	\fill (1.5,5.4) circle (5pt);

	\node at (6.5,0) {$1$};
	\node at (-.5,6) {$5$};
	\node at (6.5,6) {$3$};
	\node at (5.1,2.6) {$2$};
	\node at (4.5,5) {$4$};
	\node at (2,5.5) {$6$};

	\node at (1,2) {\#15};

\end{tikzpicture}
\\
\hline
\end{tabular}
\caption{Planar degenerations parametrized by $\overline{\mathbf{X}}_{\GP}(3,6)$, continued.}
\label{tbl:GPdegenerationspart2}
\end{table}

\begin{table}[hbtp]
\centering
\begin{tabular}{|>{\centering\arraybackslash}m{4cm}|>{\centering\arraybackslash}m{4cm}|>{\centering\arraybackslash}m{4cm}|}
\hline
\begin{tikzpicture}[scale=0.5]

	\draw[line width=1pt] (0,6) -- (6,6);
	\draw[line width=1pt] (6,6) -- (6,0);
	\draw[line width=1pt] (6,0) -- (0,6);
	\draw[line width=1pt] (12/4,6) -- (12/4,6-6/4);
	\draw[line width=1pt] (12/4,6-6/4) -- (18/4,6-12/4);
	\draw[line width=1pt] (18/4,6-12/4) -- (6,6-12/4);
	\draw[line width=1pt] (6-6/4,6-12/4) -- (6-6/4,6/4);
	\draw[line width=1pt] (6-12/4,6-6/4) -- (6/4,6-6/4);
	\draw[line width=1pt] (3,3) -- (3+6/4,3);
	\fill[gray!50,nearly transparent] (0,6) -- (12/4,6) -- (12/4,6-6/4) -- (6/4,6-6/4) -- cycle;
	\fill[gray!50,nearly transparent] (6,0) -- (6,12/4) -- (3,12/4) -- cycle;
	\fill[gray!50,nearly transparent] (12/4,6) -- (6,6) -- (6,3) -- (6-6/4,3) -- (12/4,6-6/4) -- cycle;
	\fill (6,0) circle (5pt);
	\fill (0,6) circle (5pt);
	\fill (6,6) circle (5pt);
	\fill (2,5.2) circle (5pt);
	\fill (5,5) circle (5pt);
	\fill (5.2,2) circle (5pt);

	\node at (6.5,6) {$1$};
	\node at (4.5,4.5) {$2$};
	\node at (-.5,6) {$3$};
	\node at (2.5,5.5) {$4$};
	\node at (6.5,0) {$5$};
	\node at (5.5,2.5) {$6$};

	\node at (1,2) {\#12.1};

\end{tikzpicture}
&
\begin{tikzpicture}[scale=0.5]

	\draw[line width=1pt] (0,6) -- (6,6);
	\draw[line width=1pt] (6,6) -- (6,0);
	\draw[line width=1pt] (6,0) -- (0,6);
	\draw[line width=1pt] (6/5,6-6/5) -- (18/5,6-6/5);
	\draw[line width=1pt] (18/5,6-6/5) -- (18/5,6);
	\draw[line width=1pt] (12/5,6-6/5) -- (12/5,6);
	\draw[line width=1pt] (18/5,6-6/5) -- (24/5,6-12/5);
	\draw[line width=1pt] (24/5,6-12/5) -- (6,6-12/5);
	\draw[line width=1pt] (24/5,6-12/5) -- (6,6-18/5);
	\draw[line width=1pt] (24/5,6-12/5) -- (18/5,6-12/5);
	\draw[line width=1pt] (18/5,6-12/5) -- (12/5,6-6/5);
	\draw[line width=1pt] (18/5,6-12/5) -- (18/5,6-18/5);
	\draw[line width=1pt] (12/5,6-12/5) -- (18/5,6-12/5);
	\fill[gray!50,nearly transparent] (18/5,6) -- (6,6) -- (6,6-18/5) -- (18/5,6-6/5) -- cycle;
	\fill[gray!50,nearly transparent] (0,6) -- (12/5,6) -- (12/5,6-6/5) -- (6/5,6-6/5) -- cycle;
	\fill[gray!50,nearly transparent] (12/5,6-12/5) -- (18/5,6-12/5) -- (18/5,6-18/5) -- cycle;
	\fill[gray!50,nearly transparent] (18/5,12/5) -- (18/5,18/5) -- (24/5,18/5) -- (6,12/5) -- (6,0) -- cycle;
	\fill (6,0) circle (5pt);
	\fill (0,6) circle (5pt);
	\fill (6,6) circle (5pt);
	\fill (5,2) circle (5pt);
	\fill (5,5) circle (5pt);
	\fill (1.5,5.4) circle (5pt);

	\node at (6.5,0) {$1$};
	\node at (-.5,6) {$5$};
	\node at (6.5,6) {$3$};
	\node at (5.1,2.6) {$2$};
	\node at (4.5,5) {$4$};
	\node at (2,5.5) {$6$};

	\node at (1,2) {\#15.1};

\end{tikzpicture}
&
\begin{tikzpicture}[scale=0.5]

	\draw[line width=1pt] (0,6) -- (6,6);
	\draw[line width=1pt] (6,6) -- (6,0);
	\draw[line width=1pt] (6,0) -- (0,6);
	\draw[line width=1pt] (6/5,6-6/5) -- (18/5,6-6/5);
	\draw[line width=1pt] (18/5,6-6/5) -- (18/5,6);
	\draw[line width=1pt] (12/5,6-6/5) -- (12/5,6);
	\draw[line width=1pt] (18/5,6-6/5) -- (24/5,6-12/5);
	\draw[line width=1pt] (24/5,6-12/5) -- (6,6-12/5);
	\draw[line width=1pt] (24/5,6-12/5) -- (6,6-18/5);
	\draw[line width=1pt] (24/5,6-12/5) -- (18/5,6-12/5);
	\draw[line width=1pt] (18/5,6-12/5) -- (12/5,6-6/5);
	\draw[line width=1pt] (18/5,6-12/5) -- (18/5,6-18/5);
	\draw[line width=1pt] (12/5,6-12/5) -- (12/5,6-6/5);
	\fill[gray!50,nearly transparent] (18/5,6) -- (6,6) -- (6,6-18/5) -- (18/5,6-6/5) -- cycle;
	\fill[gray!50,nearly transparent] (0,6) -- (12/5,6) -- (12/5,6-6/5) -- (6/5,6-6/5) -- cycle;
	\fill[gray!50,nearly transparent] (12/5,6-12/5) -- (12/5,6-6/5) -- (6/5,6-6/5) -- cycle;
	\fill[gray!50,nearly transparent] (18/5,12/5) -- (18/5,18/5) -- (24/5,18/5) -- (6,12/5) -- (6,0) -- cycle;
	\fill (6,0) circle (5pt);
	\fill (0,6) circle (5pt);
	\fill (6,6) circle (5pt);
	\fill (5,2) circle (5pt);
	\fill (5,5) circle (5pt);
	\fill (1.5,5.4) circle (5pt);

	\node at (6.5,0) {$1$};
	\node at (-.5,6) {$5$};
	\node at (6.5,6) {$3$};
	\node at (5.1,2.6) {$2$};
	\node at (4.5,5) {$4$};
	\node at (2,5.5) {$6$};

	\node at (1,2) {\#15.2};

\end{tikzpicture}
\\
\hline
\end{tabular}
\caption{Further planar degenerations that come from the stretches in Table~\ref{stretchesofspiders}.}
\label{tbl:furtherGPdegenerations}
\end{table}

\newpage

\begin{figure}[hbtp]
\centering
\includegraphics[scale=0.60,valign=t]{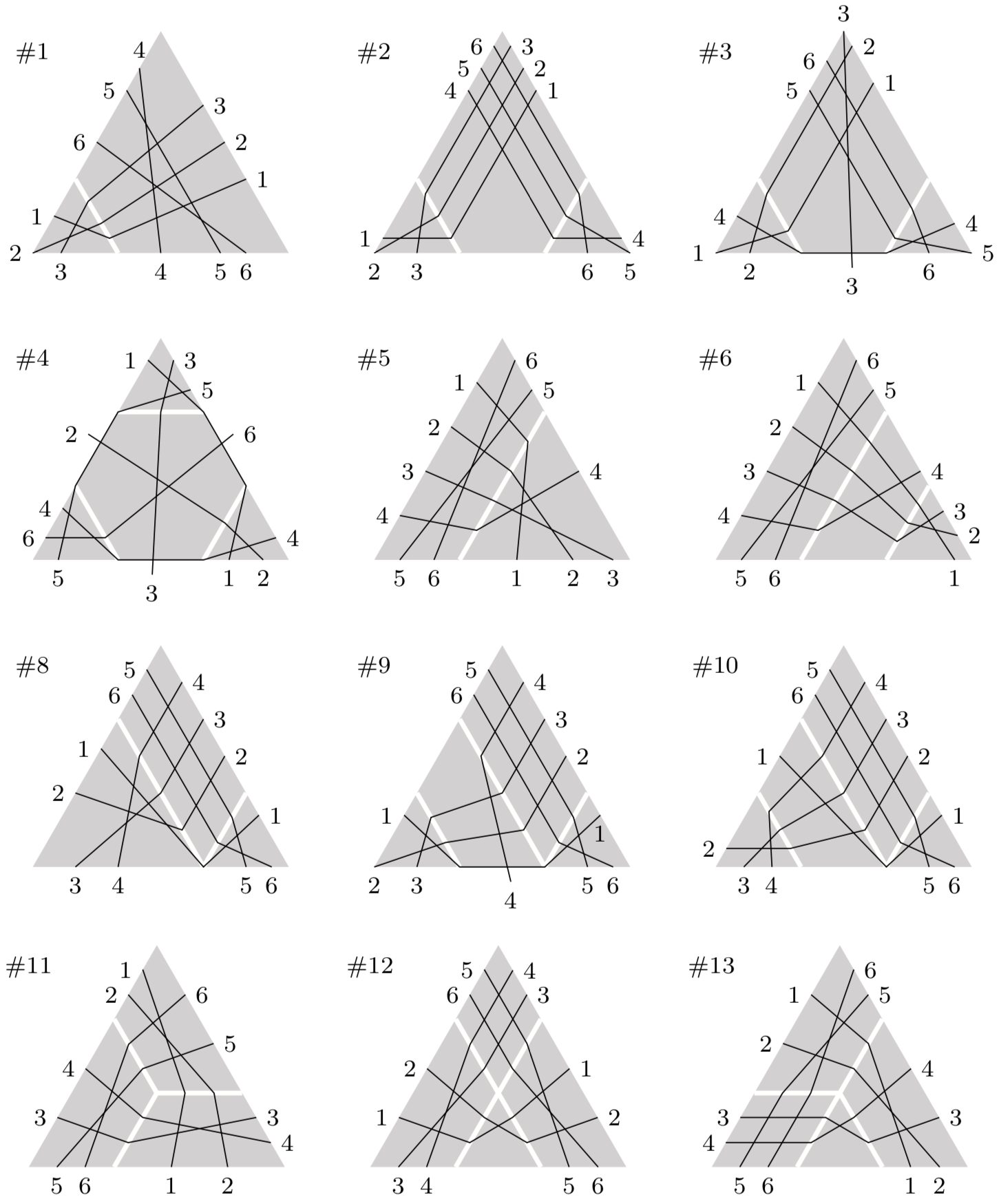}
\caption{(From \cite[Figure~5.12]{Ale15} with permission of the author.)}
\label{Alexeevdegenerations6linesA}
\end{figure}

\begin{figure}[hbtp]
\centering
\includegraphics[scale=0.60,valign=t]{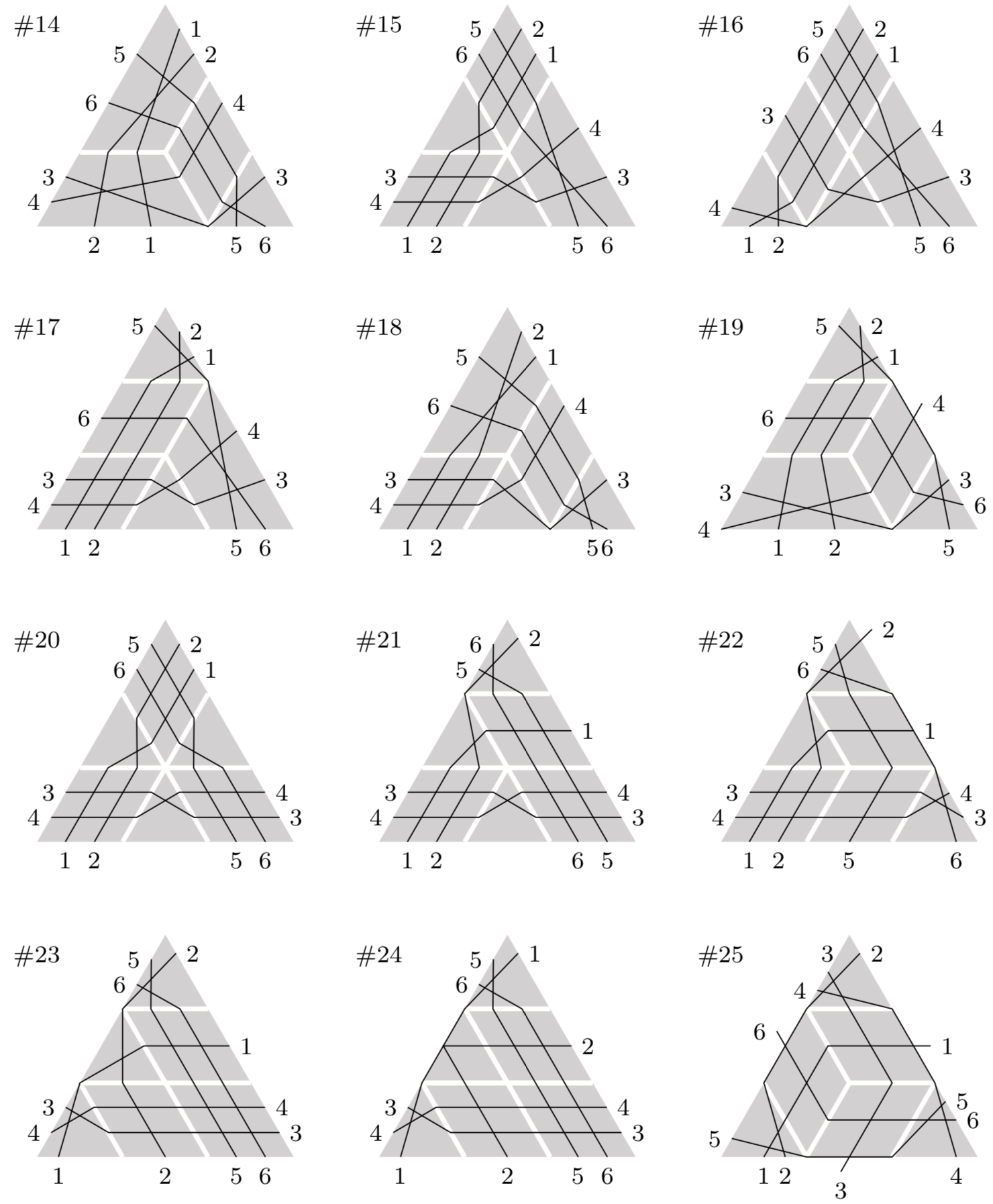}
\caption{(From \cite[Figure~5.13]{Ale15} with permission of the author.)}
\label{Alexeevdegenerations6linesB}
\end{figure}



\end{document}